\pdfoutput=1
\documentclass[10pt,reqno]{amsart}

\usepackage{graphicx}
\usepackage[margin=1in]{geometry}
\usepackage{bbm}
\usepackage{amsfonts}
\usepackage{latexsym,amssymb,amsmath,amscd,amsthm,amsxtra}
\usepackage{mathtools}
\usepackage{enumerate}
\usepackage{enumitem}
\usepackage{mathrsfs}
\usepackage{fancyhdr}
\usepackage{microtype}
\usepackage{xcolor}
\usepackage{stmaryrd}
\SetSymbolFont{stmry}{bold}{U}{stmry}{m}{n}
\usepackage{comment}
\usepackage[colorlinks=true]{hyperref}
\usepackage[capitalise,nameinlink]{cleveref}
\usepackage{tikz}
\usepackage{cite}

\usetikzlibrary{calc}

\allowdisplaybreaks
\numberwithin{equation}{section}

\newtheorem{theorem}{Theorem}[section]
\newtheorem{proposition}[theorem]{Proposition}
\newtheorem{lemma}[theorem]{Lemma}

\crefname{claim}{claim}{claims}

\theoremstyle{definition}
\newtheorem{definition}[theorem]{Definition}
\theoremstyle{remark}

\newtheorem{remark}[theorem]{Remark}

\newcommand{\C}{\mathbb C}
\newcommand{\R}{\mathbb R}
\newcommand{\N}{\mathbb N}
\newcommand{\Z}{\mathbb Z}
\newcommand{\E}{\mathbb E}
\newcommand{\Pp}{\mathbb P}
\newcommand{\1}{\mathbf 1}
\newcommand{\Tr}{\operatorname{Tr}}
\newcommand{\tr}{\operatorname{Tr}}
\newcommand{\diag}{\operatorname{diag}}

\newcommand{\supp}{\operatorname{supp}}
\newcommand{\Id}{\mathrm I}
\newcommand{\logpara}{\mathfrak L}

\newcommand{\difloop}{\mathcal D}

\newcommand{\al}{\alpha}
\newcommand{\cA}{\mathcal A}
\newcommand{\cB}{\mathcal B}

\newcommand{\cD}{\mathcal D}
\newcommand{\cE}{\mathcal E}

\newcommand{\cI}{\mathcal I}

\newcommand{\cP}{\mathcal P}
\newcommand{\cQ}{\mathcal Q}

\newcommand{\cS}{\mathcal S}

\newcommand{\cV}{\mathcal V}
\newcommand{\cW}{\mathcal W}
\newcommand{\cZ}{\mathcal Z}

\newcommand{\bS}{\mathbf S}
\newcommand{\sD}{\mathsf D}
\newcommand{\sE}{\mathsf E}

\newcommand{\Tsym}{{\mathsf L}}
\newcommand{\Qsym}{{\mathsf Q}}
\newcommand{\HS}{\mathrm{HS}}
\newcommand{\op}{\mathrm{op}}
\newcommand{\dd}{\,\mathrm d}
\newcommand{\ii}{\mathrm i}
\newcommand{\OO}{\mathrm O}
\newcommand{\oo}{\mathrm o}

\newcommand{\aux}{\mathrm{aux}}

\newcommand{\ee}{e}

\newcommand{\phy}{\nu}
\newcommand{\be}{\begin{equation}}
\newcommand{\qqq}[1]{\llbracket{#1}\rrbracket}
\newcommand{\sA}{{\mathscr{A}}}
\newcommand{\sB}{{\mathsf{B}}}

\newcommand{\sL}{{\mathsf{L}}}

\newcommand{\sQ}{{\mathsf{Q}}}
\newcommand{\Zn}{{\mathbb Z}_L^2}
\newcommand{\ZL}{{\mathbb Z}_{WL}^2}
\newcommand{\qll}[1]{[\![{#1}]\!]}

\newcommand{\rep}[1]{({#1})}
\newcommand{\enu}{\mathscr{E}_{\mathsf A_u}}

\DeclareMathOperator{\var}{Var}
\newcommand{\cal}{\mathcal}
\renewcommand{\bar}{\overline }

\DeclareMathOperator{\re}{{\mathrm{Re}}}
\DeclareMathOperator{\im}{{\mathrm{Im}}}
\newcommand{\e}{{\varepsilon}}
\newcommand{\cK}{\mathcal K}
\newcommand{\cL}{\mathcal L}
\newcommand{\sig}{\sigma}
\newcommand{\bsig}{{\boldsymbol{\sigma}}}

\newcommand{\ba}{{\mathbf{a}}}

\newcommand{\LK}{{(\mathsf{B})}}
\newcommand{\Gc}{{\mathring G}}

\newcommand{\wh}{\widehat}
\newcommand{\fn}{{n}}
\newcommand{\fc}{\mathfrak{c}}

\newcommand{\fb}{{\mathfrak b}}
\newcommand{\fs}{{\mathfrak s}}

\newcommand{\wt}{\widetilde}

\newcommand{\p}[1]{({#1})}
\newcommand{\pb}[1]{\bigl({#1}\bigr)}
\newcommand{\pB}[1]{\Bigl({#1}\Bigr)}
\newcommand{\pbb}[1]{\biggl({#1}\biggr)}

\newcommand{\pa}[1]{\left({#1}\right)}

\newcommand{\q}[1]{[{#1}]}
\newcommand{\qb}[1]{\bigl[{#1}\bigr]}
\newcommand{\qB}[1]{\Bigl[{#1}\Bigr]}
\newcommand{\qbb}[1]{\biggl[{#1}\biggr]}
\newcommand{\qBB}[1]{\Biggl[{#1}\Biggr]}
\newcommand{\qa}[1]{\left[{#1}\right]}

\newcommand{\h}[1]{\{{#1}\}}
\newcommand{\hb}[1]{\bigl\{{#1}\bigr\}}
\newcommand{\hB}[1]{\Bigl\{{#1}\Bigr\}}

\newcommand{\ha}[1]{\left\{{#1}\right\}}

\newcommand{\abs}[1]{\lvert #1 \rvert}
\newcommand{\absb}[1]{\bigl\lvert #1 \bigr\rvert}
\newcommand{\absB}[1]{\Bigl\lvert #1 \Bigr\rvert}
\newcommand{\absbb}[1]{\biggl\lvert #1 \biggr\rvert}

\newcommand{\absa}[1]{\left\lvert #1 \right\rvert}

\newcommand{\normb}[1]{\bigl\lVert #1 \bigr\rVert}
\newcommand{\normB}[1]{\Bigl\lVert #1 \Bigr\rVert}
\newcommand{\normbb}[1]{\biggl\lVert #1 \biggr\rVert}

\newcommand{\norma}[1]{\left\lVert #1 \right\rVert}

\newcommand{\avgb}[1]{\big\langle #1 \big\rangle}
\newcommand{\avgB}[1]{\Bigl\langle #1 \Bigr\rangle}

\title{Localization Lengths for Two-Dimensional Random Band Matrices: Exponential Lower Bounds}
\author{Xujie Lai$^\star$}
\author{Fan Yang$^\dagger$}

\thanks{$^\star$Qiuzhen College, Tsinghua University, Beijing, China, \href{mailto:laixj21@mails.tsinghua.edu.cn}{laixj21@mails.tsinghua.edu.cn}}
\thanks{$^\dagger$Yau Mathematical Sciences Center, Tsinghua University, and Beijing Institute of Mathematical Sciences and Applications, Beijing, China, \href{mailto:fyangmath@mail.tsinghua.edu.cn}{fyangmath@mail.tsinghua.edu.cn}}

\begin{document}

\begin{abstract}

We consider $N \times N$ complex Hermitian random band matrices $H = (H_{xy})$ with centered Gaussian entries, indexed by points $x,y$ on the two-dimensional discrete torus $(\mathbb{Z} / \sqrt{N} \mathbb{Z})^2$. The matrix entries $H_{xy}$ vanish whenever the distance between $x$ and $y$ exceeds a constant multiple of the bandwidth parameter $W$. We prove that, for any fixed \(\varepsilon>0\), if \(W\geq(\log N)^{1+\varepsilon}\), then all bulk eigenvectors are delocalized with high probability. Equivalently, this yields a lower bound of order \(\exp(W^{1-\varepsilon})\) on the localization length of two-dimensional random band matrices for arbitrarily small fixed \(\varepsilon>0\), improving the superpolynomial lower bound $W^C$ for every fixed constant $C>0$ established in \cite{DYYY25}.
A key ingredient of our proof is a new strategy that closes the resolvent flow at the level of one- and two-loop observables by combining operator and Hilbert--Schmidt norm bounds for resolvent blocks with Schatten norm interpolation. This avoids the analysis of higher-order loop observables and their deterministic limits, known as primitive loops.

\end{abstract}

\maketitle

\tableofcontents

\section{Introduction}\label{sec:main_results}

The localization--delocalization phenomenon has been a central topic in mathematical physics since Anderson's seminal work \cite{Anderson}, which introduced a tight-binding model for electron transport in disordered materials. A standard mathematical formulation of this model is the random Schr\"odinger operator (RSO) on the integer lattice $\mathbb{Z}^{d}$, \(H=-\Delta+\lambda V,\) where $\Delta$ is the graph Laplacian on $\mathbb{Z}^{d}$, $V$ is a disordered potential consisting of i.i.d.~random entries, and $\lambda>0$ measures the disorder strength. 
A fundamental question is whether the eigenstates of this operator are localized or delocalized and how their behavior depends on the disorder strength. A central physical quantity describing localization is the \emph{localization length} $\ell$. Roughly speaking, an eigenstate localized near a center $x_0$ decays on the scale $\ell$, for instance as $\exp(-|x-x_0|/\ell)$, whereas an extended (or delocalized) state spreads throughout the system and formally corresponds to $\ell=\infty$. The one-parameter scaling theory \cite{PRL_Anderson} predicts a strongly dimension-dependent picture. In dimensions $d=1,2$, arbitrarily weak disorder is expected to localize all states, and standard weak-disorder heuristics further predict that, as $\lambda\to0$, the localization length scales as $\ell\asymp\lambda^{-2}$ in $d=1$ and $\log\ell\asymp\lambda^{-2}$ in $d=2$. On the other hand, in dimensions $d\ge3$, an Anderson transition is expected: states near the spectral edges or at sufficiently strong disorder are localized, whereas bulk states at weak disorder are extended. For a detailed overview of the theory of RSOs and further background, we refer the reader to \cite{mott1961theory,ishii1973localization,Borland1963TheNature,evers2008anderson,Aizenman_book,Kirsch2007}. 

Compared with the well-developed theory of localization for one-dimensional (1D) RSOs (see, for example, \cite{GMP, KunzSou, Carmona1982_Duke, Damanik2002}), our understanding of localization and delocalization in dimensions $d\ge 2$ remains far less complete. The first rigorous result on multidimensional localization was obtained by Fr\"ohlich and Spencer \cite{FroSpen_1983,FroSpen_1985}, who proved localization at sufficiently strong disorder or sufficiently low energies using multiscale analysis. Aizenman and Molchanov \cite{Aizenman1993} subsequently developed an alternative approach based on the fractional moment method. Despite substantial progress in the theory of Anderson localization (see, for example, \cite{Carmona1987, SimonWolff, Aizenman1994, ASFH2001, Bourgain2005, Germinet2013, DingSmart2020, LiZhang2019, LiLiuZhang2026}, among many others), rigorous localization results in dimensions $d\ge2$ were, until recently, largely confined to sufficiently strong disorder or energies near the spectral edges.
Recently, after the second version of this paper was posted to arXiv, OpenAI released a preprint \cite{OpenAI2026PurePoint} claiming to resolve the pure-point spectrum conjecture for the two-dimensional (2D) Anderson model at every positive disorder strength. The exponential decay of fractional moments of the finite-volume resolvent established in Theorem 10.1 therein also implies localization of the eigenvectors. However, the preprint does not provide a quantitative upper bound on their localization length in the weak-disorder regime. Polynomial lower bounds in this regime have been established \cite{SchlagShubinWolff2002, Chen2005}, but remain far below the predicted exponential scale $\log\ell\asymp\lambda^{-2}$ in $d=2$. By contrast, in dimensions $d\ge3$, the bulk spectrum at weak disorder is expected to be purely absolutely continuous, with extended states. Substantial progress has been made toward this picture, including results on delocalization and absolutely continuous spectrum on the Bethe lattice \cite{Klein1994, Klein1998, ASW2006, aizenman2013resonant, Bethe_PRL, Bethe-Anderson, JLW2026,DS2025} and on quantum diffusion for the Anderson model \cite{ESY2008,Black:2025CMP,BDH2025}, among many others. More recently, OpenAI released a preprint \cite{OpenAI2026AbsolutelyContinuous} claiming to establish an interval of purely absolutely continuous spectrum for the Anderson model in every dimension $d\ge3$ at sufficiently weak disorder.
A quantitative description of the corresponding extended states, however, remains open.

Motivated by the delocalization problem and the search for sharper estimates on the localization length in dimensions $d\ge 2$, random band matrices (RBMs), which have been extensively studied in the physics literature \cite{PhysRevLett.64.1851, PhysRevLett.64.5, PhysRevLett.66.986}, have emerged as a prominent alternative to RSOs. They interpolate naturally between short-range RSOs and the celebrated mean-field Wigner matrices \cite{Wigner}, thereby opening the possibility of applying tools from random matrix theory to spatially structured systems. 
Roughly speaking, a $d$-dimensional random band matrix $H=(H_{xy})$ with bandwidth $W$ and system size $WL$ is a random Hermitian matrix indexed by the discrete torus \smash{$(\mathbb{Z}/(WL)\mathbb{Z})^d$}. Up to Hermitian symmetry, its entries are independent centered random variables whose variance profile $S_{xy}=\mathbb{E}|H_{xy}|^2$ is concentrated on the scale $W$. More precisely, $S_{xy}$ is typically of order $W^{-d}$ when the distance between $x$ and $y$ is at most of order $W$ and decays rapidly beyond this scale, with the normalization \(\sum_y S_{xy}\equiv 1.\)
Nonrigorous calculations based on a nonlinear $\sigma$-model \cite{PhysRevLett.67.2405}, together with numerical evidence, predict that RBMs exhibit localization--delocalization behavior analogous to that of RSOs under the heuristic correspondence $W\sim\lambda^{-1}$; see \cite{PB_review, Spencer2, Spencer3, Spencer1} for further discussion of this correspondence. Here, we provide only a brief overview of the related conjectures.

Under the correspondence $W\sim\lambda^{-1}$, the conjectures for RSOs lead to the following predictions for RBMs in the bulk of the limiting spectral density, which is given by Wigner's semicircle law on $[-2,2]$: in $d=1$, the localization length is expected to scale as $\ell\asymp W^2$; in $d=2$, $\log\ell\asymp W^2$; and in $d\ge3$, a localization--delocalization transition is expected to occur at a critical bandwidth of order 1. Equivalently, since the localization length cannot exceed the linear system size, this leads to the following conjecture. Let $N=(WL)^d$ be the dimension of $H$. For energies in $[-2+\kappa,2-\kappa]$, where $\kappa>0$ is fixed, there exists a critical bandwidth $W_c\equiv W_c(d,N)$ separating the localized regime $W\ll W_c$ from the delocalized regime $W\gg W_c$, with $W_c = \sqrt N$ for $d=1$, $W_c = \sqrt {\log N}$ for $d=2$, and $W_c=\OO(1)$ in $d\ge 3$.
A corresponding transition in the local spectral statistics, from Poisson statistics in the localized regime to GOE/GUE statistics in the delocalized regime, is also conjectured to occur at the critical bandwidth $W_c(d,N)$.

The localization–delocalization transition for 1D RBMs is now fully understood, both within and outside the critical window. On the delocalized side, a sequence of works established delocalization and related local laws, quantum diffusion, and bulk universality under conditions of the form $W\ge N^a$ for various exponents $a>1/2$ \cite{EK_band1, ErdKno2011, erdHos2013local, BaoErd2015, HeMa2018, bourgade2017universality, bourgade2019random, bourgade2020random, yang2021random, DY}. 
Yau and Yin \cite{YY_25} subsequently attained the conjecturally optimal exponent, proving delocalization and universality of bulk eigenvalue statistics under  $W\ge N^{1/2+\varepsilon}$, for every fixed $\varepsilon>0$, through a dynamical analysis of the \emph{loop hierarchy}.
Using a zigzag strategy, Erd\H{o}s and Riabov \cite{erdHos2025zigzag} extended these results to 1D RBMs with general variance profiles and arbitrary entry distributions in both the real-symmetric and complex-Hermitian symmetry classes.
On the localized side, localization was previously established under conditions of the form $W\ll N^a$ for various exponents $a<1/2$ \cite{Sch2009, PelSchShaSod, Cipolloni2024, Chen2022}.
More recently, exponential localization on the scale $W^2$ was proved throughout the conjectured localized regime $W^2\ll N$ in \cite{Localization1_2}. Building on this result and the bulk universality results in \cite{YY_25,erdHos2025zigzag}, Poisson eigenvalue statistics were subsequently established in \cite{Poisson_Anderson}.
In subsequent AI-assisted work, Ghosal and Remy \cite{GhosalRemy2026} established quantum unique ergodicity for bulk eigenvectors in the 1D Gaussian block band model under the optimal condition \(W^2/N\to\infty\). At critical bandwidth, where \(N/W^2\to\tau\in(0,\infty)\), they also identified the joint limiting law of the rescaled bulk eigenvalues and the block-mass profiles of the corresponding eigenvectors.
Together, these results establish the predicted localization-length scaling $\ell\asymp W^2$ for a broad class of 1D RBMs. By contrast, the theory in dimensions $d\ge2$ remains much less developed, particularly on the localized side, where rigorous results are still very limited. On the delocalized side, diagrammatic expansion methods were developed to prove delocalization in dimensions $d\ge7$ under the condition $W\ge N^\varepsilon$ for an arbitrarily small constant $\varepsilon>0$ \cite{yang2021delocalization, yang2022delocalization, Xu:2024aa}. These results imply the superpolynomial lower bound $\ell\ge W^C$ on the localization length for every fixed constant $C>0$. More recently, extensions of the methods introduced in \cite{YY_25} established bulk delocalization for bandwidths $W\ge N^\varepsilon$ in dimension $d=2$ \cite{DYYY25} and in all dimensions $d\ge3$ \cite{DYYY25_d3}, thereby yielding analogous superpolynomial lower bounds on the localization length.

The goal of this paper is to improve the superpolynomial lower bound obtained in \cite{DYYY25} to narrow the gap to the conjectured lower bound $\exp(\Omega(W^2))$. We make substantial progress toward this prediction by proving that all bulk eigenvectors are delocalized provided that \smash{$W\ge(\log N)^{1+\e}$} for a small constant $\e>0$. Equivalently, this yields an exponential lower bound \smash{\(\ell\ge \exp\p{W^{1-\e}}\)} on the localization length. As part of the proof of delocalization, we establish an entrywise local law for the resolvent \( G(z)=(H-z)^{-1}\) down to the nearly optimal spectral scale \(N^{-1+\varepsilon}\).
We also demonstrate that the exponent \(1\) in the bandwidth condition is sharp for the uniform entrywise local law. More precisely, for any fixed \(\varepsilon>0\) and \((\log N)^{1/2}\le W\leq(\log N)^{1-\varepsilon}\), at the spectral scale \(\eta=\exp(-W^{1-\varepsilon})\) and for every fixed bulk energy \(E\), we prove that \(\max_x \operatorname{Im}G_{xx}(E+\mathrm i\eta)\to\infty\) in probability. 
We remark that the failure of the entrywise local law does not necessarily invalidate the conjectured lower bound \(\exp(\Omega(W^2))\) on the localization length. However, it does show the limitations of the random matrix approaches that may rely on uniform bounds on the resolvent entries.

Before stating the main results, we introduce notation used throughout the paper. We use the set of natural numbers $\N=\{1,2,3,\ldots\}$ and the complex upper half-plane $\C_+:=\{z\in \C:\im z>0\}$.	In this paper, we are interested in the asymptotic regime with $N\to \infty$. When we refer to a constant, it will not depend on $N$ or $W$. 
	For any two (possibly complex) sequences $\xi_N$ and $\zeta_N$ depending on $N$, $\xi_N = \OO(\zeta_N)$, $\zeta_N=\Omega(\xi_N)$, or $\xi_N \lesssim \zeta_N$ means that $|\xi_N| \le C|\zeta_N|$ for some constant $C>0$, whereas $\xi_N=\mathrm{o}(\zeta_N)$ or $|\xi_N|\ll |\zeta_N|$ means that $|\xi_N| /|\zeta_N| \to 0$ as $N\to \infty$. We say that $\xi_N \asymp \zeta_N$ if $\xi_N = \OO(\zeta_N)$ and $\zeta_N = \OO(\xi_N)$. For any $\al,\beta\in\R$, we write
      $\al\vee \beta:=\max\{\al, \beta\}$ and $\al\wedge \beta:=\min\{\al, \beta\}$. 
	Given a vector $\mathbf v$,
  $\|\mathbf v\|_p\equiv \|\mathbf v\|_{\ell^p}$ denotes the $\ell^p$-norm. 
	Given a matrix $\cal A = (\cal A_{ij})$, $\|\cal A\| \equiv \|\cal A\|_{\rm op}$, $\|\cal A\|_{p\to q}\equiv \|\cal A\|_{\ell^p\to \ell^q}$, and $\|\cal A\|_{\max}:=\max_{i,j}|\cal A_{ij}|$ denote the operator (i.e., $\ell^2\to \ell^2$) norm, $\ell^p\to \ell^q$ norm, and maximum norm, respectively. We will use $\cal A_{ij}$ and $ \cal A(i,j)$ interchangeably in this paper. We will use $\Id$ to denote identity matrices or operators.

\subsection{The model and main results}
Let $W,L\in \N$, and take $N=(WL)^2$. Our models are defined on a 2D discrete torus \(\ZL\subset\mathbb{Z}^2\), consisting of $N$ lattice points and having side length $WL$:
\[\ZL:=\qll{ -(WL)/2+1 , (WL)/2}^2 .\] 
We partition $\ZL$ into $L^2$ disjoint blocks, indexed by \(\Zn:=\qll{ -L/2+1 , L/2}^2,\)
which we refer to as the \emph{block lattice}.  Each index $a=(a(1), a(2))\in \Zn$ corresponds to a block 
\be\label{eq:blockIa}
[a] := \prod_{i=1}^2\left\llbracket (a(i)-1)W + 1, \; a(i)W \right\rrbracket, 
\end{equation}
where $a(i)$ denotes the $i$-th coordinate of $a$. Without loss of generality, we assume that $L$ is even; the case of odd $L$ can be treated similarly by defining the block lattice as $ \Zn:=\qll{ -(L-1)/2 , (L-1)/2}^2$.

We view both $\ZL$ and $\Zn$ as discrete 2D tori. 
We denote the vertices of $\ZL$ by $x,y,\ldots$ and those of $\Zn$ by $a,b,\ldots$. Given $x,y\in \ZL$ and $a,b\in \Zn$, we denote the periodic representatives of $x-y$ and $a-b$ by
$\rep{x-y}_{WL}$ and $\rep{a-b}_L$, respectively:  
\be\label{representativeL}\rep{x-y}_{WL}:= \left((x-y)+(WL)\Z^2\right)\cap \ZL,\quad \rep{a-b}_L:= \left((a-b)+L\Z^2\right)\cap \Zn.\end{equation}
For definiteness, we use the $\ell^1$-metric to define distances: 
\begin{equation}\label{eq:periodic_distance} |x-y|\equiv \|\rep{x-y}_{WL}\|_1,\quad \forall x,y\in \ZL,\quad \text{and}\quad |a-b|\equiv \|\rep{a-b}_L\|_1,\quad \forall a,b\in \Zn,\end{equation}
which correspond to the (periodic) graph distances on $\ZL$ and $\Zn$, respectively. 
We write $x\sim y$ if $x$ and $y$ are neighbors in $\ZL$, i.e., $|x-y|=1$, and similarly $a\sim b$ if $a$ and $b$ are neighbors in $\Zn$. 

Following \cite{DYYY25}, we define our 2D random band matrix as a complex Hermitian random block Hamiltonian \( H=(H_{xy}:x,y\in \ZL) \), where the entries $H_{xy}$ are independent (up to the Hermitian symmetry $H_{xy}=\overline{ H}_{yx}$) Gaussian random variables. More precisely, given a symmetric doubly stochastic variance matrix $S=(S_{xy}:x,y\in \ZL)$, the diagonal entries of $H$ are real Gaussian random variables, and the off-diagonal entries are complex Gaussian random variables, distributed as follows: 
\be\label{bandcw0}
H_{xy}\sim \mathcal{N}_{\R}(0, S_{xy}) \cdot \mathbf 1_{x=y} + \mathcal{N}_{\C}(0, S_{xy}) \cdot \mathbf 1_{x\ne y}.
\end{equation}
For $x\in [a]$ and $y\in [b]$, the variance matrix $S$ is given by
\be\label{eq:variance-profile}
S_{xy}\equiv\var (H_{xy}) := W^{-2}S^{(\sB)}_{ab},\quad \text{with}\quad 
S^{(\sB)}_{ab}:=\frac{1}{5} \mathbf 1_{a=b} + \frac{1}{5} \mathbf 1_{a\sim b},
\end{equation}
where \( S^{(\sB)} \) denotes an \( L ^2\times L^2 \) matrix. Informally, the matrix $H$ consists of i.i.d.~GUE blocks on the diagonal and i.i.d.~Ginibre blocks (up to Hermitian symmetry and scaling) on the neighboring off-diagonal blocks.
\begin{remark}  
To simplify the proof and focus on its main challenge—namely, relaxing the condition on $W$—we follow \cite{DYYY25} and consider the simplest variance profile given by \eqref{eq:variance-profile}. Note that \smash{$S^{(\sB)}$} is the transition kernel of a lazy simple random walk on $\Zn$. More generally, one may take \smash{$S^{(\sB)}$} to correspond to other random walks on $\Zn$, as explained in \cite{RBSO1D,Block_reduction}. Even the block structure of $S$ is unnecessary, as shown in \cite{erdHos2025zigzag,PRBM}. Our argument should extend to these more general settings, possibly at the cost of requiring a larger logarithmic exponent in \eqref{eq:polylog-W} below. 
We do not pursue these extensions here.
\end{remark} 

Let \( \lambda_1 \leq \lambda_2 \leq \dots \leq \lambda_N \) denote the eigenvalues of \( H \). The corresponding normalized eigenvectors of $H$ are denoted by \( \left( \boldsymbol{\psi}_k \right)_{k=1}^N \). It is well known that as $W\to \infty$, the empirical spectral measure $N^{-1}\sum_{k=1}^N \delta_{\lambda_k}$ converges almost surely to the Wigner semicircle law \cite{Wigner} with density 
\(
\rho_{\mathrm{sc}}(x) = \sqrt{(4 - x^2)_+}/{2\pi}.
\)
Define the Green's function (or resolvent) of the Hamiltonian $H$ as
\be\label{def_Green}
G(z):=(H-z)^{-1} ,  \quad z\in \C_+.
\end{equation}
It is also known that as $W\to \infty$, $G(z)$ converges to the scalar matrix $m(z)\Id$ entrywise, where $m(z)$ denotes the Stieltjes transform of \( \rho_{\mathrm{sc}} \), defined by
\be\label{eq:defmzsc}
m(z)\equiv m_{\mathrm{sc}}(z):=\frac{-z+\sqrt{z^2-4}}{2}= \int_{\mathbb{R}} \frac{\rho_{\mathrm{sc}}(x)}{x - z} \dd x.
\end{equation}
Here, the square root is the analytic branch satisfying $\sqrt{z^2-4}\sim z$ as $|z|\to\infty$ in $\C_+$. Our first main result establishes the delocalization of the bulk eigenvectors of $H$ assuming that $W\ge (\log N)^A$ for a constant $A>1$. For simplicity of notation, we introduce the logarithmic factor 
\begin{equation}
 \logpara= 10+\log N. 
 \label{eq:prefactor}
\end{equation}

\begin{theorem}[Bulk delocalization]\label{thm:delocalization}
Consider the two-dimensional random band matrix model defined above. Fix any constants $\kappa\in (0,2)$, $A_{\ref{thm:delocalization}}>1$, and $C_{\rm bw}>A_{\ref{thm:delocalization}}+100$. Assume  
\begin{equation}
\logpara^{A_{\ref{thm:delocalization}}}\le W\le \logpara^{C_{\rm bw}}. 
 \label{eq:polylog-W}
\end{equation}
Then the following holds for every constant $D>0$ and all sufficiently large $N\ge N(\kappa,D,A_{\ref{thm:delocalization}} ,C_{\rm bw} )$:
\begin{equation}
 \Pp\left(\max_k \|\boldsymbol{\psi}_k\|_\infty^2
   \1\{\lambda_k\in[-2+\kappa,2-\kappa]\}\le\frac{(\log W)^{10}\logpara^2}{N}\right)\ge 1-N^{-D}.
 \label{eq:delocalization}
\end{equation}
Here $(\lambda_k,\boldsymbol{\psi}_k)$ are the eigenpairs of $H$ with $\|\boldsymbol{\psi}_k\|_2=1$.
\end{theorem}
\begin{remark}      
An earlier version of the paper (\href{https://arxiv.org/abs/2609.24426v2}{arXiv:2609.24426v2}) already covers the regime $W\ge\logpara^7$, establishing the delocalization bound $\logpara^7/N$. The upper bound \smash{$W\le\logpara^{C_{\rm bw}}$} in \eqref{eq:polylog-W} offers two presentation advantages.  First, it ensures that $\log W\le C\log\log N$, so any fixed power of $\log W$ can be absorbed by a factor $\logpara^\e$. We can therefore focus on tracking powers of $\logpara$.
Second, in this regime, the side length \(L\) of the block lattice differs from the side length \(WL\) of the original torus by only a polylogarithmic factor, which does not affect the localization length scale at the level of powers of \(N\).

\end{remark}

The delocalization result, \Cref{thm:delocalization}, follows from a local law on the Green's function. Following \cite{DYYY25}, for $a\in\Zn$, let $P_a$ be the projection onto block $a$, that is, $P_a$ is the block diagonal matrix defined by $(P_a)_{xy}=\delta_{xy}\mathbf 1_{x\in[a]}$. Define the block-averaging matrix $E_a$ by
\begin{equation}
 E_a=W^{-2}P_a,\qquad \text{with}\quad  \tr E_a=1,\quad
 \sum_{a}E_a=W^{-2}\Id.
 \label{eq:block-projection}
\end{equation}
Given a constant $\kappa>0$, we introduce the spectral domain:
\begin{equation}
 \bS(\kappa):=\{E+\ii\eta: |E|\le2-\kappa,\ \ (\log W)^{4}/L^2\le\eta\le1\}.
 \label{eq:main-domain}
\end{equation}

\begin{theorem}[Local laws]
\label{thm:local-law-smallW}
In the setting of \Cref{thm:delocalization}, for every fixed $D>0$ and all sufficiently large $N\ge N(\kappa,D,A_{\ref{thm:delocalization}} ,C_{\rm bw})$, the following estimates hold with probability at least $1-N^{-D}$: 
\begin{align}
 \sup_{z\in\bS(\kappa)}\max_{x,y\in \ZL}|G_{xy}(z)-m(z)\delta_{xy}| &\le (\log W)^3\logpara / W, 
       \label{eq:entry-law}\\
\sup_{z\in\bS(\kappa)}\max_{a\in \Zn} |\Tr(G(z)E_a)-m(z)| &\le (\log W)^6 \logpara^{3/2} / W^2, 
       \label{eq:trace-law}
\end{align}
Furthermore, for any constant $k>1$, there is a constant $C_{\kappa,k, D}>0$  such that, with the same probability and for all sufficiently large $N$, the blocks of $G$ satisfy the operator and Hilbert--Schmidt norm bounds
\begin{align}
\sup_{z\in\bS(\kappa)}\max_{a,b\in \Zn} \|P_aG(z)P_b\| &\le (\log W)^5 \logpara, \label{eq:norm-law}\\
\sup_{z\in\bS(\kappa)}\max_{a,b\in \Zn} W^{-2}\tr|P_aG(z)P_b|^2\big/\Pi_\eta^{(k)}(a,b)  &\le C_{\kappa,k, D}, \label{eq:HS-norm-law}
\end{align}
where $\eta=\im z$ and the profile function is defined as follows (with $\log_+ x:=(\log x)\vee 0$ for $x>0$): 
\begin{equation}
\Pi_\eta^{(k)}(a,b):=\pa{1+ \log_+\frac{\eta^{-1/2}}{|a-b|+1}}\pa{1+|a-b|\sqrt{\eta}}^{-k}.\label{eq:decay-poly}
\end{equation}
Finally, the following annealed estimate holds uniformly for
$z=E+\ii\eta\in\bS(\kappa)$ and $a,b\in\Zn$:
\begin{equation}\label{eq:annealed-two-loop-law}
 \absb{W^{-2}\E\tr|P_aG(z)P_b|^2
 -|m(z)|^2\p{1-|m(z)|^2S^{(\sB)}}^{-1}_{ab}}
 \le \e_N(1+\log\eta^{-1}),
\end{equation}
where $\e_N$ is deterministic and tends to zero as $N\to\infty$.
\end{theorem}

The usual route to \eqref{eq:delocalization} requires an entrywise local law down to the spectral scale $\eta\asymp(\log W)^{10}\logpara^2/N$. In our model, however, invariance under block unitary conjugation allows us to apply Haar concentration to the block operator bound \eqref{eq:norm-law} at the larger scale $\eta=(\log W)^4/L^2$. This yields \eqref{eq:delocalization} without extending the local law below $L^{-2}$; see \Cref{sec:pf-delocal}. 
To explain the advantage of working above this scale, following \cite{DYYY25}, we define, for $\eta\ge N^{-1}$, the characteristic length scale $\ell_\eta$ and the error parameter $M_\eta$ by
\begin{equation}\label{eq:scales}
 \ell_\eta:=\min\{\eta^{-1/2},L\},\qquad
 M_\eta:=W^2\ell_\eta^2\eta=\min\{W^2,N\eta\}.
\end{equation}
Both quantities change behavior at $\eta=L^{-2}$.
Restricting to $\eta\ge L^{-2}$ gives $\ell_\eta=\eta^{-1/2}$ and $M_\eta=W^2$ throughout, substantially simplifying the presentation.

Our second result shows that the uniform entrywise local law fails at bandwidths below the logarithmic scale. Thus, the exponent \(1\) in the bandwidth condition \eqref{eq:polylog-W} is sharp for the entrywise local law.

\begin{theorem}[Failure of the entrywise local law]\label{thm:sublog-divergence}
Consider the two-dimensional random band matrix model defined in \eqref{bandcw0} and \eqref{eq:variance-profile}. Fix $\kappa\in(0,2)$, $E\in[-2+\kappa,2-\kappa]$, and $ \epsilon\in (0,1/2)$. Suppose that $\logpara^{1/2}\le W\le \logpara^{1-\epsilon}$, and define the deterministic spectral scale \( \eta_N:=\exp(-W^{1-\epsilon}).\)
Then there is a constant $c>0$ such that
\begin{equation}\label{eq:sublog-quantitative}
 \lim_{N\to \infty}\Pp\left(\max_x\im G_{xx}(E+\ii\eta_N)
     \ge c {\logpara}/{W^{1+\epsilon}}\right)=1.
\end{equation}
Note that $\logpara/W^{1+\epsilon}\ge\logpara^{\epsilon^2}\to\infty$.
\end{theorem}

In the larger bandwidth regime $A_{\ref{thm:delocalization}}>2$, we sharpen \eqref{eq:HS-norm-law} to a quantum diffusion estimate that identifies the leading deterministic contribution.

\begin{theorem}[Quantum diffusion]
\label{thm:local-law}
In the setting of \Cref{thm:delocalization}, assume further that $A_{\ref{thm:delocalization}}>2$. There exists a constant $c_{\ref{thm:local-law}}>0$ such that, for every fixed $D>0$ and $\e>0$ and all sufficiently large $N\ge N(\kappa,\e,D,A_{\ref{thm:delocalization}},C_{\rm bw})$, the following estimate holds with probability at least $1-N^{-D}$:
\begin{equation}\label{eq:quantum-diffusion}
 \sup_{z\in\bS(\kappa)}\max_{a,b\in\Zn}
 \frac{\absb{W^{-2}\tr|P_aG(z)P_b|^2
 -|m(z)|^2\p{1-|m(z)|^2S^{(\sB)}}^{-1}_{ab}}}
 {\Omega_{c_{\ref{thm:local-law}},\eta}(a,b)}
 \le\frac{\logpara^{3+\e}}{W^2},
\end{equation}
where $\eta=\im z$ and the profile function is defined by 
\begin{equation}\Omega_{c_{\ref{thm:local-law}},\eta}(a,b):=\pa{1+ \log_+\frac{\eta^{-1/2}}{|a-b|+1}}\exp\pbb{-c_{\ref{thm:local-law}}\pbb{\frac{|a-b|\sqrt{\eta}}{\log W}}^{1/2}} .\label{eq:decay-exp}
\end{equation}
\end{theorem}

As in \cite{DYYY25,YY_25,erdHos2025zigzag}, our proofs of \Cref{thm:local-law,thm:local-law-smallW} rely on a dynamical analysis of the resolvent. Relaxing the bandwidth condition from $W\ge N^\e$ to \eqref{eq:polylog-W}, however, requires overcoming several fundamental obstacles. The works \cite{DYYY25,YY_25,erdHos2025zigzag} analyze the dynamics of higher-order resolvent loops (see \Cref{Def:G_loop}) together with their deterministic approximations, referred to as primitive loops in \cite{YY_25}. Our approach instead starts from the observation that resolvent loops can be represented as traces of products of resolvent blocks \smash{$X_{ab}:=E_a^{1/2}GE_b^{1/2}$}; see \eqref{eq:sA-loop}. This representation allows us to control the loops through Schatten norms of the $X$ blocks. Schatten interpolation then reduces the problem to bounding the operator and Hilbert--Schmidt (HS) norms of these blocks. We develop a stochastic bootstrap framework that closes the dynamical estimates for these norms using only resolvent loops of order at most two, thereby avoiding the analysis of higher-order loops and their corresponding primitive loops.
We will introduce the flow framework and the Schatten representation in \Cref{sec:flow}, and discuss the main difficulties and new ideas in \Cref{sec:ideas}.

\begin{remark}
With minor modifications, the powers of $\log W$ in \Cref{thm:delocalization,thm:local-law-smallW} can be reduced. We do not pursue these refinements here to keep the presentation focused. 
We also expect our strategy to extend readily to random band matrices and block Anderson models (see, e.g., \cite{DYYY25_d3,RBSO1D,PelSchShaSod,yang2024Del}) in dimensions $d\ge3$, yielding simpler proofs while relaxing the condition $W\ge N^\e$ in \cite{DYYY25_d3,yang2024Del} to $W\ge(\log N)^C$ for some constant $C>0$. We plan to pursue these extensions in future work.

\end{remark}

\subsection{Organization of the paper}

In \Cref{sec:flow}, we introduce the stochastic flow framework, the resolvent loops and their Schatten representation in terms of resolvent blocks, and the deterministic propagators. We also outline the main new ideas of the proof.
\Cref{sec:small-support,sec:small-trace,sec:direct-operator,sec:spatial,sec:small-closure} are devoted to proving \Cref{thm:delocalization,thm:local-law-smallW}. In \Cref{sec:small-support}, we introduce the stopping conditions and common tools and establish a bootstrap Schatten bound. We then prove the light-weight estimates in \Cref{sec:small-trace}, an a priori operator-norm estimate for the two-loop error in \Cref{sec:direct-operator}, and the two-loop bootstrap estimates in \Cref{sec:spatial}. In \Cref{sec:small-closure}, we combine these results to prove \Cref{thm:local-law-smallW} and deduce \Cref{thm:delocalization} using Haar concentration. Finally, in \Cref{app:sublog}, we prove \Cref{thm:sublog-divergence} using the annealed estimate \eqref{eq:annealed-two-loop-law} from \Cref{thm:local-law-smallW}.

In \Cref{app:quantum-diffusion}, we establish the quantum diffusion estimate in \Cref{thm:local-law} for the larger-bandwidth regime $A_{\ref{thm:delocalization}}>2$. The proof builds on the light-weight and Schatten estimates from the main text and develops the sharper spatial two-loop estimates needed for quantum diffusion.

%{\cob To streamline the presentation, we defer the proofs of several auxiliary results to the appendix. Some of these results are included primarily for illustrative purposes, others are relatively straightforward, and still others follow arguments similar to those in the existing literature. Moreover, they are all largely independent of the main proofs. The appendix is available in the online arXiv version of the paper: (\href{https://arxiv.org/abs/2609.24426}{arXiv:2609.24426}).}

\subsection{Use of AI}\label{sec:AI}

The proof strategy and main ideas were developed by the authors and presented in the first two versions of this paper on arXiv (\href{https://arxiv.org/abs/2609.24426v1}{arXiv:2609.24426v1},\href{https://arxiv.org/abs/2609.24426v2}{arXiv:2609.24426v2}). Roughly speaking, we introduced a Schatten representation of resolvent loops and developed arguments relying only on the dynamics of lower-order loops: up to order 6 in the first version and up to order 2 in the second; see \Cref{sec:ideas} for a more detailed discussion. Building on this strategy, the third version (\href{https://arxiv.org/abs/2609.24426v3}{arXiv:2609.24426v3}) extended the results to the bandwidth regime $A_{\ref{thm:delocalization}}>2$ through several technical refinements implemented with the assistance of ChatGPT:
\begin{enumerate}
  \item The first natural refinement tracks the spatial decay of the logarithmic factor $\iota_t$ in the $\Theta$-propagator bound \eqref{eq:tools-P-point}, instead of estimating this factor by its maximum.

  \item The second refinement concerns the two-sided regularization for alternating 2-loops; see \Cref{sec:double-reg}. %{\cob  of the online arXiv version (\href{https://arxiv.org/abs/2609.24426}{arXiv:2609.24426}) of the paper.}
  Inspired by Definition~4.7 of \cite{DYYY25_d3}, we use a smooth cutoff kernel, defined in \eqref{compact:kernel}, in place of the $\Theta$-propagator to lift a vector to a two-tensor. This saves powers of $\logpara$ but introduces additional drift terms into the evolution equation. The corresponding transfer estimates under the evolution kernel were proved with the assistance of ChatGPT; see \Cref{sec:transfer-two-sided}. %{\cob  of the online arXiv version (\href{https://arxiv.org/abs/2609.24426}{arXiv:2609.24426}) of the paper.}

  \item ChatGPT identified an improvement to the 1-loop bootstrap argument: tracking the dynamics of high Schatten $S_p$-norms of the $X$-matrices yields a stronger estimate than tracking their operator norms, or $S_\infty$-norms.
\end{enumerate}
The authors had planned the first two refinements since the announcement of the first version; ChatGPT independently identified the third while assisting with revisions. 

The extension to the smaller-bandwidth regime $1<A_{\ref{thm:delocalization}}\le2$ in the fourth version (\href{https://arxiv.org/abs/2609.24426v4}{arXiv:2609.24426v4}) requires further conceptual developments. It began with the authors’ observation that, when $W<\logpara^2$, obtaining a precise quantum diffusion estimate for 2-loops $W^{-2}\tr|P_aG(z)P_b|^2$ may be difficult. At sufficiently small bandwidths, fluctuations may be comparable to or larger than the deterministic leading term. In particular, when $W\le\logpara^{3/2}$, the error bound in \eqref{eq:quantum-diffusion} no longer guarantees a small error relative to that term. This motivated a strategy that retains only upper bounds on the 2-loops.

The authors asked ChatGPT to explore this strategy. By introducing additional observables into the stopping conditions, as outlined in \Cref{sec:support}, it indeed lowered the bandwidth threshold, but the resulting argument still required an exponent strictly larger than 1. Two ingredients closed the remaining gap: Gaussian integration by parts, proposed by the authors and implemented through the light-weight expansions used in our earlier works \cite{yang2021delocalization,yang2022delocalization,DYYY25_d3}, and the new observable \eqref{eq:periodic-weight}, identified by ChatGPT.  
This observable is comparable to the target observable \eqref{eq:Yrat}, but satisfies the exact identity \eqref{simp:commcalculus}, which avoids several troublesome error terms.

In addition, ChatGPT assisted with language editing, literature searches, the proof of \Cref{prop:tools-square-kernel}, and the identification of useful SDE tools (\Cref{lem:matrix-maximal,lem:stochastic-convolution-maximal}). All AI-assisted arguments were carefully verified and rewritten by the authors, who take full responsibility for the entire content of the paper.

\subsection*{Acknowledgments}  
The research of Fan Yang is supported in part by the National Key R\&D Program of China (No.~2023YFA1010400) and NSFC (No.~12526201).
We are deeply grateful to Hao Liang of the School of Mathematical Sciences at Peking University for inspiring the inequality in \Cref{lem:spatial-entropy}, which yielded a substantially stronger result than we had anticipated. 

\section{Flow framework and Schatten representation}\label{sec:flow}

We adopt the flow framework in \cite{DYYY25}. Consider the following matrix Brownian motion: 
\begin{align}\label{MBM}
	\dd (H_{t})_{xy}=\sqrt{S_{xy}}\dd (\boldsymbol{B}_{t})_{xy}, \ \ \forall x,y\in \ZL,\quad \text{where}\quad  H_{0}=0.
\end{align}
The diagonal entries of $\boldsymbol B_t$ are real Brownian motions and the strictly upper triangular entries are complex Brownian motions, all independent, with the remaining entries determined by \smash{$(\boldsymbol{B}_{t})_{xy}=\overline {(\boldsymbol{B}_{t})_{yx}}$}. Thus, for $t>0$, \smash{$t^{-1/2}\boldsymbol{B}_t$} is an $N\times N$ GUE matrix whose entries have zero mean and unit variance; $S=(S_{xy})$ is the variance matrix defined in \eqref{eq:variance-profile}. 
Following \cite{10.1214/19-ECP278, Sooster2019, DY}, we consider the Green's function of \( H_t \) with a carefully chosen time-dependent spectral parameter \( z_t \), whose dynamics are naturally renormalized at leading order.

\begin{definition}[Flow framework]\label{def_flow}
For any $\sE \in (-2,2)$, we write $m(\sE)$ for the boundary value $m(\sE+\ii 0_+)$ of the function $m$ defined in \eqref{eq:defmzsc}. Based on this, we define the \emph{spectral parameter flow} $z_t$ by
\be\label{eq:zt}
z_t(\sE) = \sE + (1-t) m(\sE),\quad \text{for}\ \ t\in [0, 1). 
\end{equation}
We refer to $\sE$ as the \emph{flow parameter}, which remains fixed throughout the flow (and throughout our proofs). We write $z_t=E_t + \ii \eta_t$ with 
\begin{align}\label{eta}
	E_t\equiv E_t(\sE)=\sE+(1-t)\re m(\sE),\quad \eta_t\equiv \eta_t(\sE) = (1-t)  \im m(\sE).
\end{align}
We write the Green's function of $H_{t}$ as $G_t(z):=(H_t-z)^{-1}$, and define the resolvent flow as 
\begin{align}\label{self_Gt}
	G_{t;\sE}\equiv G_t(z_t(\sE)):=(H_{t} -z_t{(\sE)})^{-1}.
\end{align}
By It{\^o}'s formula, $G_{t;\sE}$ satisfies the following SDE:
\begin{align}\label{eq:SDE_Gt}
\dd G_{t;\sE}=-G_{t;\sE}(\dd H_{t})G_{t;\sE}+G_{t;\sE}\{\mathcal{S}[G_{t;\sE}]-m(\sE)\}G_{t;\sE}
\dd t,
\end{align}
where $\mathcal{S}: \C^{N\times N}\to \C^{N\times N}$ is a linear operator defined as
\begin{align}\label{eq:opS}
\mathcal{S}[X]_{ij}:=\delta_{ij}\sum_{k=1}^N S_{ik}X_{kk},\quad \text{for}\quad X\in \C^{N\times N}.
\end{align}
\end{definition}

For any target spectral parameter $z$, we are interested in the original resolvent $G(z)=(H-z)^{-1}$. This can be achieved through the stochastic flow by carefully choosing the spectral parameter $\sE$.

\begin{lemma}[Lemma 2.7 of \cite{YY_25}]\label{zztE}
	Fix any $z\in \mathbb C_+$ with $\im z\in (0, 1]$ and $|\re z|\le 2-\kappa$. We choose 
	\be\label{eq:t0E0}t_0\equiv t_0(z)=|m(z)|^2=\frac{\im m(z)}{\im m(z)+ \im z},\quad \sE\equiv \sE(z)=-2\frac{\re m(z)}{|m(z)|}\, .\end{equation}
	Then, for the random band matrix model, we have  
	\begin{equation}\label{eq:zztE}
		m(z)=\sqrt{t_0}m(\sE), \quad   z =   z_{t_0}(\sE)/\sqrt{t_0} , \quad G(z) \stackrel{d}{=} \sqrt{t_0} G_{t_0;\sE} ,
	\end{equation} 
	where $\stackrel{d}{=}$ denotes equality in distribution.
\end{lemma}

In the main proofs, we will fix a target spectral parameter $z={E}+\ii \eta$ with $|E|\leq2-\kappa$ and $\eta\ge (\log W)^4/L^2$. Accordingly, we choose the parameters $t_0$ and $\sE$ as specified in \eqref{eq:t0E0}. Choose the terminal time
\begin{equation}\label{eq:t*}
 t_*:=1-c_{\kappa}(\log W)^{4}/L^2 
\end{equation}
for a sufficiently small constant $c_\kappa>0$ such that $\eta_{t_*}\le (\log W)^{4}/L^2$. The second identity in \eqref{eta} implies that $1-t\asymp \eta_t$ uniformly in $t\in [0,t_0]$.
For clarity, unless we want to emphasize their dependence on $\sE$, we will omit this variable from the notation, such as $z_t(\sE)$, $E_t(\sE)$, $\eta_t(\sE)$, $m(\sE)$, and most importantly, $G_{t;\sE}\equiv G_t$. 

\subsection{Resolvent loops and blocks}
Following previous works on random band matrices, such as \cite{YY_25,DYYY25,DYYY25_d3,erdHos2025zigzag}, we introduce the following resolvent loops.

\begin{definition}[Resolvent loops]\label{Def:G_loop}
	For a charge $\sigma\in \{+,-\}$, we denote 
		\begin{equation*}
		G_{t}(\sigma):=\begin{cases}
			(H_t-z_t)^{-1}, \ \ \text{if} \ \  \sigma=+\\
			(H_t-\bar z_t)^{-1}, \ \ \text{if} \ \ \sigma=-
		\end{cases}, \quad m (\sigma ):= \begin{cases}
			m(\sE),  &\text{if} \ \ \sigma  =+ \\
			\bar m(\sE),  &\text{if} \ \ \sigma = -
		\end{cases} .
		\end{equation*}
	In other words, we let $G_{t}(+)\equiv G_{t}$ and $G_{t}(-)\equiv G_{t}^*$. Recall that $E_{a}$ denotes the rescaled block identity matrix defined in \eqref{eq:block-projection}.
	For any $\fn\in \N$, fix indices $\bsig=(\sigma_1, \cdots, \sigma_\fn)\in \{+,-\}^\fn$ and $\ba=(a_1, \ldots, a_\fn)\in (\Zn)^\fn$. We define the corresponding resolvent loop of order $n$ by 
	\begin{equation}\label{Eq:defGLoop}
		{\cal L}^{(\boldsymbol{\sigma})}_{\ba}(t)= \tr \pa{\prod_{i=1}^\fn \left(E_{a_i} G_{t}(\sigma_i) \right) }.
	\end{equation}
	We also call these resolvent loops $\cL$-loops. 
	Moreover, we define the \emph{centered resolvent} $\Gc$ as
	\begin{equation}\label{Eq:defwtG}
		\Gc_t(\sigma) := G_t(\sigma) - m(\sigma)\Id,\quad \forall t\in[0,1), \ \sig\in \{+,-\} .
	\end{equation}
	For any $a\in \Zn$, we refer to $\tr[\Gc_t(\sig) E_a]$ as a \emph{light-weight}. 
\end{definition}

The proofs in \cite{YY_25,DYYY25,DYYY25_d3,erdHos2025zigzag} rely on an analysis of the dynamics of the resolvent loops up to arbitrarily high order. Our proof, however, only involves the dynamics of the 1-loops (i.e., the light-weights) and the 2-loops with alternating charges $\bsig=(\sigma,-\sigma)$. Our new strategy begins with the following representation of the resolvent loops:
\begin{align}\label{eq:sA-loop}
  \cL^{(\bsig)}_{\ba}(t)= \tr\pa{\prod_{i=1}^n X_{a_{i}a_{i+1}}^{\sigma_i}(t)},\quad  \quad X_{ab}^{\sigma}(t):=E_{a}^{1/2}G_t(\sigma)E_{b}^{1/2},
\end{align}
where we use the cyclic convention $a_{n+1}=a_1$. We refer to $X_{ab}^{\sigma}(t)$ as a \emph{resolvent block} or \emph{$X$-block}.
Recall that the Schatten $S_p$-norm of a matrix $X$ is defined for $1\le p<\infty$ by \smash{$\|X\|_{S_{p}}^{p}=\tr(|X|^p),$} where $|X|:=(X^*X)^{1/2}$. In particular, for $p=2k$ and $X=X_{ab}^{\sigma}$, \smash{$\|X\|_{S_{p}}^{p}$} can be expressed as an $\cL$-loop of order $p$. 
When $p=\infty$, it gives the operator norm $\|X\|_{S_{\infty}} =\|X\|_{\rm op}$. Schatten norms satisfy the following elementary properties, which will be used tacitly in our proof:
\begin{itemize}
  \item {\bf Submultiplicativity}: $\|X_1X_2\|_{S_p}\le \|X_1\|_{S_p}\|X_2\|_{S_p}.$
  \item {\bf Monotonicity}: For $1\le p \le q\le \infty$, $\|X\|_{S_1}\ge \|X\|_{S_p}\ge \|X\|_{S_q}\ge \|X\|_{S_\infty}$.
  \item {\bf Duality}: For $p,q\in[1,\infty]$ with $1/p+1/q=1$, we have $\|X\|_{S_p}=\sup\h{|\tr\p{XY^*}|:\|Y\|_{S_q}=1}$.
  \item {\bf H\"older's inequality}: For $p\ge 2$, $p_j\ge 1$, and $q\ge 1$ with $\sum_jp_j^{-1}=q^{-1}$, we have
\begin{equation}
  \|X\|_{S_p}^p \le\|X\|_{\rm op}^{p-2}\|X\|_{S_2}^{2},\quad \text{and}\quad 
 \normB{\prod_jX_j}_{S_q}\le\prod_j\|X_j\|_{S_{p_j}}.
 \label{eq:block-Schatten-product}
\end{equation}
\end{itemize}

By the second inequality in \eqref{eq:block-Schatten-product}, a bound on the $S_p$-norms of the $X$-blocks suffices to control a $p$-$\cL$-loop for $p\ge 3$. To obtain such a bound, we use the first inequality in \eqref{eq:block-Schatten-product} to interpolate between the $S_2$-norm and the operator norm. Since a $2$-$\cL$-loop bound controls the $S_2$-norm, the remaining task is to obtain a sufficiently sharp bound on $\lVert X^\sigma_{ab}\rVert_{\rm op}$. A standard approach in random matrix theory is to bound the operator norm by an $S_{2k}$-norm, with $k=1$ corresponding to the HS norm. This requires bounds on $\cL$-loops of order $2k$.
To avoid using loops of length greater than 4, we derive the operator norm bound directly from the dynamics of $X^\sigma_{ab}(t)$. 
We abbreviate $ X_{ab}(t)\equiv  X^+_{ab}(t)$ and define $X(t):=\sum_{a,b}X_{ab}(t)=W^{-2}G_t$. Under the block-extraction convention $[\cA]_{ab}:=P_a\cA P_b$, using \eqref{eq:SDE_Gt}, we can write the evolution equation for $X_{ab}$ as
\begin{equation}
\dd X_{ab} =  W^2\qb{X (t)\mathscr{S}[\Gc_t] X (t)}_{ab} \dd t - \sum_{x,y\in \ZL} K_{ab,(x,y)}\dd (\boldsymbol{B}_{t})_{xy} . 
\label{eq:dynamics-Xt}
\end{equation}
Here, the linear operator $\mathscr{S}:\C^{N\times N}\to \C^{N\times N}$ is a block version of the operator in \eqref{eq:opS}, defined as
\begin{align}\label{eq:opS-block}
\mathscr{S}[\cA]:=W^2\sum_{a,b}S_{ab}^{(\sB)}\tr\p{\cA E_{b}}E_a, \quad \text{for}\quad \cA\in \C^{N\times N},
\end{align}
and $K_{ab,(x,y)}(t)$ represents the matrix 
\begin{equation}\label{eq:K-loop}
K_{ab,(x,y)}(t)= E_a^{1/2}G_t\cV_{xy}G_tE_b^{1/2} ,\qquad \cV_{xy}:=\sqrt{S_{xy}} \mathbf e_x\mathbf e_y^* ,
\end{equation}
where $\mathbf e_x$ denotes the standard unit vector along the $x$-th direction.
To control the martingale term, we need to bound its quadratic variation densities. One is given by
\begin{align}\label{eq:K-loop-square}
 \sum_{x,y} K_{ab,(x,y)}(t)K_{ab,(x,y)}(t)^* &= E_a^{1/2}G_t\cS[G_tE_bG_t^*]G_t^*E_a^{1/2} \notag\\
 &= W^2\sum_{a'}E_a^{1/2}G_tE_{a'}G_t^*E_a^{1/2} \cdot \sum_{b'} S_{a'b'}^{(\sB)} \cL^{(+,-)}_{(b',b)}(t),
\end{align}
and the other is \smash{$\sum_{x,y} K_{ab,(x,y)}(t)^*K_{ab,(x,y)}(t)$}.

From the equation \eqref{eq:dynamics-Xt} and It\^o's formula, we can further derive the evolution equations for the light-weights, the square of the $S_2$-norm of the $X$-blocks (i.e., the 2-$\cL$-loops with alternating charges), and the square of the $S_n$-norm of the $X$-blocks. We now record them one by one. 
First, taking the trace of \eqref{eq:dynamics-Xt} with $a=b$ and using that $m(\sigma)$ is independent of $t$, we obtain 
\begin{align}
  \dd\tr\pb{\Gc_tE_a} &= W^2\sum_{a',b}
\cL^{(+,+)}_{(a,a')}(t)S^{(\sB)}_{a'b} \tr\pb{\Gc(t)E_b} \,\dd t +\dd(\mathcal M_1)_{a}(t),
\label{eq:1-D-dynamics}
\end{align}
where the martingale term is defined by 
\[\dd (\mathcal M_1)_{a}(t):=-\sum_{x,y} \tr\pa{K_{aa,(x,y)}(t)}\dd (\boldsymbol{B}_{t})_{xy}. \]
Second, we abbreviate \smash{$ \cL_{ab}(t)\equiv  \|X_{ab}(t)\|_{S_2}^2 = \cL^{(+,-)}_{(a,b)}(t)$} and \smash{$ \cL_{ab}^+(t)\equiv  \tr\p{X_{ab}(t)X_{ba}(t)} = \cL^{(+,+)}_{(a,b)}(t)$}, which we regard as $\cL$-matrices with labels $\emptyset$ and $+$, respectively. It\^o's formula gives
\begin{align}\label{eq:mainStoflow}
		\dd\mathcal{L}_{ab}^{\star}(t) &= W^{2} (\cL^{\star} S^{(\sB)} \cL^{\star})_{ab}(t) \dd t +  \mathcal{W}_{ab}^{\star}(t)\dd t + \dd\mathcal{M}^{\star}_{ab}(t) , \quad \star\in \{\emptyset,+\}, 
	\end{align}
	where $S^{\LK}$ denotes the block variance matrix in \eqref{eq:variance-profile}, and the light-weight terms $\cal W_{ab}^\star$ and the martingale terms $\dd\mathcal{M}_{ab}^\star$ are defined by
	\begin{align} 
    \label{def_EwtG}
		\mathcal{W}_{ab}(t)&:=2W^2\re \sum_{a',b'} \cL^{(+,+,-)}_{(a,a',b)}(t)S^{(\sB)}_{a'b'}\tr\pb{\Gc_t E_{b'}}  , \\
    \label{def_Edif}
		\dd\mathcal{M}_{ab}(t)&:=-2\re\qB{\sum_{x,y}\tr\pa{K_{ab,(x,y)}(t)X_{ab}(t)^*}\dd (\boldsymbol{B}_{t})_{xy}},\\
    \label{def_EwtG+}
		\mathcal{W}_{ab}^+(t)&:=W^2 \sum_{a',b'} \cL^{(+,+,+)}_{(a,a',b)}(t)S^{(\sB)}_{a'b'}\tr\pb{\Gc_t E_{b'}} +\pa{a\leftrightarrow b} , \\ 
    \label{def_Edif+}
		\dd\mathcal{M}_{ab}^+(t)&:=- \sum_{x,y}\tr\pa{K_{ab,(x,y)}(t)X_{ba}(t) }\dd (\boldsymbol{B}_{t})_{xy} - \pa{a\leftrightarrow b} .
	\end{align}
  Here $(a\leftrightarrow b)$ denotes the preceding term with $a$ and $b$ interchanged.

In general, we can derive the evolution equation of $\Tr|X_{ab}(t)|^{n}$ for each fixed even $n$. Define the matrix
\begin{equation}
 \mathcal R_{ab,(x,y)}(t) :=X_{ab}(t)^*K_{ab,(x,y)}(t) +K_{ab,(y,x)}(t)^*X_{ab}(t),
 \quad \text{with}\quad \mathcal R_{ab,(y,x)}=\mathcal R_{ab,(x,y)}^*.
 \label{eq:power-noise}
\end{equation}
Since $n$ is even, we can apply It\^o's formula to $\Tr|X_{ab}|^{n}=\Tr(X_{ab}^*X_{ab})^{n/2}$. Using
\eqref{eq:dynamics-Xt}, we obtain
\begin{align}
 \dd\Tr|X_{ab}(t)|^{n} =\mathfrak I_{n,ab}^{(1)}(t)\dd t+ \mathfrak I_{n,ab}^{(2)}(t)\dd t -\frac{n}{2}\sum_{x,y\in\ZL}
 \Tr\left(|X_{ab}|^{n-2}\mathcal R_{ab,(x,y)}\right)
       \dd(\boldsymbol B_t)_{xy},
 \label{eq:power-Ito}
\end{align}
where we denote the drift terms by 
\begin{align*}
  \mathfrak I_{n,ab}^{(1)}(t):=&~nW^2\re\Tr\left(|X_{ab}|^{n-2}X_{ab}^* [X(t)\mathscr S[\Gc_t]X(t)]_{ab}\right),\\
  \mathfrak I_{n,ab}^{(2)}(t):=&~\frac{n}{2}\sum_{x,y}
 \Tr\left(|X_{ab}|^{n-2}K_{ab,(x,y)}^*K_{ab,(x,y)}\right) +\frac{n}{4}\sum_{j=0}^{n/2-2}\sum_{x,y}
 \Tr\left(|X_{ab}|^{2j}\mathcal R_{ab,(x,y)}
 |X_{ab}|^{n-4-2j}\mathcal R_{ab,(y,x)}\right).
\end{align*}
We introduce the normalized square of the $S_n$-norm of $X_{ab}$ as 
\begin{equation}
\mathfrak R_{n,ab}(t):=\mathfrak h_{n,ab}(t)^{2/n},\qquad \mathfrak h_{n,ab}(t):=W^{-2}\Tr|X_{ab}(t)|^n.\label{eq:Rnab}
\end{equation}
Using \eqref{eq:power-Ito}, we obtain the following lemma.

\begin{lemma}
\label{lem:squared-schatten}
For fixed $a,b\in \Zn$ and even $n\ge4$, denote
\[
 \mathfrak A_{ab}:=W^2[X(t)\mathscr S[\Gc_t]X(t)]_{ab},\quad
 \Sigma^L_{ab}:=\sum_{x,y}K_{ab,(x,y)}K_{ab,(x,y)}^*,\quad
 \Sigma^R_{ab}:=\sum_{x,y}K_{ab,(x,y)}^*K_{ab,(x,y)}.
\]
Then there exists an absolute constant $C_{\mathfrak{R}}>0$ such that the following dynamical inequality holds for $\mathfrak R_{n,ab}(t)$: 
\begin{equation}\label{power:hessian}
 \dd\mathfrak R_{n,ab}
 \le\left[2\sqrt{\mathfrak R_{n,ab}}\,\|\mathfrak A_{ab}\|_{\op}
       +C_{\mathfrak{R}} n\pa{\|\Sigma^L_{ab}\|_{\op}+\|\Sigma^R_{ab}\|_{\op}}\right]\dd t
       +\dd \mathfrak M_{n,ab}.
\end{equation}
Here the martingale term is defined by
\begin{equation}\label{eq:d-mathfrak-h}
 \dd \mathfrak M_{n,ab}
 =-\re\qB{\sum_{x,y}\Tr(Q_{n,ab}^*K_{ab,(x,y)}) \dd(\boldsymbol B_t)_{xy}},\quad \text{where}\quad  Q_{n,ab}:=\frac2{W^2}\mathfrak h_{n,ab}^{2/n-1}X_{ab}|X_{ab}|^{n-2}.
\end{equation}
As a convention, let $Q_{n,ab}=0$ if $X_{ab}=0$. 
We have 
\begin{equation}\label{eq:power-gradient}
 \|Q_{n,ab}\|_{\mathrm{HS}}^2
 \le4W^{-4/n}\mathfrak R_{n,ab}.
\end{equation}
\end{lemma}
\begin{proof}
Given $a,b\in \Zn$, we regard $X_{ab}$ as a map from $\mathbb C^{[b]}$ to $\mathbb C^{[a]}$, so $|X_{ab}|$ acts on $\mathbb C^{[b]}$. The first drift term in \eqref{eq:power-Ito} can be bounded directly by 
\begin{equation}\label{eq:mathfrak-I1}
 \mathfrak I_{n,ab}^{(1)}
 \le n \Tr|X_{ab}|^{n-1}\|\mathfrak A_{ab}\|_{\op}.
\end{equation}
To bound the second drift term, we choose the left and right singular-vector bases for $X_{ab}$, with singular values $\{s_i\}_{1\le i \le W^2}$. In these bases, we can write 
\[
 (\mathcal R_{ab,(x,y)})_{ij}
 =s_i(K_{ab,(x,y)})_{ij}
   +s_j\overline{(K_{ab,(y,x)})_{ji}}.
\]
For $0\le j\le n/2-2$, expanding the traces in these bases and using $|u+v|^2\le2|u|^2+2|v|^2$ gives
\begin{align*}
 &\sum_{x,y}\Tr\left(|X_{ab}|^{2j}\mathcal R_{ab,(x,y)}
             |X_{ab}|^{n-4-2j}\mathcal R_{ab,(y,x)}\right)\\
 \le &~2\sum_{x,y,i,k}\left[
 s_i^{2j+2}s_k^{n-4-2j}|(K_{ab,(x,y)})_{ik}|^2
 +s_i^{2j}s_k^{n-2-2j}|(K_{ab,(y,x)})_{ki}|^2\right] \le  4\Tr|X_{ab}|^{n-2} \pb{\|\Sigma^L_{ab}\|_{\op}+\|\Sigma^R_{ab}\|_{\op}}.
\end{align*}
For the last step, we have applied Young's inequality $s_i^r s_k^{n-2-r}\le s_i^{n-2} r/\p{n-2} +s_k^{n-2} \p{n-2-r}/\p{n-2}$ for $0\le r\le n-2$, and bounded the resulting row and column sums of $|(K_{ab,(x,y)})_{ik}|^2$ by the diagonal entries of \smash{$\Sigma^L_{ab}$} and \smash{$\Sigma^R_{ab}$}, respectively. Summing over $j$, this bounds the second term of \smash{$\mathfrak I_{n,ab}^{(2)}$}. The first term is bounded easily by $\frac n2\Tr|X_{ab}|^{n-2}\|\Sigma^R_{ab}\|_{\op}$. In sum, we obtain 
\begin{equation}\label{eq:mathfrak-I2}
 \mathfrak I_{n,ab}^{(2)}
 \le Cn^2\Tr|X_{ab}|^{n-2}
       \pb{\|\Sigma^L_{ab}\|_{\op}+\|\Sigma^R_{ab}\|_{\op}}.
\end{equation}

To derive \eqref{power:hessian}, we regularize $\mathfrak h_{n,ab}$ by $\epsilon>0$ and denote $F_\epsilon=(\mathfrak h_{n,ab}+\epsilon)^{2/n}$. This is a smooth function of the matrix entries even when $\mathfrak h_{n,ab}=0$.
Applying It\^o's formula to $F_\epsilon$ and using \eqref{eq:power-Ito}, we derive the SDE for $F_\epsilon$. By \eqref{eq:mathfrak-I1} and H\"older's inequality \smash{$W^{-2}\Tr|X_{ab}|^r\le \mathfrak h_{n,ab}^{r/n}$} for $0<r\le n$, the drift arising from \smash{$\mathfrak I_{n,ab}^{(1)}$} is bounded by
\[
 2\pb{\mathfrak h_{n,ab}+\epsilon}^{2/n-1}W^{-2}
       \Tr|X_{ab}|^{n-1}\|\mathfrak A_{ab}\|_{\op}
 \le2\sqrt{F_\epsilon}\|\mathfrak A_{ab}\|_{\op}
\]
uniformly in $\epsilon$. Similarly, by \eqref{eq:mathfrak-I2} and H\"older's inequality, the drift arising from \smash{$\mathfrak I_{n,ab}^{(2)}$} is bounded by
\[
 Cn\pb{\mathfrak h_{n,ab}+\epsilon}^{2/n-1}W^{-2}\Tr|X_{ab}|^{n-2}
       (\|\Sigma^L_{ab}\|_{\op}+\|\Sigma^R_{ab}\|_{\op})
 \le Cn(\|\Sigma^L_{ab}\|_{\op}+\|\Sigma^R_{ab}\|_{\op}).
\]
The second-derivative term in It\^o's formula is nonpositive by concavity of the function $x\mapsto (x+\epsilon)^{2/n}$ on $[0,\infty)$ for $n\ge 2$.
Finally, by \eqref{eq:power-noise} and $\dd(\boldsymbol B_t)_{yx}=\overline{\dd(\boldsymbol B_t)_{xy}}$, we can write the martingale term as 
\[
 -2\pb{\mathfrak h_{n,ab}+\epsilon}^{2/n-1}W^{-2}\re\left[\sum_{x,y}\Tr(|X_{ab}|^{n-2}X_{ab}^*K_{ab,(x,y)})
                 \dd(\boldsymbol B_t)_{xy}\right].
\]
Taking $\epsilon\to 0$ gives \eqref{power:hessian}. Define the regularized version
\[
 Q_{n,ab,\epsilon}:=\frac2{W^2}(\mathfrak h_{n,ab}+\epsilon)^{2/n-1}
 X_{ab}|X_{ab}|^{n-2}.
\]
Using $\Tr|X_{ab}|^{2n-2}\le(\Tr|X_{ab}|^{n})^{2-2/n}$, we get 
\[
 \|Q_{n,ab,\epsilon}\|_{\mathrm{HS}}^2
 \le4W^{-4/n}\mathfrak h_{n,ab}^{2-2/n}(\mathfrak h_{n,ab}+\epsilon)^{4/n-2}
 \le4W^{-4/n}F_\epsilon,
\]
which gives \eqref{eq:power-gradient} by taking $\epsilon\to 0$.
\end{proof}

\subsection{$\Theta$-propagators and Duhamel's principle}\label{subsec:DyLK}

For $t\in[0,1)$, we define the $\Theta$-propagators by
\begin{equation}
 \Theta_t:=(\Id-t S^{(\sB)})^{-1},\quad \Theta_t^+:=m^2 (\Id-tm^2 S^{(\sB)})^{-1}.
 \label{eq:tools-theta}
\end{equation}
With this notation, we define the $\cK$-matrices $\cK^\star (t):=W^{-2} \Theta_{t}^\star$, and 
it is easy to check that they satisfy the dynamical equations
	\begin{align}\label{pro_dyncalK}
		\partial_t{\cal K}^\star(t)=W^2 \cK^\star(t) S^{(\sB)}\cK^\star(t),\quad \star\in\{\emptyset,+\}.
	\end{align}
We then define the difference matrices $\difloop^\star:=\mathcal L^\star -\mathcal K^\star $.
Subtracting \eqref{pro_dyncalK} from the evolution equation \eqref{eq:mainStoflow}, we obtain the following SDEs for the $\difloop$-loops, which have also been included in equation (5.15) of \cite{YY_25}: 
\begin{align}\label{eq_L-Keee}
\dd\difloop^{\star}(t) =  \mathscr{A}^{\star}_{t} \circ \cD^\star(t)  \dd t + \mathcal{E}^{\star}(t)\dd t+\mathcal{W}^{\star}(t)\dd t + \dd\mathcal{M}^{\star}(t),\quad \star\in\{\emptyset,+\}.
\end{align} 
Here, the quadratic error term $\mathcal{E}^\star$ is defined by $\mathcal{E}^{\star} := W^2\cD^\star S^{(\sB)}\cD^\star$, and 
we define the linear evolution operator \smash{${\mathscr A}^{\star}_{t}$} acting on matrices $\cal A$ as 
\begin{align}
 \mathscr A^{\star}_{t}\circ \cal A  := S^{(\sB)}\Theta_t^\star \cal A+\cal A S^{(\sB)}\Theta_t^\star,\quad \star\in\{\emptyset,+\}. \label{eq:L-minus-K-linear-operator}
\end{align} 
Next, by decomposing $\cL^{+}_{aa'}$ in \eqref{eq:1-D-dynamics} as $\cK^{+}_{aa'}+\cD^{+}_{aa'}$, we get
\begin{align}
 \dd\tr\pb{\Gc_tE_a} &= \sum_b \pb{S^{(\sB)}\Theta_t^+}_{ab} \tr\pb{\Gc_tE_b} \dd t + (\mathcal{E}_1)_a (t)\dd t+\dd(\mathcal M_1)_a(t),
\label{eq:1-D-dynamics-2}
\end{align}
where the quadratic error term is defined as 
\begin{equation}\label{def_ELKLK_n1}
(\mathcal{E}_1)_a(t) := W^2\sum_{b}
\pb{\cD^{+}(t)S^{(\sB)}}_{ab} \tr\pb{\Gc_tE_b} .
\end{equation}

For \(0\leq u\leq t<1\), we introduce the one-coordinate kernel
\begin{align}
 U_{u,t}
 :=\frac{I-uS^{(\mathrm B)}}{I-tS^{(\mathrm B)}},\qquad U_{u,t}^+
 :=\frac{I-um^2 S^{(\mathrm B)}}{I-tm^2 S^{(\mathrm B)}},
 \label{eq:one-coordinate-evolution-kernel}
\end{align}
and the two-coordinate evolution kernel by
\begin{equation}
 \mathcal U^{\star}_{u,t}\circ \cal A 
 :=U_{u,t}^\star\cal A U_{u,t}^\star,\quad \star\in\{\emptyset,+\}.
 \label{eq:L-minus-K-evolution-kernel}
\end{equation}
Applying Duhamel's principle to \eqref{eq_L-Keee} and \eqref{eq:1-D-dynamics-2}, and using the above evolution kernels, we obtain the integrated loop hierarchy in the following lemma. These dynamics have already been included in \cite{YY_25,DYYY25}, and will be the basis for subsequent proofs.
For simplicity, we denote the random vector for light-weights by $\difloop_{1}(t)$ with \smash{$(\difloop_{1})_a(t)=\tr\p{E_a\Gc_t}$}. When the charges of the light-weights matter, we will denote $\difloop_{1}^+(t):=\difloop_{1}(t)$ and $\difloop_{1}^-(t):=\overline{\difloop_{1}(t)}$. Without the superscript, $\difloop_{1}(t)$ is understood to carry a $+$ charge.

\begin{lemma}[Integrated loop hierarchy, Lemma 5.3 of \cite{YY_25}] \label{Sol_CalL}
For any $0\le s \le t<1$, we have 
	\begin{align}\label{int_K-LcalE}
	 \difloop^{\star}(t) & =
	 \mathcal{U}^{\star}_{s, t} \circ \difloop^{\star}\p{s}  +\int_{s}^t  \mathcal{U}^{\star}_{u, t} \circ \pa{\cE^\star(u)+\cW^\star(u)}  \dd u + \int_{s}^t \mathcal{U}^{\star}_{u, t} \circ \dd \mathcal M^{\star}\p{u} ,\quad \star\in \{\emptyset,+\},\\
  \difloop_1(t)
 &= U_{s,t}^+ \difloop_1(s)  + \int_s^tU_{u,t}^+ \cal E_1(u)\dd u  + \int_s^t  U_{u,t}^+ \dd\mathcal M_1(u).
 \label{eq:1-D-Duhamel}
	\end{align}
 \end{lemma}

We will use the stopped versions of the dynamical equations in \Cref{Sol_CalL} with a stopping time $\tau$. The corresponding stopped dynamics are applied in the sense of the following lemma, whose proof is trivial.  

\begin{lemma}[Stopped dynamics] 
\label{lem:closed-stopped-continuation}
Let \(\tau\) be a closed stopping time taking values in an interval $[a,b]$. Let $\Phi_{u,t}$ be an evolution kernel acting on $n$-tensors ($n\in\{1,2\}$) with $\Phi_{u,u}=\Id$.
Suppose an adapted process $\{Y_t\}$ of $n$-tensors
has the mild form 
\begin{equation}
 Y_{t}=\Phi_{s,t}\circ Y_{s} + \int_s^t \Phi_{u,t}\circ b_{u} \dd u + \sum_{x,y\in \ZL}\int_s^t \Phi_{u,t}\circ F_{u}(x,y) \dd (\boldsymbol{B}_{u})_{xy},\quad \forall a\le s\le t\le b,
\label{eq:Yt_mild1}
\end{equation}
where $\boldsymbol{B}_{u}$ denotes the matrix Brownian motion defined in \eqref{MBM}, and \(u\mapsto \Phi_{u,t}\circ b_{u}\) is progressively measurable and almost surely integrable (in the entrywise sense), while \smash{\(u\mapsto \Phi_{u,t}\circ F_{u}(x,y)\)} is predictable and satisfies, for each $\ba\in (\Zn)^n$,
\begin{equation}
 \E\int_a^t \norma{\pa{\Phi_{u,t}\circ F_{u}}_{\ba}}_{\rm HS}^2\,\dd u<\infty, \quad \text{where}\quad \norma{\pa{\Phi_{u,t}\circ F_{u}}_{\ba}}_{\rm HS}^2:=\sum_{x,y}\absa{\pa{\Phi_{u,t}\circ F_{u}(x,y)}_{\ba}}^2.
\label{eq:Yt_mild2}
\end{equation}
Define, for this specific time $t$,
\begin{equation}
 \widetilde Y_t=\Phi_{a,t}\circ Y_a
 +\int_a^t\1_{\{u\le\tau\}}\Phi_{u,t}\circ b_u\,\dd u
 +\sum_{x,y\in\ZL}\int_a^t\1_{\{u\le\tau\}}\Phi_{u,t}\circ F_u(x,y)\,\dd(\boldsymbol{B}_u)_{xy}.\label{eq:stopped-dynamics}
\end{equation}
Then, we have almost surely, 
\(
 \1_{\{t\le\tau\}}\widetilde Y_t
 =\1_{\{t\le\tau\}}Y_t.
\)
\end{lemma}

Finally, we collect several basic estimates of the $\Theta$-propagators. 
Along the flow \eqref{eta}, $\eta_u=(1-u)\im m\asymp_\kappa 1-u$. In the notation of \eqref{eq:scales}, we abbreviate 
\begin{equation}
\ell_u\equiv \ell_{\eta_u}=\min\{\eta_u^{-1/2},L\},\qquad \rho_u(a,b):= \frac{|a-b|}{\ell_u}, \qquad \iota_u(a,b):=  \log_+\frac{\ell_u}{1+|a-b|} , 
 \label{eq:tools-scales}
\end{equation}
where we recall that $|a-b|$ denotes the periodic graph distance introduced in \eqref{eq:periodic_distance}. We use this notation throughout the rest of the paper. 
First, for any $t\in[0,1)$, we have 
\begin{equation}\label{eq:Theta-inf-to-inf}
|\Theta_{t}^+(a,b)|\le \Theta_{t}(a,b),\quad 
\|\Theta_{t}^+\|_{\infty\to\infty}\le \|\Theta_{t}\|_{\infty\to\infty} = \p{1-t}^{-1},\quad \|\Theta_t\|= \p{1-t}^{-1}  . 
\end{equation} 
The first estimate follows immediately from the Neumann series of $\Theta_{t}$ and $\Theta_{t}^+$, along with the fact $|m|= 1$:
\[|\Theta_{t}^+(a,b)|\le |m^2|\sum_{k=0}^\infty |tm^2|^k (S^{(\sB)})^k_{ab} \le \sum_{k=0}^\infty t^k (S^{(\sB)})^k_{ab}=\Theta_{t}(a,b).\]
For the second estimate in \eqref{eq:Theta-inf-to-inf}, we use 
\[\|\Theta_{t}^+\|_{\infty\to\infty}=\max_a \sum_b |(\Theta_{t}^+)_{ab}|\le \|\Theta_{t}\|_{\infty\to\infty}=\max_a \sum_b (\Theta_{t})_{ab}= \frac{1}{1-t}. \]
The third equation in \eqref{eq:Theta-inf-to-inf} is due to the Perron–Frobenius theorem. 
We next state the pointwise estimates for the $\Theta$-propagators and their first-order finite differences. 

\begin{lemma}
\label{prop:tools-square-kernel}
For every fixed $C_*>1$, there exist structural constants $c_{\ref{prop:tools-square-kernel}},C_{\ref{prop:tools-square-kernel}},c_\kappa, C_\kappa>0$, depending only on $\kappa$, such that the following estimates hold uniformly for $a,b\in \Zn$ whenever $\eta_t\geq N^{-C_*}$:
\begin{align}
  |\Theta_{t}^+(a,b)|&\leq C_\kappa 
 \ee^{-c_\kappa |a-b|},
 \label{eq:tools-P-point-short}\\
 |\Theta_{t}(a,b)|&\leq C_{\ref{prop:tools-square-kernel}} \pa{(\eta_t\ell_t^2)^{-1}+\iota_t(a,b)}
 \ee^{-c_{\ref{prop:tools-square-kernel}}\rho_t(a,b)},
 \label{eq:tools-P-point}\\
 |(\Theta_t^2)_{ab}|
 &\leq C_{\ref{prop:tools-square-kernel}} \eta_t^{-2}\ell_t^{-2}\ee^{-c_{\ref{prop:tools-square-kernel}}\rho_t(a,b)},
 \label{eq:tools-theta-square}\\
  (\Theta_t)_{aa}&\ge c_{\ref{prop:tools-square-kernel}}\qa{1+\log \p{\eta_t^{-1}}},\label{eq:annealed-Theta-diagonal}\\
 \sum_b |\Theta_{t}(a,b)|e^{\frac12c_{\ref{prop:tools-square-kernel}}\rho_t(a,b)} &\le C_{\ref{prop:tools-square-kernel}} \eta_t^{-1}. \label{eq:high-exponential-row-moment}
\end{align}
Moreover, for every unit displacement $e\in \Zn$ with $|e|=1$, we have 
\begin{align}
 |\Theta_t(a,b+e)-\Theta_t(a,b)|
 &\le \frac{C_{\ref{prop:tools-square-kernel}}}{1+|a-b|}e^{-c_{\ref{prop:tools-square-kernel}}\rho_t(a,b)}.
 \label{eq:gradient-envelope}
\end{align}
\end{lemma}

The proof of this lemma is based on a random-walk representation of the Neumann series of $\Theta_t$. We defer the proof to the end of this section, \Cref{sec:pf-tools-square-kernel}. 

\subsection{Light-weight expansion}
An important tool for our proof is a Gaussian integration by parts argument. In particular, we will use the light-weight expansion from Lemma 3.5 of \cite{yang2021delocalization}, derived using Gaussian integration by parts: for any differentiable function $F$ of $G_t$ with the required integrability, 
\begin{equation}\label{eq:LW-expansion}
 \mathbb E[(\difloop_1)_a(t)F ]
 =tm^{-1} \sum_{b}\Theta_{t,ab}^+
   \mathbb E\qa{(\difloop_1(t))_b(S^{(\sB)}\difloop_1(t))_b F-\mathscr D_b  F}.
\end{equation}
Here, the directional derivative $\mathscr D_b$ is defined by 
\begin{align}\label{eq:direction-D}
\mathscr D_b\equiv \mathscr D_{t,b}=\sum_{x,y\in\ZL}f_{b,(x,y)}\partial_{yx},\qquad
f_{b,(x,y)}=\tr(E_b\cV_{xy}G_t),\qquad
\partial_{yx}:=\sqrt{S_{xy}}\,\partial_{(H_t)_{yx}},
\end{align}
where $\cV_{xy}$ is defined in \eqref{eq:K-loop}. Moreover, we let $\mathscr D_{t,b}^{-}\equiv \overline{\mathscr D_{t,b}}$ denote the derivative obtained by conjugating \eqref{eq:direction-D}: \smash{$\mathscr D_{t,b}^{-}=\sum_{x,y}\overline{f_{b,(x,y)}}\partial_{xy}$}. By definition, we have 
\begin{equation}\label{new:exact-two-seams}
 \mathscr D_{a'} X_{ab}^\sigma
   =-\sum_{b'}S^{(\sB)}_{a'b'}X_{ab'}^\sigma X_{b'a'}^{+} X_{a'b}^\sigma,\quad \mathscr D_{a'}^- X_{ab}^\sigma
   =-\sum_{b'}S^{(\sB)}_{a'b'}X_{aa'}^\sigma X_{a'b'}^{-} X_{b'b}^\sigma,\quad \sigma \in\{+,-\}.
\end{equation}
Applying \eqref{eq:LW-expansion} to a moment estimate, we readily obtain the following result.

\begin{lemma}\label{lem:coordinate-test}
Fix any $0\le t\le t_*$. Let $0\le\chi\le1$ be a smooth cutoff function of $H_t$ with bounded entry derivatives. Suppose that, on $\supp\chi$,
\begin{equation}
 \max_a|(\difloop_1)_a|=\oo(1),\quad \max_{a,b}\left\{|\mathscr D_{t,b}(\difloop_1)_a|
       \vee|\mathscr D_{t,b}\overline{(\difloop_1)_a} |\right\}\le\Xi_1,
 \qquad \max_b|\mathscr D_{t,b}\chi|\le\Xi_\chi, \label{eq:assumptions-GIP}
\end{equation}
where $\Xi_1,\Xi_\chi\ge0$ are deterministic control parameters. Then, for every $p\in\N$ and $a\in\Zn$,
\begin{equation}\label{new:general-trace-test}
 \left(\mathbb E|\chi(\difloop_1)_a|^{2p}\right)^{1/(2p)}
 \le  C_\kappa \bigl(\sqrt{p\Xi_1}+p\Xi_\chi\bigr),
\end{equation}
where $C_\kappa>0$ depends only on $\kappa$.
\end{lemma}
\begin{proof}
If $t=0$, then \smash{$\difloop_1(t)=0$}, and the assertion is immediate. We henceforth assume $0< t\le t_*$. Abbreviate $\mathscr D_b=\mathscr D_{t,b}$, $Y_a:=\chi|(\difloop_1)_a(t)|$, and $\Phi_{ab}:=tm^{-1}\Theta_{t,ab}^+$. 
Given $a\in \Zn$, we apply \eqref{eq:LW-expansion} to $\mathbb E Y_a^{2p}$ with
\[
F\equiv F_a:=\chi^{2p}\overline{(\difloop_1)_a}|(\difloop_1)_a|^{2p-2}
     =\chi^{2p}(\difloop_1)_a^{p-1}\overline{(\difloop_1)_a^{p}}.
\]
First, for the resulting quadratic error, we have $|(S^{(\sB)}\difloop_1(t))_b|=\oo(1)$ on $\supp\chi$. Thus, 
\begin{align}\label{eq:small-quad}
 \absB{\sum_b \Phi_{ab}\mathbb E\left[(\difloop_1)_b(S^{(\sB)}\difloop_1)_bF_a\right]}
 &=\oo(1)\cdot  \max_b \mathbb E\left[\chi^{2p}|(\difloop_1)_b(t)||(\difloop_1)_a(t)|^{2p-1}\right]=\oo(1)\cdot  \max_a\mathbb E Y_a^{2p},
\end{align}
where in the first step we used $\|\Phi\|_{\infty\to \infty}\le C_\kappa$ by \eqref{eq:tools-P-point-short} and $|m|=1$, and in the second step we applied H\"older's inequality.
Next, we have 
\begin{align*}
 \mathscr D_bF_a
 = 2p\chi^{2p-1}(\mathscr D_b\chi)(\difloop_1)_a^{p-1}\overline{(\difloop_1)_a^{p}} +\chi^{2p}\left[(p-1)(\difloop_1)_a^{p-2}\overline{(\difloop_1)_a^{p}}\cdot \mathscr D_b(\difloop_1)_a +p (\difloop_1)_a^{p-1}\overline{(\difloop_1)_a^{p-1}}\cdot \mathscr D_b\overline{(\difloop_1)_a}\right].
\end{align*}
Under the assumptions in \eqref{eq:assumptions-GIP}, we obtain, on $\supp\chi$,
\[
 |\mathscr D_bF_a|
 \le 2p\Xi_\chi Y_a^{2p-1} +(2p-1)\Xi_1 Y_a^{2p-2}.
\]
Combining it with \eqref{eq:small-quad} and $\|\Phi\|_{\infty\to \infty}\le C_\kappa$, and applying H\"older's inequality, we obtain that 
\begin{align}
  \mathbb E Y_a^{2p}&\le \oo(1)\cdot  \max_a\mathbb E Y_a^{2p} +  C_\kappa (2p)\Xi_\chi \bigl(\max_a\mathbb E Y_a^{2p}\bigr)^{1-1/(2p)}
       + C_\kappa (2p-1)\Xi_1\bigl(\max_a\mathbb E Y_a^{2p}\bigr)^{1-1/p}\notag\\
       &\le \frac{1}{2}\max_a\mathbb E Y_a^{2p} + \pa{  C_\kappa p\Xi_\chi}^{2p} + \pa{  C_\kappa p\Xi_1}^{p}, \label{eq:self-Ya}
\end{align}
where we used Young's inequality in the second step and $C_\kappa$ depends only on $\kappa$. Solving this inequality for $\mathbb E Y_a^{2p}$ and taking the $(2p)$-th root, we conclude \eqref{new:general-trace-test}.  \end{proof}

\subsection{New ideas}\label{sec:ideas}

The proof in \cite{DYYY25} analyzes the hierarchy of SDEs for the resolvent loops in \Cref{Def:G_loop} up to arbitrarily high order, comparing them with deterministic \emph{primitive loops} that describe their limiting behavior. For example, \eqref{pro_dyncalK} defines the second-order primitive loops $\cK^\star$, while higher-order primitive loops are defined recursively through a sequence of ODEs. We refer to the differences between the resolvent loops and the corresponding primitive loops as \emph{difference loops}.

A careful examination of the proof in \cite{DYYY25} shows that the main obstacle to relaxing the polynomial condition $W\ge N^\e$ to the polylogarithmic condition $W\ge(\log N)^C$ for a constant $C>0$ is the substantial probability loss incurred at several key stages:
\begin{enumerate}
 \item Deriving the averaged local law for the light-weights from the 2-$\cL$-loop bounds requires a fluctuation-averaging mechanism.
 \item Controlling the singularities arising from the evolution kernels in the analysis of the loop hierarchy requires a CLT-type cancellation mechanism.
 \item Sharpening the a priori $\cL$-loop bounds requires an iterative bootstrap argument involving $\cL$-loops of arbitrarily high order.
\end{enumerate}
Each of these arguments weakens the upper bounds on the probabilities of the bad events from $N^{-D}$ to $N^{-cD}$ for some constant $c\in(0,1)$. Consequently, they can be applied only $\OO(1)$ times, which is possible when $W\ge N^\e$. More precisely, in \cite{DYYY25}, the loop hierarchy is analyzed along the flow by induction over a sequence of multiplicatively decreasing scales $\eta_{t_k}\asymp1-t_k$ given by
\begin{equation}
 1-t_k=W^{-\delta k}\vee N^{-1+\delta},
 \label{eq:decresing_etak}
\end{equation}
where $\delta>0$ is sufficiently small and depends on $\e$. At each induction step from $k-1$ to $k$, the three arguments above are applied $\OO_\delta(1)$ times. The total number of applications is therefore at most $\OO_\e(1)$, and the resulting probability loss remains under control after choosing the initial probability-loss exponent $D$ sufficiently large depending on $\e$.
In our setting, however, the induction along \eqref{eq:decresing_etak} requires $\Omega(\log N/\log W)$ steps, and this number diverges when $W=N^{\oo(1)}$. Repeated probability losses then make the resulting high-probability estimates ineffective. One might instead try to take larger multiplicative steps, for example by setting
\(
1-t_k=W^{-C_Nk}\vee N^{-1+\delta}
\)
for some diverging sequence $C_N\to\infty$. With this choice, however, the loop hierarchy can no longer be closed at each induction step. Indeed, the bootstrap argument in \cite{DYYY25} uses the gains provided by factors of $W^{-1}$
to compensate for losses of the form
\(
\smash{({|1-t_{k-1}|}/{|1-t_k|})^C}
\)
for a constant $C>0$. Such losses become too large when $C_N\to\infty$.

To deal with the difficulties associated with items (1)--(3) and the induction scheme along \eqref{eq:decresing_etak}, we adopt a strategy developed in \cite{erdHos2025zigzag} for 1D RBMs, while introducing new ideas to address the additional complications arising in 2D and under the substantially weaker condition on $W$. Roughly speaking, following \cite{erdHos2025zigzag}, we define a stopping time $\tau_*$ through bounds on the light-weights and $\cD$-loops. (We explain below how to close the analysis using only these two orders.) By analyzing \eqref{int_K-LcalE} and \eqref{eq:1-D-Duhamel}, we establish bootstrap estimates for the light-weights and $\cD$-loops that strictly improve the a priori bounds used to define $\tau_*$. Combined with a standard stochastic continuity argument, these improvements then extend the bootstrap estimates down to the target scale
\(
\eta_t\ge (\log W)^4/L^2.
\)
In particular, as in \cite{erdHos2025zigzag}, we derive the averaged local law directly from the dynamical equation \eqref{eq:1-D-Duhamel}, rather than from the 2-$\cL$-loop bounds through a fluctuation-averaging mechanism. This resolves issue (1) above. 

Issue (2) arises in the analysis of the evolution equation for $\cD$-loops. Inspired by the local regularization introduced in \cite{erdHos2025zigzag}, we replace the CLT-type cancellation mechanism with a two-sided regularization technique.
For the dynamics of $n$-$\cL$-loops with $n\ge4$ and alternating charges, \cite{erdHos2025zigzag} applies two local regularizations independently to two disjoint pairs of vertices. The resulting gains compensate for the factor loss $\ell_t^2/\ell_u^2$ arising from the naive evolution-kernel estimate.
This construction does not apply to $\cD$-loops, since there are too few vertices to support two independent regularizations. To overcome this obstruction, we have introduced a two-sided regularization for alternating $\cD$-loops in \Cref{sec:double-reg}. %{\cob  of the online arXiv version (\href{https://arxiv.org/abs/2609.24426}{arXiv:2609.24426}) of the paper.}
It regularizes both vertices simultaneously. Although the two regularizations are no longer independent, they yield sufficient smallness under the action of the evolution kernel to compensate for the factor $\ell_t^2/\ell_u^2$, resolving issue (2).

To address issue (3), we need an alternative to the iterative bootstrap argument.
Moreover, to obtain an explicit value of $A$ in the condition $W\ge(\log N)^A$ that is as small as possible, the argument should not rely on $\cL$-loops of arbitrarily high order. 
The key question is therefore how to close the loop hierarchy without introducing very long loops or iterative arguments, since the dynamics of order-$n$ loops inevitably involve longer $\cL$-loops, as already illustrated by the low-order dynamics in \Cref{Sol_CalL}. 
In the first version of this paper (\href{https://arxiv.org/abs/2609.24426v1}{arXiv:2609.24426v1}), we introduced the Schatten representation \eqref{eq:sA-loop} and used interpolation inequalities for Schatten norms to truncate the dynamics at order 6. In the second version (\href{https://arxiv.org/abs/2609.24426v2}{arXiv:2609.24426v2}), we retained this representation and developed a new strategy that closes the dynamics at order 2, substantially simplifying the proof and lowering the exponent $A$ in the condition $W\ge(\log N)^A$. We now explain this strategy separately in the large- and small-bandwidth regimes.

\medskip
\noindent{\bf The larger-bandwidth regime $A_{\ref{thm:delocalization}}>2$.}
As shown in \Cref{Sol_CalL}, the evolution equation for $\cD_1$ contains 2-$\cD$-loops in the $\mathcal E_1$-term and a martingale whose quadratic variation involves 4-$\cL$-loops. Similarly, the evolution equation for $\cD$ contains 3-$\cL$-loops in the $\mathcal W$-term, and the quadratic variation of its martingale involves 6-$\cL$-loops. To control these loops up to order 6, one may apply the Schatten H\"older inequality \eqref{eq:block-Schatten-product}, which bounds a \(p\)-\(\cL\)-loop by the product of the Schatten \(S_p\)-norms of its resolvent blocks. Each \(S_p\)-norm can then be interpolated between the \(S_2\)- and \(S_\infty\)-norms.
Controlling the operator norm of the blocks $X_{ab}$ directly is, however, generally difficult. A standard approach uses the moment method to bound it through the $S_p$-norm for a large $p\in 2\N$. In the analysis of random band matrices, this norm is in turn estimated through a $p$-$\cL$-loop, bringing the argument back to the higher-order loop hierarchy that we seek to avoid.

To avoid this circular argument, we estimate the operator norm of $X^\sigma_{ab}$ directly from its dynamics in \eqref{eq:dynamics-Xt}. Let $O_{\op}(t):=\max_{a,b}\|X_{ab}(t)\|_{\op}$ be the random control parameter for the operator norms of the \(X\)-blocks. The blockwise drift—the first term on the right-hand side (RHS) of \eqref{eq:dynamics-Xt}—consists of a light-weight, controlled by the a priori light-weight bounds encoded in the stopping time, and products of two \(X\)-blocks, whose operator norms can be estimated self-consistently in terms of \(O_{\op}\); see \eqref{eq:Ward-op}.
To control the martingale term, we estimate its quadratic variation in \eqref{eq:K-loop-square}. The crucial observation is that it is a sum of products of a \(2\)-\(\cL\)-loop and the \emph{positive semidefinite} block matrix \smash{$E_a^{1/2}G_tE_{a'}G_t^*E_a^{1/2}$}. Thus, it suffices to control
\[ \sup_{\|\mathbf v\|_2=1}W^2\sum_{a'}\mathbf v^* E_a^{1/2}G_tE_{a'}G_t^*E_a^{1/2}\mathbf v \cdot \sum_{b'} S_{a'b'}^{(\sB)} \cL^{(+,-)}_{(b',b)}(t).\]
By positivity, the 2-\(\cL\)-loops can be bounded by their maximal size encoded in the stopping time. Ward’s identity can then be applied to the remaining sum, yielding the crucial reduction to the operator norm of a \emph{single \(X\)-block}: $\|\im X_{aa}\|_{\op}\le O_{\op}$.

The significance of this reduction can be seen heuristically as follows. After time integration, martingale concentration produces a contribution of order $\sqrt{\e_N O_{\op}}$, where $\e_N>0$ is a deterministic error. This yields a self-consistent estimate for \(O_{\op}\) in terms of \(\e_N\). Details of this argument appear in Section 4.3 of the second version of the paper (\href{https://arxiv.org/abs/2609.24426v2}{arXiv:2609.24426v2}). In the present version, we streamline the argument further using an observation suggested by ChatGPT: instead of working directly with the block dynamics \eqref{eq:dynamics-Xt}, we use \eqref{power:hessian} to analyze the squared Schatten \(S_p\)-norm of \(X_{ab}\), with \(p\asymp\log W\). This avoids losses in powers of \(\logpara\), since \smash{\(\logpara^{1/p}=\OO(1)\)} for this choice of \(p\). The details can be found in \Cref{sec:small-support}.

Once a bound on \(O_{\op}(t)\) has been established, the Schatten Hölder inequality \eqref{eq:block-Schatten-product}, together with the 2-loop bound encoded in the stopping time, allows us, in principle, to control \(\cL\)-loops of arbitrarily high order. In particular, it provides the higher-order \(\cL\)-loop estimates required in \eqref{int_K-LcalE} and \eqref{eq:1-D-Duhamel}. A further difficulty arises in analyzing \eqref{int_K-LcalE} due to the singularity of the evolution kernel \eqref{eq:L-minus-K-evolution-kernel} when \(\star=\emptyset\). As mentioned above, this singularity is handled by the two-sided regularization developed in \Cref{sec:double-reg}.
%{\cob  of the online arXiv version (\href{https://arxiv.org/abs/2609.24426}{arXiv:2609.24426}) of the paper.}

\medskip
\noindent{\bf The smaller-bandwidth regime $1<A_{\ref{thm:delocalization}}\le 2$.}
In this regime, the strategy for establishing operator-norm bounds on the \(X_{ab}\)-blocks and closing the light-weight estimates is similar to that used in the larger-bandwidth regime. The main difference is that the light-weight estimates are derived using Gaussian integration by parts, through the expansion \eqref{eq:LW-expansion}, which reduces the loss in powers of \(\logpara\). The principal challenge is to close the bootstrap estimates for the \(\cL\)-loops.
As discussed in \Cref{sec:AI}, the smaller-bandwidth regime requires a change of strategy: we seek only upper bounds on the \(\cL\)-loops, rather than a precise quantum-diffusion estimate identifying the leading deterministic contribution.\footnote{The \(\cL^{+}\)-loops can be bounded by the corresponding \(\cL\)-loops using the Cauchy–Schwarz inequality. Hence, only the \(\star=\emptyset\) version of \eqref{eq:mainStoflow} is needed in the proof.} Accordingly, the stopping time encodes only a priori upper bounds on the \(\cL\)-loops, without imposing a precise size estimate on the \(\cD\)-loops.

This strategy raises an immediate question: how can we control the first term on the RHS of \eqref{eq:mainStoflow}? This term is precisely what motivated the comparison of \(\cL\) with \(\cK\), and hence the study of the \(\cD\)-loop dynamics in \eqref{eq_L-Keee}, in previous dynamical analyses of random band matrices. Indeed, the quadratic evolution equation for \(\cL\) may by itself produce a divergent solution. It must therefore be treated as a perturbation of a known, well-behaved solution—namely, \(\cK\) in the present setting.

To address this issue, ChatGPT proposed incorporating a bound on the following spatial moment into the stopping condition:
\[
Y_{r,a}(t):=W^4\eta_t\sum_b \rho_t(a,b)^{2r}\pa{\mathcal L_{ab}(t)^2+\mathcal L_{ba}(t)^2}.
\]
For a sufficiently large \(r\in\mathbb N\), this bound, together with Markov’s inequality, yields sufficiently rapid polynomial pointwise decay of \(\cL_{ab}\). The resulting estimate, however, has a weaker prefactor and is therefore insufficient by itself to close the 2-loop estimate. Nevertheless, using this preliminary pointwise decay in the analysis of \eqref{eq:mainStoflow}, we can recover the sharper pointwise bound on \(\cL\) stated in \eqref{eq:decay-poly}.\footnote{The present strategy does not appear to recover the stronger exponential-decay profile in \eqref{eq:decay-exp}.}

The problem is thus reduced to establishing bootstrap estimates for the new observables \(Y_{r,a}\). Compared with the \(\cL\)-loops, these observables have the advantage that the prefactor \(\eta_t\) and the weights \(\rho_t(a,b)^{2r}\) decrease with \(t\). Their time derivatives therefore generate dissipative terms in the corresponding SDEs. We use this dissipation to absorb the contribution of the first quadratic term on the RHS of \eqref{eq:mainStoflow}. For this purpose, ChatGPT identified a further useful device: replacing the distance weight \(\rho_t(a,b)\) in \(Y_{r,a}\) with the periodic weights defined as in \eqref{eq:periodic-weight}. These weights admit an exact allocation among the factors in weighted sums of matrix products, as expressed in \eqref{simp:commcalculus}. This exact algebraic rule eliminates several troublesome error terms that would otherwise arise in the dynamics of \(Y_{r,a}\).

To ensure that the quadratic contributions can be absorbed by the dissipative terms, we need a sufficiently precise operator-norm bound on \(\cL\): it may exceed the norm of the leading deterministic term, $\|\cK\|=W^{-2}(1-t)^{-1}$, by at most a factor \(1+\oo(1)\). It therefore suffices to establish the a priori estimate \(\|\cD\|=\oo(W^{-2}(1-t)^{-1}).\)
In the larger-bandwidth regime, this bound follows immediately from the precise pointwise decay estimates for the \(\cD\)-loops. In the smaller-bandwidth regime, however, it must be derived from the stopping conditions and the evolution equation \eqref{eq_L-Keee}. The main difficulty is controlling the leading term on the RHS of \eqref{eq_L-Keee}. Once again, the required control is obtained by studying the rescaled quantity \(W^2(1-t)\cD(t)\). Differentiating the factor \(1-t\) generates a dissipative term that absorbs the contribution of this leading term.

Finally, the treatment of the \(\mathcal W\)-term in \eqref{eq:mainStoflow} motivates the introduction of the locally averaged light-weight observable in \eqref{eq:Jloc}. The purpose of this averaging is to control the logarithmic singularity \(\iota_t(a,b)\) in the profile function (see \eqref{eq:tools-P-point}), whose maximum is of order \(\logpara\). A suitable local average avoids this worst-case logarithmic loss. Applying Gaussian integration by parts once more then yields a bound sharper than the pointwise maximum of the light-weights; see \Cref{thm:trace}.

\subsection{Proof of \Cref{prop:tools-square-kernel}}\label{sec:pf-tools-square-kernel}

The estimate \eqref{eq:tools-P-point-short} for the stable $\Theta$-propagator was established in equation (2.66) of \cite{DYYY25_d3}. Although that paper considers random band matrices in dimensions $d\ge 3$, the proof of this estimate does not depend on the dimension. For the unstable $\Theta$-propagator, however, the corresponding estimates in \cite{DYYY25} are not strong enough. We therefore strengthen them to the bounds in \Cref{prop:tools-square-kernel}, replacing the prefactor $N^\tau$ in \cite{DYYY25} with logarithmic factors.

Let $p_k^\infty(x)\equiv p_k^\infty(0,x)$ denote the $k$-step transition kernel of the lazy simple random walk on $\Z^2$ with
\[
 p_1^\infty(x)=\frac15\1_{\{0,\pm e_1,\pm e_2\}}(x),\qquad e_1=(1,0),\quad e_2=(0,1).
\]
For $f:\Z^2\to \C$, we introduce the finite difference operator $\nabla_e $ as $\nabla_ef(x):=f(x+e)-f(x)$ for $e\in \{\pm e_1,\pm e_2\}$. Applying Dungey's theorem \cite[Theorem 1.18, p.~600]{Dungey} gives 
\begin{equation}
 p_k^\infty(x)+\sqrt{k}\,|\nabla_ep_k^\infty(x)|
 \leq Ck^{-1}\ee^{-\e |x|^2/k},
 \qquad \forall k\geq1,\quad x\in\Z^2,
 \label{eq:tools-heat-gradient}
\end{equation}
for some absolute constants $\e,C>0$, where $|\cdot|$ denotes the $\ell^1$-distance on $\Z^2$. The transition kernel of the corresponding walk on the torus $\Zn$ is given by periodization:
\begin{equation}
p_k^{\rm tor}(a,b)=\sum_{\mathbf n\in\Z^2}p_k^\infty(b-a+L\mathbf n ).\label{eq:period-RW}
\end{equation}
Combining this identity with \eqref{eq:tools-heat-gradient}, we obtain, for $k\ge1$,
\begin{equation}
 p_k^{\rm tor}(a,b)\leq  \frac C{k}
 \sum_{\substack{x\in\Z^2:x\equiv b\ ({\rm mod}\ L)}}
 \exp\left[-\frac{\e |x- a|^2}{k}\right] \le \frac{C_\e}{k}\left(1+\frac{k}{L^2}\right)
 \exp\left[-\frac{c_\e |a-b|_L^2}{k}\right],
 \label{eq:tools-periodized-heat}
\end{equation}
where $|\cdot|_L$ denotes the periodic $\ell^1$-distance on $\Zn$, to distinguish it from $|\cdot|$ on $\Z^2$. The second inequality follows from the elementary estimate
\begin{equation}
 \sum_{\mathbf n\in\Z^2}\exp\left[-\frac{\e |x+L\mathbf n|^2}{k}\right]
 \leq C_\e \left(1+\frac{k}{L^2}\right)
 \exp\left[-\frac{c_\e |x|_L^2}{k}\right],\qquad \forall k\geq1.
\label{eq:periodic_lift}
\end{equation}
We express $\Theta_t$ as the Neumann series
\begin{equation}  \label{eq:Neumann_Thetat}
\Theta_t(a,b)=\delta_{ab}+\sum_{k\ge 1} t^k p_k^{\rm tor}(a,b).
\end{equation} 
Using this series representation, we now prove \eqref{eq:tools-P-point}--\eqref{eq:gradient-envelope} one by one. 

\begin{proof}[Proof of \eqref{eq:tools-P-point}]
Note that $\delta_{ab}$ can be trivially bounded by the RHS of \eqref{eq:tools-P-point}. It remains to control $\Theta_t(a,b)-\delta_{ab}$. 
Let $r=|a-b|_L$ and recall that $1-t\asymp_\kappa\eta_t$.
By \eqref{eq:Neumann_Thetat}, \eqref{eq:tools-periodized-heat}, and
$t^k\le e^{-(1-t) k}$, we have
\begin{equation}  
 |\Theta_t(a,b)-\delta_{ab}|
 \le C\sum_{k\ge1}\left(k^{-1}+L^{-2}\right)
                  e^{-(1-t) k-c_\e r^2/k}.
\label{eq:Theta-series}
\end{equation}
We first consider the case  $\ell_t< L$. By the AM--GM inequality, there exists a constant $c>0$ such that 
\(
 e^{-(1-t) k-c_\e r^2/k}
 \le e^{-c r/\ell_t}e^{-c\eta_t k-cr^2/k}.
\)
Then, we bound \eqref{eq:Theta-series} by 
\begin{align}\label{eq:Theta-series2}
 |\Theta_t(a,b)-\delta_{ab}|\le   Ce^{-c r/\ell_t}\sum_{k\ge1}k^{-1}  e^{-c\eta_t k-c r^2/k}+C(L^2\eta_t)^{-1}e^{-c r/\ell_t}.
\end{align}
We bound the sum over $k$ by
\begin{align*}
&\sum_{k=1}^{(r+1)^2}k^{-1}  e^{-c r^2/k}+ \sum_{k=(r+1)^2}^{\eta_t^{-1}}k^{-1}  e^{-c\eta_tk-c r^2/k}+ \sum_{k\ge \eta_t^{-1}} k^{-1}  e^{-c\eta_t k}  \le C + \sum_{k=(r+1)^2}^{\eta_t^{-1}}k^{-1} \le C\qbb{1+ \log_+\frac{\ell_t}{1+r} },
\end{align*}
where the first term is bounded because $\max_{1\le k\le (r+1)^2}(r+1)^2k^{-1}e^{-c r^2/k} \le C$. Plugging this bound into \eqref{eq:Theta-series2} and using $\eta_t\ell_t^2=1$ proves \eqref{eq:tools-P-point} for the case $\ell_t< L$.
Next, suppose $\ell_t=L$, in which case $\eta_tL^2\le1$. We bound \eqref{eq:Theta-series} by 
\begin{align*}
 |(\Theta_t-\Id)(a,b)|
 &\le C\sum_{1\le k\le L^2}k^{-1}e^{-cr^2/k}
       +\frac{C}{L^2}\sum_{k>L^2}e^{-(1-t) k}\le C\sum_{k=1}^{(r+1)^2}k^{-1}e^{-cr^2/k} + C\sum_{k=(r+1)^2}^{L^2}k^{-1}
       +\frac{C}{\eta_tL^2} \\
 &\le C\left[\log_+\frac{L}{1+r} +\frac1{\eta_tL^2}\right]=C\left[\frac1{\eta_t\ell_t^2}+ \log_+\frac{\ell_t}{1+r}   \right].
\end{align*}
This concludes \eqref{eq:tools-P-point} for the regime $\ell_t= L$ by using that $\rho_t(a,b)\le 1$.
\end{proof}

\begin{proof}[Proof of \eqref{eq:tools-theta-square} and \eqref{eq:high-exponential-row-moment}]
For each fixed $j>0$, summing over dyadic shells centered at $b$ with radii comparable to $2^{-k}\ell_t$ yields
\begin{equation}\label{eq:moments-iota}
 \sum_a\iota_t(a,b)^j\le C\ell_t^2.
\end{equation}
We take $c_{\ref{prop:tools-square-kernel}}$ to be at most half the exponential decay rate obtained above for \eqref{eq:tools-P-point}. Combining \eqref{eq:tools-P-point} with the triangle inequality, the Cauchy--Schwarz inequality, and \eqref{eq:moments-iota}, we obtain
\begin{align*}
 (\Theta_t^2)_{ab}
 &\le Ce^{-c_{\ref{prop:tools-square-kernel}}\rho_t(a,b)}
 \sum_x
 \left[(\eta_t\ell_t^2)^{-1}+\iota_t(a,x)\right]
 \left[(\eta_t\ell_t^2)^{-1}+\iota_t(x,b)\right]
 e^{-c_{\ref{prop:tools-square-kernel}}[\rho_t(a,x)\wedge\rho_t(x,b)]}\\
 &\le Ce^{-c_{\ref{prop:tools-square-kernel}}\rho_t(a,b)}
 \left(\eta_t^{-2}\ell_t^{-2}+\eta_t^{-1}+\ell_t^2\right) \le C\eta_t^{-2}\ell_t^{-2}
 e^{-c_{\ref{prop:tools-square-kernel}}\rho_t(a,b)},\\
 \sum_b\Theta_t(a,b)
 e^{\frac12c_{\ref{prop:tools-square-kernel}}\rho_t(a,b)}
 &\le C\sum_b
 \left[(\eta_t\ell_t^2)^{-1}+\iota_t(a,b)\right]
 e^{-\frac12c_{\ref{prop:tools-square-kernel}}\rho_t(a,b)} \le C\left(\eta_t^{-1}+\ell_t^2\right)
 \le C\eta_t^{-1}.
\end{align*}
In the last step of each estimate, we used $\eta_t\ell_t^2\le1$.\end{proof}

\begin{proof}[Proof of \eqref{eq:annealed-Theta-diagonal}]
By the periodization formula \eqref{eq:period-RW} and symmetry of the random walk, for every $k\ge1$,
\[
 p_{2k}^{\rm tor}(a,a)\ge p_{2k}^\infty(0)
 =\sum_{x\in\Z^2}p_k^\infty(x)^2.
\]
We bound this sum from below using \eqref{eq:tools-heat-gradient}. For any $R\ge1$, we have 
\[
 \sum_{|x|>R\sqrt{k}}p_k^\infty(x)
 \le \frac{C}{k}e^{-\e R^2/2}\sum_{x\in\Z^2}e^{-\e|x|^2/(2k)}  \le Ce^{-\e R^2/2} \implies \sum_{|x|\le R\sqrt{k}}p_k^\infty(x) \ge 1/2,
\]
where we have fixed $R$ large enough so that $Ce^{-\e R^2/2}\le 1/2$.
Then, by Cauchy--Schwarz,  
\[
 p_{2k}^{\rm tor}(a,a)\ge
 \frac{\bigl(\sum_{|x|\le R\sqrt{k}}p_k^\infty(x)\bigr)^2}{CR^2 k}
 \ge \frac{1}{4CR^2 k}.
\]
Keeping the even terms in the Neumann series \eqref{eq:Neumann_Thetat}, we obtain 
\[
 (\Theta_t)_{aa}\ge 1+c\sum_{k\ge1}\frac{t^{2k}}{k}
 =1+c\log\frac{1}{1-t^2}.
\]
This proves \eqref{eq:annealed-Theta-diagonal}, since $1+\log\p{1-t^2}^{-1}\asymp 1+\log\p{1-t}^{-1}\asymp 1+\log (\eta_t^{-1})$.
\end{proof}

\begin{proof}[Proof of \eqref{eq:gradient-envelope}] 
Using \eqref{eq:tools-heat-gradient} and an argument analogous to the one above, we obtain that  
\begin{equation}\label{eq:tools-periodized-grad}
\absa{\nabla_ep_k^{\rm tor}(0,b) }
 \leq \frac{C_\e}{k^{3/2}} \left(1+\frac{k}{L^2}\right)
 \exp\left[-\frac{c_\e |b|_L^2}{k}\right],\quad \forall k\ge 1,\quad b\in \Zn. 
\end{equation}
Using this bound for $1\le k \le L^2$ and the relation $1-t\asymp \eta_t$, we obtain, for some constant $0<c\le c_\e$,
\begin{align}
 \sum_{1\le k\le L^2}t^k|\nabla_e p_k^{\rm tor}(0,b)|
 &\le C\sum_{k\ge1}k^{-3/2}{e^{-c\eta_t k-c |b|_L^2/k}}\le Ce^{-c\sqrt{\eta_t}|b|_L}\sum_{k\ge1}k^{-3/2}{e^{-\frac{1}{2}c\pa{\eta_t k+|b|_L^2/k}}} \notag\\
 &\le {C}\pa{1+|b|_L}^{-1}e^{-c\sqrt{\eta_t}|b|_L},
 \label{eq:tools-gradient-short-times}
\end{align}
where we used the AM--GM inequality in the second step, and in the third step we split the remaining sum at $k=(1+|b|_L)^2$ and bounded it by $C/(1+|b|_L)$. For $k>L^2$, we use the Fourier series representation for $p_k^{\rm tor}$, 
\begin{equation}\label{eq:Fourier_rep_pk}
p_k^{\rm tor}(0,b)= \frac{1}{L^2}\sum_{p=(p_1,p_2)\in \p{2\pi \Z_L/L}^2}[\lambda(p)]^k e^{\ii  p\cdot b }, \quad  \lambda(p):=\frac15\left(1+2\cos  p_1 
 +2\cos p_2 \right),\end{equation}
with the frequency representatives chosen in $(-\pi,\pi]^2$, along with the fact that the constant Fourier mode vanishes under $\nabla_e$. Using $|\lambda(p)|\le e^{-c_0|p|^2}$, we obtain that  
\[
 |\nabla_e p_k^{\rm tor}(0,b)|
 \le L^{-2}\sum_{p\ne0}|e^{ip\cdot e}-1|\,|\lambda(p)|^k
 \le  Ck^{-3/2}e^{-c k/L^2}
\]
for a constant $0<c<c_0$. Consequently, we bound the sum over $k>L^2$ by
\begin{equation}
 \sum_{k>L^2}t^k|\nabla_e p_k^{\rm tor}(0,b)| \le C\sum_{k>L^2}k^{-3/2}{e^{-c\eta_t k-c k/L^2}}
 \le CL^{-1}e^{-c\eta_t L^2}\le C(1+|b|_L)^{-1} e^{-c|b|_L/\ell_t}.
 \label{eq:tools-gradient-long-times}
\end{equation}
Combining \eqref{eq:tools-gradient-short-times} and
\eqref{eq:tools-gradient-long-times}, and using that the contribution $\delta_{a,b+e}-\delta_{ab}$ can be trivially bounded, we conclude \eqref{eq:gradient-envelope}.
\end{proof}

\section{Stopping conditions and Schatten bounds}\label{sec:small-support}

Starting from this section, we focus on proving the main result, \Cref{thm:local-law-smallW}, under the condition $A_{\ref{thm:delocalization}}>1$. As outlined in \Cref{sec:ideas}, our strategy combines \emph{pointwise bounds on the $\cL$-loops} with an \emph{operator-norm estimate for the $\cD$-loops}. The former provides the spatial decay needed in the estimates, while the latter yields the $\oo(1)$ error required to close the argument. We begin by introducing the stopping conditions underlying this strategy.

\subsection{Stopping conditions}\label{sec:support}

Recall that $\logpara=10+\log N$, and let $A_{\ref{thm:delocalization}}$ and $C_{\rm bw}$ be the bandwidth exponents in \eqref{eq:polylog-W}. Choose a fixed margin $0<\epsilon_{\rm sm}<\min\{1/100,2A_{\ref{thm:delocalization}}-2\}$, so that 
\begin{equation}\label{eq:window}
 \logpara^{2+\epsilon_{\rm sm}}\le W^2\le\logpara^{2C_{\rm bw}}.
\end{equation}
We then choose a large integer $k_*$ satisfying
\begin{equation}\label{eq:spatialorders}
 k_*>\max\{10,4(2+\epsilon_{\rm sm})/\epsilon_{\rm sm}\}.
\end{equation}
Denote the accumulated time along the flow by
\begin{equation}
 \vartheta_t:=1+\int_0^t\frac{\dd u}{\eta_u}= 1 +\frac{\log[(1-t)^{-1}]}{\im m} \le C\logpara \, .
\end{equation}

We first define the control parameter for the maximum light-weight:
\begin{equation}\label{eq:common-light-weight-control}
 J_1(t):=\frac{\max_a|\tr(\Gc_tE_a)|}{\mathfrak W},\qquad  \mathfrak W:=\frac{\logpara^{3/2}(\log W)^{6}}{W^2}.
\end{equation}
To define the Schatten control, set $p_*:=2\left\lceil(\log W)/2\right\rceil$ and introduce the barrier
\begin{equation}
 \mathfrak s_{t}(a,b)^2:=\frac{\vartheta_t(\log W)^{9}}{W^4} \bigl(\iota_t(a,b)+\log W\bigr).
 \label{eq:positive-barriers}
\end{equation}
Recalling $\mathfrak R_{p_*,ab}(t)=(W^{-2}\tr|X_{ab}(t)|^{p_*})^{2/p_*}$ from \eqref{eq:Rnab}, we define
\begin{equation}
 \mathcal Z_t:=\max_{a,b}
 \qa{\mathfrak R_{p_*,ab}(t)/\fs_{t}(a,b)^2}.
 \label{eq:positive-fine-barrier}
\end{equation}

Fix $0<\fc<1/10$. We measure the spatial decay of the 2-$\cL$-loops by
\begin{equation}\label{eq:J2}
J_2(t):=\max_{a,b\in \Zn}\frac{\cL_{ab}(t)}{(\log W)^{\fc}W^{-2}\pa{1+\iota_t(a,b)}\pa{1+\rho_t(a,b)}^{-k_*}}.
\end{equation}
The corresponding normalized weights are
\begin{equation}\label{eq:omega-weights}
\omega_{x,t}(a):=
 \frac{(1+\iota_t(a,x))(1+\rho_t(a,x))^{-k_*}}{\mathsf Z_{\omega,t}},\quad 
 \mathsf Z_{\omega,t}:=\sum_{a\in\Zn} (1+\iota_t(a,x))(1+\rho_t(a,x))^{-k_*} .
\end{equation}
By \eqref{eq:moments-iota} and the summable polynomial decay of the spatial profile,
$\mathsf Z_{\omega,t}\asymp\ell_t^2=\eta_t^{-1}$.
The same estimates yield, for each fixed $j>0$,
\begin{equation}\label{eq:weightnorm}
 \|\omega_{x,t}\|_\infty\le C \vartheta_t/\ell_t^2,
 \qquad\|\omega_{x,t}\|_2\le C \ell_t^{-1},\qquad 
 \sup_y\sum_a\omega_{x,t}(a)\qa{1+\iota_t(a,y)}^j\le C .
\end{equation}
The weight $\omega_{x,t}(a)$ changes by at most a constant factor when either $x$ or $a$ is shifted by $\OO(1)$. Moreover, 
\begin{equation}\label{eq:localmass}
 \omega_{x,t}(a)\ge c_{k_*}/\ell_t^2\qquad \text{whenever}\quad |a-x|\le\ell_t
\end{equation}
for some constant $c_{k_*}>0$. Using these weights, we define a local average of the squared light-weights and its control parameter:
\begin{equation}\label{eq:Jloc}
J_{\rm loc}(t):=\max_{x\in \Zn}\frac{\mathfrak D_x(t)}{\mathfrak W_{\rm loc}},\quad \mathfrak D_x(t):=\pB{\sum_a\omega_{x,t}(a)|(\difloop_1)_a(t)|^2}^{1/2},\quad
 \mathfrak W_{\rm loc}:=\frac{\logpara (\log W)^{6}}{W^2} .
\end{equation}
This complements the maximum control $J_1(t)$. 
Since $\mathfrak W_{\rm loc} =\logpara^{-1/2} \mathfrak W$, the locally averaged light-weights are controlled at a smaller scale than their maximum.
Finally, for each integer $r\ge1$, we introduce a control parameter for the radially weighted squared $\ell^2$-norms of the rows and columns of $\cL$:
\begin{align}\label{eq:Yrat}
 \mathcal Y_{r,a}(t)&:=\frac{W^4\eta_t}{(\log W)^{\fc}}\sum_b
 \rho_t(a,b)^{2r}\pa{\mathcal L_{ab}(t)^2+\mathcal L_{ba}(t)^2}.
\end{align}
With these control parameters in place, we can now define the principal stopping time.

\begin{definition}[Principal stopping time]\label{def:common-stop-small}
We define the total control parameter
\begin{align}\label{eq:support}
\mathscr J(t):=\max\hb{J_1(t), J_2(t), J_{\rm loc}(t), \cZ_t,\max_a \mathcal Y_{k_*+1,a}(t)},
\end{align}
where $J_1$, $J_2$, $J_{\rm loc}$, $\cZ_t$, and $\mathcal Y_{k_*+1,a}$ are defined in \eqref{eq:common-light-weight-control}, \eqref{eq:J2}, \eqref{eq:Jloc}, \eqref{eq:positive-fine-barrier}, and \eqref{eq:Yrat}, respectively. 
With $\mathscr J$, we define the principal stopping time as follows:
\begin{align}\label{eq:stop-small}
 \tau_*:=t_*\wedge\inf\{t\in[0,t_*]: \mathscr J(t)\ge 1\},\qquad \inf\varnothing:=+\infty.
\end{align}
At time 0, with $G_0=m\Id$ and $\|X_{ab}(0)\|_{\op}\le W^{-2}\delta_{ab}$, we see that $J_1(0)=J_{\rm loc}(0)=\mathcal Y_{r,a}(0)=0$, $J_2(0)\le (\log W)^{-\fc}$, and $\mathcal Z_0\le (\log W)^{-10}$. Hence, by continuity, we have \(\tau_*>0\). \end{definition}

\subsection{Ward's identity and probabilistic tools}\label{sec:common-tools}

Before proceeding with the proof, we collect several basic tools that will be used repeatedly, beginning with the classical Ward identity.
For a Hermitian matrix $\cal A$, define its resolvent by \smash{$R(z):=(\cal A-z)^{-1}$} for $z= E+ \ii \eta\in \C_+$. The algebraic identity
$R-R^*=2\ii\eta RR^*=2\ii\eta R^*R$
yields the Ward identity:
\begin{equation}\label{eq_Ward0}
RR^*= R^*R = \eta^{-1} \im R \, .
\end{equation}
Applying this identity to $R=G_t$ and taking block operator norms and traces gives
\begin{align}\label{eq:Ward-op}
 &\|E_a^{1/2}G_t\|_{\rm op}^2 = \|G_tE_a^{1/2}\|_{\rm op}^2 = \eta_t^{-1}\|E_a^{1/2}\p{\im G_t }E_a^{1/2}\|_{\rm op},\\
 \label{eq:sum-S2}
 &\sum_a\|X_{ab}\|_{S_2}^2
 =\sum_a\|X_{ba}\|_{S_2}^2
 =\frac{\im\Tr(G_tE_b)}{W^2\eta_t}.
\end{align}

We define the Schatten maxima by
\begin{equation}\label{eq:def-Os}
\begin{aligned}
O_2(t)&:=\max_{a,b}\|X_{ab}(t)\|_{S_2}^2,\quad  && O_{\rm op}(t):=\max_{a,b}\|X_{ab}(t)\|_{\op},\\
O_4(t)&:=\max_{a,b}\|X_{ab}(t)\|_{S_4}^4,\quad
&&\wt O_4(t):=\max_{b}\hB{\pB{\sum_a \|X_{ab}(t)\|_{S_4}^4}\vee \pB{\sum_a \|X_{ba}(t)\|_{S_4}^4}}. 
\end{aligned}
\end{equation}
These quantities satisfy the following bounds up to the stopping time $\tau_*$. 

\begin{lemma}\label{lem:entropy-operator-output-sm}
There exists a constant $C_{\ref{lem:entropy-operator-output-sm}}>0$ such that the following bounds hold simultaneously for every stopped trajectory and every $0\le t\le \tau_{*}$:  
\begin{equation}
O_2(t) \le C_{\ref{lem:entropy-operator-output-sm}}\frac{(\log W)^\fc\logpara}{W^2},\quad 
O_{\op}(t)\le C_{\ref{lem:entropy-operator-output-sm}}\frac{\vartheta_t (\log W)^{5}}{W^2}, \quad  
\wt O_4(t)\le C_{\ref{lem:entropy-operator-output-sm}}\frac{\logpara \ell_t^2 (\log W)^{10}}{W^{6}}, 
 \label{eq:entropy-operator-output-sm}
\end{equation}
\begin{equation}
 \|X_{ab}(t)\|_{\op}\le W^{2/p_*}\mathfrak s_t(a,b) \le e^2 (\log W)^{9/2}W^{-2}
 \sqrt{\vartheta_t[\iota_t(a,b)+\log W]}.
 \label{eq:spatial-operator-output-sm}
\end{equation}
\begin{equation}\label{eq:L2-calL}
W^4\eta_t\max_a\sum_b \pa{\mathcal L_{ab}^2+\mathcal L_{ba}^2}  \le  C_{\ref{lem:entropy-operator-output-sm}}(\log W)^{2\fc} .
 \end{equation}
\end{lemma}
\begin{proof}
The $O_2$ bound follows from $J_2\le1$ and $\iota_t\le C\logpara$. The condition $\mathcal Z_t\le1$, together with $p_*\ge\log W$, gives $\|X_{ab}\|_{\op}\le W^{2/p_*}\mathfrak s_t(a,b)$, which proves \eqref{eq:spatial-operator-output-sm} and the $O_{\op}$ bound. 
By \eqref{eq:J2}, \eqref{eq:moments-iota}, and the summable polynomial decay of the spatial profile, we have
 \begin{align*}
W^4\eta_t \sum_b \pa{\mathcal L_{ab}^2+\mathcal L_{ba}^2} &\le 2(\log W)^{2\fc}\eta_t \sum_b
       \pa{1+\iota_t(a,b)}^2\pa{1+\rho_t(a,b)}^{-2k_*} \le C(\log W)^{2\fc},
 \end{align*}
which concludes \eqref{eq:L2-calL}. 
To bound \smash{$\wt O_4(t)$}, we combine $\|X_{ab}\|_{S_4}^4\le\|X_{ab}\|_{\op}^2\mathcal L_{ab}$ with \eqref{eq:spatial-operator-output-sm} to obtain
\begin{align*}
 \sum_a\|X_{ab}\|_{S_4}^4
 &\le \frac{C\vartheta_t(\log W)^{9}}{W^4}
       \sum_a[\iota_t(a,b)+\log W]\|X_{ab}\|_{S_2}^2 \le \frac{C\logpara(\log W)^{10}\ell_t^2}{W^6}.
\end{align*}
For the last inequality, the contribution weighted by $\log W$ is controlled by \eqref{eq:sum-S2}, while the contribution weighted by $\iota_t$ is controlled by \eqref{eq:L2-calL}, \eqref{eq:moments-iota}, and the Cauchy--Schwarz inequality.
The same argument applies to \smash{$\sum_a\|X_{ba}\|_{S_4}^4$}, completing the proof of the \smash{$\wt O_4(t)$ bound}.
\end{proof}

The following lemma establishes an exponential-moment bound for the Schatten norm of a weighted sum of positive semidefinite matrices. All matrices in this lemma are square.

\begin{lemma}[Schatten bound via exponential moments]
\label{lem:spatial-entropy}
Let $\cI$ be a finite set of indices. Suppose $\{Z_\al:\al\in\cI\}$ is a family of positive semidefinite matrices. Let $Y=\sum_{\alpha\in\cI}Z_\alpha$. Assuming $\Tr Y>0$ (otherwise the conclusion is trivial), define the probability measure $\mu(\alpha)=\Tr Z_\alpha/\Tr Y$ on $\cI$. Let $f(\alpha)\ge0$ be a sequence of values such that $\sum_{\alpha\in\cI}\mu(\alpha)e^{f(\alpha)}\le C_0$ for some $C_0\ge e$. Then, for every $1\le p\le\infty$,
\begin{equation}
 \normB{\sum_{\alpha\in\cI}f(\alpha)Z_\alpha}_{S_p}
 \le\|Y\|_{S_p}\log\frac{C_0\Tr Y}{\|Y\|_{S_p}}.
 \label{eq:entropy-norm-bound}
\end{equation}
\end{lemma}
\begin{proof}
Let $1/p+1/q=1$. By Schatten duality and positivity, we have 
\begin{equation}\label{eq:Schatten-duality}
 \normB{\sum_\alpha f(\alpha)Z_\alpha}_{S_p}
 =\sup_{X\ge0,\,\|X\|_{S_q}=1}
       \Tr\pbb{X\sum_\alpha f(\alpha)Z_\alpha}.
\end{equation}
Fix $X\ge0$ with $\|X\|_{S_q}=1$. If $\Tr(XY)=0$, the corresponding contribution vanishes. Otherwise, define the probability measure
$p_X(\alpha):=\Tr(XZ_\alpha)/\Tr(XY)$.
Since $\|X\|_{\op}\le1$, Jensen's inequality and the assumed exponential-moment bound yield
\[
 \sum_\alpha p_X(\alpha)f(\alpha)
 \le\log\sum_\alpha p_X(\alpha)e^{f(\alpha)}
 \le\log\frac{\sum_\alpha\Tr Z_\alpha e^{f(\alpha)}}{\Tr(XY)}
 \le\log\frac{C_0\Tr Y}{\Tr(XY)}.
\]
Multiplying by $\Tr(XY)$ and using
$\Tr(XY)\le\|Y\|_{S_p}\le\Tr Y$, we obtain
\[
 \Tr\pbb{X\sum_\alpha f(\alpha)Z_\alpha}
 \le\Tr(XY)\log\frac{C_0\Tr Y}{\Tr(XY)}
 \le\|Y\|_{S_p}\log\frac{C_0\Tr Y}{\|Y\|_{S_p}}.
\]
Here the last step uses that $x\mapsto x\log(C_0\Tr Y/x)$ is
increasing on $(0,\Tr Y]$ when $C_0\ge e$.
Taking the supremum in \eqref{eq:Schatten-duality} proves the claim.
\end{proof}

We will use the following continuous-time analogue of the matrix Freedman inequality; see \cite[Theorem~1.2 and Corollary~1.3]{Tropp2011Freedman} for the discrete-time version. 

\begin{lemma}[A Freedman-type inequality for matrix martingales]
\label{lem:matrix-maximal}
Let $N_t=\sum_{\al\in \cI}\int_0^t \cA_\al(s)\,\dd B_{\al,s}$ be a $k\times k$ complex matrix martingale, where $\cI$ is a finite index set, $B_{\al}$ are independent standard (real or complex) Brownian motions, and $\cA_\al$ are predictable $k\times k$ matrix-valued processes. Suppose $\sum_\al \cA_\al \cA_\al^*\le \dot v(s)\Id$ and $\sum_\al \cA_\al^* \cA_\al\le \dot v(s)\Id $ almost surely (where the inequality is in the Loewner order), where $v$ is a deterministic, absolutely continuous, and nondecreasing function satisfying $v(0)=0$ and $v(T)\le v_*$. Then for all $R>0$,
\begin{equation}
 \mathbb P\!\left(\sup_{t\le T}\|N_t\|_{\rm op}\ge R\right)
 \le 2k\exp\qa{-R^2/(2v_*)}.
 \label{eq:matrix-tail}
\end{equation}
\end{lemma}
\begin{proof}
This inequality can be derived directly from \cite[Corollary~1.3]{Tropp2011Freedman} by discretizing time, stopping each increment at a deterministic threshold, and letting both the mesh size and the threshold tend to zero. Alternatively, it can be proved using It\^o's formula, as we now sketch. Introduce the Hermitization of $N_t$
\[
Y_t=\begin{pmatrix}0&N_t\\N_t^*&0\end{pmatrix},
\qquad \text{with}\qquad  \lambda_{\max}(Y_t)=\|N_t\|_{\rm op}.
\]
By It\^o's formula and the covariance bounds assumed in the lemma, $e^{-\vartheta^2v(t)/2}\Tr(e^{\vartheta Y_t})$ is a nonnegative local supermartingale with initial value $2k$ for every $\vartheta>0$. Applying optional stopping at
\(
\tau=T\wedge\inf\{t\ge0:\|N_t\|_{\rm op}\ge R\}
\)
yields
\[
\mathbb P\!\left(\sup_{t\le T}\|N_t\|_{\rm op}\ge R\right)
\le 2k\exp\!\left(-\vartheta R+\frac{\vartheta^2v_*}{2}\right).
\]
Optimizing over $\vartheta>0$ proves \eqref{eq:matrix-tail}. 
\end{proof}

The following lemma is a direct consequence of the Freedman-type inequality in \Cref{lem:matrix-maximal}.

\begin{lemma}[A stochastic maximal inequality]
\label{lem:stochastic-convolution-maximal}
Let ${\cal M}_t=\sum_{\al\in\cI}\int_0^t\cA_\al(s)\,\dd B_{\al,s}$ be a continuous $\mathbb C^d$-valued local martingale with ${\cal M}_0=0$, where $\cI$ is a finite index set, $B_\al$ are independent standard (real or complex) Brownian motions, and $\cA_\al(s)$ are predictable $\mathbb C^d$-valued processes.
Let \smash{$\cB:[0,T]\to\mathbb C^{d\times d}$} be a deterministic and measurable process with
\smash{\(
 A:=\int_0^T\|\cB_t\|_{\infty\to\infty}\,\dd t<\infty.
\)}
Let $\Phi_{s,t}:\C^d\to \C^d$ be the evolution operator determined by $\partial_t\Phi_{s,t}=\cB_t\Phi_{s,t}$ and $\Phi_{s,s}=\Id$, and define
\[
 N_t:=\int_0^t\Phi_{s,t}\,\dd {\cal M}_s.
\]
Suppose there are deterministic nonnegative integrable functions $q:[0,T]\to[0,\infty)$ and $p_t:[0,t]\to[0,\infty)$ for each fixed $t\in[0,T]$, such that
\begin{equation*}  
 \max_{1\le i\le d}\sum_{\al\in\cI}|(\cA_\al(s))_i|^2\le q(s),
 \qquad
 \max_{1\le i\le d}\sum_{\al\in\cI}|(\Phi_{s,t}\cA_\al(s))_i|^2\le p_t(s)
\end{equation*}
almost surely for almost every $s\in[0,t]$, and
\begin{equation}
 \int_0^t q(s)\,\dd s\le v_*,
 \qquad \int_0^tp_t(s)\,\dd s\le v_*,\label{eq:integral-bounds}
\end{equation}
for a deterministic $v_*>0$.
Then there are universal constants $c,C>0$ such that, for every $R>0$,
\begin{equation}
 \mathbb P\left(\sup_{t\le T}\|N_t\|_\infty\ge R\right)
 \le Cd(1+A)\exp\left(-\frac{cR^2}{v_*}\right).
 \label{eq:stochastic-convolution-maximal}
\end{equation}
More generally, let $w:[0,T]\to(0,\infty)$ be a deterministic, continuous, and nonincreasing positive function, and define
\smash{\(
 N_t^{(w)}:=w_t N_t.
\)}
If the two integral bounds in \eqref{eq:integral-bounds} are replaced by
\begin{equation}
 w_t^2\int_0^tq(s)\,\dd s\le v_*,
 \qquad w_t^2\int_0^tp_t(s)\,\dd s\le v_*,
\label{eq:weight-bracket}
\end{equation}
for every $t\in[0,T]$, then for every $R>0$,
\begin{equation}
 \mathbb P\left(\sup_{t\le T}\|N_t^{(w)}\|_\infty\ge R\right)
 \le Cd\left(1+A+\log\frac{w_0}{w_T}\right)
                 \exp\left(-\frac{cR^2}{v_*}\right).
 \label{eq:stochastic-convolution-weighted-maximal}
\end{equation}
\end{lemma}

\begin{proof}
It suffices to prove \eqref{eq:stochastic-convolution-weighted-maximal}, since \eqref{eq:stochastic-convolution-maximal} follows by taking $w\equiv1$. We choose a deterministic partition $0=t_0<\cdots<t_k=T$ such that
\begin{equation}\label{eq:B-op}
 \int_{t_j}^{t_{j+1}}\|\cB_s\|_{\infty\to\infty}\dd s\le\frac12,
 \qquad \frac{w_{t_j}}{w_{t_{j+1}}}\le2,
 \qquad k\le C\left(1+A+\log\frac{w_0}{w_T}\right).
\end{equation}
For each terminal time $t$ and coordinate $i$, define the martingale
\[
 N_i^{[t]}(s):=\int_0^{s\wedge t}(\Phi_{u,t}\,\dd\mathcal M_u)_i.
\]
On each cell $[t_{j},t_{j+1}]$, we have the integrated SDE
\[
 N_t^{(w)}=\frac{w_t}{w_{t_j}}N_{t_j}^{(w)}
 +\int_{t_j}^t\frac{w_t}{w_s}\cB_sN_s^{(w)}\,\dd s
 +w_t({\cal M}_t-{\cal M}_{t_j}).
\]
Using \eqref{eq:B-op} and $w_t/w_s\le1$ for $s\le t$, and absorbing the drift term, we obtain from this equation that 
\[
 \sup_{t_j\le t\le t_{j+1}}\|N_t^{(w)}\|_\infty
 \le2\|N_{t_j}^{(w)}\|_\infty
 +2w_{t_j}\sup_{t_j\le t\le t_{j+1}}\|{\cal M}_t-{\cal M}_{t_j}\|_\infty.
\]
The martingale $w_{t_j}({\cal M}_i(s)-{\cal M}_i(t_j))$, started at $t_j$, has quadratic variation density bounded by $w_{t_j}^2q(s)$, whose integral over the cell is at most
\[
 w_{t_j}^2\int_{t_j}^{t_{j+1}}q(s)\,\dd s
 \le v_*\frac{w_{t_j}^2}{w_{t_{j+1}}^2}\le4v_*.
\]
Moreover, the martingale $w_{t_j}N_i^{[t_j]}(s)$ has quadratic variation density bounded by $w_{t_j}^2p_{t_j}(s)$ with integral at most $v_*$, and its terminal value is the $i$-th component of \smash{$N_{t_j}^{(w)}$}. 
Applying \Cref{lem:matrix-maximal} with $k=1$ to these martingales and taking a union bound gives
\[
 \mathbb P\left(\sup_{t\le T}\|N_t^{(w)}\|_\infty\ge4r\right)
 \le Cd(k+1)\exp\left(-\frac{r^2}{8v_*}\right).
\]
Taking $r=R/4$ and using the bound on $k$ proves \eqref{eq:stochastic-convolution-weighted-maximal}.
\end{proof}

\subsection{Schatten bounds}\label{sec:positive}

Combining \Cref{lem:squared-schatten,lem:entropy-operator-output-sm} with the tools from \Cref{sec:common-tools}, we obtain the following bootstrap bound for $\mathcal Z_t$.

\begin{proposition}\label{prop:auxiliary-time-release-2}
Under \eqref{eq:window}, the following holds with probability at least $1-N^{-D}$ for every fixed $D>0$:
\begin{equation}\label{eq:auxiliary-barrier-improvement-2}
\sup_{0\le t\le\tau_*}\mathcal Z_t\le 1/2.
\end{equation}
\end{proposition}

We first bound the two drift terms on the RHS of \eqref{power:hessian} and the quadratic variation of the martingale.

\begin{lemma}\label{lem:linear-drift-2}
Under \eqref{eq:window}, for every stopped trajectory and $0\le t\le\tau_*$, we have 
\begin{equation}\label{eq:Acaps}
 \begin{aligned}
 \|\mathfrak A_{ab}(t)\|_{\op}
 &\le\frac{C}{\eta_t}\min\bigl\{
 (\log W)^{10}W^{-2+4/k_*}\vartheta_t\max_x\mathfrak D_x(t), (\log W)^{\fc}\mathfrak W(1+\rho_t(a,b))^{-k_*/2}\bigr\}.
 \end{aligned}
\end{equation}
\end{lemma}
\begin{proof}
Combining \eqref{eq:spatial-operator-output-sm} with $\|X_{ab}(t)\|_{\op}^2\le \|X_{ab}(t)\|_{S_2}^2= {\cL_{ab}(t)}$ and $J_2(t)\le1$, we obtain
\begin{equation}\label{eq:clipop}
 \|X_{ab}(t)\|_{\op}^2\le C (\log W)^{9} W^{-4}\vartheta_t (\iota_t(a,b)+\log W)\min\{1,W^2(1+\rho_t(a,b))^{-k_*}\}.
\end{equation}
We then apply Cauchy--Schwarz to get 
 \begin{align}\label{eq:CS-A}
 \|\mathfrak A_{ab}\|_{\op}
 &\le W^2\pB{\sum_{a'}|(S^{(\sB)}\difloop_1)_{a'}|^2\|X_{aa'}\|_{\op}^2}^{1/2} \pB{\sum_{a'}\|X_{a'b}\|_{\op}^2}^{1/2}.
 \end{align}
To control the RHS, we apply the following inequality: for every nonnegative vector $f$, 
\begin{equation}\label{eq:local-envelope}
 \sum_{a'}(\iota_t(a,a')+\log W)
       \min\{1,W^2(1+\rho_t(a,a'))^{-k_*}\}f_{a'} \le C (\log W)\ell_t^2W^{4/k_*}
       \max_x\sum_{a'}\omega_{x,t}(a')f_{a'}.
\end{equation}
Indeed, on $|a-a'|\le\ell_t$, \eqref{eq:omega-weights} gives
$\omega_{a,t}(a')\ge c\ell_t^{-2}(1+\iota_t(a,a'))$, which controls the contribution from this ball. Outside the ball, $\iota_t(a,a')=0$. Each dyadic annulus $2^l\ell_t\le|a-a'|<2^{l+1}\ell_t$ can be covered by $C2^{2l}$ balls of radius $\ell_t$. By \eqref{eq:localmass}, the sum of $f$ on each such ball is at most $C\ell_t^2\max_x\sum_{a'}\omega_{x,t}(a')f_{a'}$.
Summing the annular contributions proves \eqref{eq:local-envelope}, since
\(
 \sum_{l\ge0}2^{2l}\min\{1,W^22^{-k_*l}\}\le C_{k_*}W^{4/k_*}
\)
by splitting the sum at \smash{$2^l=W^{2/k_*}$}.

We now apply \eqref{eq:clipop} and \eqref{eq:local-envelope} with $f_{a'}=1$ and
\smash{$f_{a'}=|(S^{(\sB)}\difloop_1)_{a'}|^2$} to bound the RHS of \eqref{eq:CS-A} by 
 \begin{align}
 \|\mathfrak A_{ab}\|_{\op} \le C (\log W)^{10}\vartheta_t {\ell_t^2}W^{-2+4/k_*}\max_x\mathfrak D_x(t),\label{eq:case1}
 \end{align}
where we have also used that
\[
 \max_x\sum_{a'}\omega_{x,t}(a')|(S^{(\sB)}\difloop_1)_{a'}|^2
 \le C \max_x\mathfrak D_x(t)^2,
\]
since $\omega_{x,t}(a')$ changes by at most a constant factor when $a'$ is shifted by a unit coordinate step. Alternatively, Schatten H\"older gives
\smash{\(
 \|\mathfrak A_{ab}\|_{\op}
 \le W^2\mathfrak W\sum_{a'}\sqrt{\mathcal L_{aa'}\mathcal L_{a'b}}.
\)}
Applying $J_2\le1$ and summing the polynomial profile, we obtain
\begin{align}
 \|\mathfrak A_{ab}\|_{\op}
 \le C\eta_t^{-1} (\log W)^{\fc} \mathfrak W
                         (1+\rho_t(a,b))^{-k_*/2}.
\label{eq:case2}
 \end{align}
Here the convolution inequality for $\sqrt{1+\iota_t(a,b)}(1+\rho_t(a,b))^{-k_*/2}$ follows by splitting the sum over $a'$ into the region $|a-a'|\ge |a-b|/2$ and its complement, where $|a'-b|\ge |a-b|/2$. Combining \eqref{eq:case1} and \eqref{eq:case2} and using $\ell_t^2=\eta_t^{-1}$, we obtain \eqref{eq:Acaps}.
\end{proof}

\begin{lemma}\label{lem:block-entropy-variance-2}
Under \eqref{eq:window}, for every stopped trajectory, $0\le t\le\tau_*$, and $\star\in \{L,R\}$, we have 
\begin{equation}\label{eq:Vcaps}
 \|\Sigma_{ab}^{\star}\|_{\op}\le\frac{C(\log W)^{6+\fc}\vartheta_t}{\eta_tW^4},\quad \Tr\Sigma_{ab}^{\star}
 \le\frac{C(\log W)^{2\fc}}{\eta_tW^2}(1+\rho_t(a,b))^{-k_*}.
\end{equation}
\end{lemma}
\begin{proof}
For the left covariance matrix, \eqref{eq:K-loop-square} gives
 \begin{align}
\Sigma^L_{ab}=\sum_{x,y}K_{ab,(x,y)}K_{ab,(x,y)}^* =\frac1{\eta_t W^2}\sum_{a'\in\Zn} f_b(a')Z_{a}(a') ,
 \label{eq:left-block-covariance}
\end{align}
where we denote
\begin{equation}
Z_{a} (a'):=\eta_tW^2 E_a^{1/2}G_t E_{a'}G^*_t E_a^{1/2},\quad f_b(a'):=W^2\sum_{b'}S^{(\sB)}_{a'b'}\cL_{b'b}(t).\label{eq:def-Z-f}
\end{equation} 
By Ward's identity, we have 
\[\sum_{a'}Z_{a}(a')= E_a^{1/2}\p{\im G_t} E_a^{1/2}=:Y_{a},\quad \tr\qa{Z_{a}(a')}=\eta_tW^2 \cL_{aa'}(t). \]
Since $J_1(t)\le1$ up to $\tau_*$ and $\mathfrak{W}=\oo(1)$, the trace of $Y_a$ satisfies
\begin{equation}
 0<c_\kappa\le \tr Y_{a} = \im \tr\p{G_t E_a}   \le 2.
 \label{eq:trace-reference-bounds}
\end{equation}
Thus the normalized traces define a probability measure on $\Zn$,
\begin{equation}\label{eq:def-mu-a}
 \mu_a(a'):=\frac{\tr Z_a(a')}{\tr Y_a}
 =\frac{\eta_tW^2\cL_{aa'}(t)}{\tr Y_a}.
\end{equation}
The stopping bound $J_2(t)\le1$ gives
\[
 \mu_a(a')\le {C\ell_t^{-2}(\log W)^{\fc}} 
 (1+\iota_t( a,a' ))(1+\rho_t(a,a'))^{-k_*},\qquad
 f_b(a')\le C(\log W)^{\fc}(1+\iota_t( a',b )).
\]
In particular, the first bound gives $\|\mu_a\|_2\le C(\log W)^{\fc}/\ell_t$ by \eqref{eq:moments-iota}. Choose $\lambda_0>0$ small enough that 
\[e^{\lambda_0f_b(a')/(\log W)^{\fc}} \le C+C[\ell_t/(1+|a'-b|)]^\gamma\1_{|a'-b|\le\ell_t}\] 
for a fixed $0<\gamma<1/2$. Applying Cauchy--Schwarz, we get 
\[
 \sum_{a'}\mu_a(a')e^{\lambda_0f_b(a')/(\log W)^{\fc}}
 \le C +C\|\mu_a\|_2
 \qBB{\sum_{|a'-b|\le\ell_t}
       \pbb{\frac{\ell_t}{1+|a'-b|}}^{2\gamma}}^{1/2} \le C (\log W)^{\fc} .
\]
Now, applying \Cref{lem:spatial-entropy} to $Z_a(a')$ with $p=\infty$ and weights $\lambda_0 f_b(a')/(\log W)^{\fc}$, we obtain 
\begin{equation}\label{eq:entropy-norm-application}
 \|\Sigma_{ab}^L\|
 \le\frac{C(\log W)^{\fc}}{\eta_tW^2}
 \|Y_a\|\log\frac{C (\log W)^{\fc} \Tr Y_a}{\|Y_a\|}.
\end{equation}
Since $Y_a=\im X_{aa}$ has rank at most $W^2$, $\Tr Y_a\le W^2\|Y_a\|$ and $\|Y_a\|\le O_{\op}(t)$. By \eqref{eq:entropy-operator-output-sm},
\[
 \|\Sigma_{ab}^L\|
 \le\frac{C(\log W)^{1+\fc}}{\eta_tW^2}\|Y_a\|
 \le\frac{C(\log W)^{6+\fc}\vartheta_t}{\eta_tW^4},
\]
which implies the first bound in \eqref{eq:Vcaps} for $\star=L$. The case $\star=R$ follows by symmetry.

For the trace bound, we get from \eqref{eq:K-loop-square} that
\begin{equation}\label{eq:common-trace-covariance}
 \Tr\Sigma^L_{ab}=\Tr\Sigma^R_{ab}
 =W^2\sum_{a',b'}S^{(\sB)}_{a'b'}\mathcal L_{aa'}\mathcal L_{b'b}.
\end{equation} 
Using $J_2(t)\le 1$ and the following polynomial convolution bound, we conclude the second bound in \eqref{eq:Vcaps}:
\[\sum_{a'}(1+\iota_t(a,a'))(1+\rho_t(a,a'))^{-k_*}
          (1+\iota_t(a',b))(1+\rho_t(a',b))^{-k_*} \le C \ell_t^2(1+\rho_t(a,b))^{-k_*}.\]
For $|a-b|\le2\ell_t$, this inequality follows from Cauchy--Schwarz and \eqref{eq:moments-iota}. Otherwise, we split the sum according to $|a-a'|\ge|a-b|/2$ or $|a'-b|\ge|a-b|/2$. On each part, the polynomial decay along the longer leg is bounded by $C(1+\rho_t(a,b))^{-k_*}$, while summing the remaining factor using \eqref{eq:moments-iota} and its polynomial decay contributes a factor $C\ell_t^2$.
\end{proof}

\begin{lemma}\label{lem:drifts-QV-2}
Let $\Gamma_{ab}$ denote the quadratic-variation density of the martingale in \eqref{eq:d-mathfrak-h} for $n=p_*$,
\begin{equation}\label{eq:QV-Zt}
 \Gamma_{ab}(t):=\partial_t\langle\mathfrak M_{p_*,ab}\rangle_t.
\end{equation}
Under \eqref{eq:window}, for every stopped trajectory and $0\le t\le\tau_*$, we have
\begin{equation}\label{simp:power-gram}
 \Gamma_{ab}(t)
 \le\frac{C (\log W)^{20}\vartheta_t^2}
              {\eta_tW^{6-4/k_*}}
 \mathfrak R_{p_*,ab}(t).
\end{equation}
\end{lemma}
\begin{proof}
Writing $Q=Q_{p_*,ab}^*$ and grouping coordinate indices in \eqref{eq:K-loop} by blocks, we get
\begin{align}
 \Gamma_{ab}
\le \sum_{x,y}|\Tr(QK_{ab,(x,y)})|^2
 = W^2\sum_{a',b'}S^{(\sB)}_{a'b'}\|X_{b'b}QX_{aa'}\|_{S_2}^2.\label{eq:bound-QV-Z}
\end{align}
We bound $Q$ in the HS norm and use \eqref{eq:clipop} to bound both $X$-factors in the operator norm. Applying Cauchy--Schwarz in $(a',b')$ and summing over shells as in the proof of \eqref{eq:local-envelope} gives an additional factor of $C(\log W)^2\ell_t^2W^{4/k_*}$. We omit the details of this argument. Combining this bound with \eqref{eq:power-gradient} gives \eqref{simp:power-gram}.\end{proof}

We now combine the above estimates to complete the proof of \Cref{prop:auxiliary-time-release-2}.

\begin{proof}[\bf Proof of \Cref{prop:auxiliary-time-release-2}]
Stopping \eqref{power:hessian} at $\tau_*$ and dividing by $\mathfrak s_t(a,b)^2$ gives
\begin{align}\label{eq:dynamics-Ht}
 \wt{\mathfrak R}_{ab}(t):=\frac{\mathfrak R_{p_*,ab}(t\wedge\tau_*)}{\fs_t(a,b)^2}
 \le \frac{\mathfrak R_{p_*,ab}(0)}{\fs_t(a,b)^2}
 +\frac1{\fs_t(a,b)^2}\int_0^t\1_{\{u\le\tau_*\}}
 \qa{\mathfrak K_{ab}^{(1)}(u)+\mathfrak K_{ab}^{(2)}(u)}\dd u
 +N_{ab}^{(t)}(t),
\end{align}
where $\mathfrak K_{ab}^{(1)}:=2\sqrt{\mathfrak R_{p_*,ab}}\,\|\mathfrak A_{ab}\|_{\op}$ and $\mathfrak K_{ab}^{(2)}:=C_{\mathfrak R}p_*(\|\Sigma^L_{ab}\|_{\op}+\|\Sigma^R_{ab}\|_{\op})$, and the martingale term (with the terminal time $t$ fixed) is defined by
\begin{align}\label{eq:def-Nab}
 N_{ab}^{(t)}(s):=\frac1{\fs_t(a,b)^2}\int_0^s
 \1_{\{u\le\tau_*\}}\dd\mathfrak M_{p_*,ab}(u).
\end{align}
The initial term is at most $(\log W)^{-10}=\oo(1)$. To control the drift and martingale terms in \eqref{eq:dynamics-Ht}, we repeatedly use the following estimate, valid for every fixed $r>0$ and all $A\ge1$:
\begin{equation}\label{eq:clock}
 \int_0^t\min\ha{1,A\pa{1+\rho_u(a,b)}^{-r}}\frac{\dd u}{\eta_u} \le C_{\kappa,r}[ 1+\iota_t(a,b) +\log A].
\end{equation}
To prove \eqref{eq:clock}, first suppose that $a\ne b$. Changing variables to $v=\rho_u(a,b)$ gives $\dd u/\eta_u=-2(\im m)^{-1}\dd v/v$. The contribution from $v\le A^{1/r}$ is at most $C_{\kappa,r}(1+\iota_t(a,b)+\log A)$. For $v>A^{1/r}$, integrating $Av^{-r-1}$ gives a bound of $C_{\kappa,r}$. When $a=b$, the estimate follows directly from \smash{$\int_0^t\dd u/\eta_u\le C_\kappa(1+\iota_t(0,0))$}.

The integral on the left-hand side (LHS) of \eqref{eq:clock} is also bounded by $\vartheta_t$. Combining the two bounds, we obtain, for $1\le A\le W^{C_0}$ with $C_0>0$ fixed,
\begin{equation}\label{eq:minimum-clock}
 \int_0^t\min\ha{1,A\pa{1+\rho_u(a,b)}^{-r}}\frac{\dd u}{\eta_u}
 \le C \min\{\vartheta_t,\iota_t(a,b)+\log W\}.
\end{equation}

\noindent{\bf Bounding $\mathfrak K^{(1)}$.}
By applying \eqref{eq:minimum-clock} to the integral of \eqref{eq:Acaps}, and using $\max_x\mathfrak D_x(u)\le\mathfrak W_{\rm loc}$, $\mathfrak s_u\le\mathfrak s_t$, $\iota_t(a,b)\le C\vartheta_t\le C\logpara$, and \eqref{eq:window}, we obtain 
 \begin{align}
 \frac1{\mathfrak s_t(a,b)^2}\int_0^{t\wedge\tau_*}\mathfrak K_{ab}^{(1)}(u)\,\dd u
 &\le\frac2{\mathfrak s_t(a,b)}\int_0^{t\wedge\tau_*}\|\mathfrak A_{ab}(u)\|\,\dd u  \le C(\log W)^{11/2} W^{4/k_*}\mathfrak W_{\rm loc} \sqrt{\vartheta_t(\iota_t(a,b)+\log W)} \notag\\
 &\le C(\log W)^{23/2} \logpara^2 {W^{-2+4/k_*}}=\oo(1).\label{eq:powerdrift}
\end{align}

\noindent{\bf Bounding $\mathfrak K^{(2)}$.}
By \eqref{eq:Vcaps}, we have
\[
 \|\Sigma_{ab}^{L}(u)\|\vee \|\Sigma_{ab}^{R}(u)\|
 \le\frac{C (\log W)^{6+\fc}\vartheta_u}{\eta_uW^4}
 \min\{1, W^2(1+\rho_u(a,b))^{-k_*}\}.
\]
Thus \eqref{eq:minimum-clock} gives
\begin{equation}\label{eq:normalized-production}
 \frac1{\mathfrak s_t(a,b)^2}\int_0^{t\wedge\tau_*}\mathfrak K_{ab}^{(2)}(u)\,\dd u
 \le C p_*(\log W)^{-3+\fc} \le C (\log W)^{-2+\fc}=\oo(1).
\end{equation}

\noindent{\bf Bounding $N_{ab}^{(t)}$.}
Using the Schatten norm monotonicity, we obtain
\begin{equation}\label{eq:common-Schatten-mono}
 \mathfrak R_{p_*,ab}(u)\le W^{-4/p_*}\|X_{ab}(u)\|_{S_2}^2.
\end{equation}
For fixed terminal time $t$, denote
\begin{equation}\label{eq:common-terminal-cutoff}
 f_{ab}^{(t)}(u):=\min\{1,\|X_{ab}(u)\|_{S_2}^2/
                         (W^{4/p_*}\mathfrak s_t(a,b)^2)\}.
\end{equation}
For $u\le t\wedge\tau_*$, the stopping condition gives $\mathfrak R_{p_*,ab}(u)\le\mathfrak s_u(a,b)^2\le\mathfrak s_t(a,b)^2$. Combining this with \eqref{eq:common-Schatten-mono} yields $\mathfrak R_{p_*,ab}(u)\le\mathfrak s_t(a,b)^2f_{ab}^{(t)}(u)$. By the two-loop bound $J_2(u)\le 1$, we have 
\begin{equation*}
 f_{ab}^{(t)}(u)
 \le \min\{1, CW^2(1+\rho_u(a,b))^{-k_*}\}.
\end{equation*}
Consequently, \Cref{lem:drifts-QV-2} and \eqref{eq:minimum-clock} yield
\begin{equation}\label{eq:powerbracket}
 \begin{split}
 \langle N_{ab}^{(t)}\rangle_t
 &\le\frac{C(\log W)^{11}\vartheta_t}
              {W^{2-4/k_*} (\iota_t(a,b)+\log W)}
       \int_0^{t\wedge\tau_*}f_{ab}^{(t)}(u)\,\frac{\dd u}{\eta_u} \le\frac{C (\log W)^{11}\vartheta_t}{W^{2-4/k_*}}.
 \end{split}
\end{equation}
The scalar case of \Cref{lem:matrix-maximal} therefore gives
\begin{equation}\label{eq:fixed-terminal-maximal}
 \mathbb P\left(\sup_{0\le s\le t}|N_{ab}^{(t)}(s)|\ge\frac18\right)
 \le2\exp\left[-\frac{c W^{2-4/k_*}}
                         {(\log W)^{11}\logpara}\right].
\end{equation}

The bound \eqref{eq:fixed-terminal-maximal} at \(s=t\) provides the control needed for \eqref{eq:dynamics-Ht}. We now extend the estimate on \smash{\(|N_{ab}^{(t)}(t)|\)} uniformly to all \(t\le t_*\). For each pair $(a,b)$, define a sequence of terminal times by $t_{ab,0}=0$, $t_{ab,n_{ab}}=t_*$, and
\begin{align*}
 t_{ab,j}&:=\inf\left\{s\in[0,t_*]:
\fs_s(a,b)^2\ge2^j\fs_0(a,b)^2\right\},
 \quad 1\le j<n_{ab},\quad n_{ab}:=\left\lceil\log_2
 \frac{\fs_{t_*}(a,b)^2}{\fs_0(a,b)^2}\right\rceil.
\end{align*}
The continuity and monotonicity of \(\fs_u(a,b)\) in \(u\) ensure that these times are well defined. Since $\vartheta_0=1$, $\vartheta_{t_*}\le C\logpara$, and $\iota_{t_*}(a,b)+\log W\le C\logpara$, the ratio $\fs_{t_*}(a,b)^2/\fs_0(a,b)^2$ is at most $C\logpara^2$. Thus $n_{ab}\le C(1+\log\logpara)$ uniformly in $(a,b)$.
For \(s\in[t_{ab,j-1},t_{ab,j}]\), the definition gives
\(
 1\le {\fs_{t_{ab,j}}(a,b)^2}\big/
             {\fs_s(a,b)^2}\le2,
\)
which implies that \smash{$\abs{N_{ab}^{(s)}(s)}\le 2\abs{N_{ab}^{(t_{ab,j})}(s)}$}. Thus, \eqref{eq:fixed-terminal-maximal} yields, for any $1\le j\le n_{ab}$,
\begin{equation*}
 \Pp\pa{\sup_{t_{ab,j-1}\le s\le t_{ab,j}}|N_{ab}^{(s)}(s)|\ge \frac{1}{4}} \le \Pp\pa{\sup_{0\le s\le t_{ab,j}}|N_{ab}^{(t_{ab,j})}(s)|\ge \frac{1}{8}} \le 2\exp\left[-\frac{c W^{2-4/k_*}} {(\log W)^{11}\logpara}\right]. 
\end{equation*}
Taking a union bound over all intervals, we obtain
\begin{equation}\label{eq:running-terminal-maximal}
 \Pp\left(\sup_{0\le s\le t_*}|N_{ab}^{(s)}(s)|\ge \frac14\right)
 \le C(1+\log\logpara)
 \exp\left[-\frac{cW^{2-4/k_*}}{(\log W)^{11}\logpara}\right].
\end{equation}

Substituting the bound \eqref{eq:running-terminal-maximal} and the drift estimates \eqref{eq:powerdrift} and \eqref{eq:normalized-production} into \eqref{eq:dynamics-Ht}, then taking a union bound over $(a,b)$, gives
\begin{equation}\label{a1:cl:eq:powerprob}
 \begin{split}
 \mathbb P\left(\sup_{0\le t\le\tau_*}\mathcal Z_t\ge\frac12\right)
 &\le CL^4(\log\logpara)
       \exp\left[-\frac{c W^{2-4/k_*}} {(\log W)^{11}\logpara}\right] \le N^{-D}.
 \end{split}
\end{equation}
Here the last step follows from \eqref{eq:window} and \eqref{eq:spatialorders}, since $(2+\epsilon_{\rm sm})(1-2/k_*)>2$.
This proves \eqref{eq:auxiliary-barrier-improvement-2}.
\end{proof}

\section{Light-weight bootstrap estimates}\label{sec:small-trace}

To establish the bootstrap estimate for \(J_1\), we use the Gaussian integration-by-parts argument in \Cref{lem:coordinate-test}. For \(J_{\rm loc}\), we use the following weighted counterpart of \Cref{lem:coordinate-test}. Define the following more general notion of \eqref{eq:direction-D}: for a vector \smash{$\zeta\in\C^{\Zn}$}, let 
\[\mathscr D_{t,\zeta}^+:=\sum_b\zeta_b\mathscr D_{t,b},\quad \mathscr D_{t,\zeta}^-:=\sum_b\zeta_b\overline{\mathscr D_{t,b}}.\] 

\begin{lemma}\label{lem:weighted-trace-test}
Fix $0\le t\le t_*$ and $x\in\Zn$. Let $0\le\chi\le1$ be a smooth cutoff function of $H_t$ with bounded entry derivatives.
Suppose that, on $\supp\chi$,
\begin{equation}\label{eq:weighted-trace-assumptions}
 \max_a|(\difloop_1)_a(t)|=\oo(1),\qquad
 \max_{\sigma\in\{+,-\}}|\mathscr D_{t,b}(\difloop_1^\sigma)_a(t)|
 \le \xi_1[1+\iota_t(a,b)],\quad \forall a,b\in\Zn,
\end{equation}
and, for every $\zeta\in\C^{\Zn}$,
\begin{equation}\label{eq:weighted-cutoff-assumptions}
 \max_{\sigma\in\{+,-\}}|\mathscr D_{t,\zeta}^\sigma\chi|
 \le\Xi_{\chi,2}\|\zeta\|_{\ell^2}
       +\Xi_{\chi,4/3}\|\zeta\|_{\ell^{4/3}},
\end{equation}
where $\xi_1,\Xi_{\chi,2},\Xi_{\chi,4/3}\ge0$ are deterministic control parameters. Then, for every $p\in2\N$,
\begin{equation}\label{eq:weighted-trace-test}
 \pa{\E\absa{\chi\mathfrak D_x(t)}^p}^{1/p}
 \le C\left[
 \sqrt{\xi_1}\pa{\sqrt{\vartheta_t}+\sqrt p}
 +p\Xi_{\chi,2}{\sqrt{\vartheta_t}}/{\ell_t}
 +p\Xi_{\chi,4/3}/\sqrt{\ell_t}\right],
\end{equation}
where $C>0$ depends only on $\kappa$ and $k_*$.
\end{lemma}
\begin{proof}
On the $\ell^2(\omega_{x,t})$-space with \smash{$\|f\|_{\ell^2(\omega_{x,t})}^2:=\sum_a\omega_{x,t}(a)|f_a|^2$}, the stable kernel satisfies 
\begin{equation}\label{eq:l2-omega}
 \left\|\Theta_t^+ f\right\|_{\ell^2(\omega_{x,t})}
 \le C \|f\|_{\ell^2(\omega_{x,t})}.
\end{equation}
In fact, \eqref{eq:l2-omega} follows easily by conjugating the kernel $\Theta_t^+$ with $\omega_{x,t}^{1/2}$ and applying \eqref{eq:tools-P-point-short} and \eqref{eq:omega-weights}. Let 
\[
 Y:=\mathfrak D_x(t)^2,\qquad
 \zeta:=m^{-2}(\Theta_t^+)^\top (\omega_{x,t}\overline{\difloop_1 (t)}),
\]
where $\omega_{x,t}\overline{\difloop_1(t)}$ denotes the entrywise product of the two vectors.
The exponential decay in \eqref{eq:tools-P-point-short}, \eqref{eq:weightnorm}, and H\"older's inequality give
\[
 \|\zeta\|_{\ell^2}\le C \pb{{\sqrt{\vartheta_t}}/{\ell_t}}\sqrt Y, \qquad \|\zeta\|_{\ell^{4/3}}\le C \sqrt {Y/\ell_t}.
\]
Thus, \eqref{eq:weighted-cutoff-assumptions} implies  
\begin{equation}
|\mathscr D_{t,\zeta}^+\chi|\le C \pb{\Xi_{\chi,2}\sqrt{\vartheta_t}/\ell_t+\Xi_{\chi,4/3}\ell_t^{-1/2}}\sqrt Y.\label{eq:cond-derv1}
\end{equation}
Furthermore, using \eqref{eq:weighted-trace-assumptions}, the exponential decay in \eqref{eq:tools-P-point-short}, and $\iota_t(a,b)\le C\vartheta_t$, we obtain 
 \begin{equation}
 \sum_{a,b}\omega_{x,t}(a)\absa{m^{-2}\Theta_t^+(a,b)}
       \max_{\sigma\in\{+,-\}}|\mathscr D_{t,b}(\difloop_1^\sigma)_a(t)|
 \le C \xi_1\vartheta_t.
\label{eq:cond-derv2}
\end{equation}
For the derivative of $Y$, we write
\smash{\(
 \mathscr D_{t,\zeta}^+Y
 =\sum_{a,b}\omega_{x,t}(a)\zeta_b
 \qb{\overline{(\difloop_1)_a}\,\mathscr D_{t,b}(\difloop_1)_a
 +(\difloop_1)_a\mathscr D_{t,b}\overline{(\difloop_1)_a}}.
\)}
Substituting the definition of $\zeta$ and using \eqref{eq:weighted-trace-assumptions}, we obtain
 \begin{equation}
 |\mathscr D_{t,\zeta}^+Y|
 \le C\xi_1\sum_{a,a'}\omega_{x,t}(a)\omega_{x,t}(a')
 |(\difloop_1)_a||(\difloop_1)_{a'}|
 \sum_b|\Theta_t^+(a',b)|[1+\iota_t(a,b)]\le C\xi_1Y.
\label{eq:cond-derv3}
\end{equation}
In the second step, we used $\sum_b|\Theta_t^+(a',b)|[1+\iota_t(a,b)]\le C[1+\iota_t(a,a')]$ by the exponential decay in \eqref{eq:tools-P-point-short}, the Cauchy--Schwarz inequality, and the definition \eqref{eq:Jloc}.

We now apply \eqref{eq:LW-expansion} with $F_a=\omega_{x,t}(a)\overline{(\difloop_1)_a}\chi^pY^{p/2-1}$ and sum over $a$, as in the proof of \Cref{lem:coordinate-test}.
The first quadratic term in $\cD_1$ is $\oo(1)\E[\chi^pY^{p/2}]$ by \eqref{eq:l2-omega} and the first assumption in \eqref{eq:weighted-trace-assumptions}. The three derivative bounds \eqref{eq:cond-derv1}--\eqref{eq:cond-derv3} and H\"older's inequality then give an inequality analogous to \eqref{eq:self-Ya}:
\begin{align}
 \E\absa{\chi\mathfrak D_x(t)}^p
 \le &~\oo(1)\E\absa{\chi\mathfrak D_x(t)}^p
 +C\xi_1(\vartheta_t+p)
 \pa{\E\absa{\chi\mathfrak D_x(t)}^p}^{1-2/p}\notag \\
 &~+Cp\left(
 \Xi_{\chi,2} {\sqrt{\vartheta_t}}/{\ell_t}
 +\Xi_{\chi,4/3}\ell_t^{-1/2}\right)
 \pa{\E\absa{\chi\mathfrak D_x(t)}^p}^{1-1/p}.
 \label{eq:self-local-trace}
\end{align}
Again, applying Young's inequality and solving the inequality proves \eqref{eq:weighted-trace-test}.
\end{proof}

To apply \Cref{lem:coordinate-test,lem:weighted-trace-test}, we need to check the assumptions in \eqref{eq:assumptions-GIP}, \eqref{eq:weighted-trace-assumptions}, and \eqref{eq:weighted-cutoff-assumptions} with certain control parameters $\Xi_1$, $\Xi_\chi$, $\xi_1$, $\Xi_{\chi,2}$, and $\Xi_{\chi,4/3}$. The derivative bounds used to verify these assumptions are given in \Cref{lem:far-Gaussian-support}. To state it, we construct a smooth cutoff function as follows.
Fix any $0\le t\le t_*$, and let $\phi_1,\phi_2,\phi_{\op}>0$ be deterministic control parameters with $\phi_1=\oo(1)$. We define 
\begin{equation}\Xi_{2}(a,b):=\phi_2 W^{-2}[1+\iota_t(a,b)],\quad \Xi_{\op}(a,b):=\phi_{\op}{(\log W)}^2 W^{-2}\sqrt{\vartheta_t(\log W+\iota_t(a,b))},\label{eq:xi-parameters}
\end{equation} 
and introduce the normalized maximum
\begin{equation}\label{new:three-list-cutoff}
\mathcal X_t:= \max\pa{\max_a |(\difloop_1)_a(t)|/\phi_1 ,\ \max_{a,b}\cL_{ab}/\Xi_{2}(a,b),\  \max_{a,b} \|X_{ab}\|_{\rm op} /\Xi_{\op}(a,b)}.
\end{equation}
We now construct a cutoff function $0\le \chi_0 \le 1$ such that $\chi_0\equiv 1$ on $\{\mathcal X_t\le 1\}$ and $\chi_0\equiv 0$ on $\{\mathcal X_t\ge 2\}$. To define it explicitly, we use the following list of smooth ratios with $q\in \N$: for $a,b\in\Zn$,
\begin{equation}  
 \phi_1^{-1}\sqrt{|(\difloop_1)_a(t)|^2+\epsilon^2\phi_1^2},\quad
 {\cL_{ab}}/{\Xi_2(a,b)},\quad
{\left[\tr\bigl(X_{ab}^*X_{ab}
       +\epsilon^2\Xi_{\op}(a,b)^2P_b\bigr)^q\right]^{1/(2q)}}\big/
      {\Xi_{\op}(a,b)}.\label{eq:norm-ratios}
\end{equation}
Choose an integer $q\ge C_q \log W $ for a large absolute constant $C_q>0$ such that $W^{1/q}\le1.01$. Moreover, fix a sufficiently small constant $0<\epsilon\le 0.01$ such that all the smooth ratios are at most 1.02 on $\{\cal X_t\le 1\}$. 
We then enumerate this list as $\mathsf r_1,\ldots,\mathsf r_{k}$, and choose an even integer $p\ge C\log k$ with
$k^{1/p}\le1.02$. Define
\begin{equation}
 R:=\pB{\sum_{i=1}^k \mathsf r_i^p}^{1/p},\qquad
 \chi_0:=\psi(R),\label{eq:def-chi0}
\end{equation}
where $0\le \psi \le 1$ is a smooth non-increasing function on $\mathbb R$ such that it equals 1 on $(-\infty,1.1]$ and vanishes on $[1.8,\infty)$. Thus $\chi_0$ is smooth, equals 1 on $\{\mathcal X_t\le1\}$, and has support in $\{\mathcal X_t\le1.8\}\subset\{\mathcal X_t<2\}$.

\begin{lemma}[Bounds on directional derivatives]\label{lem:far-Gaussian-support}
Fix $0\le t\le t_*$. Let $\chi_0$ be the cutoff constructed above, and let
$\Phi_2,\Phi_{\op}>0$ be deterministic control parameters such that $\max_{a,b} \cL_{ab} \le\Phi_2$ and
$\max_{a,b} \|X_{ab}\|_{\rm op} \le\Phi_{\op}$ on $\supp\chi_0$. 
There exists an absolute constant $C>0$ such that 
\begin{equation}\label{eq:spatial-cutoff-derivatives}
 \max_{\sigma\in\{+,-\}}|\mathscr D_{t,\zeta}^\sigma\chi_0|
 \le C \Phi_{\op}\pb{\wt O_2(t)}^{1/2}\phi_1^{-1}\|\zeta\|_{\ell^2}+C\Phi_{\op}\pb{\wt O_4(t)}^{1/4}\|\zeta\|_{\ell^{4/3}} ,
\end{equation}
where $\wt O_4(t)$ is defined in \eqref{eq:def-Os}, and we denote \[\wt O_2(t):=\max_a\sum_b(\cL_{ab}^2+\cL_{ba}^2).\]
Moreover, on $\supp\chi_0$, there exists an absolute constant $C>0$ such that 
\begin{align}\label{new:spatial-trace-test}
\max_{\sigma\in\{+,-\}}|\mathscr D_{t,a'} (\difloop_1^\sigma)_{a}(t)| &\le \Phi_{\op}\min\hb{\Phi_2,\max_{b':|b'-a'|\le 1}\p{\mathcal L_{ab'}+\mathcal L_{b'a}}},\quad  a,a'\in\Zn,\\
 \max_{a'} |\mathscr D_{t,a'} \chi_0| &\le C(\Phi_{\op}^2+\phi_1^{-1}\Phi_2\Phi_{\op} ).\label{new:spatial-trace-test-2}
\end{align}
\end{lemma}
\begin{proof}
The time $t$ is fixed throughout the proof. By symmetry, it suffices to prove the assertions for the plus derivatives $\mathscr{D}_\zeta\equiv \mathscr{D}_\zeta^+$ and $\mathscr{D}_b$. 
Taking the trace of the diagonal block in the first formula of \eqref{new:exact-two-seams} gives
\[
 \mathscr D_{a'}(\difloop_1^\sigma)_a(t)
 =-\sum_{b'}S^{(\sB)}_{a'b'}
       \tr(X_{ab'}^\sigma X_{b'a'}^+X_{a'a}^\sigma).
\]
Applying the Schatten H\"older inequality \eqref{eq:block-Schatten-product}, with the middle $X$-matrix in operator norm, we get 
\begin{equation}
 |\mathscr D_{a'}(\difloop_1^\sigma)_a(t)|
 \le\Phi_{\op}\sum_{b'}S^{(\sB)}_{a'b'}
       \|X_{ab'}^\sigma\|_{S_2}\|X_{a'a}^\sigma\|_{S_2}
 \le  \Phi_{\op}\min\hb{\Phi_2,\max_{b':|b'-a'|\le 1}\p{\mathcal L_{ab'}+\mathcal L_{b'a}}},
\label{eq:spatial-derv-LW1}
\end{equation}
which concludes \eqref{new:spatial-trace-test}.
Taking the square of the RHS and bounding the sum over $a'$ using the AM--GM inequality, we obtain
\begin{equation}
 \sum_{a'}|\mathscr D_{a'}(\difloop_1^\sigma)_a(t)|^2
 \le C\Phi_{\op}^2\sum_{a'} \pa{\cL_{a'a}+\cL_{aa'}}^2
 \le C\Phi_{\op}^2\wt O_2(t).
\label{eq:spatial-derv-LW}
\end{equation}
Combining it with Cauchy--Schwarz and using \eqref{eq:spatial-derv-LW1}, we bound the derivatives of the first ratio in \eqref{eq:norm-ratios} by 
\begin{equation}
      \begin{split}            
\phi_1^{-1}\left|\mathscr{D}_\zeta\sqrt{|(\difloop_1)_a(t)|^2+\epsilon^2\phi_1^2}\right| &\le C\phi_1^{-1}\Phi_{\op}\pb{\wt O_2(t)}^{1/2}\|\zeta\|_{\ell^2} ,\\ 
\phi_1^{-1}\left|\mathscr{D}_{a'} \sqrt{|(\difloop_1)_a(t)|^2+\epsilon^2\phi_1^2}\right| &\le C\phi_1^{-1}\Phi_{\op}\Phi_2 .
      \end{split}
\label{eq:spatial-derv-LW2}
\end{equation}

We next consider the directional derivatives of the second and third ratios in \eqref{eq:norm-ratios}. For the first formula in \eqref{new:exact-two-seams}, if $|a'-b'|\le1$, then the triangle inequality implies that at least one of $|a-b'|$ and $|a'-b|$ is at least $(|a-b|-1)/2$. For either such edge, say $(a,b')$, the parameters in \eqref{eq:xi-parameters} satisfy
\(\Xi_2(a,b')\le C\Xi_2(a,b)\) and \( \Xi_{\op}(a,b')\le C\Xi_{\op}(a,b)\) for an absolute constant $C>0$.
Then, for the operator norm ratio, we bound $X_{ab'}^\sigma$ by $C\Xi_{\op}(a,b)$ and the middle factor $X_{b'a'}^+$ by $\Phi_{\op}$ on $\supp\chi_0$:
\begin{equation}
{\|X_{ab'}^\sigma X_{b'a'}^+X_{a'b}^\sigma\|_{\op}}\big/
      {\Xi_{\op}(a,b)}
 \le C\Phi_{\op}\|X_{a'b}^\sigma\|_{\op}.
\label{eq:spatial-derv-op0}
\end{equation}
For the $S_2$-norm ratio, differentiate $\tr(X_{ab} X_{ab}^*)$ using \eqref{new:exact-two-seams}. We then bound the unchanged factor (say, $X_{ab}^*=X_{ba}^-$) and the long $X$-factor (say, $X_{ab'}$ with $|a-b'|\ge (|a-b|-1)/2$) in the $S_2$-norm and the remaining two $X$-factors by their operator norms, i.e.,
\begin{equation}
 \frac{\|X_{ba}^-\|_{S_2}
        \|X_{ab'}\|_{S_2}
        \|X_{b'a'}\|_{\op}\|X_{a'b}\|_{\op}}
      {\Xi_2(a,b)}
 \le C\Phi_{\op}\|X_{a'b}\|_{\op}.  \label{eq:spatial-derv-S20}
\end{equation}
The derivative acting on $X_{ab}^*$ can be treated similarly using $\mathscr D_{a'}(X^*)=(\mathscr D_{a'}^{-}X)^*$. With \eqref{eq:spatial-derv-S20} and $B_{ab}:= \|X_{ab}\|_{\op}\vee\|X_{ba}\|_{\op} $, we can bound the derivatives of the second ratio in \eqref{eq:norm-ratios} by
\begin{align}
\Xi_2(a,b)^{-1}\bigl|\mathscr{D}_{a'} \cL_{ab}\bigr| \le C\Phi_{\op}\sum_{b'}S^{(\sB)}_{a'b'}
       (B_{ab'}+B_{a'b}+B_{aa'}+B_{b'b}) \le C\Phi_{\op}^2.
  \label{eq:spatial-derv-S2}
\end{align}
Similarly, using \eqref{eq:spatial-derv-op0}, H\"older's inequality, and $W^{1/q}\le 1.01$, we can also bound the derivatives of the third ratio in \eqref{eq:norm-ratios} by the RHS of \eqref{eq:spatial-derv-S2}. Consequently, applying H\"older's inequality with exponents $(4,4/3)$ and using \smash{$\sum_b B_{ab}^4\le C\wt O_4(t)$}, we can bound the $\mathscr{D}_\zeta$ derivative of the second and third ratios in \eqref{eq:norm-ratios} by  
\begin{align}
\Xi_2(a,b)^{-1}\bigl|\mathscr{D}_{\zeta} \cL_{ab}\bigr| + \Xi_{\op}(a,b)^{-1} \left|\mathscr{D}_{\zeta} \left[\tr\bigl(X_{ab}^*X_{ab}
       +\epsilon^2\Xi_{\op}(a,b)^2P_b\bigr)^q\right]^{1/(2q)}\right|  \le C\Phi_{\op}\pb{\wt O_4(t)}^{\frac 1 4}\|\zeta\|_{\ell^{4/3}} .
  \label{eq:spatial-derv-S2-op}
\end{align}

With \eqref{eq:spatial-derv-LW2}, \eqref{eq:spatial-derv-S2}, \eqref{eq:spatial-derv-S2-op}, and the bound $k^{1/p}\le1.02$, we can readily bound the directional derivatives of $\chi_0$ defined in \eqref{eq:def-chi0} as in \eqref{eq:spatial-cutoff-derivatives} and \eqref{new:spatial-trace-test-2}. \end{proof}

Let $\chi_0$ be the cutoff function from \eqref{eq:def-chi0}, with
\[
 \phi_1=(2\log W)^{-1},\qquad \phi_2=(\log W)/2,\qquad
 \phi_{\op}=\frac12(\log W)^{7/2}.
\]
Under this choice, \eqref{eq:xi-parameters} and \eqref{eq:spatial-operator-output-sm} imply that, for each fixed $0\le t\le t_*$ and sufficiently large $N$, $\chi_0\equiv 1$ on $\{\mathscr{J}\le 1\}$. On $\supp\chi_0$, the bound $\mathcal X_t\le1.8$ gives
\begin{equation}\label{eq:base}
\begin{aligned}
 &\|\difloop_1(t)\|_\infty\le(\log W)^{-1},\quad
 \mathcal L_{ab}(t)\le (\log W)W^{-2}[1+\iota_t(a,b)],\\ 
 &\|X_{ab}(t)\|_{\op}\le (\log W)^{11/2}W^{-2} \sqrt{\vartheta_t(\iota_t(a,b)+\log W)}.
\end{aligned}
\end{equation}
Applying Ward's identity and using $\max_a |\tr(E_aG_t)|\le2$ under the first condition of \eqref{eq:base}, we obtain  
\begin{equation}\label{eq:base2}
\max_a \sum_b \pa{\mathcal L_{ab}+\mathcal L_{ba}}\le 4W^{-2}\eta_t^{-1}.
\end{equation}
With \eqref{eq:base} and \eqref{eq:base2}, we check the scales of the parameters appearing in \Cref{lem:far-Gaussian-support}. First, under \eqref{eq:base}, we can choose 
\begin{equation}
\Phi_2=C{(\log W)\vartheta_t}/{W^2},\quad \Phi_{\op}=C{(\log W)^{6}\vartheta_t}/{W^2} \label{eq:cut0}
\end{equation}
for some absolute constant $C>0$. Second, using the second bound in \eqref{eq:base}, \eqref{eq:base2}, and \eqref{eq:moments-iota}, we get 
\begin{align}
 \wt O_2(t)
 &\le\frac{C\log W}{W^2}\max_a\sum_b(1+\iota_t(a,b))\pa{\mathcal L_{ab}+ \mathcal L_{ba}} \notag\\
 &\le\frac{C\log W}{W^2}\max_a\sum_b\pa{\mathcal L_{ab}+ \mathcal L_{ba}} +\frac{C(\log W)^2}{W^4}\max_a\sum_b \iota_t(a,b)\qa{1+\iota_t(a,b)}
 \le \frac{C(\log W)^2}{W^4\eta_t}.\label{eq:cut}
\end{align}
By a similar argument with the third bound in \eqref{eq:base}, \eqref{eq:base2}, \eqref{eq:moments-iota}, and $\|X_{ab}\|_{S_4}^4\le \|X_{ab}\|_{\op}^2\mathcal L_{ab}$, we have 
\begin{align}
 \wt O_4(t)&
 \le\frac{C(\log W)^{11}\vartheta_t}{W^4}\max_a\sum_b[\iota_t(a,b)+\log W]\pa{\mathcal L_{ab}+\mathcal L_{ba}} \le\frac{C(\log W)^{12}\vartheta_t}{W^6\eta_t}.\label{eq:cut2}
 \end{align}

\begin{proposition}[Bootstrap estimates for light-weights]\label{thm:trace}
Under the cutoff $\chi_0$ chosen in \eqref{eq:def-chi0}, the following bounds hold for every $p\in\N$, $t\in[0,t_*]$, and $a,x\in\Zn$, with a constant $C>0$ independent of $p$:
\begin{align}\label{eq:point-tracebound}
& \pa{\E \absa{\chi_0(\difloop_1)_{a}(t)}^{2p}}^{1/(2p)}
 \le C 
 \frac{\sqrt p\,(\log W)^{7/2}\vartheta_t}{W^2}
 +C\frac{p(\log W)^{12}\vartheta_t^2}{W^4} ,\\
      \label{eq:tracebound}
&\pa{\E \absa{\chi_0\mathfrak D_x(t)}^{2p}}^{1/(2p)} \le   
C \frac{(\log W)^{7/2}(\vartheta_t+\sqrt{p\vartheta_t})}{W^2}
 +C\frac{p(\log W)^{8}\vartheta_t^{3/2}}{W^4} +C\frac{p(\log W)^{9}\vartheta_t^{5/4}}{W^{7/2}} .
\end{align}
For $p=\logpara$, under \eqref{eq:window} and \eqref{eq:spatialorders}, the right-hand sides are $\oo(\mathfrak W)$ and $\oo(\mathfrak W_{\rm loc})$, respectively. Consequently, for every fixed $D>0$ and all sufficiently large $N$, 
\begin{equation}\label{eq:barrier-improvement-J1loc}
 \mathbb P\left(\sup_{0\le t\le t_*}\qa{\chi_0 J_1(t)+\chi_0 J_{\rm loc}(t)}\le 1/10\right) \ge 1-N^{-D}.
\end{equation}
\end{proposition}
\begin{proof}
By \Cref{lem:far-Gaussian-support}, the assumptions of \Cref{lem:coordinate-test} hold with
\[
 \Xi_1=C\Phi_2\Phi_{\op}
 \le \frac{C(\log W)^{7}\vartheta_t^2}{W^4},\qquad
 \Xi_\chi=C(\Phi_{\op}^2+\phi_1^{-1}\Phi_2\Phi_{\op})
 \le\frac{C(\log W)^{12}\vartheta_t^2}{W^4}.
\]
Thus \eqref{new:general-trace-test} gives \eqref{eq:point-tracebound}. 
By \eqref{eq:base}, \eqref{eq:spatial-cutoff-derivatives}--\eqref{new:spatial-trace-test-2}, and  \eqref{eq:cut0}--\eqref{eq:cut2}, the assumptions of \Cref{lem:weighted-trace-test} hold with
\[\xi_1=\frac{C(\log W)^{7}\vartheta_t}{W^4},\quad 
 \Xi_{\chi,2}=\frac{C(\log W)^{8}\vartheta_t\ell_t}{W^4},\qquad
 \Xi_{\chi,4/3}=\frac{C(\log W)^{9}\vartheta_t^{5/4}\ell_t^{1/2}}{W^{7/2}}.
\]
Substituting into \eqref{eq:weighted-trace-test} proves \eqref{eq:tracebound}. 
By \eqref{eq:window}, \eqref{eq:spatialorders}, and $\vartheta_t\le C\logpara$, the right-hand sides of \eqref{eq:point-tracebound} and \eqref{eq:tracebound} are $\oo(\mathfrak W)$ and $\oo(\mathfrak W_{\rm loc})$, respectively, when $p=\logpara$. Applying Markov's inequality at thresholds equal to $C_D$ times the respective right-hand sides, where $C_D>1$ is chosen sufficiently large that $C_D^{-\logpara}\le N^{-D-1}$, we obtain that for each $t\in[0,t_*]$,
\begin{equation}
\mathbb P\left( \chi_0 J_1(t)+\chi_0 J_{\rm loc}(t) \le 1/20\right) \ge 1-N^{-D}.\label{eq:J1-sigle-time}
\end{equation}
To obtain the simultaneous estimate \eqref{eq:barrier-improvement-J1loc}, we apply a standard $N^{-C}$-net argument in $t$, enlarging $D$ in \eqref{eq:J1-sigle-time} if necessary, together with the resolvent perturbation estimates. We omit the routine details.\end{proof}

\section{A priori two-loop estimate}\label{sec:direct-operator}

In this section, we prove an a priori operator-norm bound for the $\cD$-loops in \Cref{direct:continuation}. We use this bound in \Cref{sec:spatial} to establish the bootstrap estimates for $J_2$ and $\mathcal Y_{k_*+1}$. 

\subsection{Two-sided regularization}
Our proof relies on a two-sided regularization. Given a vector $f\in \C^{\Zn}$, let $\diag(f)$ denote the diagonal matrix with $(\diag(f))_{aa}=f(a)$ for $a\in\Zn$. We define the lifting operator $\sL_t$ on vectors and the corresponding regularization operator $\sQ_t$ on matrices by
\begin{equation}\label{direct:transported-lift}
 \mathsf L_tf:=(1-t)\Theta_t\diag(\Theta_t^{-1}f)\Theta_t,
 \qquad \mathsf Q_t:=\Id-\mathsf L_t\circ\cP,
\end{equation}
where $\cP$ is the row-sum operator
\begin{equation}
 (\cP \circ \cal A)(a):=(\cA\1)(a) =\sum_b \cal A(a,b),
 \label{eq:tools-relative-projection}
\end{equation}
and $\1$ is the column vector with all entries equal to one. Since $\sL_t f$ is symmetric and $(\sL_t f)\1=f$, the matrix $\sQ_t\circ F$ satisfies the two-sided sum-zero property \(\p{\Qsym_t \circ F}\1=\p{\Qsym_t \circ F}^\top\1=0\) whenever $F\1=F^\top\1$.

We define $\mathsf f(t)$ to be the common row and column marginal of $\difloop(t)$. By Ward's identity,
\begin{equation}\label{eq:mathsf-f}
\mathsf f_a(t):=\pa{\difloop(t)\mathbf1}_{a}=\pa{\difloop(t)^{\mathsf T}\mathbf1}_{a}={\im \tr\p{\Gc_t E_a}}\big/\p{W^2\eta_t}.\end{equation}
We then define the marginal $Y$-loops and the regularized $Z$-loops by
\begin{equation}\label{eq:D+Y-2}Y(t):=\mathsf L_t\mathsf f(t),\quad Z(t):=\mathcal D(t)-Y(t).
\end{equation}
Let $\mathfrak g_t:=(1-t)^{-1}\Id+S^{(\sB)}\Theta_t$. For each fixed vector $f$, differentiating \eqref{direct:transported-lift} gives
\[
 \partial_t(\mathsf L_tf)
 =\mathscr A_t\circ(\mathsf L_tf)-(1-t)^{-1}\mathsf L_tf
       -(1-t)\Theta_t\diag(S^{(\sB)}f)\Theta_t
 =\mathscr A_t\circ(\mathsf L_tf)-\mathsf L_t(\mathfrak g_tf).
\]
Moreover, applying \(\cP\) to equation \eqref{eq_L-Keee} with $\star=\emptyset$ gives the evolution of the marginal,
\begin{equation}
 \dd \mathsf f(t) =\left[\mathfrak g_t\mathsf f(t)+\cP\circ\pa{\cE(t)+\cW(t)}\right]\dd t
  +\cP\circ\dd\mathcal M(t).
 \label{eq:low-density-dynamics}
\end{equation}
Then, differentiating \(Z(t)=\difloop(t)-\Tsym_t\mathsf f(t)\), subtracting the equation \( \dd(\Tsym_t\mathsf f(t)) =(\partial_t\Tsym_t)\circ \mathsf f(t)\,\dd t+\Tsym_t\circ\dd \mathsf f(t)\), and using \eqref{eq:low-density-dynamics}, we obtain the SDE
\begin{equation}\label{direct:zero-defect-flow}
 \dd Z(t)=\bigl[\mathscr A_t\circ Z(t)
       +\mathsf Q_t\circ(\mathcal E(t)+\mathcal W(t))\bigr]\dd t
       +\mathsf Q_t\circ\dd\mathcal M(t).
\end{equation}

\subsection{Fourier tools}\label{sec:timeoperators}
We next collect several deterministic estimates that will be used in the argument. To this end, it is convenient to work in Fourier space. For Fourier modes $p=(p_1,p_2)\in(2\pi/L)\Zn$, \eqref{eq:variance-profile} gives
\[
 \wh S^{(\sB)}(p)=\frac{1+2\cos p_1+2\cos p_2}{5},\qquad
 \lambda(p):=1-\wh S^{(\sB)}(p)=\frac25(2-\cos p_1-\cos p_2).
\]
For a translation-invariant matrix $A$, we write $\wh A(p):=\sum_b e^{-\ii p\cdot(a-b)}A_{ab}$ for its Fourier transform, which is independent of $a$. For a vector $f$, we use the unitary Fourier transform \smash{$\wh f(p):=L^{-1}\sum_a e^{-\ii p\cdot a}f_a$}. We define $\mathsf A_u:=(\Id-S^{(\sB)})\Theta_u$, whose Fourier transform is 
$$\wh {\mathsf A}_u(p)=\frac{\lambda(p)}{1-u+u\lambda(p)}.$$ 
Throughout the following proof, all inverses of $\mathsf A_u$ and $\Id-S^{(\sB)}$ are understood as pseudoinverses: they act as inverses on $\mathbf1^\perp$ and vanish on the constant mode. Their Fourier multipliers are $\wh{\mathsf A}_u(p)^{-1}$ and $\lambda(p)^{-1}$ for $p\ne0$, respectively, and zero for $p=0$. These formulas give
\begin{equation}\label{direct:inverse-energy}
 \begin{gathered}
 \mathsf A_u^{-1}=u\Id+(1-u)(\Id-S^{(\sB)})^{-1}
 \ \text{ on }\mathbf1^\perp,\quad \|\mathsf A_u\|_{\op}\le \frac{8}{5},\quad \|\mathsf A_u^{1/2}(\Id-S^{(\sB)})^{-1}\mathsf A_u^{1/2}\|_{\op} \le(1-u)^{-1},\\
 \max_a (\Id-S^{(\sB)})^{-1} _{aa}
 =\frac1{L^2}\sum_{p\ne0}\frac1{\lambda(p)}
 \le C\log L \le C\logpara.
 \end{gathered}
\end{equation}

\begin{lemma}\label{lem:energy-form}
We introduce the energy form
\(\enu(f):=\langle f,\mathsf A_uf\rangle=f^*\mathsf A_uf .\)
On $\{u\le \tau_*\}$, we have 
\begin{equation}\label{direct:Dirichlet}
\sum_{a,b}\max_{a',b': |a-a'|\le 1,|b-b'|\le 1}\qa{\mathcal L_{a'b'}(u)+\cK_{a'b'}(u)}|f_a-f_b|^2
 \le C \frac{(\log W)^{\fc}}{W^{2}\eta_u}\enu(f).
\end{equation}
Furthermore, if a vector $f$ has Fourier support $\supp \wh f\subset\{\lambda(p)\le1-u\}$, then on $\{u\le \tau_*\}$,
\begin{equation}\label{direct:Dirichlet-loc}
\sum_{a,b}|(\difloop_1)_a|\max_{a',b': |a-a'|\le 1,|b-b'|\le 1}\qa{\mathcal L_{a'b'}(u)+\cK_{a'b'}(u)}|f_a-f_b|^2 \le C\frac{ (\log W)^{\fc}\mathfrak{W}_{\rm loc}}{W^{2}\eta_u} \enu(f).
\end{equation}
The bounds \eqref{direct:Dirichlet} and \eqref{direct:Dirichlet-loc} also hold with $\cL$ replaced by $\mathcal L^\top$. 
\end{lemma}
\begin{proof}
The kernel \smash{$K_u(x,a)=\omega_{x,u}(a)+S^{(\sB)}_{xa}$} is symmetric, has total mass \(2\), and satisfies the second-moment bound $\sum_{a}K_{u}(0,a)|a|^2\le C\ell_u^2$. Consequently, its associated difference-form multiplier satisfies 
\[
\widehat K_u^{\rm dif}(p)
 := 2\sum_{a}K_{u}(0,a)[1-\cos(p\cdot a)]\le C\sum_{a}K_{u}(0,a)\min\{1, |p\cdot a|^2\} \le C \min\{1,\ell_u^2\lambda(p)\}.\]
Since $\ell_u^2=\eta_u^{-1}$, we have
\(
 \min\{1,\ell_u^2\lambda(p)\}
 \le C \wh{\mathsf A}_u(p).
\)
Parseval's identity therefore gives
\begin{equation}\label{direct:profile-form}
 \sum_{x,a}\bigl[\omega_{x,u}(a)+S^{(\sB)}_{xa}\bigr]|f_a-f_x|^2
 =\sum_p\widehat K_u^{\rm dif}(p)|\widehat f(p)|^2
 \le C\enu(f).
\end{equation}
The point profile \eqref{eq:J2} for $\cL$ and the propagator bound \eqref{eq:tools-P-point} for $\cK=W^{-2}\Theta$ then imply \eqref{direct:Dirichlet}, since $\omega_{x,u}(a)$ changes by at most a constant factor under bounded displacements of $x$ and $a$.

Under the additional Fourier support condition $\supp \wh f\subset\{\lambda(p)\le1-u\}$, we use the following fact for any vector \smash{$v\in\C^{\Zn}$}:
\begin{equation}\label{direct:local-average}
 \sum_a|v_a||f_a|^2
 \le C \max_x
       \pB{\sum_a\omega_{x,u}(a)|v_a|^2}^{1/2}\|f\|^2.
\end{equation}
In fact, writing $f=(1-u)\Theta_u[(1-u)^{-1}\Theta_u^{-1}f]$ and applying weighted Cauchy--Schwarz with $(1-u)\Theta_u\mathbf1=\mathbf1$, we obtain
\[ 
 \sum_a|v_a||f_a|^2\le \sum_{a,b}|v_a|[(1-u)\Theta_u]_{ab}\left|[(1-u)^{-1}\Theta_u^{-1}f]_b\right|^2\le 4\|f\|^2\max_x\sum_a|v_a|[(1-u)\Theta_u]_{ax},
\]
where the second inequality uses $\|(1-u)^{-1}\Theta_u^{-1}f\|\le2\|f\|$ under the Fourier support assumption on $f$. By \eqref{eq:tools-P-point} and \eqref{eq:omega-weights}, \(
 0\le[(1-u)\Theta_u]_{ax}\le C\omega_{x,u}(a).
\)
A further application of weighted Cauchy--Schwarz, using $\sum_a\omega_{x,u}(a)=1$, proves \eqref{direct:local-average}.

Note that if $f$ has Fourier support in $\{\lambda\le1-u\}$, then so does $x\mapsto f_{x+a}-f_x$ for every displacement $a\in \Zn$. Applying \eqref{direct:local-average} with $v=\difloop_1$ and using $\max_x\mathfrak D_x(u)\le\mathfrak W_{\rm loc}$ from \eqref{eq:Jloc} and $J_{\rm loc}(u)\le 1$ gives
\begin{equation} 
 \sum_x|(\difloop_1)_x||f_{x+a}-f_x|^2
 \le C\mathfrak W_{\rm loc}\sum_x|f_{x+a}-f_x|^2.
\label{eq:energy-bdd2}  
\end{equation}
Multiplying by $\omega_{0,u}(a)+S^{(\sB)}_{0a}$, summing in $a$, and using \eqref{direct:profile-form}, we obtain
\begin{equation}    
 \sum_{x,a}|(\difloop_1)_x|\bigl[\omega_{x,u}(a)+S^{(\sB)}_{xa}\bigr]|f_a-f_x|^2
 \le C\mathfrak W_{\rm loc}\enu(f).
\label{eq:energy-bdd1}  
\end{equation}
Again, the point profile \eqref{eq:J2} for $\cL$ and the propagator bound \eqref{eq:tools-P-point} for $\cK=W^{-2}\Theta$ imply \eqref{direct:Dirichlet-loc}, using the same bounded-displacement comparison as above.
\end{proof}

For each unit displacement $e\in \Zn$, define the translation operator $(\mathsf T_ef)(a):=f(a+e)$, and introduce the five Fourier multipliers
\begin{equation}\label{direct:difference-factors}
 \mathsf B_{0,u}:=\Id-(1-u)\Theta_u,\qquad
 \mathsf B_{e,u}:=\sqrt{\frac{u(1-u)}5}\,\Theta_u(\Id-\mathsf T_e),
 \qquad |e|=1.
\end{equation}
\begin{lemma}\label{lem:Bu-f-Bu}
For any vector $f$, we have 
\begin{align}\label{direct:lift-comparison}
& \diag f-\mathsf L_uf=\sum_{e\in \Zn:|e|\le 1}\mathsf B_{e,u}\pa{\diag f} \mathsf B_{e,u}^* ,\\
\label{direct:reconstruction}
& \|\diag f-\mathsf L_uf\|_{\op}\le\|f\|_\infty,
 \quad \|\mathsf L_uf\|_{\op}\le2\|f\|_\infty.
\end{align}
\end{lemma}
\begin{proof}
By definition, we have that  
\[
 \frac15\sum_{|e|= 1}(\Id-\mathsf T_e)\pa{\diag f}(\Id-\mathsf T_e)^*
 =\diag f+\diag(S^{(\sB)}f)-S^{(\sB)}\pa{\diag f}-\pa{\diag f}S^{(\sB)}.
\]
Substituting it into \eqref{direct:transported-lift} proves \eqref{direct:lift-comparison}. Taking $f=\mathbf1$ then gives 
\begin{equation}\label{direct:difference-energy0}
 \sum_{|e|\le 1}\mathsf B_{e,u}\mathsf B_{e,u}^*
 =\Id-(1-u)^2\Theta_u^2\preceq\Id, 
\end{equation}
where $\preceq$ denotes the Loewner order for Hermitian matrices. Combining this bound with \eqref{direct:lift-comparison} yields the first bound in \eqref{direct:reconstruction}, and the second bound follows immediately. 
\end{proof}

\subsection{Drift and martingale estimates}\label{sec:drift-martingale}

Expanding the two $\cD$ factors in $\mathcal E$ as $(\cD-\diag \mathsf f)+\mathsf f$, we can rewrite the drift term in \eqref{direct:zero-defect-flow} as
\begin{equation}\label{direct:completed-drift}
 \begin{split}
 \mathsf Q_t\circ(\mathcal E(t)+\mathcal W(t))
 ={}&W^2[\mathcal D(t)-\diag\mathsf f(t)]S^{(\sB)}
               [\mathcal D(t)-\diag\mathsf f(t)] +\mathsf Q_t\circ\mathfrak F(t),
 \end{split}
\end{equation}
where the first term has two zero marginals, so $\mathsf Q_u$ preserves it, and $\mathfrak F(t)$ is defined by
\begin{equation}\label{a1:eq:Fcomplete}
 \mathfrak F(t):=\mathcal W(t)
 +W^2\qa{\mathcal D(t)S^{(\sB)}\p{\diag\mathsf f(t)}
       +\p{\diag\mathsf f(t)}S^{(\sB)}\mathcal D(t) -\p{\diag\mathsf f(t)}S^{(\sB)}\p{\diag\mathsf f(t)}} .
\end{equation}
The first term of \eqref{direct:completed-drift} is easy to control. For the second one, we resolve \eqref{def_EwtG} as $\mathcal W(t)=\sum_x\mathcal W^{(x)}(t)$, where
\[
 \mathcal W^{(x)}_{ab}(t):= 2{W^2}\re\qB{
 (\difloop_1)_x(t)\sum_{a'}S^{(\sB)}_{a'x}
                  \Tr(E_aG_tE_{a'}G_tE_bG_t^*)},
 \quad x,a,b\in \Zn.
\]
Then, we define 
\[
 K_x(t):= \mathcal W^{(x)}(t)
 +W^2\mathsf f_x(t)
 \qa{\mathcal D(t)S^{(\sB)}\mathbf e_x\mathbf e_x^*+\mathbf e_x\mathbf e_x^*S^{(\sB)}\mathcal D(t)},\quad x\in \Zn,
\]
where $\mathbf e_x$ denotes the standard block coordinate vector. Hence, we can write 
\begin{equation}
 \mathfrak F(t)=\sum_x K_x(t)- \cQ(t),\quad 
\cQ(t):= W^2\pa{\diag\mathsf f(t)}S^{(\sB)}\pa{\diag\mathsf f(t)}.
\label{eq:decompose-F-rootx}
\end{equation}
Ward's identity gives the row sums
\begin{align}
 \sum_b\mathcal L^{(+,+,-)}_{(a,x,b)}
 &=\frac{\mathcal L^+_{ax}-\mathcal L_{ax}}
          {2\ii W^2\eta_u},\qquad 
 \sum_b\mathcal L^{(+,-,-)}_{(a,b,x)}
 =\frac{\mathcal L_{ax}-\overline{\mathcal L^+_{ax}}}
          {2\ii W^2\eta_u}.
 \label{eq:raw-third-loop-marginals}
\end{align}
The corresponding column sums are obtained by replacing $(a,x)$ with $(x,a)$ on the RHS. Substituting these identities into the definition of $K_x$ gives 
\begin{align} \label{eq:equal-marg-Kx}
\pa{K_x\mathbf1}(a)=\pa{K_x^{\top}\mathbf1}(a)
={}&\frac1{\eta_t}\im\qB{(\difloop_1)_x
       \pb{\mathcal L^+S^{(\sB)}}_{ax}}
       -\mathsf f_x\pb{\Theta_tS^{(\sB)}}_{ax} +W^2\mathsf f_x\delta_{ax}\pb{S^{(\sB)}\mathsf f}_x.
\end{align}

\begin{lemma} \label{lem:root}
For arbitrary real vectors $f,g\in \R^{\Zn}$, we have
\begin{equation}\label{simp:rootbound}
 \left|\langle f,K_x(t)g\rangle-\sum_a(K_x(t)\mathbf1)_af_ag_a\right|
 \le\frac4{\eta_t}|(\difloop_1)_x(t)|
       \sqrt{\mathscr E_{t,x}^-(f)\mathscr E_{t,x}^+(g)},
\end{equation}
where the two energy forms are defined by 
\[
 \begin{aligned}
 \mathscr E_{t,x}^-(f)&:=\sum_{a'}S^{(\sB)}_{a'x}
 \sum_a[\mathcal L_{aa'}(t)+\cK_{aa'}(t)]|f_a-f_x|^2,\\
 \mathscr E_{t,x}^+(g)&:=\sum_{a'}S^{(\sB)}_{a'x}
 \sum_b[\mathcal L_{a'b}(t)+\cK_{a'b}(t)]|g_b-g_x|^2.
 \end{aligned} 
\]
\end{lemma}
\begin{proof}
Writing $f_a=(f_a-f_x)+f_x$ and $g_b=(g_b-g_x)+g_x$, and using \eqref{eq:equal-marg-Kx}, we can expand 
\begin{align}
 &\langle f,K_xg\rangle-\sum_a(K_x\mathbf1)_af_ag_a
 = \sum_{a,b}(K_x)_{ab}(f_a-f_x)(g_b-g_x) -\sum_a(K_x\mathbf1)_a(f_a-f_x)(g_a-g_x)\notag \\
 &=\frac2{W^2}\re\left[(\difloop_1)_x
 \sum_{a'}S^{(\sB)}_{a'x}\Tr(\mathbf F_xG_t  E_{a'}G_t\mathbf G_xG_t^*)\right]-\frac1{\eta_t W^2}\im\left[(\difloop_1)_x
 \sum_{a'}S^{(\sB)}_{a'x}\Tr(\mathbf F_xG_t E_{a'}G_t\mathbf G_x)\right]\notag\\
 &\quad  +\mathsf f_x(t)
 \sum_a\pb{\Theta_tS^{(\sB)}}_{ax}(f_a-f_x)(g_a-g_x),\label{simp:anchored}
\end{align}
where $\mathbf F_x,\mathbf G_x\in \C^{N\times N}$ are defined by $\mathbf F_x:=\sum_a(f_a-f_x)P_a$ and $\mathbf G_x:=\sum_b(g_b-g_x)P_b$ (recall \eqref{eq:block-projection}).
By Schatten H\"older, the resolvent bound $\|G_t\|\le\eta_t^{-1}$, and Cauchy--Schwarz in $a'$, the sum of the absolute values of the first two terms on the RHS of \eqref{simp:anchored} is bounded by $3|(\difloop_1)_x(t)|/\eta_t$ times the geometric mean of the $\mathcal L$-contributions to $\mathscr E_{t,x}^-(f)$ and $\mathscr E_{t,x}^+(g)$.
For the last term, the bound $W^2|\mathsf f_x(t)|\le |(\difloop_1)_x(t)|/\eta_t$, the identity $\cK=W^{-2}\Theta$, and Cauchy--Schwarz give the analogous estimate with coefficient $|(\difloop_1)_x(t)|/\eta_t$ and the $\mathcal K$-contributions to these energies. Combining the two estimates proves \eqref{simp:rootbound}.
\end{proof}

Combining this lemma with \Cref{lem:energy-form,lem:Bu-f-Bu}, we can bound the drift $\mathsf Q_u\circ\mathfrak F(u)$ as follows. 
\begin{lemma}[Drift estimate]
      \label{direct:source-lemma}
On $\{u\le \tau_*\}$, for arbitrary test vectors \smash{$f,g\in \C^{\Zn}$}, we have 
\begin{equation}\label{direct:sourceform}
 |\langle f,(\mathsf Q_u\circ\mathfrak F(u))g\rangle| \le \frac{C(\log W)^{\fc}\mathfrak W}{W^2\eta_u^2}
             \sqrt{\enu(f)\enu(g)}.
\end{equation}
If $\p{\supp \wh f}\cup \p{\supp \wh g}\subset\{\lambda(p)\le1-u\}$, then $\mathfrak W$ can be replaced by $\mathfrak W_{\rm loc}$.
\end{lemma}
\begin{proof}
By decomposing $f$ and $g$ into their real and imaginary parts, it suffices to consider real-valued $f$ and $g$, up to an absolute constant. This decomposition preserves the Fourier support condition. Fix $u\le\tau_*$.
By \eqref{direct:Dirichlet} and $\|\difloop_1\|_\infty\le\mathfrak W$, we have
\begin{align}
  \sum_x|(\difloop_1)_x|\mathscr E_{u,x}^\pm(f)
 \le C(\log W)^{\fc} \p{W^2\eta_u}^{-1}\mathfrak W\enu(f).
\label{eq:total-energy-bdd1} 
\end{align}
We prove \eqref{direct:sourceform} using \eqref{eq:total-energy-bdd1}, without imposing the Fourier support condition on $f$ and $g$. Under this additional condition, \eqref{direct:Dirichlet-loc} gives \eqref{eq:total-energy-bdd1} with $\mathfrak W$ replaced by $\mathfrak W_{\rm loc}$. The same argument, using this sharper bound together with \eqref{eq:energy-bdd1} and \eqref{eq:get-Wloc} below, yields the desired local improvement from $\mathfrak W$ to $\mathfrak W_{\rm loc}$. We indicate the necessary modifications as we proceed.

By \eqref{direct:lift-comparison}, we can write
\begin{equation}
 \langle f,(\mathsf Q_u\circ\mathfrak F(u))g\rangle
 = \langle f,\mathfrak F(u)g\rangle
       -\sum_a(\mathfrak F(u)\mathbf1)_af_ag_a +\sum_{|e|\le1}\langle\mathsf B_{e,u}^*f,
       \diag(\mathfrak F(u)\mathbf1)\mathsf B_{e,u}^*g\rangle.
\label{eq:normform-f-g}
\end{equation}
To estimate the first two terms on the RHS, we apply \Cref{lem:root}, Cauchy--Schwarz in $x$, and \eqref{eq:total-energy-bdd1} to obtain
\begin{equation}   
 \sum_x \absB{\langle f,K_x(u)g\rangle-\sum_a(K_x(u)\mathbf1)_af_ag_a}
 \le \frac{C(\log W)^{\fc}\mathfrak W}{W^2\eta_u^2}
 \sqrt{\enu(f)\enu(g)},
\label{eq:root-Kx-sum}   
\end{equation}
with $\mathfrak W$ replaced by $\mathfrak W_{\rm loc}$ under the Fourier support condition. By the symmetry of $\cQ$ in \eqref{eq:decompose-F-rootx}, its contribution to the centered form is
\[
-\langle f,\cQ(u)g\rangle+\sum_a(\cQ(u)\mathbf1)_af_ag_a = \frac{W^2}{2}\sum_{x,y}S^{(\sB)}_{xy}
 \mathsf f_x(u)\mathsf f_y(u)(f_x-f_y)(g_x-g_y).
\]
Using $W^2\eta_u|\mathsf f_x|\le |(\difloop_1)_x|\le\mathfrak W$, we obtain
\[
\absB{\langle f,\cQ(u)g\rangle-\sum_a(\cQ(u)\mathbf1)_af_ag_a} \le 
 \frac{\mathfrak{W}^2}{2W^2\eta_u^2}\sum_{x,y}S^{(\sB)}_{xy}
 |f_x-f_y|\,|g_x-g_y| \le \frac{C\mathfrak{W}^2}{W^2\eta_u^2} \sqrt{\enu(f)\enu(g)},
\]
where the second inequality follows from Cauchy--Schwarz and \eqref{direct:profile-form}. Under the Fourier support condition, we retain one $\mathsf f(u)$ factor inside the sum and apply Cauchy--Schwarz and \eqref{eq:energy-bdd1} to obtain
\[\absB{\langle f,\cQ(u)g\rangle-\sum_a(\cQ(u)\mathbf1)_af_ag_a} \le\frac{\mathfrak W}{2W^2\eta_u^2}
 \sum_{x,y}S^{(\sB)}_{xy}|(\difloop_1)_x|
 |f_x-f_y|\,|g_x-g_y| \le\frac{C\mathfrak W\mathfrak W_{\rm loc}}{W^2\eta_u^2}
 \sqrt{\enu(f)\enu(g)}.
\]
Combining these bounds with \eqref{eq:root-Kx-sum}, we conclude that
\begin{equation}\label{simp:formmajorant}
\absB{\langle f,\mathfrak F(u)g\rangle
       -\sum_a(\mathfrak F(u)\mathbf1)_af_ag_a}
 \le \frac{C(\log W)^{\fc}\mathfrak W}{ W^2\eta_u^2}
       \sqrt{\enu(f)\enu(g)},
\end{equation}
again with $\mathfrak W$ replaced by $\mathfrak W_{\rm loc}$ under the Fourier support condition.

To estimate the third term in \eqref{eq:normform-f-g}, we sum \eqref{eq:equal-marg-Kx} over $x$ and subtract $\cQ(u)\mathbf1$, obtaining
\begin{equation}\label{direct:source-marginal}
 \mathfrak F(u)\mathbf1
 =\frac1{\eta_u}\im\left[
 (\mathcal L^+(u)-\cK(u))S^{(\sB)}\difloop_1(u)\right],
\end{equation}
where we used $\mathsf f(u)=\im\difloop_1(u)/(W^2\eta_u)$ and $\cK(u)=W^{-2}\Theta_u$. By the Cauchy--Schwarz bound $|\mathcal L^+_{ab}|\le\sqrt{\mathcal L_{ab}\mathcal L_{ba}}$ and Ward's identity, we have 
\begin{equation}
\|\mathfrak F(u)\mathbf1\|_\infty
 \le C \eta_u^{-1}\mathfrak W \max_a \sum_b\pa{\mathcal L_{ab}  + \mathcal L_{ba}  +\cK_{ab} }\le C\mathfrak W/(W^2\eta_u^2).
\label{eq:get-W}
\end{equation}
For the local improvement, the point profile \eqref{eq:J2} and the propagator bound \eqref{eq:tools-P-point} yield
\begin{equation}
\|\mathfrak F(u)\mathbf1\|_\infty
 \le \frac{C (\log W)^{\fc}}{W^2\eta_u^{2}} \max_a \sum_b \omega_{a,u}(b)|(\difloop_1)_b|  \le \frac{C(\log W)^{\fc}\mathfrak W_{\rm loc}}{W^2\eta_u^2}.
\label{eq:get-Wloc}
\end{equation}
The identity in \eqref{direct:difference-energy0} also gives 
\begin{equation}\label{direct:difference-energy}
 \sum_{|e|\le1}\mathsf B_{e,u}\mathsf B_{e,u}^*
 =u\mathsf A_u[\Id+(1-u)\Theta_u]\preceq2\mathsf A_u.
\end{equation}
Thus, Cauchy--Schwarz, \eqref{eq:get-W}, and \eqref{direct:difference-energy} imply
\[
 \absB{\sum_{|e|\le1}\langle\mathsf B_{e,u}^*f,
       \diag(\mathfrak F(u)\mathbf1)\mathsf B_{e,u}^*g\rangle}\le 2\|\mathfrak F(u)\mathbf1\|_\infty\sqrt{\enu(f)\enu(g)}\le \frac{ C\mathfrak W}{W^2\eta_u^2}\sqrt{\enu(f)\enu(g)}.
\]
Using \eqref{eq:get-Wloc} instead gives the same bound with $\mathfrak W$ replaced by $(\log W)^{\fc}\mathfrak W_{\rm loc}$. Combining these estimates with \eqref{simp:formmajorant} proves \eqref{direct:sourceform} and its local improvement.
\end{proof}

To estimate the quadratic variation of the two-loop martingale, we write its differential from \eqref{def_Edif} as
\begin{align}\label{eq:MG-Brownian}
 \dd\mathcal M_{ab}(u) = \sum_{x,y\in \ZL} 
    \pa{\partial_{xy}\cL_{ab}(u)} \left(\dd \boldsymbol{B}_u\right)_{xy},
\end{align}
where $\partial_{xy}$ is defined in \eqref{eq:direction-D}. For clarity, we express the martingale in terms of independent real Brownian motions. More precisely, for \(x\ne y\in \ZL\), denote 
\begin{align}
 &V_{xy}^{\re}:=\sqrt{\frac{S_{xy}}2} \pa{\mathbf e_x\mathbf e_y^* + \mathbf e_y\mathbf e_x^* },\ \ 
 &&V_{xy}^{\im}:=\ii\sqrt{\frac{S_{xy}}2}\pa{\mathbf e_x\mathbf e_y^* - \mathbf e_y\mathbf e_x^*} ,
 \ \ 
 &&V_{xx}^{\dd}:=\sqrt{S_{xx}}\mathbf e_x\mathbf e_x^*  ,
 \label{eq:low-physical-frame-main}\\
 &B_{u,(x,y)}^{\re}:=\sqrt{2}\re\left(\boldsymbol{B}_u\right)_{xy},\quad &&B_{u,(x,y)}^{\im}:=\sqrt{2}\im\left(\boldsymbol{B}_u\right)_{xy},\quad &&B_{u,(x,x)}^{\dd}:= \left(\boldsymbol{B}_u\right)_{xx}.\label{eq:low-physical-frame-main-B}
\end{align}
Fix a bijective labeling $\bar i:\ZL\to\qqq{N}$ and define
the index set of independent real coordinates by
\[ \cI:=\ha{((x,y),\al):x,y\in \ZL, \ \bar i(x) < \bar i(y),\  \al\in\{\re,\im\}}\cup \ha{((x,x),\dd):x\in\ZL}.\]
Each $\phy\in\cI$ determines a matrix $V_{\phy}$ and a scalar Brownian motion $B_{u,\phy}$ through \eqref{eq:low-physical-frame-main} and \eqref{eq:low-physical-frame-main-B}. The processes $\{B_{u,\phy}:\phy\in\cI\}$ are independent standard real Brownian motions, and \eqref{eq:MG-Brownian} becomes
 \begin{align}\label{eq:MG-Brownian-physical}
 \dd\mathcal M_{ab}(u) = \sum_{\phy\in \cI} 
		h_{u,\phy}(a,b) \dd  {B}_{u,\phy} = \sum_{\mu\in \Zn}\sum_{\phy\in \cI_\mu} 
	h_{u,\phy}(a,b) \dd  {B}_{u,\phy},
\end{align}
where $\cI_\mu:=\{((x,y),\al)\in\cI: x\in[\mu]\}$ and the coefficient matrices $h_{u,\phy}$ for the $\cL$-loop are given by
\begin{align}\label{eq:hue}
  h_{u,\phy}(a,b): = \partial_{\phy} \cL_{ab}(u)=-\tr\pa{G_u E_bG_u^* E_aG_u V_{\phy}}-\tr\pa{G_u^*E_aG_uE_bG_u^*V_{\phy}}.
\end{align}
Here, $\partial_{\phy}$ denotes the directional derivative with respect to the matrix argument in the Hermitian direction $V_{\phy}$. It is crucial to notice that \( h_{u,\phy}\1=h_{u,\phy}^\top\1\) for all $\phy\in \cI$ by Ward's identity: 
\begin{align}
 \sum_b h_{u,\phy}(a,b) = \partial_{\phy} \sum_b \cL_{ab}(u)= \partial_{\phy} \sum_b \cL_{ba}(u)= \sum_b h_{u,\phy}(b,a).\label{eq:h_equal_marginal}
\end{align}

\begin{lemma}[Martingale estimate]\label{direct:noise-lemma}
On $\{u\le\tau_*\}$, for every $f\in\C^{\Zn}$ and every deterministic real symmetric positive semidefinite matrix $T\in\R^{\Zn\times\Zn}$, the following estimates hold:
\begin{align}
\sum_{\phy}\|(\mathsf Q_u\circ h_{u,\phy})f\|^2
 &\le \frac{C(\log W)^{\fc}}{W^4\eta_u^3}O_{\op}(u)^2\enu(f),\label{direct:physicalnoise}\\
 \sum_{\phy}\|T^{1/2}(\mathsf Q_u\circ h_{u,\phy})f\|^2
 &\le \frac{C}{W^4\eta_u^3}\left[\frac{(\log W)^{\fc}(\max_xT_{xx})O_{\op}(u)}{W^2\eta_u}
       +O_{\op}(u)^2\|\mathsf A_u^{1/2}T\mathsf A_u^{1/2}\|\right]
          \enu(f).\label{direct:physicalnoise2}
\end{align}
Both estimates also hold with \(\mathsf Q_u\circ h_{u,\phy}\) replaced by \((\mathsf Q_u\circ h_{u,\phy})^*\).
\end{lemma}
\begin{proof}
Again, it suffices to consider a real vector $f$. Let $\mathbf F_f:=\sum_af_aP_a$ and $\mathbf X_f:=G_u[\mathbf F_f,G_u^*]$. Applying Ward's identity and \eqref{direct:lift-comparison}, we obtain
\begin{equation}\label{simp:schurcoefficient}
 \begin{aligned}
 (\mathsf Q_u\circ h_{u,\phy})f &= \qa{h_{u,\phy} - \diag(h_{u,\phy}\mathbf1)}f + \ha{\diag(h_{u,\phy}\mathbf1)  -\sL_u(h_{u,\phy}\mathbf1)} f \\
 &=-\mathsf K_{\phy}f+ \sum_{|e|\le1}
       \mathsf B_{e,u}\diag(h_{u,\phy}\mathbf1)\mathsf B_{e,u}^*f.
 \end{aligned}
\end{equation}
where the kernel $\mathsf K_{\phy}$ is defined by 
\( (\mathsf K_{\phy}f)_a= W^{-4}
       \Tr \q{V_{\phy}\p{\mathbf X_fP_aG_u+G_u^*P_a\mathbf X_f^*}}.\)
Using $\|G_u\|_{\op}\le \eta_u^{-1}$ and \eqref{direct:Dirichlet}, we obtain
\begin{equation}
\label{eq:Xf-HS}
 \|\mathbf X_f\|_{\HS}^2 \le \eta_u^{-2}\|[\mathbf F_f,G_u^*]\|_{\HS}^2=\frac{W^4}{\eta_u^2}
  \sum_{a,b}\mathcal L_{ab} |f_a-f_b|^2
 \le C(\log W)^{\fc}W^2\eta_u^{-3}\enu(f).
\end{equation}

 By the definition of $\mathsf K_{\phy}f$, we have 
\begin{equation}
 \sum_{\phy}\|T^{1/2}\mathsf K_{\phy}f\|^2
 \le\frac{4}{W^{10}}\sum_{a_1,a_2}S^{(\sB)}_{a_1a_2}
       \sum_{a,a'}T_{aa'}
          \Tr[(P_{a_1}\mathbf X_fP_aG_uP_{a_2})^*P_{a_1}\mathbf X_fP_{a'}G_uP_{a_2}].\label{eq:HS-Kphy}
\end{equation}
Let $\mathbf T_b=\sum_a(T^{1/2})_{ba}P_a$ for $b\in\Zn$. Using $T_{aa'}=\sum_b (T^{1/2})_{ab}(T^{1/2})_{ba'}$, we can write
\begin{align}\label{simp:blockSchur}
 \sum_{a,a'}T_{aa'}\Tr[(P_{a_1}\mathbf X_fP_aG_uP_{a_2})^*P_{a_1}\mathbf X_fP_{a'}G_uP_{a_2}] =\sum_{b}\|P_{a_1}\mathbf X_f\mathbf T_bG_uP_{a_2}\|_{\HS}^2 \notag\\
 \le \|G_uP_{a_2}\|_{\op}^2
       \sum_b\|P_{a_1}\mathbf X_f\mathbf T_b\|_{\HS}^2\le\eta_u^{-1} (\max_xT_{xx})W^2O_{\op}(u) \|P_{a_1}\mathbf X_f\|_{\HS}^2,
 \end{align}
where the third step uses \eqref{eq:Ward-op} and $\sum_{b}\mathbf T_b^2\preceq(\max_xT_{xx})\Id$. 
For $T=\Id$, we have a sharper bound 
$$\sum_b\|P_{a_1}\mathbf X_fP_bG_uP_{a_2}\|_{\HS}^2
\le W^4O_{\op}(u)^2\|P_{a_1}\mathbf X_f\|_{\HS}^2.$$
Applying the above two block estimates to \eqref{eq:HS-Kphy} and summing over $a_1,a_2$ using \eqref{eq:Xf-HS}, we can bound the $\mathsf K_{\phy}$ contribution by the right-hand sides of \eqref{direct:physicalnoise} and \eqref{direct:physicalnoise2}.

For the contribution of the second term on the RHS of \eqref{simp:schurcoefficient}, we write
$\mathbf f_{a,x}=((h_{u,\phy}\mathbf1)(a):\phy\in\cI_x)$ 
for $a,x\in\Zn$. Ward's identity gives 
\begin{align}\label{eq:mart-marginal}
 W^2\eta_u\mathbf f_{a,x}
 =\bigl(\partial_{\phy}\im\tr(G_uE_a):\phy\in\cI_x\bigr)=\bigl(-\im\tr(G_uV_{\phy}G_uE_a):\phy\in\cI_x\bigr). \end{align}
By the definition of $V_{\phy}$ and Schatten H\"older's inequality, we have 
\begin{align}
\|\mathbf f_{a,x}\|_{\ell^2(\cI_x)}^2 &\le CW^{-2}\eta_u^{-2} \sum_y  S^{(\sB)}_{xy}\pa{\|X_{ya}X_{ax}\|_{S_2}^{2} +\|X_{xa}X_{ay}\|_{S_2}^{2}}.\label{eq:ell2-fax}
 \end{align}
Summing \eqref{eq:ell2-fax} over $x$, bounding one $X$-factor by $O_{\op}(u)$, and applying Ward's identity as in \eqref{eq:sum-S2} gives  
\begin{equation}\label{direct:noise-marginal}
 \begin{split}
 \sum_{\phy}|(h_{u,\phy}\mathbf1)_a|^2
 &\le\frac{CO_{\op}(u)^2}{W^2\eta_u^2}
       \sum_b\bigl(\|X_{ab}(u)\|_{S_2}^2+\|X_{ba}(u)\|_{S_2}^2\bigr) \le\frac{C_\kappa O_{\op}(u)^2}{W^4\eta_u^3}.
 \end{split}
\end{equation}
By \eqref{direct:difference-energy} (which also implies $\|\mathsf B_{e,u}^*T\mathsf B_{e,u}\|\le2\|\mathsf A_u^{1/2}T\mathsf A_u^{1/2}\|$ for $|e|\le1$), we have
\[
 \sum_{\phy}\normB{T^{1/2}\sum_{|e|\le1}\mathsf B_{e,u}
       \diag(h_{u,\phy}\mathbf1)\mathsf B_{e,u}^*f}^2
 \le C\|\mathsf A_u^{1/2}T\mathsf A_u^{1/2}\|
       \max_a\sum_{\phy}|(h_{u,\phy}\mathbf1)_a|^2\,\enu(f).
\]
Plugging \eqref{direct:noise-marginal} into this inequality bounds the remaining term by the right-hand sides of \eqref{direct:physicalnoise} and \eqref{direct:physicalnoise2}. Interchanging $G_u$ and $G_u^*$ does not affect the proof.
\end{proof}

\subsection{The operator-norm estimate}\label{direct:sec:causal}

Using the estimates established in \Cref{sec:drift-martingale}, we prove the following a priori operator-norm bound for the $\cD$-loops.
\begin{proposition}[Operator-norm estimate for $\cD$-loops]\label{direct:continuation}
For every fixed $D>0$, there exists a constant $C_D>0$ such that the following bound holds with probability at least $1-N^{-D}$:
\begin{equation}\label{direct:unbalanced-conclusion}
\sup_{t\le\tau_*}  W^2(1-t) \|\mathcal D(t)\| \le C_D(\log W)^{6} {\logpara}/{W}.
\end{equation}
\end{proposition}
\begin{remark}
The bound \eqref{direct:unbalanced-conclusion} is stronger than needed: the two-loop bootstrap estimates in \Cref{sec:spatial} require only an $\oo(1)$ bound on its LHS. Indeed, the Schur test and Ward's identity give
$$\|\cL(t)\|\le \max_a \max\hB{\sum_{b}\cL_{ab},\sum_{b}\cL_{ba}} \le \frac{\max_a \im\tr\p{G_tE_a}}{W^2\eta_t} \le \frac{1+C_\kappa\mathfrak{W}}{W^2(1-t)}\quad \text{on}\ \{t\le \tau_*\}.$$
Moreover, by \eqref{eq:Theta-inf-to-inf}, $\|\cK(t)\|=W^{-2}\|\Theta_t\|=W^{-2}(1-t)^{-1}$. Thus, the trivial estimate is $W^2(1-t) \|\mathcal D(t)\| \le 2+\oo(1)$ on $\{t\le \tau_*\}$, whereas \Cref{direct:continuation} improves this bound to $\oo(1)$ by exploiting the cancellation between $\cL$ and $\cK$. Such a bound would follow directly from a sufficiently precise pointwise estimate on the $\cD$-loops imposed through the stopping time $\tau_*$, as in the quantum diffusion argument in \Cref{app:quantum-diffusion}. %{\cob  of the online arXiv version (\href{https://arxiv.org/abs/2609.24426}{arXiv:2609.24426}) of the paper.} 
In the smaller-bandwidth regime $1<A_{\ref{thm:delocalization}}\le2$, however, our argument does not close such a strong pointwise two-loop estimate. We therefore establish the weaker operator-norm bound by analyzing the dynamics of the $\cD$-loops.
\end{remark}

To prove \Cref{direct:continuation}, we follow the argument used for \Cref{prop:auxiliary-time-release-2}. More precisely, define
\smash{\(
 \mathscr X(t):=W^2(1-t)\mathsf A_t^{-1/2}Z(t)\mathsf A_t^{-1/2}
\)}
and its normalized squared Schatten norm
\begin{equation}\label{direct:alternative-square}
 \mathscr R_p(t):=L^{-2/p}\|\mathscr X(t)\|_{S_{2p}}^2.
\end{equation}
We introduce the auxiliary stopping time
\begin{equation}      
 \tau_{\aux}:=\tau_*\wedge
 \inf\{t\in[0,t_*]:\|\mathscr X(t)\|\ge  (\log W)^{-1} \},
 \qquad \inf\varnothing:=\infty.
\label{eq:tau-aux2}
\end{equation}

\begin{proposition}\label{direct:causal-moment}
Let $p=\lceil\logpara\rceil$.
For every fixed $D>0$, there exists a constant $C_D>0$ such that the following bound holds simultaneously for all $0\le t\le\tau_{\aux}$ with probability at
least $1-N^{-D}$:
\begin{equation}\label{direct:causal-moment-bound}
 \begin{split}
 \mathscr R_p(t)
 \le C_D(\log W)^{\fc} \logpara \mathfrak W_{\rm loc} +C_D(\log W)^{2\fc}\logpara  \mathfrak W^2+ {C_D} (\log W)^{5+\fc/2}\logpara^2 W^{-2}.
 \end{split}
\end{equation}

\end{proposition}

\begin{proof}  
We first derive a dynamical inequality for $ \mathscr R_p(t)$, similar to \Cref{lem:squared-schatten}. Differentiating the deterministic weights and using \eqref{direct:zero-defect-flow} and \eqref{direct:completed-drift}, we obtain the following SDE at time $u$:
\begin{equation}\label{direct:alternative-matrix-flow}
 \dd\mathscr X =\left[-\frac12(\mathsf A_u\mathscr X+\mathscr X\mathsf A_u)
       +(\mathscr X-\mathfrak{f})\mathsf A_u S^{(\sB)}(\mathscr X-\mathfrak{f})+\mathscr{F}\right]\frac{\dd u}{1-u}
       +\sum_{\phy}\mathscr H_{\phy}\,\dd B_{u,\phy},
\end{equation}
where we denote 
\[
 \begin{aligned}
 \mathfrak{f}(u):=W^2(1-u)\mathsf A_u^{-1/2}
             & [\diag\mathsf f(u)-\Tsym_u \mathsf f(u)]\mathsf A_u^{-1/2},\quad \mathscr{F}(u):=W^2(1-u)^2\mathsf A_u^{-1/2}
             (\mathsf Q_u\circ\mathfrak F(u))\mathsf A_u^{-1/2},\\
 &\mathscr H_{\phy}(u):=W^2(1-u)\mathsf A_u^{-1/2}
             (\mathsf Q_u\circ h_{u,\phy})\mathsf A_u^{-1/2}.
 \end{aligned}
\]
In deriving the first term on the RHS of \eqref{direct:alternative-matrix-flow}, we have used 
\[
-\frac{\mathscr X}{1-u} +\frac12\bigl(S^{(\sB)}\Theta_u\mathscr X
             +\mathscr XS^{(\sB)}\Theta_u\bigr) = -\frac{\mathsf A_u\mathscr X+\mathscr X\mathsf A_u}{2(1-u)} \]
because \(S^{(\sB)}\Theta_u=(\Id-\mathsf A_u)/(1-u)\) by definition.

With a slight abuse of notation, we denote the total drift in equation \eqref{direct:alternative-matrix-flow} by $\mathscr B(u)$ and set
\[
 \mathfrak h_p:=L^{-2}\Tr|\mathscr X|^{2p},\quad
 Q_p:=\frac2{L^2}\mathfrak h_p^{1/p-1}\mathscr X|\mathscr X|^{2p-2},\quad
 \Sigma^L:=\sum_{\phy}\mathscr H_{\phy}\mathscr H_{\phy}^*,\quad
 \Sigma^R:=\sum_{\phy}\mathscr H_{\phy}^*\mathscr H_{\phy},
\]
with the convention $Q_p=0$ at $\mathscr X=0$. As in the proof of \Cref{lem:squared-schatten}, applying It\^o's formula to $(\mathfrak h_p+\epsilon)^{1/p}$ and letting $\epsilon\downarrow0$ gives the following dynamical inequality for the stopped process \smash{$\widetilde{\mathscr R}_p(t):=\mathscr R_p(t\wedge\tau_{\aux})$}:
\begin{equation}\label{direct:alternative-power-hessian}
 \dd\widetilde{\mathscr R}_p(u)
 \le\1_{\{u\le\tau_{\aux}\}}
       [\mathfrak K_p^{(1)}(u)+\mathfrak K_p^{(2)}(u)]\,\dd u
       +\dd N_p(u),
\end{equation}
where $\mathfrak K_p^{(1)}:=\re\Tr(Q_p^*\mathscr B)$, $\mathfrak K_p^{(2)}:=C_{ \mathfrak K}p(\|\Sigma^L\|+\|\Sigma^R\|)$ for an absolute constant $C_{ \mathfrak K}>0$, and 
\[
 \dd N_p(u):=\1_{\{u\le\tau_{\aux}\}}\sum_{\phy}
       \re\Tr(Q_p(u)^*\mathscr H_{\phy}(u))\,\dd B_{u,\phy}.
\]
By H\"older's inequality for the normalized trace, $\|Q_p\|_{S_1}\le2\sqrt{\mathscr R_p}$. 
To bound the bracket, consider the singular value decomposition (SVD) $Q_p=\sum_jq_jf_jg_j^*$, where $f_j$ and $g_j$ are unit left and right singular vectors and $q_j$ are the singular values. Then the Cauchy--Schwarz inequality gives
\[
 \sum_{\phy}|\Tr(Q_p^*\mathscr H_{\phy})|^2
 \le\Bigl(\sum_i q_i\Bigr)
       \sum_j q_j\sum_{\phy}|\langle f_j,\mathscr H_{\phy}g_j\rangle|^2
 \le\|Q_p\|_{S_1}^2\|\Sigma^R\|,
\]
which implies the bound 
\begin{equation}\label{direct:alternative-power-bracket}
 \dd\langle N_p\rangle_u
 \le \1_{\{u\le\tau_{\aux}\}}\cdot 4\mathscr R_p(u)\|\Sigma^R(u)\|\,\dd u.
\end{equation}
We now estimate the two drift terms and the martingale in \eqref{direct:alternative-power-hessian} for $u\le \tau_{\aux}$.  

\medskip
\noindent{\bf Bounding $\mathfrak K_p^{(1)}$.}
Since $\mathsf A_u$ is a positive semidefinite matrix, the linear term in $\mathscr{B}$ makes a nonpositive contribution to $\Tr(Q_p^*\mathscr B)$. This negative drift also absorbs the quadratic drift in $\mathscr X$ under the condition $\|\mathscr X\|\le(\log W)^{-1}$. To see this, consider the SVD of $\mathscr X=\sum_i s_i f_i g_i^*$, where $f_i$ and $g_i$ are unit left and right singular vectors, and $s_i$ are the singular values. For $u\le \tau_{\aux}$, we have $s_i\le\|\mathscr X\|\le(\log W)^{-1}$.
Then,
\begin{equation}\label{eq:absorption0}
 \begin{aligned}
 &{} \re\Tr \ha{Q_p^*\left[-\frac12(\mathsf A_u\mathscr X+\mathscr X\mathsf A_u)
       +(\mathscr X-\mathfrak{f})\mathsf A_u S^{(\sB)}(\mathscr X-\mathfrak{f}) \right]}\\
 \le &{}\frac{2\mathfrak h_p^{1/p-1}}{L^2}
 \sum_i s_i^{2p-1}\left\{-\frac{s_i}{2}e_i
 +|\langle   s_ig_i-\mathfrak f^*f_i,
       \mathsf A_uS^{(\sB)}(s_if_i-\mathfrak f g_i)\rangle|\right\}.
 \end{aligned}
\end{equation}
where we abbreviate $e_i:=\enu(f_i)+\enu(g_i)$. 
Note that the matrix $\mathfrak f$ is Hermitian and satisfies $\|\mathfrak f\|\le C\mathfrak W$ by $W^2(1-u)\|\mathsf f\|_{\infty}\le C\mathfrak{W}$, \eqref{direct:lift-comparison}, and \eqref{direct:difference-energy}.
Since $S^{(\sB)}$ commutes with $\mathsf A_u$, we can combine these bounds with $\|S^{(\sB)}\|\le1$, $\|\mathsf A_u\|<2$, and Cauchy--Schwarz to obtain
 \begin{align*}
 |\langle s_ig_i-\mathfrak f f_i,
       \mathsf A_uS^{(\sB)}(s_if_i-\mathfrak f g_i)\rangle| & \le
 \pa{s_i\sqrt{\enu(g_i)}+\sqrt2\|\mathfrak f\|}
 \pa{s_i\sqrt{\enu(f_i)}+\sqrt2\|\mathfrak f\|}\le s_i^2 e_i+C\mathfrak W^2.
 \end{align*}
Plugging this bound into \eqref{eq:absorption0} and using $s_i\le (\log W)^{-1}$, we obtain 
 \begin{align}\label{eq:absorption}
(\ref{eq:absorption0})\le-\frac{\mathfrak h_p^{1/p-1}}{2L^2}
                  \sum_i s_i^{2p}e_i
       +\frac{C\mathfrak W^2\mathfrak h_p^{1/p-1}}{L^2}\sum_i s_i^{2p-1}.
 \end{align}
For the term $\mathscr{F}$, let $P_{\rm lo}$ and $P_{\rm hi}$ be the Fourier projections onto $\{p:\lambda(p)\le1-u\}$ and its complement.
Applying \eqref{direct:sourceform} to \smash{$\mathsf A_u^{-1/2}f$ and $\mathsf A_u^{-1/2}g$} gives
\begin{align*}
|\langle f,\mathscr Fg\rangle-\langle P_{\rm lo}f,\mathscr F P_{\rm lo} g\rangle|&\le C(\log W)^{\fc}\mathfrak W\pa{\|P_{\rm hi}f\|\|g\|+ \|f\|\|P_{\rm hi}g\|},\\
|\langle P_{\rm lo}f,\mathscr F P_{\rm lo} g\rangle| &\le C(\log W)^{\fc}\mathfrak W_{\rm loc}\|f\|\|g\|.
\end{align*}
Furthermore, on the high modes, $\mathsf A_u\ge\frac12\Id$, so $\|P_{\rm hi}f\|\le\sqrt{2\enu(f)}$. Thus, we get 
\begin{equation}\label{eq:absorption2}
 \begin{aligned}
 |\Tr(Q_p^*\mathscr F)|
 &\le\frac{2\mathfrak h_p^{1/p-1}}{L^2}
       \sum_i s_i^{2p-1}|\langle f_i,\mathscr Fg_i\rangle| \le\frac{C(\log W)^{\fc}\mathfrak h_p^{1/p-1}}{L^2}
       \sum_i s_i^{2p-1}
          \bigl(\mathfrak W_{\rm loc}+\mathfrak W\sqrt{e_i}\bigr).
 \end{aligned}
\end{equation}
For any sufficiently small constant $\delta>0$, Young's inequality gives
\[
 C(\log W)^{\fc}\mathfrak W s_i^{2p-1}\sqrt{e_i}
 \le\delta s_i^{2p}e_i
       +C_\delta(\log W)^{2\fc}\mathfrak W^2s_i^{2p-2}.
\]
The first term is absorbed by the negative term in \eqref{eq:absorption}. For the remaining terms, we use H\"older's inequality for the average \smash{$L^{-2}\sum_i s_i^{2p-j}\le\mathfrak h_p^{1-j/(2p)}$} for $j\in\{1,2\}$ and $\mathscr R_p\le\|\mathscr X\|^2\le(\log W)^{-2}$ up to $\tau_{\aux}$. Including the factor $(1-u)^{-1}$ now gives 
\begin{equation}\label{direct:alternative-first-drift}
 \mathfrak K_p^{(1)}(u)
 \le \frac{C(\log W)^{\fc}\mathfrak W_{\rm loc}}{1-u}
 + \frac{C(\log W)^{2\fc}\mathfrak W^2}{1-u}.
\end{equation}

\medskip
\noindent{\bf Bounding $\mathfrak K_p^{(2)}$.} 
With the first identity in \eqref{direct:inverse-energy}, for every unit vector $f$, we have that 
\[ W^{-4}(1-u)^{-2}\avgb{f,\Sigma^Rf}\le \sum_{\phy}\|(\mathsf Q_u\circ h_{u,\phy})(\mathsf A_u^{-1/2}f)\|^2 + (1-u)\sum_{\phy}\|T^{1/2}(\mathsf Q_u\circ h_{u,\phy})(\mathsf A_u^{-1/2}f)\|^2 ,   \]
where $T=(\Id-S^{(\sB)})^{-1}.$ 
Then, applying \eqref{direct:physicalnoise} and \eqref{direct:physicalnoise2} and using \eqref{direct:inverse-energy}, we obtain that
\begin{equation}
     \label{eq:cov-density}
 \|\Sigma^R(u)\| 
 \le\frac{C(\log W)^{\fc}}{\eta_u}
       \left(O_{\op}^2+\frac{\logpara}{W^2}O_{\op}\right)
 \le \frac{C(\log W)^{10+\fc}\logpara^2}{W^4\eta_u}:=v_u,
\end{equation}
where the second inequality uses the bound on $O_{\op}$ in \eqref{eq:entropy-operator-output-sm}. By symmetry, $\|\Sigma^L(u)\|$ satisfies the same bound. Consequently, we have 
\begin{equation}\label{direct:alternative-2nd-drift}\mathfrak K_p^{(2)}\le Cp v_u\le Cv_u\logpara .\end{equation}

\medskip
\noindent{\bf Bounding $N_p$.}
For the covariance density bound $v_u$ given by \eqref{eq:cov-density}, we have 
\[
 V(t):=\int_0^t v_u\,\dd u
 \le\frac{C(\log W)^{10+\fc}\logpara^3}{W^4}.
\]
Since $\mathscr R_p\le(\log W)^{-2}$ up to $\tau_{\aux}$, \eqref{direct:alternative-power-bracket} implies $\langle N_p\rangle_{t_*}\le4(\log W)^{-2}V(t_*)$ for sufficiently large $N$.
The scalar case of \Cref{lem:matrix-maximal} therefore gives, for any $R>0$,
\[
 \mathbb P\left(\sup_{0\le t\le t_*}|N_p(t)|^2
       \ge\frac{CR(\log W)^{10+\fc}\logpara^3}{W^4}\right)
 \le2\exp(-R/2) 
\]
Taking $R=C_D\logpara$ for a large enough constant $C_D>0$ yields, with probability at least $1-N^{-D}$,
\[
 \sup_{t\le t_*}|N_p(t)|
 \le C_D(\log W)^{5+\fc/2}\logpara^2W^{-2}.
\]
Integrating \eqref{direct:alternative-power-hessian}, and using $\mathscr R_p(0)=0$, \eqref{direct:alternative-first-drift}, \eqref{direct:alternative-2nd-drift}, and $\int_0^{t_*}(1-u)^{-1}\dd u\le C\logpara$, we obtain
\[
 \sup_{t\le t_*}\widetilde{\mathscr R}_p(t)
 \le C_D(\log W)^{\fc}\logpara\mathfrak W_{\rm loc}
       +C_D(\log W)^{2\fc}\logpara\mathfrak W^2
       +C_D(\log W)^{5+\fc/2}\logpara^2W^{-2}.
\]
This proves \eqref{direct:causal-moment-bound} by \Cref{lem:closed-stopped-continuation}.
\end{proof}

\begin{proof}[\bf Proof of \Cref{direct:continuation}]
For $p=\lceil\logpara\rceil$, inserting $\mathfrak W$ in \eqref{eq:common-light-weight-control} and $\mathfrak W_{\rm loc}$ in \eqref{eq:Jloc} into \eqref{direct:causal-moment-bound}, we get 
\[
 \sup_{t\le\tau_{\aux}}\mathscr R_p(t)
 \le C_D\left[
 (\log W)^{6+\fc}\frac{\logpara^2}{W^2}
 +(\log W)^{12+2\fc}\frac{\logpara^4}{W^4}
 +(\log W)^{5+\fc/2}\frac{\logpara^2}{W^2}\right]
 \le (\log W)^{12}\frac{\logpara^2}{W^2}
\]
with probability at least $1-N^{-D}$. Here the second inequality uses $0<\fc<1/10$ and \eqref{eq:window}.
Since $L^{1/p}\le C$, it follows that
\begin{equation}\label{direct:closedbootstrap}
 \sup_{t\le\tau_{\aux}}\|\mathscr X(t)\|
 \le L^{1/p}\sup_{t\le\tau_{\aux}}\sqrt{\mathscr R_p(t)}
\le C(\log W)^{6}{\logpara}/{W}.
\end{equation}
By \eqref{eq:window}, the bound in \eqref{direct:closedbootstrap} satisfies
\(
 (\log W)\cdot C (\log W)^{6}{\logpara}/{W}
 \le  C(\log W)^{7}\logpara^{-\epsilon_{\rm sm}/2} \ll (\log W)^{-1}.
\)
Continuity therefore excludes the auxiliary threshold $(\log W)^{-1}$ up to $\tau_*$, so \eqref{direct:closedbootstrap} holds up to $\tau_*$ with probability at least $1-N^{-D}$. Using \smash{$Z(t)=\mathsf A_t^{1/2}\mathscr X\mathsf A_t^{1/2}/[W^2(1-t)]$} (because $Z$ has two zero marginals), $\|\mathsf A_t\|\le 8/5$, and \eqref{direct:reconstruction}, we get that on $\{t\le \tau_*\}$, with probability at least $1-N^{-D}$,
\[
 W^2(1-t)\|\mathcal D(t)\|
 \le W^2(1-t)(\|Z(t)\|+\|Y(t)\|)
 \le2\|\mathscr X(t)\|+C \mathfrak W .
\]
Since $\mathfrak W=(\log W)^{6}\logpara^{3/2}/W^2 \ll (\log W)^{6}\logpara/W$ under \eqref{eq:window}, this proves \eqref{direct:unbalanced-conclusion} together with \eqref{direct:closedbootstrap}.
\end{proof}

\section{Two-loop bootstrap estimates}\label{sec:spatial}

In this section, we establish the bootstrap estimates for $J_2$ and $\mathcal Y_{k_*+1,a}$. We work up to the stopping time
\begin{equation}      
 \tau :=\tau_*\wedge
 \inf\{t\in[0,t_*]:W^2(1-t)\|\mathcal D(t)\|\ge  (\log W)^{-1} \},
 \qquad \inf\varnothing:=\infty.
\label{eq:tau-aux3}
\end{equation}
By \Cref{direct:continuation}, $\tau=\tau_*$ with probability at least $1-N^{-D}$ for every fixed $D>0$. 
Our proof uses \eqref{eq:mainStoflow} with $\star=\emptyset$.
We begin by estimating the drift and quadratic variation at the order $r$ appearing in \eqref{eq:Yrat}.

\subsection{Drift and martingale estimates at a fixed order}
We introduce the following variant of $\mathcal Y_{r,a}$:
\begin{equation}
 \wh Y_{r,a}(t):=W^4\eta_t\sum_{j=1}^2\sum_b
 \left|\frac{\chi_j(a)-\chi_j(b)}{\ell_t}\right|^{2r}
 (\mathcal L_{ab}(t)^2+\mathcal L_{ba}(t)^2),
\label{eq:periodic-weight}
\end{equation}
where the two periodic coordinates are defined by
\[
 \chi_j(a):=\frac{L}{2\pi}e^{2\pi\ii a_j/L},\qquad  j=1,2, \ a=(a_1,a_2).
\]
These periodic coordinates satisfy
\begin{equation}
 |\chi_j(a)-\chi_j(b)|
 =\frac L\pi\left|\sin\frac{\pi(a_j-b_j)}L\right|
 \asymp\min_{n\in\mathbb Z}|a_j-b_j+nL| \implies \sum_{j=1}^2|\chi_j(a)-\chi_j(b)| \asymp \|(a-b)_L\|_1,
\label{eq:chi-period}
\end{equation}
where we use the notation from \eqref{eq:periodic_distance}. 
%Any other choice of $\chi_j$ satisfying the same comparisons also works for the proof. However, the geometry of the torus requires periodic coordinates, so the identity function is not applicable. 
By \eqref{eq:chi-period}, 
$ \wh Y_{r,a}(t)$ is comparable to $\mathcal Y_{r,a}(t)$ (up to a $(\log W)^{\fc}$ factor):
\begin{equation}\label{eq:coordinate-radial-comparison}
 \wh Y_{r,a}(t)\asymp_r (\log W)^{\fc} \mathcal Y_{r,a}(t),\quad 
 W^4\eta_t\sum_b(1+\rho_t(a,b))^{2r}
 (\mathcal L_{ab}^2+\mathcal L_{ba}^2)
 \le C_r \bigl(\wh Y_{0,a}(t)+\wh Y_{r,a}(t)\bigr),
\end{equation}
where all coordinate weights in \(\wh Y_{0,a}(t)\) are interpreted as \(1\). The advantage of these periodic coordinate weights is that the associated weighted matrices satisfy the exact product rule in \eqref{simp:commcalculus}, yielding cancellations unavailable for the radial weights defining \(\mathcal Y_{r,a}\).
Denote the local martingale part of $\wh Y_{r,a}(t)$ by $\wh Y_{r,a}^{\rm mart}(t)$:  
\begin{align}\label{eq:dY-mart}
 \dd\wh Y_{r,a}^{\rm mart}(t)
 ={}&2W^4\eta_t\sum_{j=1}^2\sum_b
 \left|\frac{\chi_j(a)-\chi_j(b)}{\ell_t}\right|^{2r} \qa{\mathcal L_{ab}(t)\dd\mathcal M_{ab}(t) +\mathcal L_{ba}(t)\dd\mathcal M_{ba}(t)},
\end{align}
where $\dd\mathcal M_{ab}(t)=\sum_{\phy\in\cI}h_{t,\phy}(a,b)\,\dd B_{t,\phy}$
as in \eqref{eq:MG-Brownian-physical}.
Define the symmetric majorants and their row sums by
\begin{equation}
 \mathcal B_{ab}:=\|X_{ab}\|_{\op}\vee\|X_{ba}\|_{\op},\qquad
 \mathcal C_{ab}:=\mathcal B_{ab}^2\le \mathcal L_{ab}\vee \mathcal L_{ba},\qquad
 \Lambda_i:=\max_a\sum_b\mathcal B_{ab}^i,\quad i=1,2.\label{eq:Lambda12}      
\end{equation}
On $\{t\le \tau\}$, using \eqref{eq:clipop} and the dyadic shell calculation in \eqref{eq:local-envelope}, we get 
\begin{equation}\label{eq:clip}
 \Lambda_i \le\frac{C(\log W)^{5i}\vartheta_t^{i/2}}  { W^{2i-4/k_*}\eta_t},\qquad i=1,2.
\end{equation}

\begin{lemma}
\label{thm:noise} 
For any fixed $r\ge0$, suppose
\begin{equation}\label{simp:row-bound-B}
 W^4\eta_t\max_a\sum_b(1+\rho_t(a,b))^{2r}
 (\mathcal L_{ab}^2+\mathcal L_{ba}^2)\le A_r
\end{equation}
for some $A_r>0$, which may depend on $N$. Then there exists a constant $C_r>0$ such that 
\begin{align}\label{eq:sameorder}
\sum_{b} \pa{1+\rho_t(a,b)}^{2r}|\mathcal W_{ab}(t)|^2  & \le C_rA_r\eta_t^{-1}\Lambda_1(t)^2\|\difloop_1(t)\|_\infty^2  ,\\
 \sum_{b}\pa{1+\rho_t(a,b)}^{2r}\sum_{\phy}|h_{t,\phy}(a,b)|^2 & \le C_rA_r \eta_t^{-1}W^{-2}\Lambda_2(t),\label{eq:sameorder2}\\
 \partial_t\langle \wh Y_{r,a}^{\rm mart}\rangle
 &\le C_rA_r W^2 \Lambda_2\wh Y_{r,a}(t).\label{eq:bracket}
 \end{align}
By symmetry, \eqref{eq:sameorder} and \eqref{eq:sameorder2} also hold with summation over $a$ for each fixed $b$.
Finally, on $\{t\le \tau\}$, the control parameters on the RHS satisfy 
\begin{equation}\label{eq:parameters}
 \begin{gathered}
 \Lambda_1\|\difloop_1(t)\|_\infty
 \le \frac{C (\log W)^{11}\logpara^2}{W^{4-4/k_*}\eta_t},
 \quad 
\Lambda_2\le \frac{C (\log W)^{10}\logpara}{W^{4-4/k_*}\eta_t}.
 \end{gathered}
\end{equation}
\end{lemma}
\begin{proof}
Using \eqref{eq:Lambda12}, the definition of $h_{t,\phy}$ in \eqref{eq:hue}, and the Schatten H\"older inequality, we obtain 
\begin{equation}\label{eq:rawpoint}
 \sum_{\phy}|h_{t,\phy}(a,b)|^2\le4W^2\mathcal L_{ab}Q_{ab},\quad 
 Q_{ab}:=\sum_{a',b'}S^{(\sB)}_{a'b'}\|X_{ab'}\|_{\op}^2\|X_{a'b}\|_{\op}^2
 \le(\mathcal C S^{(\sB)}\mathcal C)_{ab} .
\end{equation}
For simplicity, denote $\mathcal C^{[r]}_{ab}=\pa{1+\rho_t(a,b)}^r\mathcal C_{ab}$. By the triangle inequality $\rho_t(a,b)\le \rho_t(a,b')+\ell_t^{-1}+\rho_t(a',b)$ for $|a'-b'|\le 1$, and \eqref{simp:row-bound-B}, we have 
\begin{equation}
 \pa{1+\rho_t(a,b)}^rQ_{ab}
 \le C_r(\mathcal C^{[r]}S^{(\sB)}\mathcal C+\mathcal C S^{(\sB)}\mathcal C^{[r]})_{ab},\qquad
 \sup_a\normb{(\mathcal C^{[r]})_{a\bullet}}_{\ell^2}
 \le {\sqrt{A_r}}/\p{W^2\sqrt{\eta_t}}.
\label{eq:C-distribute}
\end{equation}
Combining these bounds with $\|\mathcal C\|\le\Lambda_2$ and $\sum_{a'}(\mathcal C S^{(\sB)})_{aa'}\le\Lambda_2$, we obtain
\begin{align}
\normb{\pa{1+\rho_t(a,\bullet)}^rQ_{a\bullet}}_{\ell^2} &\le C_r\|(\mathcal C^{[r]}S^{(\sB)}\mathcal C)_{a\bullet}\|_2+C_r\|(\mathcal C S^{(\sB)}\mathcal C^{[r]})_{a\bullet}\|_2 \label{eq:Q-weighted}\\
&\le C_r\|(\mathcal C^{[r]})_{a\bullet}\|_2\|\mathcal C\|+ C_r\sum_{a'}(\mathcal C S^{(\sB)})_{aa'}\|(\mathcal C^{[r]})_{a'\bullet}\|_2 \le\frac{C_r\Lambda_2\sqrt{A_r}}{W^2\sqrt{\eta_t}} .\notag
\end{align}
Together with \eqref{eq:rawpoint} and \eqref{simp:row-bound-B}, this yields \eqref{eq:sameorder2}: 
\begin{align*}
  \sum_{b,\phy}\pa{1+\rho_t(a,b)}^{2r}|h_{t,\phy}(a,b)|^2&\le4W^2
       \|\pa{1+\rho_t(a,\bullet)}^r\mathcal L_{a\bullet}\|_{\ell^2}
       \| \pa{1+\rho_t(a,\bullet)}^rQ_{a\bullet} \|_{\ell^2} \le \frac{C_rA_r\Lambda_2}{W^{2}\eta_t }.
 \end{align*}

For \eqref{eq:sameorder}, we bound the 3-$\cL$-loop in \eqref{def_EwtG} using the Schatten H\"older inequality by
\begin{equation}\label{eq:minsource}
 |\mathcal W_{ab}(t)|\le2W^2\sqrt{\mathcal L_{ab}}
 \sum_{a'}|(S^{(\sB)}\difloop_1)_{a'}|
 \min\{\|X_{a'b}\|_{\op}\sqrt{\mathcal L_{aa'}},
                   \|X_{aa'}\|_{\op}\sqrt{\mathcal L_{a'b}}\}.
\end{equation}
We decompose the summation over $a'$ according to whether $|a-a'|\ge |a'-b|$ or $|a-a'|<|a'-b|$. Denote
\[
\mathscr L_{ab}=\pa{1+\rho_t(a,b)}^{r/2}\sqrt{W^2\eta_t (\mathcal L_{ab}\vee \mathcal L_{ba})}.
\]
Then, using \eqref{eq:minsource} and the bound $\rho_t(a,a')\vee \rho_t(a',b)\ge \rho_t(a,b)/2$, we obtain 
\begin{equation}\label{eq:minsource2}
 \pa{1+\rho_t(a,b)}^r|\mathcal W_{ab}(t)|
 \le C_r\eta_t^{-1}\|\difloop_1\|_\infty{}\mathscr L_{ab}(\mathscr L\mathcal B+\mathcal B \mathscr L)_{ab}.
\end{equation}
Applying H\"older's inequality to the RHS yields 
\begin{equation}\label{eq:minsource3}
 \|\mathscr{L}_{a\bullet}(\mathscr{L}\mathcal B+\mathcal B \mathscr{L})_{a\bullet}\|_{\ell^2} \le\|\mathscr{L}_{a\bullet}\|_4
       \pB{\|\p{\mathscr{L}\mathcal B}_{a\bullet}\|_4+\sum_{a'}\mathcal B_{aa'}\|\mathscr{L}_{a'\bullet}\|_4}\le2\Lambda_1(\eta_t A_r)^{1/2},
\end{equation}
where in the second step, we used \eqref{eq:Lambda12}, $ \|\mathscr{L}_{a\bullet}\|_4\le(\eta_t A_r)^{1/4}$ by \eqref{simp:row-bound-B}, and 
\begin{align*}
  \|\p{\mathscr{L}\mathcal B}_{a\bullet}\|_4^4 \le  \sum_b\pB{\sum_{a'}\mathcal B_{a'b}}^3 \sum_{a'}|\mathscr{L}_{aa'}|^4 \mathcal B_{a'b} \le \Lambda_1^4\|\mathscr{L}_{a\bullet}\|_{4}^4.   
\end{align*}
Plugging \eqref{eq:minsource3} into \eqref{eq:minsource2} proves \eqref{eq:sameorder}.
For \eqref{eq:bracket}, we apply the Cauchy--Schwarz inequality to get 
\[
 \begin{split}
 \partial_t\langle \wh Y_{r,a}^{\rm mart}\rangle
 &\le4W^4\eta_t\wh Y_{r,a}(t)
 \sum_{j=1}^2\sum_{b,\phy}
 \left|\frac{\chi_j(a)-\chi_j(b)}{\ell_t}\right|^{2r}
 \bigl(|h_{t,\phy}(a,b)|^2+|h_{t,\phy}(b,a)|^2\bigr) \le C_rA_rW^2\Lambda_2\wh Y_{r,a}(t),
 \end{split}
\]
where the second inequality uses $\left|\p{\chi_j(a)-\chi_j(b)}/{\ell_t}\right|\le C\rho_t(a,b)$ and \eqref{eq:sameorder2}.
Finally, \eqref{eq:parameters} follows from \eqref{eq:clip}, $\|\difloop_1\|_\infty\le\mathfrak W=(\log W)^{6}\logpara^{3/2}/W^2$, and $\vartheta_t\le C_\kappa\logpara$.
\end{proof}

\subsection{Displacement weights on the torus}
\label{sec:commutators}

For each direction $\al\in\{1,2\}$ and tested order $r$, define
\[
 (A_{:r}^{(\al)})_{ab}:=
 \left(\frac{\chi_\al(a)-\chi_\al(b)}{\ell_t}\right)^rA_{ab},\qquad r\ge0.
\]
If the direction $\al$ is fixed, we will also abbreviate \smash{$A_{:r}\equiv A_{:r}^{(\al)}$} and $\chi=\chi_\al$. We use the convention $A_{:0}\equiv A_{:0}^{(\al)}=A$. By definition, it is easy to see that  
\begin{equation}\label{simp:commcalculus0}
\dd(A_{:r})=(\dd A)_{:r}-\frac r2A_{:r}\,\frac{\dd t}{1-t},\qquad
 (\diag(A\mathbf1))_{:r}=0,\qquad r\ge1 .
\end{equation}
For matrices weighted by these periodic displacement weights, we have the following \emph{exact product rule} for any matrices $A,B,F\in \C^{\Zn\times \Zn}$:
\begin{equation}\label{simp:commcalculus}
 (ABF)_{:r}=\sum_{i+j+k=r}\binom{r}{i,j,k}A_{:i}B_{:j}F_{:k}.
\end{equation} 
Indeed, expanding $[(\chi_a-\chi_{a'})/\ell_t+(\chi_{a'}-\chi_{b'})/\ell_t +(\chi_{b'}-\chi_{b})/\ell_t]^rA_{aa'}B_{a'b'}F_{b'b}$ and summing over $a',b'$ proves \eqref{simp:commcalculus}. We will also use the following simple interpolation bound. 

\begin{lemma}\label{lem:interp}
Fix $t\in[0,t_*]$ and a direction $\al\in\{1,2\}$, and omit the superscript $(\al)$ from the displacement-weighted matrices. For every $A\in\mathbb C^{\Zn\times\Zn}$,
\begin{equation}\label{eq:interp-elementary}
 \|A_{:1}\|\le8\|A\|^{2/3}\|A_{:3}\|^{1/3},\qquad
 \|A_{:2}\|\le8\|A\|^{1/3}\|A_{:3}\|^{2/3}.
\end{equation}
\end{lemma}
\begin{proof}
Define \smash{$\mathbf E_\al:=\diag(e^{2\pi\ii a_\al/L})_{a\in\Zn}$} and abbreviate $\delta:=L/(2\pi\ell_t)$. For any integer $j\ge1$ and $h>0$, we introduce the perturbation of $A_{:j-1}$ as  
\(
 \mathcal A:=A_{:j-1}+hA_{:j}\mathbf E_\al^*.
\)
We can check that 
\begin{equation}    
 (1+\delta h)\mathcal A-\delta h\mathbf E_\al\mathcal A\mathbf E_\al^*
 =A_{:j-1}-h^2A_{:j+1}(\mathbf E_\al^*)^2,
\label{eq:recursive-Aj}  
\end{equation}
by applying $(1+x)(1-x)=1-x^2$ entrywise with $x=h(\chi_\al(a)-\chi_\al(b))e^{-2\pi\ii b_\al/L}/\ell_t$.

By the unitarity of $\mathbf E_\al$ and the triangle inequality, we get from \eqref{eq:recursive-Aj} that
\[
 \|\mathcal A\|\le\|(1+\delta h)\mathcal A-\delta h\mathbf E_\al\mathcal A\mathbf E_\al^*\|
 \le\|A_{:j-1}\|+h^2\|A_{:j+1}\|,
\]
which implies 
\[
 h\|A_{:j}\|=\|\mathcal A-A_{:j-1}\|
 \le2\|A_{:j-1}\|+h^2\|A_{:j+1}\|\implies \|A_{:j}\|^2\le8\|A_{:j-1}\|\,\|A_{:j+1}\|.
\]
Here the implication is derived by dividing by $h$ and minimizing over $h>0$. Combining this inequality for $j=1,2$ proves \eqref{eq:interp-elementary}. \end{proof}

Applying \eqref{eq:mainStoflow} with $\star=\emptyset$ and \eqref{simp:commcalculus0}--\eqref{simp:commcalculus}, for each fixed direction $\al\in\{1,2\}$ (which is omitted from the superscript) and every integer $r\ge0$, we obtain the displacement-weighted evolution equation 
\begin{equation}\label{simp:commflow}
 \dd\mathcal L_{:r}(t)
 = \biggl[W^2 
       \sum_{i+j+k=r}\binom r{i,j,k}
       \mathcal L_{:i}(t)S^{(\sB)}_{:j}\mathcal L_{:k}(t)
       -\frac r2\frac{\mathcal L_{:r}(t)}{1-t}
       + \mathcal W_{:r}(t)\biggr] \dd t  +\bigl(\dd\mathcal M(t)\bigr)_{:r}.
\end{equation}
By the definition of $S^{(\sB)}$ and \eqref{eq:high-exponential-row-moment}, we see that for each fixed order $j\ge 0$,   
\begin{align}\label{simp:interp0}
 \|S^{(\sB)}_{:j}\|\le C_j\ell_t^{-j},\qquad \|\pa{\Theta_t}_{:j}\|\le C_j\eta_t^{-1}.
\end{align}
To control the equation \eqref{simp:commflow}, we collect the corresponding drift and martingale estimates in the following lemma. Recall the drift term $\mathfrak F$ defined in \eqref{a1:eq:Fcomplete}.

\begin{lemma}
For every fixed integer $r\ge 1$, there exists a constant $C_r>0$ such that on $\{t\le \tau\}$,
\begin{align}
  W^2\eta_t^2
 \|(\mathfrak F)_{:r}\| &\le C_r \|\difloop_1(t)\|_\infty\sqrt{\mu_{2r}}+C_r \|\difloop_1(t)\|_\infty^2 ,\label{simp:comm-inputs}\\
 \normB{\sum_{\phy}(h_{t,\phy})_{:r}^*(h_{t,\phy})_{:r}}
 +\normB{ \sum_{\phy}(h_{t,\phy})_{:r}(h_{t,\phy})_{:r}^*} 
 &\le C_r \eta_t^{-3} W^{-4}O_{\op}(t)^2\mu_{2r},\label{simp:comm-inputs2}
 \end{align}
where $\mu_k(t)$, for integers $k\ge0$, is defined by 
\begin{equation}
 \mu_k(t):=1+W^2\eta_t 
 \max_a\sum_b(1+\rho_t(a,b))^k\qa{\mathcal L_{ab}(t)+\mathcal L_{ba}(t)}.
\label{eq:def-muk}
\end{equation}
\end{lemma}
\begin{proof}
Given $x\in \Zn$ and vectors $f,g$, we denote
\[
 f_a^{x,i}:=\left[ \p{\overline \chi_a-\overline \chi_x}/{\ell_t}\right]^i f_a,
 \qquad
 g_b^{x,j}:=\left[\p{\chi_b-\chi_x}/{\ell_t}\right]^j g_b.
\]
Then, writing $\chi_a-\chi_b=\p{\chi_a-\chi_x}+ \p{\chi_x-\chi_b}$ and applying the binomial theorem, we can write 
\[
 \langle f,(K_x)_{:r}g\rangle
 =\sum_{i=0}^r(-1)^{r-i}\binom r i  \langle f^{x,i},K_xg^{x,r-i}\rangle.
\]
By decomposing each vector into its real and imaginary parts, \Cref{lem:root} extends to complex vectors, up to an absolute constant. Applying this extension to each summand, we find that the marginal terms involving $K_x\1$ in \eqref{simp:rootbound} cancel: they share the same weight, while $\sum_{i=0}^r(-1)^{r-i}\binom ri=0$ for $r\ge1$. Combining this cancellation with the Cauchy--Schwarz inequality gives
\begin{equation}     
 \sum_x \absa{\langle f,(K_x)_{:r}g\rangle} \le C_r\eta_t^{-1} \|\difloop_1\|_\infty \sum_{i=0}^r \pB{\sum_x\mathscr E_{t,x}^-(f^{x,i})}^{1/2} \pB{\sum_x\mathscr E_{t,x}^+(g^{x,r-i})}^{1/2} .
\label{eq:f-Kxr-g} 
\end{equation}
To bound the RHS, we claim that  
\begin{equation}     
\sum_x\mathscr E_{t,x}^\pm(f^{x,i})
 \le C_i\eta_t^{-1}W^{-2}\mu_{2i}\|f\|^2,\qquad 0\le i\le r.
\label{eq:f-Kxr-g0} 
\end{equation}
By symmetry, we only need to prove \eqref{eq:f-Kxr-g0} for the $-$ sign case. For $i>0$, we have $f_x^{x,i}=0$. For the $\mathcal L$-contribution to $\mathscr E_{t,x}^-(f^{x,i})$, we have
\begin{align*}
\sum_{x,a',a}S^{(\sB)}_{a'x}\mathcal L_{aa'}(t)
       |f_a^{x,i}-f_x^{x,i}|^2 &=\sum_a|f_a|^2\sum_{a'}\mathcal L_{aa'}(t)
       \sum_xS^{(\sB)}_{a'x}
       \left|\frac{\chi_a-\chi_x}{\ell_t}\right|^{2i}\\
 &\le C_i\sum_a|f_a|^2\sum_{a'}
       (1+\rho_t(a,a'))^{2i}\mathcal L_{aa'}(t)
 \le C_i\eta_t^{-1}W^{-2}\mu_{2i} \|f\|^2.
 \end{align*}
The $\cK$-contribution satisfies, for every fixed $i\ge0$,
\[
 \sum_{x,a',a}S^{(\sB)}_{a'x}\cK_{aa'}(t)
 |f_a^{x,i}-f_x^{x,i}|^2\le C_i \eta_t^{-1} W^{-2}\|f\|^2,
\]
by $\cK=W^{-2}\Theta_t$ and \eqref{eq:high-exponential-row-moment}. Note this also handles the $\cK$-contribution to the $i=0$ case of $\mathscr E_{t,x}^-(f^{x,i})$. It remains to bound the $\mathcal L$-contribution to \smash{$\mathscr E_{t,x}^-(f^{x,i})$} with $i=0$:
\begin{align*}
&\sum_{x,a',a}S^{(\sB)}_{a'x}\mathcal L_{aa'}|f_a-f_x|^2 \le
2\sum_a|f_a|^2\sum_{a'}\mathcal L_{aa'}
+2\sum_x|f_x|^2\sum_{a'}S^{(\sB)}_{a'x}\sum_a\mathcal L_{aa'} \le C\eta_t^{-1}W^{-2}\|f\|^2 ,
\end{align*}
where the second inequality uses Ward's identity. This concludes \eqref{eq:f-Kxr-g0}.

Plugging \eqref{eq:f-Kxr-g0} into \eqref{eq:f-Kxr-g}, we obtain 
 \begin{align}\label{eq:bound-f-Kxr-g}
 W^2\eta_t^2\absB{\avgB{ f,\sum_x(K_x)_{:r}g}}
 &\le C_r\|f\|\|g\| \|\difloop_1\|_\infty\sum_{i=0}^r
        \sqrt{\mu_{2i}\mu_{2r-2i}} \le C_r\|f\|\|g\| \|\difloop_1\|_\infty\sqrt{\mu_{2r}}\,.
 \end{align}
In the second step, we used $\mu_{2i}\le C_r\mu_{2r}^{i/r}$ on $\{t\le \tau\}$. Indeed, H\"older's inequality gives
\begin{align*}
&\sum_b(1+\rho_t(a,b))^{2i}
(\mathcal L_{ab}+\mathcal L_{ba}) \le \qB{\sum_b(\mathcal L_{ab}+\mathcal L_{ba})}^{1-i/r} \qB{\sum_b(1+\rho_t(a,b))^{2r}
(\mathcal L_{ab}+\mathcal L_{ba})}^{i/r} \le \frac{C_r \mu_{2r}^{i/r}}{W^2\eta_t},
\end{align*}
where we used Ward's identity for the second inequality. For the second $\mathcal Q$ term in \eqref{eq:decompose-F-rootx}, we have 
\begin{align}\label{eq:bound-Qr}
 \|\cal Q_{:r}\|=W^2\norma{ \bigl[\diag\p{\mathsf f(t)} S^{(\sB)}\diag\p{\mathsf f(t)}\bigr]_{:r}}
 &=W^2\norma{\diag\p{\mathsf f(t)} S^{(\sB)}_{:r}\diag\p{\mathsf f(t)}} \le \frac{C_r\ell_t^{-r}}{W^{2}\eta_t^{2}}\|\difloop_1\|_\infty^2,   
\end{align} 
where we used \eqref{simp:interp0} in the last step.
Together with \eqref{eq:bound-f-Kxr-g}, this concludes \eqref{simp:comm-inputs}.

For \eqref{simp:comm-inputs2}, we claim that
\begin{equation}\label{simp:rowtest}
\sum_{\phy}\absB{\sum_b f_bh_{t,\phy}(a,b)}^2
 \le \frac{C_\kappa O_{\op}(t)^2}{ W^2\eta_t^2}\sum_b \mathcal L_{ab}|f_b|^2.
\end{equation}
It suffices to prove this bound for a real vector $f$. Using the definition of $h_{t,\phy}$ in \eqref{eq:hue} and the notation $\mathbf U_a=\sqrt{\eta_t/W^2} G_t^*P_a$ and $\mathbf F=\sum_bf_bP_b$, we can write 
\[
 \sum_bf_bh_{t,\phy}(a,b)
   =-\frac{2 }{W^2\eta_t} \re\Tr(V_{\phy}\mathbf U_a\mathbf U_a^*\mathbf FG_t^*),\qquad
 \|P_{b'}\mathbf U_a\|^2=\eta_t W^2\|X_{ab'}\|_{\op}^2\le\eta_t W^2O_{\op}(t)^2.
\]
Using these equations and the bound $\|G_t^*\|\le \eta_t^{-1}$, we then bound the LHS of \eqref{simp:rowtest} by
 \begin{align*}
\sum_{\phy}\absB{\sum_bf_bh_{t,\phy}(a,b)}^2
 &\le\frac{C}{ W^6\eta_t^2}
       \sum_{a',b'}S^{(\sB)}_{a'b'}\|P_{b'}\mathbf U_a\mathbf U_a^*\mathbf FG_t^*P_{a'}\|_{\HS}^2 \le\frac{CO_{\op}(t)^2}{W^4\eta_t} \sum_{a'}\|\mathbf U_a^*\mathbf FG_t^*P_{a'}\|_{\HS}^2\\
 &\le\frac{CO_{\op}(t)^2}{W^4\eta_t^3}\|\mathbf U_a^*\mathbf F\|_{\HS}^2 = \frac{CO_{\op}(t)^2}{W^2\eta_t^2}\sum_b \mathcal L_{ab}|f_b|^2.
 \end{align*}
 This concludes \eqref{simp:rowtest}. 
Applying \eqref{simp:rowtest} to $((\chi_a-\chi_b)/\ell_t)^rf_b$ and then summing in $a$, we conclude
\[
 \sum_{a,\phy}\left|
       \sum_b\left(\frac{\chi_a-\chi_b}{\ell_t}\right)^r
                       h_{t,\phy}(a,b)f_b\right|^2
 \le \frac{C_rO_{\op}(t)^2}{W^4\eta_t^3}\mu_{2r}\|f\|^2.
\]
This proves the desired bound for the first variance in \eqref{simp:comm-inputs2}. The second variance can be bounded in a similar way.  \end{proof}

\subsection{Bootstrap estimate for $\mathcal Y_{k_*+1}$} \label{sec:anchors}

In this subsection, we establish the bootstrap estimates for $\mathcal Y_{k_*+1,a}$ using the dynamics of $\cL_{:r}$ and $\difloop_{:r}$ derived from \eqref{simp:commflow}. In addition to the highest order $k_*+1$, we use the third-order dynamics for the following reason. 
Fix one coordinate direction $\al\in\{1,2\}$ and omit the superscript $(\al)$ from the notation below.
We define the third-order displacement-weighted comparison operator $T$ by
\[
 T(t):=W^2(1-t)\difloop_{:3}(t)
 =W^2(1-t)\mathcal L_{:3}(t)-(1-t) \pa{\Theta_t}_{:3}.
\]
By \eqref{eq:tau-aux3}, $W^2(1-t)\|\mathcal D(t)\|\le(\log W)^{-1}$ on $\{t\le \tau\}$, so 
\begin{equation}      
W^2(1-t)\|\mathcal L(t)\|\le 1+(\log W)^{-1}. \label{eq:L-norm-t}
\end{equation}
Multiplying \eqref{simp:commflow} at $r=3$ by $W^2(1-t)$, subtracting the equation for $(1-t) \pa{\Theta_t}_{:3}$, and using \eqref{simp:commcalculus} gives
\begin{equation}     
 dT=\qa{W^2(1-t) (TS^{(\sB)}\mathcal L +\mathcal L S^{(\sB)}T)-\frac52 T 
       +W^2(1-t)^2\pa{\mathcal R_3+ \mathcal W_{:3}}} \frac{\dd t}{1-t} 
       +W^2(1-t)\bigl(\dd\mathcal M(t)\bigr)_{:3},
\label{eq:drift-T} 
\end{equation}
where the drift remainder $\mathcal R_3$ is given by
 \begin{align*}
 \mathcal R_3= \pa{\Theta_t}_{:3}S^{(\sB)}\mathcal D +\mathcal DS^{(\sB)}\pa{\Theta_t}_{:3} + \sum_{\substack{i+j+k=3\\i,k<3}}\binom3{i,j,k}
 \Bigl[\mathcal D_{:i} S^{(\sB)}_{:j}\pa{\Theta_t}_{:k} +\pa{\Theta_t}_{:i}S^{(\sB)}_{:j}\mathcal D_{:k} +W^2\mathcal D_{:i} S^{(\sB)}_{:j}\mathcal D_{:k} \Bigr].
\end{align*} 
The first term in brackets on the RHS of \eqref{eq:drift-T} is linear in $T$ and contributes to the leading drift. In the larger-bandwidth regime, we handle this term by writing $\cL=\cD+\cK$ and incorporating the leading $\cK$-contribution into the propagator via Duhamel's principle. Here, we absorb it into the damping term $-\tfrac52T$.
Indeed, by \eqref{eq:L-norm-t}, the linear contribution on this time scale is bounded by $2(1+(\log W)^{-1})\|T\|$. For an order-$r$ comparison operator, the damping coefficient is $1+r/2$, which strictly exceeds this bound for $r\ge3$. This is why we use the third-order comparison operator $T$ and its dynamics \eqref{eq:drift-T}.
Combining the bound on $\|T\|$ with the interpolation inequality \eqref{eq:interp-elementary}, we obtain bounds on $\|\mathcal L_{:j}(t)\|$ for $1\le j\le 3$, as well as on \smash{$\wh Y_{0,a}$ and $\wh Y_{3,a}$}. We then use these bounds to complete the bootstrap estimate for \smash{$\wh Y_{k_*+1,a}$}.

\begin{proposition}[Bootstrap estimate for $\mathcal Y_{k_*+1}$]  \label{lem:spatialclosure}
Let $k_*$ be chosen as in \eqref{eq:spatialorders}. There exist constants $A_1,A_3,A_{k_*+1}>0$, depending only on the structural parameters and $k_*$ and independent of $N,W,$ and the probability depth $D$, such that the following holds. For every fixed $D>0$, there exists $N_0(D)$ such that, for $N\ge N_0(D)$, with probability at least $1-N^{-D}$, the bounds
\begin{align}
& W^2\eta_t\max_{1\le j \le 3}\|\mathcal L_{:j}(t)\|\le A_1, \label{eq:bootstrap-L123}\\ 
&  W^4\eta_t\max_a\sum_b(1+\rho_t(a,b))^6
       (\mathcal L_{ab}^2+\mathcal L_{ba}^2)\le A_3,\label{eq:bootstrap-T}\\ 
&  W^4\eta_t\max_a\sum_b(1+\rho_t(a,b))^{2(k_*+1)}
       (\mathcal L_{ab}^2+\mathcal L_{ba}^2)\le A_{k_*+1},\label{eq:bootstrap-k*} 
 \end{align}
hold simultaneously for all $0\le t\le\tau_*$ and both coordinate directions (whose superscripts are omitted). In particular, \eqref{eq:bootstrap-k*} gives that
\(
 \max_a\mathcal Y_{k_*+1,a}(t)\le (\log W)^{-\fc}{A_{k_*+1}}=\oo(1).
\)
\end{proposition}
\begin{proof}
By \Cref{direct:continuation}, it suffices to establish these estimates for $t\le\tau$, which we assume throughout the proof.
We first bound the third-order comparison operator $T$, which will yield the bound \eqref{eq:bootstrap-L123}. 

\par\medskip\noindent{\bf Bounding $T$.} 
By the interpolation bound \eqref{eq:interp-elementary} and \eqref{eq:tau-aux3}, we have 
\begin{align}\label{eq:interp-D3}
  W^2(1-t)\|\mathcal D_{:i}(t)\|&\le C(\log W)^{i/3-1}\|T\|^{i/3},\qquad 0\le i \le 3.    
\end{align}
Combining this with \eqref{simp:interp0} and Young’s inequality $(\log W)^{\theta-1}\|T\|^\theta\le \e \|T\|+C_\e (\log W)^{-1}$ for $0\le\theta<1$, we obtain
\begin{equation}    
 W^2\eta_t^2\|\mathcal R_3\|\le \pa{\e+C_\e(\log W)^{-1}}\|T\|+C_\e (\log W)^{-1}.
\label{eq:T3-bound-R3}  
\end{equation}

To estimate \(\mathcal W_{:3}\) in \eqref{eq:drift-T}, we define $\mathsf F(t):=W^2(1-t)\diag\mathsf f(t)$. On $\{t\le \tau\}$, we have 
\begin{equation}\label{eq:F:r}
\|\mathsf F(t)\|\le C\|\difloop_1(t)\|_{\infty}\le C\mathfrak W,\quad  \text{and}\quad \mathsf F_{:j}(t)=0 \ \ \text{for}\ \ j>0. 
\end{equation}
Then, equation~\eqref{a1:eq:Fcomplete} gives
\begin{equation}\label{eq:spatial-quadratic-completion}
 W^2(1-t)^2\mathcal W =W^2(1-t)^2\mathfrak F -W^2(1-t) (\mathcal D S^{(\sB)}\mathsf F(t)+\mathsf F(t)S^{(\sB)}\mathcal D)+ \mathsf F(t)S^{(\sB)}\mathsf F(t).
\end{equation}
Applying \eqref{simp:commcalculus}, we express $W^2(1-t)^{2}\mathcal W_{:3}$ as a sum of terms controlled by \eqref{simp:interp0}, \eqref{simp:comm-inputs}, \eqref{eq:F:r}, and \eqref{eq:interp-D3}.
To apply \eqref{simp:comm-inputs}, we first use \eqref{eq:L2-calL}, \eqref{eq:Yrat}, and the stopping condition $\mathcal {Y}_{k_*+1,a}\le 1$ to obtain  
\begin{equation}    
 W^4\eta_t \max_a\sum_b(1+\rho_t(a,b))^{2(k_*+1)}
 (\mathcal L_{ab}^2+\mathcal L_{ba}^2)
 \le C (\log W)^{2\fc}.
\label{eq:coordinate-Y-k*}  
\end{equation}
We then get from Cauchy--Schwarz that 
\begin{equation} \mu_6\le 1+ \qB{\eta_t\sum_b(1+\rho_t(a,b))^{-4}}^{1/2} \qB{W^4\eta_t\sum_b(1+\rho_t(a,b))^{16}(\mathcal L_{ab}^2+\mathcal L_{ba}^2)}^{1/2} \le  C (\log W)^{\fc}.\label{eq:coordinate-Y-6}  
\end{equation}
The second inequality follows from \eqref{eq:coordinate-Y-k*}, since $2(k_*+1)>16$ and $\eta_t\sum_b(1+\rho_t(a,b))^{-4}\le C\eta_t\ell_t^2=C$. 
Using \eqref{simp:comm-inputs} with \(r=3\) and the input \eqref{eq:coordinate-Y-6}, together with \eqref{eq:interp-D3} (along with an application of Young’s inequality), \eqref{simp:interp0}, and \eqref{eq:F:r}, bounds the terms obtained from \eqref{eq:spatial-quadratic-completion} (by applying \eqref{simp:commcalculus} to the RHS). This yields
\begin{align}\label{eq:T3-bound-W3}
 W^2\eta_t^2\|\mathcal W_{:3}\|
 \le C(\log W)^{\fc} \mathfrak W +C\mathfrak W\|T\| .
 \end{align}

We finally bound the quadratic variation of the martingale term in \eqref{eq:drift-T}. By \eqref{simp:comm-inputs2}, \eqref{eq:coordinate-Y-6}, and \eqref{eq:entropy-operator-output-sm},
\begin{align}\label{eq:T3-bound-MG}
 W^4\eta_t^3\pbb{\normB{\sum_{\phy}(h_{t,\phy})_{:3}^{*}(h_{t,\phy})_{:3}}
 +\normB{\sum_{\phy}(h_{t,\phy})_{:3}(h_{t,\phy})_{:3}^{*}}}
 \le C(\log W)^{10+\fc}\logpara^2/W^4.
 \end{align}
Define the stopped martingale
\[
 M_3(t):=W^2\int_0^{t\wedge\tau}(1-u)
                      \bigl(\dd\mathcal M(u)\bigr)_{:3}.
\]
By \eqref{eq:T3-bound-MG} and since $\int_0^{t_*}(1-u)^{-1}\dd u\le C\logpara$, its two quadratic variations are bounded by
\(
 C(\log W)^{10+\fc}\logpara^3/W^4.
\)
Applying \Cref{lem:matrix-maximal} with $R=C_D(\log W)^{5+\fc/2}\logpara^2/W^2$ for a large enough constant $C_D>0$, we conclude
\begin{align}\label{eq:T3-MG-high-prob}
 \Pp\left\{\sup_{t\le t_*}\|M_3(t)\|>C_D(\log W)^{5+\fc/2} \frac{\logpara^2}{W^2}\right\}\le N^{-D}.\end{align}

We choose $\e>0$ in \eqref{eq:T3-bound-R3} sufficiently small and then take $N$ sufficiently large. To make use of the linear damping pathwise, let $U(t):=T(t)-M_3(t)$ for $t\le\tau$. By \eqref{eq:drift-T}, $U$ is absolutely continuous and satisfies
\[
 \partial_t U=\frac{1}{1-t}\Bigl[
 W^2(1-t)\bigl((U+M_3)S^{(\sB)}\mathcal L
 +\mathcal L S^{(\sB)}(U+M_3)\bigr)
 -\frac52(U+M_3)+W^2(1-t)^2(\mathcal R_3+\mathcal W_{:3})\Bigr].
\]
The linear term $W^2(1-t) \p{US^{(\sB)}\cL+\cL S^{(\sB)}U}$ has operator norm at most $2(1+(\log W)^{-1})\|U\|$ by \eqref{eq:L-norm-t}. Moreover, \eqref{eq:T3-bound-R3} and \eqref{eq:T3-bound-W3}, together with $\eta_t\asymp_\kappa1-t$ and $\|T\|\le\|U\|+\|M_3\|$, bound the remainder by a sufficiently small multiple of $\|U\|$, plus
$C\|M_3\|+C[(\log W)^{-1}+(\log W)^{\fc}\mathfrak W]$.
Consequently, the operator norm of this absolutely continuous process satisfies, for almost every $t<\tau$ and a constant $c>0$,
\begin{align}\label{eq:Gronwall-T}
 \frac{\dd}{\dd t}\|U(t)\|
 \le \frac{-c\|U(t)\|+C\|M_3(t)\|
 +C[(\log W)^{-1}+(\log W)^{\fc}\mathfrak W]}{1-t}.
\end{align}
Since $U(0)=0$, Gronwall's inequality and $T=U+M_3$ give
\begin{align}\label{eq:Gronwall-T2}
 \|T(t)\|\le \|M_3(t)\|+C\pB{\sup_{u\le t}\|M_3(u)\|
 + (\log W)^{-1}+(\log W)^{\fc}\mathfrak W}
 \int_0^t\left(\frac{1-t}{1-u}\right)^c\frac{\dd u}{1-u},
 \qquad t\le\tau.
\end{align}
Note that the uniform damping prevents a loss of $\logpara$ in the integral. Indeed, it introduces the integral factor $(1-t)^c/(1-u)^c$, and 
\smash{\(
 \int_0^t \p{1-t}^c/\p{1-u}^c\cdot \p{1-u}^{-1} \dd u 
 \le c^{-1}.
\)}
Taking the supremum of \eqref{eq:Gronwall-T2} over $t\le\tau$ and using \eqref{eq:T3-MG-high-prob} therefore gives, with probability at least $1-N^{-D}$,
 \begin{align*}
 \sup_{t\le\tau}\|T(t)\|
 &\le C_D\left[(\log W)^{-1}+(\log W)^{\fc}\mathfrak W
       +(\log W)^{5+\fc/2}\frac{\logpara^2}{W^2}\right] =\oo(1).
 \end{align*}
Combining this bound with \eqref{simp:interp0} and \eqref{eq:interp-D3}, we get that for $1\le j\le3$,
\[
 W^2(1-t)\|\mathcal L_{:j}\|\le (1-t)\|\pa{\Theta_t}_{:j}\|+W^2(1-t)\|\mathcal D_{:j}\|
       \le A_1
\]
for a structural constant $A_1>0$. This proves \eqref{eq:bootstrap-L123}.

\par\medskip\noindent{\bf Bounding $\wh Y_{0}$ and $\wh Y_{3}$.}
We next prove \eqref{eq:bootstrap-T} and \eqref{eq:bootstrap-k*} using the dynamics for $\wh Y_{r,a}$. Combining \eqref{simp:commflow} with It\^o's formula, we get  
\begin{align}\label{eq:Yhat-dynamics}
 \dd\wh Y_{r,a}(t)
 = &-\frac{r+1}{1-t}\wh Y_{r,a}(t)\,\dd t +2W^4\eta_t\re\sum_{\al=1}^2\sum_b
 \qa{\overline{(\mathcal L_{:r}^{(\al)})_{ab}}
       (\mathcal W_{:r}^{(\al)})_{ab}
       +(a\leftrightarrow b)}\,\dd t\notag\\
 &+2W^6\eta_t\re\sum_{\al=1}^2\sum_b
       \sum_{i+j+k=r}\binom r{i,j,k}
 \qa{\overline{(\mathcal L_{:r}^{(\al)})_{ab}}
       \pa{\mathcal L_{:i}^{(\al)}(S^{(\sB)})_{:j}^{(\al)}
        \mathcal L_{:k}^{(\al)}}_{ab}
       +(a\leftrightarrow b)}\,\dd t\notag\\
 &+W^4\eta_t\sum_{\al=1}^2\sum_b\sum_{\phy\in\cI}
 \ha{\bigl|\qb{(h_{t,\phy})_{:r}^{(\al)}}_{ab}\bigr|^2
       +(a\leftrightarrow b)}\,\dd t
 +\dd\wh Y_{r,a}^{\rm mart}(t).
\end{align}
Here all matrices on the RHS are evaluated at $t$, $(a\leftrightarrow b)$ denotes the preceding term with $a$ and $b$ exchanged, and \smash{$\dd\wh Y_{r,a}^{\rm mart}(t)$} is defined in \eqref{eq:dY-mart}.
Applying the Cauchy--Schwarz inequality together with \eqref{eq:sameorder}--\eqref{eq:parameters} (with $A_r=C(\log W)^{2\fc}$ given by \eqref{eq:coordinate-Y-k*}), we obtain 
\begin{align}
 & W^4\eta_t^2\sum_{\al=1}^2\sum_b
 \pbb{2|(\mathcal L_{:r}^{(\al)})_{ab}(\mathcal W_{:r}^{(\al)})_{ab}|
 +\sum_{\phy\in\cI}\bigl|\qb{(h_{t,\phy})_{:r}^{(\al)}}_{ab}\bigr|^2
 +(a\leftrightarrow b)}\notag\\
 &\quad\le C_r(\log W)^{2\fc}W^2\eta_t
\qa{\Lambda_1\|\difloop_1\|_\infty+\Lambda_2} \le C_r(\log W)^{11+2\fc} {\logpara^2}/{W^{2-4/k_*}} =\oo(1),\label{eq:source-production}\\
& \eta_t \partial_t\langle \wh Y_{r,a}^{\rm mart}\rangle
\le C_r (\log W)^{10+2\fc} \logpara /W^{2-4/k_*} \wh Y_{r,a}(t)=\oo(1)\cdot \logpara^{-1}\wh Y_{r,a}(t) .
\label{eq:source-bracket}
\end{align}
It remains to estimate the third term on the RHS of \eqref{eq:Yhat-dynamics}, denoted by $\partial_t\wh Y_{r,a}|_{\rm quad}$.

For $r=0$, we expand the rightmost $\cL$ factor as $\mathcal L=W^{-2}\Theta_t+\mathcal D$ to obtain 
\[
 \|(\mathcal L S^{(\sB)}\mathcal L)_{a\bullet}\|_2
 \le \|\mathcal L_{a\bullet}\|_2\|\mathcal D\|
 +W^{-2}\sum_{a'}(\mathcal L S^{(\sB)})_{aa'}\|(\Theta_t)_{a'\bullet}\|_2.
\]
Ward's identity gives $\sum_{a'}(\mathcal L S^{(\sB)})_{aa'}\le C/(W^2\eta_t)$. Moreover, $\|(\Theta_t)_{a'\bullet}\|_2^2=(\Theta_t^2)_{a'a'}\le C/\eta_t$ by \eqref{eq:tools-theta-square}. Substituting these estimates into the above bound, squaring, and applying \eqref{eq:tau-aux3}, we obtain that on \(\{t\le\tau\}\)
\begin{align}\label{eq:ell2-LSL}
 W^8\eta_t^3\sum_b|(\mathcal L S^{(\sB)}\mathcal L)_{ab}|^2
 &\le C(\log W)^{-2}\wh Y_{0,a}+C.
\end{align}
Applying Cauchy--Schwarz and then \eqref{eq:ell2-LSL} and its column analogue, we get
\[
 (1-t)\partial_t\wh Y_{0,a}\big|_{\rm quad}
 \le C(\log W)^{-1}\wh Y_{0,a}+C\sqrt{\wh Y_{0,a}}
 \le \wh Y_{0,a}/10+C,
\]
where the second step uses Young's inequality. Together with \eqref{eq:source-production} and \eqref{eq:source-bracket}, this bounds \eqref{eq:Yhat-dynamics} by
\begin{equation}      
 \dd\wh Y_{0,a}
 \le\pa{-\tfrac45\wh Y_{0,a}+C}\p{1-t}^{-1} {\dd t}
       +\dd\wh Y_{0,a}^{\rm mart},\quad (1-t)\partial_t\langle \wh Y_{0,a}^{\rm mart}\rangle=\oo(1)\cdot \logpara^{-1}\wh Y_{0,a}(t) .
\label{eq:SDE-ineq-Y0}
\end{equation}
Choose a fixed $A_0>\max\{16,5C\}$. Since
$\wh Y_{0,a}(0)=4\eta_0\le4<A_0/4$, \Cref{lem:damped-scalar}, applied up to $\tau$, and a union bound over $a$ give that, for every fixed $D>0$ and sufficiently large $N$,
\begin{equation}\label{eq:bootstrap-Y0}
 \Pp\pa{\sup_{t\le\tau}\max_a\wh Y_{0,a}(t)\ge A_0}
 \le N^{-D}.
\end{equation}
Note that $A_0$ is independent of the probability depth $D$. 

Let $\tau_0$ be the first time that either \eqref{eq:bootstrap-L123} or $\max_a\wh Y_{0,a}(t)\le A_0$ fails. Then, with probability at least $1-N^{-D}$, we have $\tau_0\ge\tau$.
For $r=3$, we work up to $\tau_0\wedge\tau$. 
By H\"older's inequality, on $\{t\le \tau_0\wedge\tau\}$,
\begin{equation}     
 \wh Y_{i,a}^{1/2}
 \le\wh Y_{0,a}^{(3-i)/6}\wh Y_{3,a}^{i/6}
 \le A_0^{(3-i)/6}\wh Y_{3,a}^{i/6},\qquad 0\le i<3.
\label{eq:interp-Yi} 
\end{equation}
In $\partial_t\wh Y_{3,a}|_{\rm quad}$, we distinguish the two orientations of the extreme products. By \eqref{eq:L-norm-t} and Cauchy--Schwarz, 
\[2W^6\eta_t (1-t)\re\sum_{\al=1}^2\sum_b
 \qa{\overline{(\mathcal L_{:3}^{(\al)})_{ab}}
       \pa{\mathcal L_{:3}^{(\al)} S^{(\sB)} 
        \mathcal L}_{ab}+ \overline{(\mathcal L_{:3}^{(\al)})_{ba}}
       \pa{\mathcal L S^{(\sB)} 
        \mathcal L_{:3}^{(\al)}}_{ba}}\le 2(1+(\log W)^{-1}) \wh Y_{3,a}. \]
For the opposite orientations, \eqref{eq:bootstrap-L123} gives, in each coordinate direction,
\[
 W^8\eta_t^3\|(\mathcal LS^{(\sB)}\mathcal L_{:3})_{a\bullet}\|_2^2 + W^8\eta_t^3\|(\mathcal L_{:3}S^{(\sB)}\mathcal L)_{\bullet a}\|_2^2
 \le C W^4\eta_t\pa{\|\mathcal L_{a\bullet}\|_2^2+\|\mathcal L_{ \bullet a}\|_2^2} \le C \wh Y_{0,a}.
\]
Thus, by Cauchy--Schwarz and Young's inequality, 
\[2W^6\eta_t (1-t)\re\sum_{\al=1}^2\sum_b
 \qa{\overline{(\mathcal L_{:3}^{(\al)})_{ab}}
       \pa{\mathcal L S^{(\sB)} 
        \mathcal L_{:3}^{(\al)}}_{ab}+ \overline{(\mathcal L_{:3}^{(\al)})_{ba}}
       \pa{\mathcal L_{:3}^{(\al)} S^{(\sB)} 
        \mathcal L}_{ba}} \le  C \sqrt{\wh Y_{3,a}\wh Y_{0,a}}
 \le \e\wh Y_{3,a}+C_\e. \]
For the other terms with $i,k<3$, fix a coordinate direction $\al\in\{1,2\}$. Using Cauchy--Schwarz and \eqref{eq:bootstrap-L123}, we obtain that on $\{t\le \tau_0\wedge\tau\}$,
\begin{align*}
W^6\eta_t^2 \sum_b \absa{\overline{(\mathcal L_{:3})_{ab}} \pb{\mathcal L_{:i}(S^{(\sB)})_{:j}
        \mathcal L_{:k}}_{ab}}&\le \sqrt{\wh Y_{3,a}}  \pbb{ W^8\eta_t^3\sum_b\absB{\pb{\mathcal L_{:i}(S^{(\sB)})_{:j}\mathcal L_{:k}}_{ab}}^2}^{1/2}  \\
        &\le C\max_{0\le i<3}\sqrt{\wh Y_{3,a}\wh Y_{i,a}}  \le  \e\wh Y_{3,a}+C_\e,
\end{align*}
where in the last step, we used \eqref{eq:interp-Yi} and Young's inequality again. The contribution from the terms with $(a\leftrightarrow b)$ is bounded similarly.
Combining these estimates, and applying \eqref{eq:source-production} and \eqref{eq:source-bracket}, we obtain an inequality analogous to \eqref{eq:SDE-ineq-Y0} provided $\e$ is sufficiently small, independently of $D$: for $t\le \tau_0\wedge\tau$,
\[
 \dd\wh Y_{3,a}
 \le(-c_3\wh Y_{3,a}+C)\p{1-t}^{-1}{\dd t}
        +\dd\wh Y_{3,a}^{\rm mart},\qquad
 (1-t)\partial_t\langle\wh Y_{3,a}^{\rm mart}\rangle
 \le \oo(1)\cdot \logpara^{-1}\wh Y_{3,a}.
\]
Choose a fixed $C_3>\max\{1,4C/c_3\}$. Since
$\wh Y_{3,a}(0)=0$, applying \Cref{lem:damped-scalar} and a union bound over $a$ yields 
\begin{equation}\label{eq:bootstrap-Y3}
 \Pp\pa{\sup_{t\le\tau\wedge\tau_0}\max_a\wh Y_{3,a}(t)\ge C_3}
 \le N^{-D}
\end{equation}
for any fixed $D>0$. Combining \eqref{eq:bootstrap-Y0} with \eqref{eq:bootstrap-Y3} and using \eqref{eq:coordinate-radial-comparison}, we conclude \eqref{eq:bootstrap-T}.

\par\medskip\noindent{\bf Bounding $\wh Y_{k_*+1}$.}
Finally, we prove \eqref{eq:bootstrap-k*} and hence establish the bootstrap estimate for $\max_a \mathcal Y_{k_*+1,a}$. Let $\tau_3$ be the first time that either \eqref{eq:bootstrap-L123} or \eqref{eq:bootstrap-T} fails, and define $\tau':=\tau\wedge\tau_3$. 
By \eqref{eq:coordinate-radial-comparison}, it suffices to establish a constant bound for \smash{$\wh Y_{k_*+1,a}(t)$} with probability at least $1-N^{-D}$ for all $t\le \tau'$. 

For $3\le l\le k_*+1$, H\"older's inequality and \eqref{eq:bootstrap-T} give, with $\theta_l:=(l-3)/(k_*-2)$, 
 \begin{align*}
 \sum_b(1+\rho_t(a,b))^{2l}\mathcal L_{ab}^2
 \le\left[\sum_b(1+\rho_t(a,b))^6\mathcal L_{ab}^2\right]^{1-\theta_l} \left[\sum_b(1+\rho_t(a,b))^{2(k_*+1)}\mathcal L_{ab}^2\right]^{\theta_l}
 \le \frac{C_l(1+\wh Y_{k_*+1,a})^{\theta_l}}{W^4\eta_t}.
 \end{align*}
For $l<3$, we use the bound \eqref{eq:bootstrap-T} directly. Together with Cauchy--Schwarz and the bound $\eta_t\sum_b(1+\rho_t(a,b))^{-4}\le C$, this yields, for $0\le i \le k_*+1$ and $0\le j \le k_*-1$,
\begin{equation}\label{simp:twoanchor}
 W^4 \eta_t \|(\mathcal L_{:i})_{a\bullet}\|_2^2
 \le C\pb{1+\wh Y_{k_*+1,a}}^{(i-3)_+/(k_*-2)},\quad 
 \mu_j\le C\pa{1+Y_{\max}}^{(j-1)_+/(2k_*-4)} ,
\end{equation}
where we abbreviate $Y_{\max}:=\max_x\wh Y_{k_*+1,x}$ and $\mu_j$ is defined in \eqref{eq:def-muk}. 
The first bound also holds for $W^4 \eta_t \|(\mathcal L_{:i})_{\bullet a}\|_2^2$ by symmetry. 

We now estimate the quadratic-product term $\partial_t\wh Y_{k_*+1,a}|_{\rm quad}$ in \eqref{eq:Yhat-dynamics} for $r=k_*+1$. Fix a coordinate direction $\al\in\{1,2\}$. For $i+j+k=k_*+1$ and $i,k<k_*+1$, first suppose $k\le k_*-2$. Using \eqref{eq:L-norm-t}, \eqref{eq:bootstrap-L123} when $k\le3$, and the bound \eqref{simp:twoanchor} when $k\ge4$, we obtain
\[
W^8 \eta_t^3
 \|(\mathcal L_{:i}S^{(\sB)}_{:j}\mathcal L_{:k})_{a\bullet}\|_2^2 \le C\pa{1+Y_{\max}}^{1-\gamma},\quad \gamma:=1/(k_*-2).
\]
For $k\in\{k_*-1,k_*\}$, applying Minkowski's inequality and using \eqref{eq:def-muk}, we obtain
\[
 W^8 \eta_t^3\|(\mathcal L_{:i}S^{(\sB)}_{:j}\mathcal L_{:k})_{a\bullet}\|_2^2  \le  C\mu_i^2 \max_x W^4\eta_t\|(\mathcal L_{:k})_{x\bullet}\|_2^2 \le C\pa{1+Y_{\max}}^{[(i-1)_++k-3]/(k_*-2)} \le C\pa{1+Y_{\max}}^{1-\gamma},
\]
where we used \eqref{simp:twoanchor} in the second step and $i\le k_*+1-k$ in the last. Combining these bounds with Cauchy--Schwarz, we obtain, for $i,k<k_*+1$,  
\begin{align}\label{eq:mixed-product-bound}
W^6\eta_t^2 \sum_b \pa{\absa{ {(\mathcal L_{:(k_*+1)})_{ab}}
       \pb{\mathcal L_{:i}(S^{(\sB)})_{:j}
        \mathcal L_{:k}}_{ab}}+\pa{a\leftrightarrow b}}\le C \sqrt{\wh Y_{k_*+1,a}}
       \pa{1+Y_{\max}}^{(1-\gamma)/2}.
\end{align}
For the extreme cases involving $\mathcal L_{:k_*+1}S^{(\sB)}\mathcal L$ and $\mathcal L S^{(\sB)}\mathcal L_{:k_*+1}$, \eqref{eq:L-norm-t} and Cauchy--Schwarz give
\begin{align}\label{eq:extreme-product-bound1}
&2W^6\eta_t (1-t) \sum_{\al=1}^2\sum_b \pa{\absa{ {\pb{\mathcal L^{(\al)}_{:(k_*+1)}}_{ab}} \pb{\mathcal L^{(\al)}_{:(k_*+1)}S^{(\sB)}
        \mathcal L}_{ab}}+\absa{ {\pb{\mathcal L^{(\al)}_{:(k_*+1)}}_{ba}} \pb{\mathcal LS^{(\sB)}
        \mathcal L^{(\al)}_{:(k_*+1)}}_{ba}}} \notag\\
        &\le 2(1+(\log W)^{-1}) \wh Y_{k_*+1,a}.
\end{align}
For the other extreme product, using Minkowski's inequality and the bound $W^2(1-t)\sum_b\mathcal L_{ab}
 =W^2(1-t)\sum_b\mathcal L_{ba} =1+\OO(\mathfrak{W})\le1+(\log W)^{-1}$ by Ward's identity, we get
 \begin{align*}
 &W^4\sqrt{\eta_t}(1-t)\left[\sum_{\al=1}^2\normb{(\mathcal L S^{(\sB)}\mathcal L^{(\al)}_{:k_*+1})_{a\bullet}}_{\ell^2}^2\right]^{1/2}\le (1+(\log W)^{-1}) \sqrt{Y_{\max}},
 \end{align*}
and the same bound holds for $[\sum_{\al=1}^2\|(\mathcal L^{(\al)}_{:k_*+1}S^{(\sB)}\mathcal L)_{\bullet a}\|_{\ell^2}^2]^{1/2}$. Combining these estimates with Cauchy--Schwarz yields
\begin{align}\label{eq:extreme-product-bound2}
&2W^6\eta_t (1-t) \sum_{\al=1}^2\sum_b \pa{\absa{ {\pb{\mathcal L^{(\al)}_{:(k_*+1)}}_{ab}} \pb{\mathcal L S^{(\sB)} \mathcal L^{(\al)}_{:(k_*+1)}}_{ab}}+\absa{ {\pb{\mathcal L^{(\al)}_{:(k_*+1)}}_{ba}} \pb{\mathcal L^{(\al)}_{:(k_*+1)} S^{(\sB)} \mathcal L}_{ba}}} \notag\\
        &\le 2\sqrt{2}(1+(\log W)^{-1})\sqrt{ \wh Y_{k_*+1,a}Y_{\max}}.
\end{align}
Combining \eqref{eq:mixed-product-bound}--\eqref{eq:extreme-product-bound2} and using Young's inequality, we get
\begin{align*}
 (1-t)\partial_t\wh Y_{k_*+1,a}\big|_{\rm quad}\le 4\wh Y_{k_*+1,a} +4Y_{\max} +C_{\gamma} \pa{1+Y_{\max}}^{1-\gamma}.
\end{align*}
Together with the damping term in \eqref{eq:Yhat-dynamics} and \eqref{eq:source-production}--\eqref{eq:source-bracket}, this gives
\begin{equation}\label{eq:SDE-ineq-Yk}
 \dd\wh Y_{k_*+1,a}
 \le\Bigl[-c_{k_*+1}\wh Y_{k_*+1,a}
          +4Y_{\max} +C_{k_*} \pa{1+Y_{\max}}^{1-\gamma}
           \Bigr]\frac{\dd t}{1-t}
          +\dd\wh Y_{k_*+1,a}^{\rm mart},
\end{equation}
where $(1-t)\partial_t\langle\wh Y_{k_*+1,a}^{\rm mart}\rangle =\oo(1)\cdot\logpara^{-1}\wh Y_{k_*+1,a}$ and $c_{k_*+1}:=k_*-3$.

To apply \Cref{lem:damped-scalar}, choose a fixed
$B_*\ge\max\{1,A_3\}$ and stop additionally when
$Y_{\max}$ reaches $B_*$, with the corresponding stopping time denoted by $\tau''$. 
The choice of $k_*$ gives
$4/c_{k_*+1}<1/8$. Since $\gamma>0$, we may take $B_*$ sufficiently
large, independently of $D$, so that
\[
 4B_*+C_{k_*}(1+B_*)^{1-\gamma}\le c_{k_*+1}B_*/8.
\]
Thus \eqref{eq:SDE-ineq-Yk} has drift at most $[-c_{k_*+1}\wh Y_{k_*+1,a}+c_{k_*+1}B_*/8]/(1-t)$ before $\tau'\wedge\tau''$. Since $\wh Y_{k_*+1,a}(0)=0$, \Cref{lem:damped-scalar} and a union bound over $a$ show that 
\begin{equation}\label{eq:bootstrap-Yk}
 \Pp\pa{\sup_{t\le\tau'\wedge\tau''}\max_a\wh Y_{k_*+1,a}(t)\ge B_*} \le N^{-D}.
\end{equation}
By continuity, $\tau''\ge\tau'=\tau$ with probability at least $1-N^{-D}$. Together with \eqref{eq:coordinate-radial-comparison}, this proves \eqref{eq:bootstrap-k*}.
\end{proof}

\begin{lemma}\label{lem:damped-scalar}
Let $N_t$ be a scalar martingale with $N_0=0$ satisfying the hypotheses of \Cref{lem:matrix-maximal}. Assume also that
$(1-t)\partial_t\langle N\rangle_t\le A^2$ for a constant $A>0$. Let $Y_t$ be a nonnegative
process satisfying
\[
 \dd Y_t\le(-\e Y_t+a)\p{1-t}^{-1} {\dd t}+\dd N_t,  \qquad Y_0\le B/4,
\]
where $\varepsilon,a>0$ are deterministic constants.
Then there exist constants $C_\varepsilon, c_\varepsilon>0$, depending only on $\varepsilon$, such that, for every $B\ge\max\{A,4a/\varepsilon\}$ and $0\le T<1$,
\begin{equation}      
 \mathbb P\pB{\sup_{0\le t\le T}Y_t\ge B}
 \le C_\varepsilon\bigl(1+\log[(1-T)^{-1}]\bigr)e^{-c_\varepsilon B/A}.
\label{eq:maximal-damped}
\end{equation}
The same conclusion holds if we replace the bracket bound
by $(1-t)\partial_t\langle N\rangle_t\le Y_t A$.
\end{lemma}
\begin{proof}
Applying Grönwall's inequality with the integrating factor $(1-t)^{-\varepsilon}$ gives
\[
 Y_t\le(1-t)^{\varepsilon}Y_0
 +\frac a{\varepsilon}\qa{1-(1-t)^{\varepsilon}}
 +(1-t)^{\varepsilon}\int_0^t(1-u)^{-\varepsilon}\,\dd N_u.
\]
Since $Y_0\le B/4$ and $a/\varepsilon\le B/4$, the sum of the first two terms is at most $B/4$. Hence, on $\{Y_t\ge B\}$, the last stochastic term must be at least $3B/4$.
Applying the scalar version of \eqref{eq:stochastic-convolution-weighted-maximal} with weight $w_t=(1-t)^{\varepsilon}$ and its bracket bounded by 
\[
 (1-t)^{2\varepsilon}\int_0^t(1-u)^{-2\varepsilon}\,\dd\langle N\rangle_u
 \le A^2(1-t)^{2\varepsilon}\int_0^t(1-u)^{-2\varepsilon-1}\,\dd u
 \le\frac{A^2}{2\varepsilon},
\]
we obtain the maximal inequality
\[
 \mathbb P\pB{\sup_{0\le t\le T}Y_t\ge B}
 \le C_\varepsilon\bigl(1+\log[(1-T)^{-1}]\bigr)
       \exp\pa{-c_\varepsilon B^2/A^2}.
\]
Since $B\ge A$, this implies \eqref{eq:maximal-damped}.
For the alternative bracket assumption, we stop $N$ at $\tau_B:=\inf\{t\ge0:Y_t\ge B\}$, with $\inf\varnothing=\infty$. Before $\tau_B$, the bracket density is bounded by $AB/(1-t)$. Applying the preceding argument up to $\tau_B$, with $A^2$ replaced by $AB$, completes the proof.
\end{proof}

\subsection{Bootstrap estimate for $J_2$}\label{sec:pointprofile}
Finally, we are ready to establish the bootstrap estimate for the pointwise profile $J_2$ in \eqref{eq:J2} with the evolution equation \eqref{eq:mainStoflow} and the estimates proved in \Cref{lem:spatialclosure,direct:continuation}.

\begin{proposition}[Bootstrap estimate for $J_2$]\label{prop:pointcontinue}
There exists a structural constant $C_2>0$ such that, for every fixed $D>0$, the following holds with probability at least $1-N^{-D}$:
\begin{align}
\sup_{0\le t\le \tau_*}J_2(t) \le C_2(\log W)^{-\fc} .\label{eq:bootstrap-J2} 
 \end{align}
\end{proposition}
\begin{proof}
Let $\tau_{k_*}$ be the first time any of the bounds \eqref{eq:bootstrap-L123}--\eqref{eq:bootstrap-k*} fails, and define $\wt \tau:=\tau\wedge\tau_{k_*}$. Fix $t\le \wt\tau$ throughout the proof. 
We first bound the quadratic drift in \eqref{eq:mainStoflow}. Let \smash{$\mathcal L^{[k_*+1]}_{ab}:=\pa{1+\rho_t(a,b)}^{k_*+1}\mathcal L_{ab}$}. Similar to \eqref{eq:C-distribute},  
\begin{equation*}
 \pa{1+\rho_t(a,b)}^{k_*+1}(\mathcal L S^{(\sB)}\mathcal L)_{ab}
 \le C_r(\mathcal L^{[k_*+1]}S^{(\sB)}\mathcal L+\mathcal L S^{(\sB)}\mathcal L^{[k_*+1]})_{ab}.
\end{equation*}
Combining \eqref{eq:bootstrap-k*} with Cauchy--Schwarz, we obtain for a constant $C_{k_*}>0$,
 \begin{align}\label{eq:quad-drift}
 W^4\eta_t(\mathcal L S^{(\sB)}\mathcal L)_{ab} \le C_{k_*}  \pa{1+\rho_t(a,b)}^{-(k_*+1)} .
 \end{align}

We then bound the drift term $\mathcal W$ and the quadratic variation of the martingale term in \eqref{eq:mainStoflow} using \Cref{lem:spatialclosure}.
For $|a-b|\ge\ell_t$, we apply \eqref{eq:minsource} with the summation over $a'$ decomposed according to whether $|a-a'|\ge |a'-b|$ or $|a-a'|<|a'-b|$ as in the proof of \eqref{eq:minsource2}. We bound the longer $\cL$-factor by \smash{$C(\log W)^{\fc}W^{-2}(1+\iota_t(a,b))\pa{1+\rho_t(a,b)}^{-k_*}$}.
Then, using \eqref{eq:minsource}, \eqref{eq:minsource2}, \eqref{eq:Lambda12}, \eqref{eq:parameters}, and the fact that $\iota_t(a,b)=0$ for $|a-b|\ge \ell_t$, we obtain
\begin{equation}\label{simp:far-LW}
 \begin{aligned}
 W^2\eta_t|\mathcal W_{ab}|
 &\le \frac{C (\log W)^{\fc}W^2\eta_t\Lambda_1\|\difloop_1\|_\infty}{\p{1+\rho_t(a,b)}^{k_*}} \le \frac{C(\log W)^{11+\fc}\logpara^2}{W^{2-4/k_*}}\pa{1+\rho_t(a,b)}^{-k_*}.
 \end{aligned}
\end{equation}
Similarly, using $\|X_{ab}\|_{\op}\le \|X_{ab}\|_{S_2}$, we bound $Q_{ab}$ in \eqref{eq:rawpoint} as 
\begin{equation}\label{simp:farpoint}
 Q_{ab}\le C (\log W)^{\fc}\Lambda_2W^{-2}\pa{1+\rho_t(a,b)}^{-k_*}.
\end{equation}
Combining this with \eqref{eq:rawpoint} and \eqref{eq:parameters}, we obtain
\begin{equation}\label{simp:far-MG}
 W^4\eta_t\sum_{\phy}|h_{t,\phy}(a,b)|^2\le \frac{C (\log W)^{2\fc}W^2\eta_t\Lambda_2}{\p{1+\rho_t(a,b)}^{2k_*}}\le \frac{C(\log W)^{10+2\fc}\logpara}{W^{2-4/k_*}}\pa{1+\rho_t(a,b)}^{-2k_*}.
\end{equation}
For $|a-b|\le\ell_t$, $Q_{ab}\le O_{\op}(t)^2\Lambda_2$. Substituting \eqref{eq:entropy-operator-output-sm}, \eqref{eq:parameters}, and \eqref{eq:J2} into \eqref{eq:rawpoint} gives 
\begin{equation}\label{simp:near-MG}
 W^4\eta_t\sum_{\phy}|h_{t,\phy}(a,b)|^2
 \le C (\log W)^{20+\fc}\logpara^3W^{-4+4/k_*}\pa{1+\iota_t(a,b)}.
\end{equation}
For the drift term $\cW$, applying \eqref{eq:J2} and \eqref{eq:omega-weights}, we get
\[
 \sum_x\mathcal L_{ax}|(S^{(\sB)}\difloop_1)_x|^2
 \le\frac{C(\log W)^{\fc}}{W^2\eta_t}
       \sum_x\omega_{a,t}(x)|(S^{(\sB)}\difloop_1)_x|^2
 \le\frac{C(\log W)^{\fc}}{W^2\eta_t}\max_x\mathfrak D_x(t)^2.
\]
Applying Cauchy--Schwarz to \eqref{eq:minsource} therefore gives
\begin{equation}\label{simp:near-LW}
 \begin{aligned}
 W^2\eta_t|\mathcal W_{ab}(t)|
 &\le C (\log W)^{\fc/2}W^3\sqrt{\eta_t}\mathfrak W_{\rm loc}
          \sqrt{\mathcal L_{ab}\Lambda_2} \le \frac{C(\log W)^{11+\fc}\logpara^{3/2}}{W^{2-2/k_*}} \sqrt{1+\iota_t(a,b)}.
 \end{aligned}
\end{equation}
where we plugged in \eqref{eq:parameters} and $\mathfrak W_{\rm loc}$ in \eqref{eq:Jloc}.   

Combining \eqref{eq:quad-drift}, \eqref{simp:far-LW}, and \eqref{simp:near-LW} gives the following total drift bound under \eqref{eq:window} and \eqref{eq:spatialorders}:
\begin{equation}\label{eq:total-drift}
 \eta_t\left[W^4(\mathcal L S^{(\sB)}\mathcal L)_{ab}
                 +W^2|\mathcal W_{ab}(t)|\right]
 \le C \pa{1+\rho_t(a,b)}^{-k_*}.
\end{equation}
For simplicity, denote $\Phi_t(a,b):=\pa{1+\iota_t(a,b)}\pa{1+\rho_t(a,b)}^{-k_*}$. We split the time integral according to whether $|a-b|\ge\ell_u$ or $|a-b|<\ell_u$.
On the far interval, the substitution $x=|a-b|/\ell_u$ gives
\begin{equation}      
 \int_0^t\mathbf1_{\{|a-b|\ge\ell_u\}}
       \pa{1+\rho_u(a,b)}^{-k_*}\frac{\dd u}{1-u}
 \le2\int_{\max(1,|a-b|/\ell_t)}^\infty
       (1+x)^{-k_*}\frac{\dd x}{x}
 \le C\Phi_t(a,b).
\label{eq:integrated-drift-far}
\end{equation}
On the near interval, we use $\pa{1+\rho_u(a,b)}^{-k_*}\le1$. If this interval is nonempty, then $1+\rho_t(a,b)<2$, and
\begin{equation}
 \int_0^t\mathbf1_{\{|a-b|<\ell_u\}}\frac{\dd u}{1-u}
 \le C \pa{1+\iota_t(a,b)}
 \le C \Phi_t(a,b).
\label{eq:integrated-drift-near}
\end{equation}
Thus, \eqref{eq:total-drift} gives
\begin{equation}\label{eq:gain-clock}
\int_0^{t\wedge\wt\tau}
 \left[W^2(\mathcal L S^{(\sB)}\mathcal L)_{ab}(u)
             + |\mathcal W_{ab}(u)|\right]\dd u
 \le C  W^{-2}\Phi_t(a,b).
\end{equation}

For the martingale stopped at $\wt\tau$, \eqref{simp:far-MG} and \eqref{simp:near-MG} give the covariance density bound
\begin{equation}\label{eq:variance-MG-t}
 \mathbf1_{\{u\le\wt\tau\}}\sum_{\phy}|h_{u,\phy}(a,b)|^2
 \le v(u):=\frac{C}{W^4\eta_u}
 \frac{(\log W)^{12}\logpara}{W^{2-4/k_*}}
 \pa{1+\iota_u(a,b)}\pa{1+\rho_u(a,b)}^{-2k_*}.
\end{equation}
The same far- and near-interval calculations as in \eqref{eq:integrated-drift-far} and \eqref{eq:integrated-drift-near} with exponent $2k_*$, together with monotonicity of $\iota_u(a,b)$, give
\[
 \int_0^t\pa{1+\iota_u(a,b)}\pa{1+\rho_u(a,b)}^{-2k_*}
       \frac{\dd u}{\eta_u}
 \le C\Phi_t(a,b)^2.
\]
Consequently, we have 
\[
 W^4\int_0^{t\wedge\wt\tau}\sum_{\phy}|h_{u,\phy}(a,b)|^2\dd u
 \le W^4\int_0^t v(u)\dd u
 \le C\frac{(\log W)^{12}\logpara}{W^{2-4/k_*}}\Phi_t(a,b)^2.
\]
Applying the scalar maximal inequality in \eqref{eq:matrix-tail} to the stopped martingale $W^2\mathcal M_{ab}(\cdot\wedge\wt\tau)$ with $R^2=C_D(\log W)^{12}\logpara^2 W^{-2+4/k_*}\Phi_t(a,b)^2$ for a large enough constant $C_D>0$, using a doubling argument analogous to the one leading to \eqref{eq:running-terminal-maximal}, and taking a union bound over $a,b\in\Zn$, we obtain, with probability at least $1-N^{-D}$,
\begin{equation}\label{eq:point_errors}
\sup_{0\le t\le \wt\tau}\max_{a,b}\frac{|\mathcal M_{ab}(t\wedge\wt\tau)|}{W^{-2}\Phi_t(a,b)}  \le C_{D} \frac{(\log W)^{6}\logpara}{W^{1-2/k_*}}=\oo(1).
\end{equation}
Integrating equation \eqref{eq:mainStoflow}, and using $\mathcal L_{ab}(0)=W^{-2}\delta_{a,b}$, \eqref{eq:gain-clock} and \eqref{eq:point_errors}, we conclude that with probability at least $1-N^{-D}$,
\begin{equation}\label{eq:Pstar}
 \sup_{0\le t\le \wt\tau}\max_{a,b}\mathcal L_{ab}(t)/\qa{W^{-2}{\Phi_t(a,b)}} \le C \implies \sup_{0\le t\le \wt\tau}J_2(t)\le C(\log W)^{-\fc}=\oo(1)
\end{equation}
for a constant $C>0$ independent of $D$. Together with \Cref{direct:continuation,lem:spatialclosure}, this concludes \eqref{eq:bootstrap-J2}.
\end{proof}

Combining \Cref{thm:trace,prop:auxiliary-time-release-2,lem:spatialclosure,prop:pointcontinue}, we readily obtain the following proposition, which shows that, with probability at least $1-N^{-D}$, the strict inequality $\mathscr J(t)<1$ holds up to $\tau_*$, and hence $\tau_*=t_*$.

\begin{proposition}[Release of the principal stopping time]
      \label{prop:small-flow-release}
Under \eqref{eq:window} and \eqref{eq:spatialorders}, for every fixed $D>0$, $\tau_*=t_*$ with probability at least $1-N^{-D}$. On the same event, the estimates of \Cref{lem:entropy-operator-output-sm,prop:auxiliary-time-release-2,direct:continuation,lem:spatialclosure,prop:pointcontinue} hold on $[0,t_*]$, and
\begin{equation}\label{eq:small-flow-inherited}
 \sup_{t\le t_*}J_1(t)\le C_D/\log W.
\end{equation}
\end{proposition}
\begin{proof}
Apply \Cref{thm:trace,prop:auxiliary-time-release-2,lem:spatialclosure,prop:pointcontinue} at an increased probability depth $D'>D$. These estimates improve every barrier in \eqref{eq:stop-small}, so continuity gives $\tau_*=t_*$. 
For the bound \eqref{eq:small-flow-inherited}, use \eqref{eq:point-tracebound} with $p=C_D\logpara$, $\vartheta_t\le C\logpara$, and $\chi_0=1$ on the good event. Dividing by $\mathfrak W$ yields
\[
 J_1(t)\le C_D\left[(\log W)^{-5/2}
             +\frac{(\log W)^{6}\logpara^{3/2}}{W^2}\right]
 \le C_D/\log W.
\]
The same time-net argument as in the proof of \Cref{thm:trace}
makes this bound uniform in $t$, after increasing the probability
depth. This proves the proposition.
\end{proof} 

\section{Proofs of the main results}
\label{sec:small-closure}

In this section, we prove \Cref{thm:local-law-smallW} using \Cref{prop:small-flow-release}. We also deduce
\Cref{thm:delocalization} from the block operator estimate \eqref{eq:norm-law} and the Haar concentration bound below.

\begin{lemma}[Haar concentration]\label{lem:Haar-entries}
Let \smash{$U=\bigoplus_{a\in\Zn}U_a$}, where the $U_a\in U(W^2)$ are independent Haar unitary matrices. For a deterministic matrix $M$, let
\[
 F_2(M):=\max_{a,b}\|P_aMP_b\|_{S_2},\quad
 F_{\op}(M):=\max_{a,b}\|P_aMP_b\|.
\]
There exist absolute constants $c,C>0$ such that, for $1\le p\le cW^2$ and $x\in[a]$, $y\in[b]$,
\begin{equation}\label{eq:Haar-entry-tail}
 \mathbb P_U\left(
 |(UMU^*)_{xy}-\delta_{xy}W^{-2}\Tr(P_aMP_a)|
 >\frac{C}{W^2}\qa{\sqrt p F_2(M)+p F_{\op}(M)}\right)
 \le Ce^{-cp}.
\end{equation}
Moreover, for any deterministic unit vector $v_a$ supported on $[a]$, any $x\in[a]$, and $0\le u<1$, we have
\begin{equation}\label{eq:sublog-Haar-lower}
 \Pp_U(|\langle \mathbf e_{x},U_a v_a\rangle|^2\ge u)
        =(1-u)^{W^2-1}.
\end{equation}
\end{lemma}
\begin{proof}
The estimate \eqref{eq:Haar-entry-tail} follows from Gaussian concentration. Let \(g\in\mathbb C^{W^2}\) be a standard complex Gaussian vector and \(T\) a deterministic \(W^2\times W^2\) Hermitian matrix. Diagonalizing \(T\) and using the exponential moments of \(|g_j|^2\), we obtain, for \(p\ge1\),
\begin{equation}
 \mathbb P\left(|g^*Tg-\Tr T|>
  C\pa{\sqrt p\,\|T\|_{S_2}+p\|T\|}\right)\le2e^{-cp}.
\label{eq:Gauss-concentration}
\end{equation}
This bound also follows directly from the more general Hanson--Wright inequality (see, e.g., \cite[Theorem~1.1]{HansonW}). Taking $T=\Id$ and $p=\e W^2$, with \(\varepsilon>0\) sufficiently small, gives
\begin{equation}
 \mathbb P\left(\|g\|^2\le W^2/2\right)\le 2e^{-c\e W^2}.
\label{eq:Gauss-concentration2}
\end{equation}
Applying \eqref{eq:Gauss-concentration} to $T-W^{-2}\Tr(T)\Id$ and dividing by $\|g\|^2$ yields that for some absolute constants $c,C>0$,
\begin{equation}
 \mathbb P\left(\absa{\pa{\frac{g}{\|g\|}}^*T\pa{\frac{g}{\|g\|}}-W^{-2}\Tr T}>
  \frac{C}{W^2}\pa{\sqrt p\,\|T\|_{S_2}+p\|T\|}\right)\le C e^{-c(p\wedge W^2)} .
\label{eq:Gauss-concentration3}
\end{equation}
Here \(g/\|g\|\) is uniformly distributed on the unit sphere of \(\mathbb C^{W^2}\). Splitting a general matrix into its Hermitian real and imaginary parts extends the above bounds to arbitrary deterministic complex matrices.

The same estimates hold in dimension \(2W^2\), with adjusted absolute constants. For \(a\ne b\), the relevant rows of \(U_a\) and \(U_b\) are independent uniform unit vectors. We apply the Gaussian quadratic-form bound \eqref{eq:Gauss-concentration} in dimension \(2W^2\) to $T=\left(\begin{smallmatrix}0&P_aMP_b\\0&0\end{smallmatrix}\right)$. Dividing by the product of the two Gaussian lengths and using the lower-tail estimate \eqref{eq:Gauss-concentration2} proves \eqref{eq:Haar-entry-tail}.
For \(a=b\) and \(x=y\), we apply the bound \eqref{eq:Gauss-concentration3} directly. For \(a=b\) and \(x\ne y\), let \(u,v\) be the corresponding orthonormal Haar vectors. Polarization expresses \(u^*(P_aMP_a)v\) through quadratic forms at $(u\pm v)/\sqrt2$ and $(u\pm\ii v)/\sqrt2$. 
Each of these four vectors is uniform on the unit sphere, and the scalar trace terms cancel. Applying \eqref{eq:Gauss-concentration3} to each quadratic form and taking a union bound proves the desired \eqref{eq:Haar-entry-tail}.

Finally, by Haar invariance, we can write \smash{$|\langle \mathbf e_{x},U_av_a\rangle|^2\stackrel{\rm d}{=}X_1/\sum_{j=1}^{W^2}X_j$}, where $X_j=|g_j|^2$ are independent mean-one exponential variables. Conditioning on $X_2,\ldots,X_{W^2}$, we obtain 
\[
 \Pp\pbb{X_1\ge\frac{u}{1-u}\sum_{j=2}^{W^2}X_j} =\E\exp\pbb{-\frac{u}{1-u}\sum_{j=2}^{W^2}X_j} =(1-u)^{W^2-1},
\]
which proves \eqref{eq:sublog-Haar-lower}.
\end{proof}

\subsection{Proof of \Cref{thm:delocalization}}\label{sec:pf-delocal}

Let $\eta_0=(\log W)^4/L^2$ and $I=[-2+\kappa,2-\kappa]$.
Applying \eqref{eq:norm-law} gives $\Pp(\mathcal E) \ge 1-N^{-D-2}$, where the event $\mathcal E$ is defined by
\begin{equation}\label{eq:uniform-diagonal-block}
\mathcal E:=\hB{ \max_a\sup_{E\in I}\|P_aG(E+\ii\eta_0)P_a\| \le (\log W)^5\logpara}.
\end{equation}
For an eigenvector $\boldsymbol\psi_j$, denote $v_{aj}=P_a\boldsymbol\psi_j$. At the random energy $E=\lambda_j\in I$, spectral decomposition gives the positive-semidefinite inequality
\(
\eta_0^{-1}v_{aj}v_{aj}^* \preceq P_a\im G(\lambda_j+\ii\eta_0)P_a  .
\)
Consequently, on $\mathcal E$, we have
\begin{equation}\label{eq:block-spectral-mass}
 \|v_{aj}\|^2
 \le\eta_0\|P_a\im G(\lambda_j+\ii\eta_0)P_a\|
 \le \eta_0(\log W)^5\logpara.
\end{equation}
Let $U=\bigoplus_aU_a$ be independent of $H$, with the distribution in \Cref{lem:Haar-entries}. Then we have 
\begin{equation}     
 UHU^*\stackrel{\rm d}=H,\qquad \pa{UHU^*-z}^{-1}=UG(z)U^*.
\label{eq:invariant-H-distr} 
\end{equation}
Conditional on $H$, we apply \eqref{eq:sublog-Haar-lower} to $v_a=v_{aj}/\|v_{aj}\|$ whenever $v_{aj}\ne0$, with $u=C_D\logpara/W^2<1$. Since \smash{$(1-u)^{W^2-1}\le e^{-cC_D\logpara}$}, a union bound over all eigenvectors and coordinates gives, for sufficiently large $C_D$,
\[
 |(U\boldsymbol\psi_j)_x|^2
 \le C_D\frac{\logpara}{W^2}\|v_{aj}\|^2,
 \qquad \forall x\in[a],\quad 1\le j\le N,
\]
with conditional probability at least $1-N^{-D-2}$. Since $H$ has simple spectrum almost surely, the normalized eigenvector of $UHU^*$ at $\lambda_j$ agrees with $U\boldsymbol\psi_j$ up to phase.
%{\cop The discriminant of the characteristic polynomial is a polynomial in the independent real Gaussian coordinates of $H$ and is nonzero at a diagonal matrix with distinct entries in the support of $H$. Thus $H$ has simple spectrum almost surely, and the normalized eigenvector of $UHU^*$ at $\lambda_j$ agrees with $U\boldsymbol\psi_j$ up to phase.}
By \eqref{eq:invariant-H-distr} and the invariance of $\mathcal E$ under block unitary conjugation, \eqref{eq:block-spectral-mass} yields
\begin{equation}\label{eq:Haar-delocalization}
 \max_{\lambda_j\in I}\|\boldsymbol\psi_j\|_\infty^2
 \le C_{\kappa,D}\frac{\eta_0(\log W)^5\logpara^2}{W^2}
 = C_{\kappa,D}\frac{(\log W)^9\logpara^2}{N}
\end{equation}
with probability at least $1-N^{-D}$. This concludes \eqref{eq:delocalization}.

\subsection{Proof of \Cref{thm:local-law-smallW}}\label{sec:pf-local-law-smallW}

For a fixed $z=E+\ii\eta\in\bS(\kappa)$, choose $t_0=t_0(z)$ and $\sE=\sE(z)$ as in \eqref{eq:t0E0}. Then $t_0=|m(z)|^2\asymp_\kappa1$ and
\begin{equation}\label{eq:scaling}
 G(z)\stackrel{\rm d}=\sqrt{t_0}\,G_{t_0;\sE},\qquad
 m(z)=\sqrt{t_0}\,m(\sE),\qquad
 \eta_{t_0}=\sqrt{t_0}\eta\asymp_\kappa\eta.
\end{equation}
The choice of $c_\kappa$ in \eqref{eq:t*} ensures $t_0\le t_*$. Choose an integer $k_*\ge k$ satisfying \eqref{eq:spatialorders}. 

We apply \Cref{prop:small-flow-release} at an increased probability depth $D'>D$. Thus $\tau_*=t_*$ outside an event of probability at most $N^{-D'}$. Since $X_{ab}=W^{-2}P_aG_tP_b$ and $\mathcal L_{ab}=W^{-4}\tr|P_aG_tP_b|^2$, \eqref{eq:bootstrap-J2} and \eqref{eq:spatial-operator-output-sm} give
\begin{equation}\label{eq:static-small-blocks}
 W^{-2}\tr|P_aG(z)P_b|^2 \le C_{\kappa,k}\Pi_\eta^{(k)}(a,b),\qquad
 \max_{a,b}\|P_aG(z)P_b\|
 \le C_\kappa(\log W)^{9/2}\logpara.
\end{equation}
Here we used $k_*\ge k$, $\eta_{t_0}\asymp_\kappa\eta$, and $\vartheta_{t_0}+\iota_{t_0}(a,b)+\log W\le C \logpara$. This concludes \eqref{eq:norm-law} and \eqref{eq:HS-norm-law} at the given $z$.
Moreover, \eqref{eq:small-flow-inherited} gives that outside an event of probability at most $N^{-D'}$,
\begin{equation}\label{eq:static-small-trace}
 \begin{aligned}
 \max_a|\Tr(G(z)E_a)-m(z)|
 \le C_{D'}{(\log W)^{5}\logpara^{3/2}}/{W^2}.
 \end{aligned}
\end{equation}
This gives \eqref{eq:trace-law} at the given $z$. 
To establish the entrywise local law \eqref{eq:entry-law}, we again use \Cref{lem:Haar-entries}. If $U$ is independent of $H$ with the distribution in \Cref{lem:Haar-entries}, then \eqref{eq:invariant-H-distr} holds. The block estimates \eqref{eq:static-small-blocks} and \eqref{eq:static-small-trace} are invariant under this conjugation. They give
\[
 F_2(G(z))\le C_{\kappa,k} W\sqrt{\logpara},\qquad
F_{\op}(G(z))\le C_\kappa(\log W)^{9/2}\logpara.
\]
Conditioning on $H$, we apply \Cref{lem:Haar-entries} with $p=C_{D'}\logpara $ and take a union bound over the entries of $UG(z)U^*$. Averaging over $H$ and using \eqref{eq:invariant-H-distr} and the invariance of $\tr(GE_a)$ under block unitary conjugation, we obtain
\[  \max_{a,b}\max_{x\in [a],y\in[b]}|G_{xy}(z)-\delta_{xy}\tr\p{G(z)E_a}|
 \le C_{D'}\left[
 \frac{\logpara}{W}
 +\frac{(\log W)^{9/2}\logpara^2}{W^2}\right]
 \le C_{D'}\frac{\logpara}{W} \]
 with probability at least $1- N^{-D' + 1}$. Together with \eqref{eq:static-small-trace}, this gives \eqref{eq:entry-law} at the given $z$.

We next obtain the uniform estimates in \eqref{eq:entry-law}--\eqref{eq:HS-norm-law} by taking an $N^{-C}$-net in $z$, with $C>0$ sufficiently large depending on $k$, and applying the preceding bounds at probability depth $D'>D+2C+10$. A union bound and the resolvent perturbation estimate $\|G(z)-G(z')\|\le |z-z'| \p{\im z \cdot \im z'}^{-1}$
then extend the estimates to the entire spectral domain. We omit the routine details.

Finally, we prove \eqref{eq:annealed-two-loop-law} by combining the flow equation \eqref{eq_L-Keee} for $\star=\emptyset$ with the bootstrap bounds from \Cref{prop:small-flow-release}. We increase $k_*$, if necessary, so that $4/k_*<1-1/A_{\ref{thm:delocalization}}$. Let $\Omega$ be the event on which the bounds of \Cref{prop:small-flow-release} hold throughout $[0,t_*]$. For any sufficiently large fixed $D'$, these bounds may be taken at probability depth $D'$, so that
\begin{equation}\label{eq:annealed-good-prob}
 \Pp(\Omega^c)\le N^{-D'}.
\end{equation}
The following estimates hold on $\Omega$, uniformly for $0\le t\le t_*$. By \eqref{direct:unbalanced-conclusion},
\begin{align}
W^2(1-t)\|\mathcal D(t)\|
 \le C_{D'}(\log W)^6 {\logpara}/{W}.
 \label{eq:annealed-operator-input}\end{align}
 Since \eqref{eq:bootstrap-T} gives $A_0=\OO(1)$, taking $r=0$ in \eqref{eq:sameorder}--\eqref{eq:parameters} and averaging over the row index yields
\begin{align}
 \frac{W^4}{L^2}\|\mathcal W(t)\|_{\HS}^2
 &\le \frac{C(\log W)^{22}\logpara^4}
 {W^{4-8/k_*}(1-t)^3}, \qquad 
 \frac{W^4}{L^2}\sum_{\phy}\|h_{t,\phy}\|_{\HS}^2
 \le \frac{C(\log W)^{10}\logpara}
 {W^{2-4/k_*}(1-t)^2}.
 \label{eq:annealed-noise-HS}
\end{align}
Finally, setting $b=a$ in \eqref{simp:near-LW} and averaging over $a$ gives
\begin{equation}\label{eq:annealed-drift-trace}
 \frac{W^2}{L^2}|\Tr\mathcal W(t)|
 \le\frac{C(\log W)^{11+\fc}\logpara^2}
 {W^{2-4/k_*}(1-t)}.
\end{equation}

To estimate the expectation of the quadratic term $ \mathcal E$ in \eqref{eq_L-Keee}, we first bound the second moment of $\mathcal D(t)$.
Let
\(B(t):=\mathscr{A}_{t} \circ \cD(t)   + \mathcal{E}(t)+\mathcal{W}(t)\)
be the full drift.
On $\Omega$, summing \eqref{eq:bootstrap-T} over $a$ and using $\mathcal D(u)=\mathcal L(u)-W^{-2}\Theta_u$, we obtain
\begin{align}
 \frac{W^2}{L}\|\mathcal D(u)\|_{\HS}\le C(1-u)^{-1/2}.
\label{eq:renom-D-HS}
\end{align}
Here we used the symmetry and translation invariance of $\Theta_u$, together with \eqref{eq:tools-theta-square}, to obtain
\begin{equation}\label{eq:annealed-Theta-HS}
 L^{-2}\|\Theta_u\|_{\HS}^2=(\Theta_u^2)_{aa}
 \le C\eta_u^{-2}\ell_u^{-2}\le C(1-u)^{-1}.
\end{equation}
Since $\|S^{(\sB)}\|\le1$, these bounds, together with \eqref{eq:annealed-operator-input} and the first bound in \eqref{eq:annealed-noise-HS}, give on $\Omega$
\[
 \frac{W^2}{L}\|B(u)\|_{\HS}
 \le \frac{W^2}{L}\|\mathcal D(u)\|
       \bigl(2\|\Theta_u\|_{\HS}+W^2\|\mathcal D(u)\|_{\HS}\bigr)
       +\frac{W^2}{L}\|\mathcal W(u)\|_{\HS}
= \frac{\oo(1)}{(1-u)^{3/2}}.
\]
Combining this estimate with the second bound in \eqref{eq:annealed-noise-HS} and taking $D'$ sufficiently large in \eqref{eq:annealed-good-prob}, we obtain the following bounds uniformly for $0\le u\le t_*$:\footnote{The deterministic bound $\|G_u\|\le\eta_u^{-1}\le N$ gives \smash{$\|B(u)\|_{\HS}^2+\sum_{\phy}\|h_{u,\phy}\|_{\HS}^2\le N^C$} for some absolute constant $C>0$. Since $\Pp(\Omega^c)\le N^{-D'}$, the contribution of $\Omega^c$ to both expectations is negligible for sufficiently large $D'$.}
\[
 \frac{W^2}{L}\bigl(\E\|B(u)\|_{\HS}^2\bigr)^{1/2}
 = \frac{\oo(1)}{(1-u)^{3/2}},\qquad
 \frac{W^4}{L^2}\sum_{\phy}\E\|h_{u,\phy}\|_{\HS}^2
 = \frac{\oo(1)}{(1-u)^2}.
\]
Since $\mathcal D(0)=0$, Minkowski's inequality for the drift integral and the It\^o isometry for the martingale give
\[
 \frac{W^2}{L}\bigl(\E\|\mathcal D(t)\|_{\HS}^2\bigr)^{1/2}
= \oo\left(\int_0^t\frac{\dd u}{(1-u)^{3/2}}
       +\left(\int_0^t\frac{\dd u}{(1-u)^2}\right)^{1/2}\right)
=\frac{\oo(1)}{\sqrt{1-t}}.
\]
Thus, uniformly for $0\le t\le t_*$,
\begin{equation}\label{eq:annealed-averaged-error}
 \frac{W^4(1-t)}{L^2}\E\|\mathcal D(t)\|_{\HS}^2=\oo(1).
\end{equation}

We now show that, uniformly for $0\le t\le t_*$ and $a,b\in\Zn$,
\begin{equation}\label{eq:annealed-flow-law}
 W^2\absa{\E\mathcal D_{ab}(t)} =\oo\bigl(1+\log(1-t)^{-1}\bigr).
\end{equation}
This implies \eqref{eq:annealed-two-loop-law} by \eqref{eq:scaling}, which gives
\[t_0=|m(z)|^2,\qquad 
 W^{-2}\tr|P_aG(z)P_b|^2
 \stackrel{\rm d}=W^2t_0\mathcal L_{ab}(t_0),\qquad
 \Theta_{t_0}=(1-|m(z)|^2S^{(\sB)})^{-1}.
\]

To prove \eqref{eq:annealed-flow-law}, we integrate \eqref{eq_L-Keee} and take expectations. Since $\mathcal D(0)=0$ and the martingale has mean zero, translation invariance gives
\begin{equation}
 W^2 \E\mathcal D_{ab}(t) 
=\frac{W^2}{L^2}\int_0^t\sum_{x\in\Zn} \E B_{a+x,b+x}(u) \dd u .
\label{eq:ED-flow}
\end{equation}
For any matrices $\cal A$ and $\cal B$ indexed by $\Zn$, we have the Cauchy--Schwarz inequality \(
 \sum_x\left|(\cal A\cal B)_{a+x,b+x}\right| \le \|\cal A\|_{\HS}\|\cal B\|_{\HS}\). Applying this inequality to the two products in $\mathscr A_u\circ\cD(u)$ and to $\cE(u)=W^2\cD(u)S^{(\sB)}\cD(u)$, and using $\|S^{(\sB)}\|\le1$, we obtain from \eqref{eq:annealed-Theta-HS} and \eqref{eq:annealed-averaged-error} that
\begin{align*}
      \frac{W^2}{L^2}\sum_x \absa{\E\pa{\mathscr{A}_{u} \circ \cD(u)}_{a+x,b+x}}&\le \frac{2\|\Theta_u\|_{\HS}}{L}\,
 \frac{W^2}{L}\bigl(\E\|\cD(u)\|_{\HS}^2\bigr)^{1/2}
 = \frac{\oo(1)}{1-u},\\
 \frac{W^2}{L^2}\sum_x \absa{\E\cE_{a+x,b+x}(u)} &  \le \frac{W^4}{L^2}\E\|\cD(u)\|_{\HS}^2= \frac{\oo(1)}{1-u}.
\end{align*}
For the light-weight term, \eqref{simp:far-LW} and \eqref{simp:near-LW} give that on $\Omega$,
\[
 W^2(1-u)\max_{a,b}|\mathcal W_{ab}(u)|
 \le \frac{C(\log W)^{11+\fc}\logpara^2}{W^{2-4/k_*}}=\oo(1)\implies \frac{W^2}{L^2}\sum_x \absa{\E\cW_{a+x,b+x}(u)}= \frac{\oo(1)}{1-u},
\]
where we have again used that the contribution of \(\Omega^c\) to the expectation is negligible for sufficiently large \(D'\). Substituting these bounds into \eqref{eq:ED-flow} yields, uniformly for $a,b\in\Zn$ and $0\le t\le t_*$,
\[
 W^2\bigl|\E\mathcal D_{ab}(t)\bigr|
= \oo(1)\int_0^t\frac{\dd u}{1-u}
 =\oo(1)\log\frac{1}{1-t},
\]
which proves \eqref{eq:annealed-flow-law}.

\section{Failure of the entrywise local law below logarithmic bandwidth}
\label{app:sublog}

The proof of \Cref{thm:sublog-divergence} relies on \Cref{thm:local-law-smallW} and the following consequence.
\begin{lemma}\label{lem:lower-2-loop}
In the setting of \Cref{thm:local-law-smallW}, there exists a constant $c_\kappa>0$ such that, for all sufficiently large $N$, uniformly over $z\in \bS(\kappa)$ and $a\in\Zn$,
\begin{equation}\label{eq:positive-block-second-moment}
 W^{-2}\E\Tr\bigl[(P_a\im G(E+\ii \eta)P_a)^2\bigr]
 \ge c_\kappa \pa{1+\log \eta^{-1}}\, .
\end{equation}
\end{lemma}

\begin{proof}
We start from the identity
\[
 2\Tr[(P_a\im G (z) P_a)^2]  = \Tr(P_aG(z)P_aG(z)^*)- \re\Tr(P_aG(z)P_aG(z)).
\]
Combining the case $a=b$ of \eqref{eq:annealed-two-loop-law} with \eqref{eq:annealed-Theta-diagonal}, we obtain 
\begin{equation}\label{eq:sublog-physical-RA}
 W^{-2}\E\Tr(P_aG(z)P_aG(z)^*) =|m(z)|^2\bigl(1-|m(z)|^2S^{(\sB)}\bigr)^{-1}_{aa} +\oo(1+\log\eta^{-1}) \ge c_\kappa  \pa{1+\log\eta^{-1}}
\end{equation}
for some constant $c_\kappa>0$. For the second term, we claim that
\begin{equation}\label{eq:sublog-RR-small}
 W^{-2}\left|\E\Tr(P_aG(z)P_aG(z))\right|
 \le C_\kappa+\oo(1+\log\eta^{-1}).
\end{equation}
Together, these estimates prove \eqref{eq:positive-block-second-moment} for $\eta\le\delta_\kappa$, where $\delta_\kappa>0$ is a sufficiently small constant. For $\eta\ge \delta_\kappa$, the averaged local law \eqref{eq:trace-law} gives $W^{-2}\Tr(P_a\im G(z)P_a)\ge c_\kappa$ with high probability. By Cauchy--Schwarz,
\[
 W^{-2}\Tr[(P_a\im G(z)P_a)^2]
 \ge \bigl[W^{-2}\Tr(P_a\im G(z)P_a)\bigr]^2.
\]
Taking expectations yields the required lower bound, since $1+\log\eta^{-1}$ is bounded in this regime.

It remains to prove \eqref{eq:sublog-RR-small}. One approach is to analyze the $+$ version of \eqref{int_K-LcalE}, using the spatial decay of the stable kernel $\mathcal U_{u,t}^+$ provided by \eqref{eq:tools-P-point-short}. We instead give a direct argument based on the $GG$ expansion in \cite[Lemma~3.14]{yang2021delocalization}, obtained by Gaussian integration by parts. Write $m=m(z)$, \smash{$\Gc=G-m\Id$}, $T_{ab}:=W^{-2}\Tr(P_aGP_bG)$, and $\theta^+(z):=(1-m^2S^{(\sB)})^{-1}$. Then
\begin{equation}\label{eq:sublog-RR-expansion}
 \E T_{ab} =m\theta^+_{ab}\E \Tr(GE_b) + m\sum_{a',b'}\theta^+_{aa'}S^{(\sB)}_{a'b'} \ha{\E\bigl[\Tr(\Gc E_{b'})T_{a'b}\bigr] +   \E\bigl[\Tr(\Gc E_{a'})T_{b'b}\bigr]}.
\end{equation}
By Cauchy--Schwarz and \eqref{eq:HS-norm-law}, we have $\max_{a,b}|T_{ab}|\le C(1+\log\eta^{-1})$  on the good event. Moreover, \eqref{eq:tools-P-point-short} gives $\|\theta^+(z)\|_{\infty\to\infty}\le C_\kappa$. 
Combining these bounds with \eqref{eq:trace-law} in \eqref{eq:sublog-RR-expansion}, and using $\|G(z)\|\le\eta^{-1}$ on the exceptional event, we obtain
\[| \E T_{ab}| \le C+ C\frac{(\log W)^6\logpara^{3/2}}{W^2} (1+\log\eta^{-1}) \le C+\oo(1+\log\eta^{-1}).\]
This proves \eqref{eq:sublog-RR-small}.
\end{proof}

To apply \Cref{thm:local-law-smallW,lem:lower-2-loop} under the bandwidth condition \eqref{eq:polylog-W}, we use a classical renormalization argument and consider the model on smaller block tori. An \emph{auxiliary torus} of side $n$ is a periodic block lattice $(\mathbb Z/n\mathbb Z)^2$ with $W^2$ sites per block. On this torus, we define a Gaussian band matrix $h$ by the same law as $H$, with $L$ and $N=(WL)^2$ replaced by $n$ and $M=(Wn)^2$, respectively. We construct independent copies of $h$ from disjoint $n\times n$ squares of the original block lattice $\Zn$ by retaining the entries within each square and adding independent Gaussian entries across opposite edges according to the periodic variance profile.

We choose $n\asymp\exp(2W^{1-\epsilon})$, so that
$\log M\asymp W^{1-\epsilon}$ and the bandwidth condition \eqref{eq:polylog-W} holds at dimension $M$. Applying \Cref{thm:local-law-smallW,lem:lower-2-loop} on each auxiliary torus gives the lower bound \eqref{eq:sublog-positive-block} below for its central resolvent block with probability bounded away from zero. By independence, a positive fraction of the squares satisfy this lower bound with high probability. Resolvent decay then allows us to transfer the bound to the corresponding blocks of the original resolvent $G$. Finally, independent Haar rotations within these blocks, as in \Cref{lem:Haar-entries}, yield the diagonal lower bound \eqref{eq:sublog-quantitative}.

\begin{lemma}
\label{lem:sublog-positive-block}
Let $n\in 4\N$ satisfy $n\asymp \exp(2W^{1-\epsilon})$, and let $h$ be the model on the auxiliary block torus of side $n$. 
Fix $z=E+\ii\eta_N$ with \( \eta_N=\exp(-W^{1-\epsilon})\), and define the resolvent $g(z):=(h-z)^{-1}$.
There exist constants $c_1,c_2>0$ such that for all $a\in \Z_n^2$,
\begin{equation}\label{eq:sublog-positive-block}
 \Pp(\|P_a \im g (z) P_a\|\ge c_1W^{1-\epsilon})\ge c_2.
\end{equation}
\end{lemma}
\begin{proof}

We apply \Cref{thm:local-law-smallW,lem:lower-2-loop} with $N$ replaced by $M=W^2n^2$ and $L$ replaced by $n$. Choose $A_{\ref{thm:delocalization}}\in(1,(1-\epsilon)^{-1})$. Since $10+\log M\asymp W^{1-\epsilon}$ and $\eta_N=\exp(-W^{1-\epsilon})\ge(\log W)^4/n^2$, the hypotheses hold for sufficiently large $W$. By \eqref{eq:trace-law} and \eqref{eq:HS-norm-law} with $k=4$, the bounds
\begin{equation}\label{eq:sublog-positive-upper}
 \Tr(P_a\im g(z)P_a)\le CW^2,\qquad
 \Tr\bigl[(P_a\im g(z)P_a)^2\bigr]\le CW^{3-\epsilon}, 
\end{equation}
hold simultaneously with probability at least $1-M^{-D}$. Set $X:=\Tr[(P_a\im g(z)P_a)^2]$. By \eqref{eq:positive-block-second-moment}, $\E X\ge cW^{3-\epsilon}$. Combining this with \eqref{eq:sublog-positive-upper} and splitting the expectation at $(c/2)W^{3-\epsilon}$ gives
\begin{align*}
 cW^{3-\epsilon} \le\E X \le \frac c 2 W^{3-\epsilon}
 +CW^{3-\epsilon}\Pp\left(X
 \ge \frac c 2 W^{3-\epsilon} \right) +\oo(W^{3-\epsilon}) ,
\end{align*}
where the final term accounts for the exceptional event on which \eqref{eq:sublog-positive-upper} fails. Consequently,
%{\cop Denote this event by $\mathcal B$. Since $\|g(z)\|\le\eta_N^{-1}$ and $P_a$ has rank $W^2$, we have $X\le W^2\eta_N^{-2}$. Taking $D=1$ and using $n\asymp\eta_N^{-2}$, we obtain $\E[X\mathbf1_{\mathcal B}]\le W^2\eta_N^{-2}M^{-1}\le C\eta_N^2=\oo(W^{3-\epsilon})$.}  
\[
 \Pp\left(X
 \ge\frac c2W^{3-\epsilon}\right)\ge\frac{c}{4C}.
\]
Combining the inequality $X\le \|P_a\im g(z)P_a\|\Tr(P_a\im g(z)P_a)$ with the first bound in
\eqref{eq:sublog-positive-upper}, we get 
\begin{align*}
 &\Pp\left(\|P_a\im g(z)P_a\|
 \ge\frac{c}{2C}W^{1-\epsilon}\right) 
 \ge\frac{c}{4C}-M^{-D}\ge\frac{c}{8C}
\end{align*}
for all sufficiently large $W$. This proves \eqref{eq:sublog-positive-block}
with $c_1=c/(2C)$ and $c_2=c/(8C)$.
\end{proof}

\begin{proof}[\bf Proof of \Cref{thm:sublog-divergence}]
Write $\eta=\eta_N$ and choose $n$ as in
\Cref{lem:sublog-positive-block}. Since $W^{1-\epsilon}\le \logpara^{(1-\epsilon)^2}=\oo(\logpara)$, we have $n\ll L=\sqrt N/W$. We pack the block torus $\Zn$ with $J:=\lfloor L/n\rfloor^2$ \emph{disjoint squares} of side $n$, leaving any remaining strips unused. Since $n/L\to0$, we have 
\begin{equation}
     J\asymp L^2/n^2=N/(W^2n^2)=N^{1-\oo(1)}.
\label{eq:box-count}  
\end{equation}

For each $n\times n$ square, we retain the original entries $H_{xy}$ whose endpoints both lie in the square, and add independent Gaussian entries along the wrap-around edges to form a periodic auxiliary torus. The resulting matrices $h_1,\ldots,h_J$ are independent two-dimensional random band matrices on these auxiliary tori. In each square, choose one of its four central blocks and denote its index by $a_j$. Its distance from every boundary block is at least $n/4$ in the auxiliary torus metric.
We say that the $j$th square \emph{succeeds} if \eqref{eq:HS-norm-law} with $k=4$ holds at $z$ for $g_j(z):=(h_j-z)^{-1}$ and
\begin{equation}
 \left\|\im\left(P_{a_j}g_j(z)P_{a_j}\right)\right\|
 \ge c_1W^{1-\epsilon}.
\label{eq:Im-g-lower}
\end{equation}
By \Cref{thm:local-law-smallW,lem:sublog-positive-block}, each square succeeds with probability at least $c_2-\oo(1)$. The success events are independent, so Chebyshev's inequality implies that at least $c_2J/2$ squares succeed with probability tending to 1. Denote this event by $\Omega$.

Each row of $H$ and of every $h_j$ contains at most $5W^2$ nonzero Gaussian entries, each with variance at most $1/(5W^2)$. There are at most $CNW^2$ independent nonzero entries in total. A Gaussian tail bound and a union bound therefore give $\Pp(\Omega_K)=1-\oo(1)$ for a sufficiently large absolute constant $C>0$, where 
\begin{equation}\label{eq:sublog-K-bound}
 \Omega_K:=
 \hB{
 \max_{A\in\{H,h_1,\ldots,h_J\}}
 \max_x\sum_y|A_{xy}|\le K
},
 \qquad K:=CW\sqrt{\logpara}.
\end{equation}
Since all these matrices are Hermitian, Schur's test yields $\|A\|_{\op}\le K$ on $\Omega_K$ for every $A\in\{H,h_1,\ldots,h_J\}$.

We work on $\Omega\cap\Omega_K$ and fix a successful square with block index set $\cI_j$.
Let $\partial\cI_j$ denote the set of blocks in $\cI_j$ with a nearest neighbor outside $\cI_j$, and write $P_I:=\sum_{b\in I}P_b$ for any set $I$ of blocks. Let $H_{\cI_j^c}$ be the principal minor of $H$ obtained by removing the rows and columns corresponding to sites in the square. Define
\[
 \wt H_j:=h_j\oplus H_{\cI_j^c},\qquad
 \wt G_j(z):=(\wt H_j-z)^{-1}=g_j(z)\oplus(H_{\cI_j^c}-z)^{-1}.
\]
Since $a_j\in\cI_j$, the direct-sum decomposition gives $\wt G_j(z)P_{a_j}=P_{\cI_j}\wt G_j(z)P_{a_j}$. To form $\wt H_j$ from $H$, we remove the couplings crossing the boundary of the square and add periodic couplings between its opposite edges. Every endpoint of a modified coupling that lies inside the square belongs to $\partial\cI_j$. Hence
$$(H-\wt H_j)P_{\cI_j}=(H-\wt H_j)P_{\partial\cI_j}.$$
The resolvent identity therefore gives
\[
 P_{a_j}\bigl(G(z)-\wt G_j(z)\bigr)P_{a_j}
 =-P_{a_j}G(z)(H-\wt H_j)P_{\partial\cI_j}\wt G_j(z)P_{a_j}.
\]
By orthogonality of the boundary block projections, we have 
\[
 \|P_{\partial\cI_j}\wt G_j(z)P_{a_j}\|_{\HS}^2
 =\sum_{b\in\partial\cI_j}\|P_b g_j(z)P_{a_j}\|_{\HS}^2.
\]
The sum runs over the $4n-4$ boundary blocks, each at distance at least $n/4$ from $a_j$ in the auxiliary torus. Using $\|G(z)\|\le\eta^{-1}$, \smash{$\|H-\wt H_j\|\le 2K$} on $\Omega_K$, and \eqref{eq:HS-norm-law} with $k=4$ on this successful square, we obtain
\begin{align}\label{eq:sublog-box-error}
 \left\|P_{a_j}G(z)P_{a_j}
       -P_{a_j}(h_j-z)^{-1}P_{a_j}\right\| &\le \frac{2K}{\eta}
 \pbb{\sum_{b\in\partial\cI_j}
 \|P_b(h_j-z)^{-1}P_{a_j}\|_{\HS}^2}^{1/2}\notag\\
 &\le \frac{CKW\sqrt n}{\eta (n\sqrt\eta)^{2}}
 \le CKW \exp(-W^{1-\epsilon})=\oo(1).
\end{align}
Combining this estimate with \eqref{eq:Im-g-lower}, we conclude that there exist constants $c,c'>0$ such that, with probability tending to 1, at least $cJ$ central blocks of the original resolvent satisfy
\begin{equation}\label{eq:sublog-original-blocks}
 \|\im(P_{a_j}G(z)P_{a_j})\|\ge c'W^{1-\epsilon}.
\end{equation}

As in \Cref{lem:Haar-entries}, let $U_a\in U(W^2)$ be independent Haar unitaries, independent of $H$, and set $U=\bigoplus_a U_a$. Condition on a realization of $H$ for which at least $cJ$ central blocks satisfy \eqref{eq:sublog-original-blocks}. For each such block $a_j$, choose a unit eigenvector $v_{a_j}$ corresponding to the largest eigenvalue of $\im(P_{a_j}G(z)P_{a_j})$, and fix a coordinate $x_j\in[a_j]$. By positivity of $\im G$ and \eqref{eq:sublog-original-blocks}, we have 
\[
 \im(UG(z)U^*)_{x_jx_j}\ge c'W^{1-\epsilon}
                  |\langle \mathbf e_{x_j},U_{a_j}v_{a_j}\rangle|^2.
\]
Take $u=\logpara/(16W^2)\le 1/16$ and $c_0=c'/16$. By \eqref{eq:sublog-Haar-lower},
\[
 \Pp_U\pa{\im(UG(z)U^*)_{x_jx_j}\ge c_0\logpara/W^{1+\epsilon}} \ge \Pp_U\left(|\langle e_{x_j},U_{a_j}v_{a_j}\rangle|^2\ge u\right) = (1-u)^{W^2-1}\ge N^{-1/4},
\]
where the last inequality follows from $\log(1-u)\ge-2u$ and $W\ge \logpara^{1/2}$. Since the rotations are independent across blocks, the conditional probability that none of these coordinates reaches the threshold is at most
\[
 (1-N^{-1/4})^{cJ}\le\exp(-cJN^{-1/4})=\oo(1)
\]
by \eqref{eq:box-count}. Since the conditioning event has probability tending to 1, the same conclusion holds after averaging over $H$. Finally, $UHU^*$ has the same law as $H$ by \eqref{eq:invariant-H-distr}. Hence, with probability tending to 1,
\[
 \max_{x\in\ZL}\im G_{xx}(E+\ii\eta_N)
 \ge c_0 {\logpara}/{W^{1+\epsilon}},
\]
which proves \eqref{eq:sublog-quantitative}.
\end{proof}

\bibliographystyle{abbrv}
\bibliography{ref.bib} 

\newpage
\appendix
\section{Quantum diffusion estimates}\label{app:quantum-diffusion}

In this appendix, we prove the quantum diffusion estimate in \Cref{thm:local-law} for the larger-bandwidth regime $A_{\ref{thm:delocalization}}>2$, using the estimates from \Cref{prop:small-flow-release}.
We first introduce the profile function used to control the pointwise decay of the $\cL$-loops and define the associated stopping conditions.

\subsection{Profile function and principal stopping time}

Define the decay rate
\begin{equation}
 \nu_\sharp=\pa{C_\sharp \log W}^{-1/2},
 \label{eq:tools-profile-nu}
\end{equation}
where the structural constant $C_\sharp\ge 10$ is chosen sufficiently large relative to the decay rate in \eqref{eq:tools-P-point}. Constants below may depend on $C_\sharp$. For each rate $0<\nu\le\nu_\sharp$, define the sharper profile \smash{$\Omega_t^\sharp$} (with $\sharp$ standing for ``sharp'') and the weaker profile $\Omega_{\nu,t}$ by 
\begin{equation}
 \Omega_{\nu,t}(a,b)
 :=e^{-\nu\sqrt{\rho_t(a,b)}}, \quad \forall a,b\in \Zn,
 \quad \text{and}\quad \Omega_t^\sharp:=\Omega_{\nu_\sharp,t}.
 \label{eq:tools-profile}
\end{equation}
Note that the profile functions increase with time, and they are of constant order when $|a-b|\le C\ell_t \log W$. The following estimates are elementary to check. 

\begin{lemma}
\label{lem:tools-profile-calculus}
For every fixed $c_0>0$, $j\ge 0$, $0<\nu\le \nu_\sharp$, and $t\in[0,t_*]$, we have
\begin{align}
 \sup_a\sum_b(1+\rho_t(a,b))^j\Omega_{\nu,t}(a,b)
 &\le C_{j} (\nu_\sharp/\nu)^{2j+4}(\log W)^{j+2}\ell_t^2,
 \label{eq:tools-profile-moments}\\
 \sum_{a'}\ee^{-c_0\rho_t(a,a')}\Omega_{\nu,t}(a',b)
 &\leq C_{c_0}\ell_t^2\Omega_{\nu,t}(a,b),
 \label{eq:tools-profile-mixed}\\
 \sum_{a'}\Omega_{\nu,t}(a,a') \Omega_{\nu,t}(a',b)
 &\leq C (\nu_\sharp/\nu)^4(\log W)^2\ell_t^2\Omega_{\nu, t}(a,b).
 \label{eq:tools-profile-convolution}
\end{align}
These estimates remain true after a bounded displacement of any block argument upon enlarging the constants. 
\end{lemma}

\begin{proof}
We divide the summation region into shells according to the condition $k\leq\rho_t(a,b)<k+1$, $k\ge 0$, where each shell has cardinality $\#\{b\in\Zn:k\le\rho_t(a,b)<k+1\} \le C\ell_t^2(1+k)$.
Consequently, for $j\ge 0$,
\begin{equation}
 \sum_b(1+\rho_t(a,b))^j\Omega_{\nu,t}(a,b)
 \le C\ell_t^2\sum_{k\ge0}(1+k)^{j+1}e^{-\nu\sqrt{k}}
 \le C_{j} \nu^{-2j-4} \ell_t^2.
 \label{eq:pf-tools-profile-mass}
\end{equation}
This proves \eqref{eq:tools-profile-moments}.
For \eqref{eq:tools-profile-convolution},
using the triangle inequality
\begin{equation}
 \sqrt {\rho_t(a,a')}+\sqrt {\rho_t(a',b)}\ge\sqrt {\rho_t(a,b)}+\tfrac12\sqrt{\min\{\rho_t(a,a'),\rho_t(a',b)\}},
 \label{eq:tools-profile-triangle}
\end{equation}
we obtain that 
\begin{equation}  
 \Omega_{\nu,t}(a,a')\Omega_{\nu,t}(a',b)
 \le \Omega_{\nu,t}(a,b) e^{-\frac12\nu\sqrt{\min\{\rho_t(a,a'),\rho_t(a',b)\}}}.
\label{eq:product_Omega}
\end{equation}
When we sum over $a'$, we split the sum into $\rho_t(a,a')\le \rho_t(a',b)$ and $\rho_t(a,a')>\rho_t(a',b)$, and apply \eqref{eq:pf-tools-profile-mass} with $j=0$. This yields \eqref{eq:tools-profile-convolution}. For \eqref{eq:tools-profile-mixed}, using \eqref{eq:tools-profile-triangle} and completing the square in $\sqrt{\rho_t(a,a')}$, we get 
\[
 e^{-c_0\rho_t(a,a')}\Omega_{\nu,t}(a',b) 
 \le\Omega_{\nu,t}(a,b) e^{-c_0\rho_t(a,a')+ \nu\sqrt{\rho_t(a,a')}}
 \le C_{c_0}\Omega_{\nu,t}(a,b) e^{-c_0\rho_t(a,a')/2}.
\]
Summing the last exponential yields a $C\ell_t^2$ factor, which proves \eqref{eq:tools-profile-mixed}.\end{proof}

Under the condition \eqref{eq:polylog-W} with $A_{\ref{thm:delocalization}}>2$, we choose a fixed margin $\epsilon>0$ and the bootstrap exponent $\fb$ by
\begin{align}
   0<\epsilon <\min(1/100,2A_{\ref{thm:delocalization}}-4),\quad \text{and}\quad 
\fb=3+\epsilon/3
.\label{ac:small-band-parameters}
\end{align}
Under this choice, we have  
\begin{equation}
 W^2\ge\logpara^{4+\epsilon}\, .
 \label{eq:bandwidth-margin}
\end{equation}
For \(0\le u\le t_*\), where $t_*$ is defined in \eqref{eq:t*}, we first define the principal stopping time for the rescaled light-weights and $\cD$-loops. 

\begin{definition}[Principal stopping time]\label{def:common-stop}
We retain $\mathfrak W$ and $J_1$ from \eqref{eq:common-light-weight-control}, as well as $p_*$, $\mathfrak s_t$, and $\mathcal Z_t$ from \eqref{eq:positive-barriers} and \eqref{eq:positive-fine-barrier}. With a slight abuse of notation, we redefine the 2-loop control parameter in this appendix by
\begin{equation}\label{eq:J1-J2}
  J_{2}(t):=\max_{\star\in\{\emptyset,+\}}\max_{a,b}
{\absb{\mathcal D^{\star}_{ab}(t)}}\big/\pa{{\mathfrak w}\Pi_t(a,b)},
\end{equation}
where the weight $\mathfrak w$ and profile $\Pi_t$ are defined by
\begin{equation}
{\mathfrak w}:={\logpara^{\fb}}/{W^4},\quad 
\Pi_t(a,b):= \qa{1+\iota_t(a,b)} \Omega_{t}^\sharp(a,b).  
\label{eq:principal-envelopes}
\end{equation}
Note that under \eqref{ac:small-band-parameters} and \eqref{eq:bandwidth-margin}, we have 
\begin{equation} {\mathfrak W} \le (\log W)^{6}\logpara^{-5/2-\epsilon}\ll \logpara^{-1-\epsilon/2},\quad W^2{\mathfrak w}\le \logpara^{-1-2\epsilon/3}.
  \label{eq:small-Wt}\end{equation}
We also redefine the principal stopping time for this appendix as
\begin{equation}
\tau_*:=t_*\wedge\inf\left\{t\in[0,t_*]:
    \max\{J_1(t),J_2(t),\mathcal Z_t\}\ge1\right\},\quad \inf\varnothing:=+\infty.
\label{eq:principal-stop}
\end{equation}
At time zero, $J_1(0)=J_2(0)=0$ and
$\mathcal Z_0\le(\log W)^{-10}<1$, so $\tau_*>0$ by continuity.
\end{definition}

By \Cref{prop:small-flow-release}, for every fixed $D>0$ and sufficiently large $N$, the controls $J_1$ and $\mathcal Z$ remain below $1/2$ throughout $[0,t_*]$ with probability at least $1-N^{-D}$. Thus only the redefined $J_2$ requires a new bootstrap estimate. We retain all three controls in \eqref{eq:principal-stop} so that the estimates below hold pathwise up to $\tau_*$. In particular, the stopping conditions imply that, for some constant $c_\kappa>0$, on $\{t\le\tau_*\}$,
\begin{equation}
\max_{a}\absb{\tr\pb{\Gc_tE_a}}\le {\mathfrak W},\quad  c_\kappa\le\im\Tr(G_tE_a)\le 2,\quad
 |\Tr(G_t(\sigma) E_a)|\le 2,\quad
 |\mathcal D^{\star}_{ab}(t)|\le {\mathfrak w}\Pi_t(a,b).
 \label{new:common-prior}
\end{equation}
Similar to \Cref{lem:entropy-operator-output-sm}, the stopping conditions yield the following bounds for the Schatten maxima in \eqref{eq:def-Os}.

\begin{lemma}
\label{lem:entropy-operator-output}
There exists a constant $C_{\ref{lem:entropy-operator-output}}>0$ depending only on $C_{\ref{prop:tools-square-kernel}}$ in \Cref{prop:tools-square-kernel} such that the following bounds hold simultaneously for every stopped trajectory and every $0\le t\le \tau_*$:  
\begin{equation}
  \begin{aligned} 
O_2(t) &\le C_{\ref{lem:entropy-operator-output}}\logpara/W^2,\quad 
 &&O_{\op}(t)\le C_{\ref{lem:entropy-operator-output}}\vartheta_t (\log W)^{5}/W^2, \\ 
O_4(t) &\le C_{\ref{lem:entropy-operator-output}}\logpara^3(\log W)^{10}/W^6,\quad 
 &&\wt O_4(t)\le C_{\ref{lem:entropy-operator-output}}\logpara \ell_t^2 (\log W)^{10}/W^{6}.
  \end{aligned}
 \label{eq:entropy-operator-output}
\end{equation}
Moreover, we have the stronger entrywise bounds: for all $a,b\in \Zn$, 
\begin{equation}
 \|X_{ab}(t)\|_{\op}\le C_{\ref{lem:entropy-operator-output}}(\log W)^{9/2}W^{-2}
 \sqrt{\vartheta_t[\iota_t(a,b)+\log W]},
 \label{eq:spatial-operator-output}
\end{equation}
\begin{equation}
 \|X_{ab}\|_{S_2}^{2}\le C_{\ref{prop:tools-square-kernel}}W^{-2}[1+ \iota_t(a,b)]e^{-c_{\ref{prop:tools-square-kernel}}\rho_t(a,b)}+{\mathfrak w}\Pi_{t}(a,b)\le C_{\ref{lem:entropy-operator-output}}W^{-2}\Pi_{t}(a,b).
 \label{eq:Gram-sharp-bounds}
\end{equation}
For $\rho_t(a,b)>\nu_\sharp^{-2}/2$, we also have
\begin{equation}
 \|X_{ab}\|_{S_2}^{2}\le C_{\ref{lem:entropy-operator-output}}{\mathfrak w}\Omega_{t}^{\sharp}(a,b).\label{eq:entropy-far-S2}
\end{equation}
\end{lemma}

\begin{proof}
Fix $0\le t \le \tau_*$ throughout the proof. 
Since \smash{$\|X_{ab}\|_{S_2}^2=\mathcal L_{ab}(t)$}, the last bound in \eqref{new:common-prior}, together with \eqref{eq:tools-P-point}, gives the first inequality in \eqref{eq:Gram-sharp-bounds}. Then using $ \iota_t(a,b)\le C\logpara$, $W^2{\mathfrak w}=o(1)$, and \smash{$e^{-c_{\ref{prop:tools-square-kernel}}\rho_t(a,b)} \le C\Omega_t^\sharp(a,b)$} proves the second inequality in \eqref{eq:Gram-sharp-bounds}, and hence the bound on $O_2(t)$.
If $\rho_t(a,b)>\nu_\sharp^{-2}/2$, then $\iota_t(a,b)=0$.
By choosing $C_\sharp$ in \eqref{eq:tools-profile-nu} sufficiently large relative to $c_{\ref{prop:tools-square-kernel}}$, we can ensure that 
\begin{equation*}
 e^{-c_{\ref{prop:tools-square-kernel}}\rho_t(a,b)}
 \le CW^{-2}\Omega_t^\sharp(a,b).
\end{equation*}
As ${\mathfrak w}\ge W^{-4}$, the first inequality in
\eqref{eq:Gram-sharp-bounds} now gives \eqref{eq:entropy-far-S2}. The proof of \Cref{lem:entropy-operator-output-sm} derives the bounds on $O_{\op}$ and \smash{$\wt O_4$} from $\mathcal Z_t\le1$, Ward's identity, and an upper bound on $\mathcal L_{ab}$. Applying the same argument with the sharper profile \eqref{eq:Gram-sharp-bounds} yields \eqref{eq:spatial-operator-output} and the corresponding bounds in \eqref{eq:entropy-operator-output}. Finally, the inequality
$\|X_{ab}\|_{S_4}^4\le\|X_{ab}\|_{\op}^2\|X_{ab}\|_{S_2}^2$,
together with the bounds on $O_{\op}$ and $O_2$, gives the stated bound on $O_4$.
\end{proof}

\subsection{Two-sided regularization}\label{sec:double-reg}

We estimate $J_2$ by analyzing \eqref{int_K-LcalE} using the controls $J_1(t)\le1$ and $\mathcal Z_t\le1$, together with the bounds in \Cref{lem:entropy-operator-output}. This analysis requires tracking the spatial decay of the two-loops through the profile defined in \eqref{eq:tools-profile}. To bound the light-weight and martingale terms, we also introduce the half-decay profile $\psi_t$ for resolvent blocks and the corresponding triangle (three-loop) profile $\mathfrak P_t$:
\begin{equation}
 \psi_t(a,b):=\Omega_t^\sharp(a,b)^{1/2},\quad 
 \mathfrak P_t(x;a,b):= \psi_t(a,b)\psi_t(a,x)\psi_t(b,x) ,
 \quad
 \mathfrak P_t(a,b):=\mathfrak P_t(0;a,b).
 \label{eq:low-triangle-profile-definition}
\end{equation}
We also abbreviate $\Omega_t^\sharp(r)\equiv \Omega_t^\sharp(0,r)$ and $\psi_t(r)\equiv \psi_t(0,r)$. Using \eqref{eq:product_Omega} and \eqref{eq:tools-profile-convolution} with $\nu=\nu_\sharp/2$, we get
\begin{align}
\mathfrak P_t(x;a,b)
 &\le C\Omega_t^\sharp(a,b), \quad \Omega_t^\sharp(a,x)\Omega_t^\sharp(x,b)\le C\mathfrak P_t(x;a,b),\quad 
 \sum_x\mathfrak P_t(x;a,b)
 \le C(\log W)^2\ell_t^2\Omega_t^\sharp(a,b).
 \label{eq:low-triangle-root-mass}
\end{align}

Our analysis also requires kernel estimates for $\mathcal U_{s,t}$. We begin with the one-coordinate kernels defined in \eqref{eq:one-coordinate-evolution-kernel}, which can be written as
\[U_{u,t}^+=\Id+(t-u)S^{(\sB)}\Theta^+_t,\qquad U_{u,t}=\Id+\Xi_{u,t},\qquad \Xi_{u,t}:=(t-u)S^{(\sB)}\Theta_t.\] 
The following lemma for the stable kernel $U^+$ follows directly from \eqref{eq:tools-P-point-short}; we omit the straightforward proof.

\begin{lemma}[Stable kernel]
\label{lem:tools-stable-transport}
There is a constant $C_\kappa>0$ depending only on $\kappa$ such that, for every $0<\nu\le \nu_{\sharp}$, $s\in[0,1)$, and arbitrary $0\leq u\leq t<1$,
\begin{align}
\|U_{u,t}^+\|_{\infty\to\infty}&\le C_\kappa,\quad  \sum_{a'}|U_{u,t}^+(a,a')|\Omega_{\nu,s}(a',b) \leq C_\kappa \Omega_{\nu,s}(a,b). \label{eq:tools-stable-inf2inf}
\end{align}
\end{lemma}

For the singular kernel, we can derive the following estimate using \eqref{eq:tools-P-point} and \eqref{eq:high-exponential-row-moment}. 
\begin{lemma}[Singular kernel]
\label{prop:tools-bounded-transport}
For any fixed constant $c_\nu\in(0,1]$, let $\nu=c_\nu\nu_\sharp$. Then, for all $0\le u\le t\le t_*$,
\begin{align}
\sum_{a'}|U_{u,t} (a,a')|\Omega_{\nu,u}(a',b) &\le \sum_{a'}|U_{u,t} (a,a')|\Omega_{\nu,t}(a',b)   \leq C \frac{\eta_u}{\eta_t} \Omega_{\nu,t}(a,b). 
 \label{eq:tools-bounded-transport}
\end{align}
\end{lemma}

\begin{proof}
Fix $u,t,\nu$. We abbreviate $U=U_{u,t} $ and $\Omega_s=\Omega_{\nu,s}$ for $s\in\{u,t\}$. Rewrite $U_{u,t}$ as 
\begin{equation}
 U(a,a')=\1_{\{a=a'\}}+\frac{t-u}{1-t}\pi_{t}(a,a'),\quad  \pi_{t}=(1-t)S^{(\sB)}\Theta_t.
 \label{eq:tools-singular-propagator-bound}
\end{equation}
By definition and the bounds \eqref{eq:tools-P-point} and \eqref{eq:high-exponential-row-moment}, we have 
\begin{align}\label{eq:simple_Pt}
 \sum_{a'}\pi_t(a,a')=1, \quad 
 \pi_t(a,a')\le C[\ell_t^{-2}+\eta_t\iota_t(a,a')] 
 \ee^{-c\rho_t(a,a')},\quad \sum_b \pi_t(a,b)e^{c\rho_t(a,b)}\le C,
\end{align}
for a sufficiently small constant $0<c\le c_{\ref{prop:tools-square-kernel}}/2$. 
Using the simple bound 
\begin{equation}
 \Omega_u(a',b)\leq\Omega_t(a',b)
 \leq\ee^{\nu \sqrt{\rho_t(a,a')}}\Omega_t(a,b)
, \label{eq:Omega_u_t}
\end{equation}
along with \eqref{eq:simple_Pt}, we derive that
\begin{equation}\label{eq:Pu-Omega}
  \sum_{a'}\pi_t(a,a')\Omega_u(a',b)\le\Omega_t(a,b) 
 \sum_{a'}\pi_t(a,a')\ee^{\nu\sqrt{\rho_t(a,a')}} \le C\Omega_t(a,b).
\end{equation}
With this bound and \eqref{eq:tools-singular-propagator-bound}, we obtain
\[
 \sum_{a'}|U(a,a')|\Omega_u(a',b) 
 \leq\Omega_t(a,b) +\frac{\eta_u}{\eta_t}
 \sum_{a'}\pi_t(a,a')\Omega_u(a',b)
 \leq C \frac{\eta_u}{\eta_t}\Omega_t(a,b).
\]
This concludes \eqref{eq:tools-bounded-transport}.
\end{proof}

The singular kernel estimate \eqref{eq:tools-bounded-transport} is substantially weaker than the stable kernel estimate \eqref{eq:tools-stable-inf2inf}, making the $\cD$-loops more difficult to control than the $\cD^+$-loops. The main difficulty is the singular factor $\eta_u/\eta_t$ in \eqref{eq:tools-bounded-transport}. To compensate for this factor, we introduce a two-sided regularization of the $\cD$-loops using a compactly supported averaging kernel, inspired by Definition~4.7 of \cite{DYYY25_d3}.

Fix a symmetric nonnegative $\varphi\in C_c^\infty(\mathbb R^2)$, supported in the $\ell^1$ unit ball around the origin and bounded below by a positive constant on the $\ell^1$ ball of radius $1/2$. 
For $0\le t\le t_*$, with $t_*$ defined in \eqref{eq:t*}, we define the locally smoothed cutoff kernel on the scale $R_t:=(\log W)\ell_t\le C L/\log W$ as 
\begin{equation}
 \mathsf K_t(a,b):=\frac1{\mathsf Z_t}  \varphi\pa{\frac{\rep{a-b}_L}{R_t}},
 \qquad \mathsf P_t:=\Id-\mathsf K_t,
 \label{compact:kernel}
\end{equation}
where $\mathsf Z_t$ is the normalization factor \smash{$\mathsf Z_t:=\sum_{x\in\Zn}\varphi(x/R_t)$}. Note that $\mathsf K_t$ is symmetric and doubly stochastic, and $\mathsf Z_t$ is of order $R_t^2$. Consequently, there exists a constant $C_\varphi>0$ such that for all $t\in[0,t_*]$, 
\begin{equation}
 0\le\mathsf K_t(a,b)\le C_\varphi R_t^{-2}\mathbf1_{|a-b|\le R_t},\qquad
 |\partial_t\mathsf K_t(a,b)|\le C_\varphi\eta_t^{-1}R_t^{-2}
                 \mathbf1_{|a-b|\le R_t}.
 \label{compact:kernel-bounds}
\end{equation}
Then we define the two-sided lifting operator $\sL_t$ acting on vectors and the two-sided regularization operator $\sQ_t$ acting on matrices by\footnote{With a slight abuse of notation, we reuse $\sL_t$ and $\sQ_t$ for these operators, which differ from those defined in \eqref{direct:transported-lift}. The compactly supported lift in \eqref{impl:lift} could also be used in \Cref{sec:direct-operator}, but it would lengthen the proof by introducing additional drift terms into the $Z$-dynamics \eqref{eq:low-Z-dynamics}. Conversely, the lift in \eqref{direct:transported-lift} cannot be used directly for the pointwise quantum diffusion estimates in this appendix: the $\Theta$-propagator bound \eqref{eq:tools-P-point} contains logarithmic factors $\iota_t$, which lead to losses in the pointwise transfer estimates. For the operator-norm estimates in \Cref{sec:direct-operator}, \eqref{eq:tools-theta-square} prevents these additional logarithmic losses.}
\begin{equation}
 \mathsf L_tf=\diag(f)\mathsf K_t+\mathsf K_t\diag(f)-\mathsf K_t\diag(f)\mathsf K_t,
 \qquad \mathsf Q_t=\Id-\mathsf L_t\circ\cP,
 \label{impl:lift}
\end{equation}
where $\cP$ is the row-sum operator defined in \eqref{eq:tools-relative-projection}. Note that \(\sL_tf\) is symmetric and $(\mathsf L_tf)\mathbf1=f$.  As a consequence, \(\Qsym_t \circ F\) satisfies the two-sided sum-zero property \(\p{\Qsym_t \circ F}\1=\p{\Qsym_t \circ F}^\top\1=0\) whenever \(F\mathbf1=F^{\mathsf T}\mathbf1\). Moreover, the entries $(\mathsf L_tf)(a,b)$ vanish if $|a-b|>2R_t$.

With $\mathsf f$ in \eqref{eq:mathsf-f}, we define the marginal $Y$-loops and the regularized $Z$-loops by
\begin{align}
Y(t):=\sL_t\circ \mathsf f(t),\qquad 
Z(t):=\Qsym_t\circ \difloop(t) \, .\label{def:Z-Y-f} \end{align}
These definitions give the decomposition
\begin{equation}\label{eq:D=Z+Y}
 \difloop(t)=Z(t)+Y(t). 
\end{equation}
 By Ward's identity and the fact that $J_1(t)\le 1$ on $\{t\le \tau_*\}$, we have 
$\|\mathsf f(t)\|_\infty\le \mathfrak W/(W^2\eta_t)$ on $\{t\le \tau_*\}$. Together with the bound on $\mathsf K_t$ in \eqref{compact:kernel-bounds}, it implies that for all $a,b\in \Zn$,
\begin{align}
 |Y_{ab}(t)|&=|(\sL_t\circ \mathsf f(t))(a,b)|\le C\|\mathsf f(t)\|_\infty (\mathsf K_t+\mathsf K_t^2)(a,b)\le CR_t^{-2} \|\mathsf f(t)\|_\infty\1_{|a-b|\le 2R_t} \notag\\
 &\le C\logpara^{3/2}(\log W)^{4}W^{-4} \Omega_t^{\sharp}(a,b).
 \label{compact:reconstruction}
\end{align}
On the other hand, we control the $Z$-loops through their dynamics, stated in the following lemma. To this end, define the additional drift induced by the two-sided regularization:
\begin{equation}
 \mathsf D_t \circ f := 
  \sA_t\circ (\mathsf L_tf) -(\partial_t\mathsf L_t)\circ f-\mathsf L_t\circ (\mathfrak g_tf),
 \label{compact:defect}
\end{equation}
where $\mathfrak g_t:=(1-t)^{-1}\Id+S^{(\sB)}\Theta_t$ was introduced below \eqref{eq:D+Y-2}. Note that \(\sD_tf\) is symmetric and satisfies the two-sided sum-zero property $(\sD_tf)\1=(\sD_tf)^\top\1=0$, since $\cP\circ \mathsf L_t=\Id$, $\cP\circ (\partial_t\mathsf L_t)=0$, and
$\cP\circ \mathscr A_t \circ \mathsf L_t=\mathfrak g_t$.
We refer to this term as the \emph{lift-defect drift}.

\begin{lemma}\label{lem:low-lift-lemma}
The $Z$-loops satisfy the dynamics 
\begin{equation}
 \dd Z(t)=\qa{\mathscr A_t\circ Z(t)+\mathsf Q_t\circ \pa{\mathcal E(t)+\cW(t)}+
          \mathsf D_t\circ \mathsf f(t)}\dd t+\mathsf Q_t\circ \dd\mathcal M(t),
 \label{eq:low-Z-dynamics}
\end{equation}
For deterministic $0\le s\le t\le t_*$, applying Duhamel's principle to \eqref{eq:low-Z-dynamics} gives the mild form
\begin{align}
 Z_t=\mathcal U_{s,t}\circ Z(s)
 +\int_s^t\mathcal U_{u,t}\circ
       \qa{\mathsf Q_u\circ (\mathcal E(u)+\mathcal W(u))+\mathsf D_u\circ \mathsf f(u)}\dd u +\int_s^t\mathcal U_{u,t}\circ\mathsf Q_u\circ \dd\mathcal M(u).
 \label{eq:two-side-Z-Duhamel}
\end{align}
\end{lemma}
\begin{proof}
Differentiating $Z(t)=\difloop(t)-\Tsym_t\mathsf f(t)$ and using the product rule
\(
 \dd(\Tsym_t\mathsf f(t))
 =(\partial_t\Tsym_t)\mathsf f(t)\dd t
 +\Tsym_t\dd\mathsf f(t),
\)
together with \eqref{eq_L-Keee} and \eqref{eq:low-density-dynamics}, yields \eqref{eq:low-Z-dynamics}.
\end{proof} 

We do not regularize $\cD^+$-loops and set $Z^+:=\difloop^+$ and $Y^+:=0$. In the remainder of the proof, we focus on the $\cD$-loops. The $\cD^+$-loops admit a similar but simpler treatment, with all regularization operators replaced by the identity. Indeed, the drift and martingale estimates below are uniform in the loop charges and, together with the stable kernel estimates in \Cref{lem:tools-stable-transport}, give the corresponding bounds. It remains to control $Z(t)$ using \eqref{eq:two-side-Z-Duhamel} with $s=0$. We define the stopped process \smash{$\wt Z(t)$} by inserting $\1_{\{u\le\tau_*\}}$ into each integrand and estimate the resulting terms separately. We begin with transfer estimates for the evolution kernel $\mathcal U_{u,t}$ under two-sided regularization.

\subsection{Transfer estimates under two-sided regularization}\label{sec:transfer-two-sided}

Throughout this subsection, we fix initial and terminal times $0\le u\le t\le t_*$ and consider the evolution kernel $\mathcal U_{u,t}=U_{u,t}\otimes U_{u,t}$. 
For simplicity, we write $U=U_{u,t}=\Id+\Xi$ and $\Xi=\Xi_{u,t}=(t-u)S^{(\sB)}\Theta_t$. For $r,s,x,a\in\Zn$, define
\begin{align}
 g_r^x(a)&=U(r,a)-\Xi(r,x)=\delta_{ra}+\Xi(r,a)-\Xi(r,x),
 \label{compact:point-rows}\\
 \lambda_x(r,s;a)&=g_r^x(a)(\mathsf K_u g_s^x)(a)+(\mathsf K_u g_r^x)(a)g_s^x(a)
                    -(\mathsf K_u g_r^x)(a)(\mathsf K_u g_s^x)(a).
 \label{compact:marginal-kernel}
\end{align}
Here $x$ is a fixed reference point. Apart from the identity contribution $\delta_{ra}$, the row $g_r^x$ is a spatial difference of $\Xi$, which
provides an additional decay factor through the gradient estimate. The function $\lambda_x$ describes the corresponding contribution of
the lift. More precisely, if a matrix $F_x$ satisfies $F_x\mathbf1=F_x^{\mathsf T}\mathbf1=f_x$, then
\begin{equation}
 [U(\mathsf Q_u\circ F_x)U](r,s)
 =\sum_{a,b}g_r^x(a)F_x(a,b)g_s^x(b)
        -\sum_af_x(a)\lambda_x(r,s;a).
 \label{compact:projected-source}
\end{equation}
In fact, by the two-sided sum-zero property of $\mathsf Q_u\circ F_x$ and the symmetry of $U$, we may subtract the constants $\Xi(r,x)$ and $\Xi(s,x)$ from the corresponding rows to obtain
\[
 \begin{aligned}
 [U(\mathsf Q_u\circ F_x)U](r,s)
 &=\sum_{a,b}[U(r,a)-\Xi(r,x)]
       (\mathsf Q_u\circ F_x)(a,b)[U(s,b)-\Xi(s,x)]\\
 &=\sum_{a,b}g_r^x(a)F_x(a,b)g_s^x(b)
   -\sum_{a,b}g_r^x(a)(\mathsf L_u\circ f_x)(a,b)g_s^x(b).
 \end{aligned}
\]
Expanding $\mathsf L_u$ according to \eqref{impl:lift}, the last sum equals $\sum_a f_x(a)\lambda_x(r,s;a)$, proving \eqref{compact:projected-source}.
We now establish the following estimates for the two terms in \eqref{compact:projected-source}.

\begin{proposition}
  \label{impl:transfer}
Let $\{F_x:x\in \Zn\}$ be a family of matrices whose entries belong to some Banach space equipped with norm $\|\cdot\|$ and satisfy
\[\|F_x(a,b)\|\le (1+\iota_u(a,x))(1+\iota_u(b,x)) \mathfrak P_u(x;a,b),\quad a,b\in \Zn.\] 
Denote the first term on the RHS of \eqref{compact:projected-source} by 
$\mathcal R_x(F)=\sum_{a,b}g_r^x(a)F_x(a,b)g_s^x(b)$.
Then
\begin{align}
  \sum_{x}\|\mathcal R_x(F)\| 
 \le C(\log W)^{10}{\ell_u^2}(1+\iota_t(r,s))\Omega_t^{\sharp}(r,s),
 \ \  
 \pB{\sum_{x}\|\mathcal R_x(F)\|^2}^{1/2}  \le C(\log W)^{9}\ell_u \Psi_{u,t}\Omega_t^{\sharp}(r,s).
 \label{impl:root-transfer}
\end{align}
The transfer gain $\Psi_{u,t}\equiv \Psi_{u,t}(r,s)$ is defined by 
\begin{equation}
\Psi_{u,t}(r,s)=
 \begin{cases}
 1,&\ell_t/\ell_u<(\log W)^2,\\
 (1+\min\{|r-s|,\ell_t\}/R_u)^{-1/4},
 &\ell_t/\ell_u\ge(\log W)^2.
 \end{cases}
\label{eq:Psi-gain}  
\end{equation}
Let $\{f_x:x\in \Zn\}$ be a family of vectors whose entries belong to the same Banach space. If $\|f_x(a)\|\le (1+\iota_u(a,x))\Omega_u^{\sharp}(a,x)$, then
\begin{align}
  \sum_x  \sum_a \|f_x(a)\|\,|\lambda_x(r,s;a)| 
 &\le C(\log W)^{8} (1+\iota_t(r,s))\Omega_t^{\sharp}(r,s),
 \label{eq:marginal-scalar-response}\\
 \qbb{ \sum_x \pB{\sum_a \|f_x(a)\|\,|\lambda_x(r,s;a)|}^2 }^{1/2}
 &\le C(\log W)^{7}\ell_u^{-1}\Psi_{u,t} \Omega_t^{\sharp}(r,s) .
 \label{eq:marginal-square-response}
\end{align}
\end{proposition}

The proof of this proposition depends on several technical lemmas. First, to describe the spatial decay of the regularized kernel $g_r^x$, we introduce the endpoint profile
\begin{equation}
 \xi_r(x)=\left(1+\log_+\frac{R_u}{1+|r-x|}\right)
       \frac{R_u}{R_u+|r-x|}
       \exp\qa{-c_\xi\pa{\frac{|r-x|}{R_u+\ell_t}}^{1/2}}.
 \label{eq:near-J}
\end{equation}
The three factors serve distinct purposes. The logarithmic factor controls the local logarithmic singularity. The factor $R_u/(R_u+|r-x|)$ captures the spatial gain from subtracting $\Xi(r,x)$, as quantified by \eqref{eq:gradient-envelope}. The exponential factor combines the initial-time profile decay in \eqref{eq:tools-profile} with the terminal-time $\Theta$-propagator decay. Here $c_\xi>0$ is a sufficiently small constant depending only on $C_\sharp$ and $c_{\ref{prop:tools-square-kernel}}$. The inclusion of $R_u$ in the decay scale ensures compatibility with the cutoff scale in \eqref{compact:kernel} when $t$ is close to $u$. When $\ell_t/\ell_u$ is large, the second factor in \eqref{eq:near-J} yields a gain in products of endpoint profiles. We record this gain by defining
\begin{equation}
  \begin{aligned}  
\Psi^{(1)}_{u,t}(r,s)&:=1+\log_+\frac{\ell_t}{|r-s|+\ell_u},\quad \text{and}\quad \Psi^{(p)}_{u,t}(r,s): =
\left(1+\frac{|r-s|}{R_u}\right)^{-\left(2-2/p-\e\mathbf1_{\{p=2\}}\right)},
 \label{end:cap}
  \end{aligned}
\end{equation}
for fixed $1<p\le2$ and an arbitrarily small constant $\e>0$. Note that for $8/7<p\le 2$, we have that $\Psi^{(p)}_{u,t}(r,s)$ is controlled by $\Psi_{u,t}(r,s)$ defined in \eqref{eq:Psi-gain} for a sufficiently small $\e>0$.
We have the following estimates, whose proof is deferred to \Cref{sec:pf-endpoint-product}.

\begin{lemma}\label{end:master}
Define the ball $B_x=B(x,12{C_\sharp R_u}):=\{a\in\Zn:|a-x|<12{C_\sharp R_u}\}$ where $C_\sharp$ is the constant in \eqref{eq:tools-profile-nu}. Let $\mathfrak r\ge1$ be an arbitrary deterministic parameter which may depend on $N$. 
For each fixed $1\le p\le 2$, there exists a structural constant $C>0$ (independent of $\mathfrak r$) such that
\begin{align}
 \|\xi_r\xi_s\|_{\ell^p} &\le  CR_u^{2/p} \Psi^{(p)}_{u,t}(r,s),
 \label{eq:compact-L2-product}\\
 \|(g_r^x-\delta_r)\mathbf1_{B_x}\|_{\ell^p}
 &\le C(\log W)^{2/p}\ell_u^{2/p-2}\xi_r(x)\Omega_t^{\sharp}(r,x),
 \label{eq:near-endpoint-X1}\\
\sum_{b\in B(x,\mathfrak r R_u)}|g_r^x(b)|(1+\iota_u(b,x)) &\le C\mathfrak r^4(\log W)^2
 \exp\left( \nu_\sharp\sqrt{\frac{2\mathfrak r R_u}{\ell_t}}
            +c_\xi\sqrt{\frac{2\mathfrak r R_u}{R_u+\ell_t}}\right)
       \xi_r(x)\Omega_t^\sharp(r,x). \label{eq:near-endpoint-X}
\end{align}
\end{lemma}

\begin{proof}[\bf Proof of \Cref{impl:transfer}]
For simplicity, we present the proof for complex-valued entries $F_x(a,b)$ and $f_x(a)$; the Banach-valued case follows by replacing absolute values with norms throughout. 
For the proof of \eqref{impl:root-transfer}, we decompose each $F_x(a,b)$ into dyadic shells according to its displacement from $x$:
\[
 F_x^{[0]}(a,b)=F_x(a,b)\mathbf1_{ |a-x|\vee|b-x| \le C_\sharp R_u},\quad \text{and}\quad F_x^{[i]}(a,b)=F_x(a,b)\mathbf1_{2^{i-1}C_\sharp R_u< |a-x|\vee |b-x| \le 2^iC_\sharp R_u}, \ \ i\ge 1.
\]
For shell 0, using $\mathfrak P_u\le1$ and
$\Omega_t^\sharp(r,x)\Omega_t^\sharp(s,x)\le\Omega_t^\sharp(r,s)$, we obtain from \eqref{eq:compact-L2-product} and
\eqref{eq:near-endpoint-X} that
\[
 \|\mathcal R_{\bullet}(F^{[0]})\|_{\ell^p}
 \le C(\log W)^4\|\xi_r\xi_s\|_p\Omega_t^\sharp(r,s)
 \le C(\log W)^{4+2/p}\ell_u^{2/p}
       \Psi^{(p)}_{u,t}(r,s)\Omega_t^\sharp(r,s),\quad p\in\{1,2\}.
\]
Here $\bullet$ denotes the root index $x$, over which the $\ell^p$-norm is taken. This gives the desired contribution to \eqref{impl:root-transfer}, since
\smash{$\Psi^{(1)}_{u,t}(r,s)\le C(1+\iota_t(r,s))$ and
$\Psi^{(2)}_{u,t}(r,s)\le C\Psi_{u,t}(r,s)$.}
It remains to bound
$\|\mathcal R_{\bullet}(F^{[i]})\|_{\ell^p}$ for $i\ge1$.
On the support of $F_x^{[i]}$, the identity
$C_\sharp R_u/\ell_u=\nu_\sharp^{-2}$ gives
\begin{equation}  
 \mathfrak P_u(x;a,b)
 \le \exp\left[-\frac{\nu_\sharp}{2}\sqrt{\rho_u(a,x)}-\frac{\nu_\sharp}{2}\sqrt{\rho_u(b,x)}\right]
 \le e^{-2^{(i-3)/2}}.
\label{eq:exp-triangle-profile}
\end{equation}

We first consider the regime $\ell_t/\ell_u\ge(\log W)^2$.
For $i\ge1$, applying \eqref{eq:near-endpoint-X} with
$\mathfrak r=2^iC_\sharp$, and using
\eqref{eq:exp-triangle-profile} and
$\Omega_t^\sharp(r,x)\Omega_t^\sharp(s,x)\le\Omega_t^\sharp(r,s)$, gives
\begin{align}
 |\mathcal R_x(F^{[i]})|
 &\le C(\log W)^4 2^{8i}\xi_r(x)\xi_s(x)\Omega_t^\sharp(r,s)  \exp\pbb{-2^{(i-3)/2}
       +2\nu_\sharp\sqrt{\frac{2^{i+1}C_\sharp R_u}{\ell_t}}
       +2c_\xi\sqrt{\frac{2^{i+1}C_\sharp R_u}{R_u+\ell_t}}}\notag\\
&\le C(\log W)^4 2^{8i}e^{-2^{(i-5)/2}}\xi_r(x)\xi_s(x)\Omega_t^\sharp(r,s)  . \label{eq:sum-i>i*}
\end{align}
In the second step, the two positive terms in this exponent can be absorbed into the negative term by using $\ell_t \ge(\log W)^2\ell_u = (\log W)R_u$, and 
\[
 \frac{2\nu_\sharp\sqrt{2^{i+1}C_\sharp R_u/\ell_t}}
      {2^{(i-3)/2}}
 =8\sqrt{\frac{\ell_u}{\ell_t}}\le\frac8{\log W},\qquad
 \frac{2c_\xi\sqrt{2^{i+1}C_\sharp R_u/(R_u+\ell_t)}}
      {2^{(i-3)/2}}
 \le8c_\xi\sqrt{C_\sharp}\le\frac14.
\]
Taking the $\ell^p$ norm over $x$ and applying
\eqref{eq:compact-L2-product} yields
\begin{equation}  
 \|\mathcal R_\bullet(F^{[i]})\|_{\ell^p}
 \le C(\log W)^{4+2/p}2^{8i}e^{-c2^{i/2}}
       \ell_u^{2/p}\Psi^{(p)}_{u,t}(r,s)\Omega_t^\sharp(r,s)
\label{eq:sum-i<i*}
\end{equation}
for an absolute constant $c>0$.
Summing over $i\ge1$ and using the triangle inequality, we bound the contribution of these shells by the right-hand sides of \eqref{impl:root-transfer}.
On the other hand, if $\ell_t/\ell_u<(\log W)^2$, we use $|g_r^x(a)|\le U(r,a)+\Xi(r,x)$ to bound $|g_r^x(a)||g_s^x(b)|$ by a sum of four terms. For $p\in\{1,2\}$, we have
\begin{align}
 \norma{ (1+\iota_u(a,\bullet))(1+\iota_u(b,\bullet))\mathfrak P_u(\bullet;a,b) }_p &\le  \norma{\mathfrak P_u(\bullet;a,b) }_p +  \norma{ \iota_u(a,\bullet)+\iota_u(b,\bullet)+\iota_u(a,\bullet)\iota_u(b,\bullet)}_p \Omega_u^\sharp(a,b) \notag\\
 &\le C(\log W)^{2/p}\ell_u^{2/p}\Omega_u^\sharp(a,b).\label{eq:two-iota}
\end{align}
The first inequality follows from the first bound in
\eqref{eq:low-triangle-root-mass}. For the second inequality, \eqref{eq:moments-iota} and H\"older's inequality control the $\iota$ terms. The last inequality in \eqref{eq:low-triangle-root-mass} bounds $\norma{\mathfrak P_u(\bullet;a,b)}_1$; combining it with the first pointwise bound gives the corresponding $\ell^2$ estimate.

By Minkowski's inequality, \eqref{eq:two-iota}, and
\eqref{eq:tools-bounded-transport}, the term containing
$U(r,a)U(s,b)$ satisfies
\begin{equation}
 \normB{\sum_{a,b}U(r,a)U(s,b) F_{\bullet}(a,b)}_p \le C(\log W)^{2/p}\pa{\ell_t^4/\ell_u^4}\ell_u^{2/p}\Omega_t^\sharp(r,s).
\label{eq:U-U-transfer}
\end{equation}
Next, consider a mixed term, say the one containing
$U(r,a)\Xi(s,x)$. Summing first over $b$ and using
\eqref{eq:low-triangle-root-mass} and \eqref{eq:moments-iota}, we obtain
\begin{align}
 &~\sum_{a,b}U(r,a)\Xi(s,x) |F_{x}(a,b)| \le \sum_{a,b} U(r,a)\Xi(s,x)(1+\iota_u(a,x)) (1+\iota_u(b,x))\mathfrak P_u(x;a,b) \notag\\
 \le &~C(\log W)^2\ell_u^2\Xi(s,x)\sum_a U(r,a)(1+\iota_u(a,x))\Omega_u^\sharp(a,x)\notag\\
 \le &~ C(\log W)^2\ell_u^2\Xi(s,x) \qbb{\frac{\ell_t^2}{\ell_u^2}\Omega_t^\sharp(r,x) +  \iota_u(r,x)\Omega_u^\sharp(r,x)+\sum_{a:|a-x|\le \ell_u} \Xi(r,a) \iota_u(a,x)}  \notag\\
  \le &~ C(\log W)^2|t-u|(1+\iota_t(s,x))e^{-\frac{1}{2}c_{\ref{prop:tools-square-kernel}}\rho_t(s,x)} \Omega_t^\sharp(r,s)\qa{ {\ell_t^2}  +  \ell_u^2\iota_u(r,x)}  .
 \label{eq:U-Xi-transfer1}
\end{align}
For the third inequality, we write $U(r,a)=\delta_{ra}+\Xi(r,a)$ and apply \eqref{eq:tools-bounded-transport} to the unweighted term. In the fourth step, we used \eqref{eq:tools-P-point}, \eqref{eq:moments-iota}, and the bound $e^{-c\rho_t(s,x)}\le C\Omega_t^\sharp(s,x)$ for any constant $c>0$ to get  
\[\Xi(s,x) \le C|t-u|(1+\iota_t(s,x))e^{-\frac{1}{2}c_{\ref{prop:tools-square-kernel}}\rho_t(s,x)}\Omega_t^\sharp(s,x),\quad 
\sum_{a:|a-x|\le \ell_u} \Xi(r,a) \iota_u(a,x)
 \le C|t-u|{\ell_t^2}\Omega_t^\sharp(r,x).
\]
We have also used the triangle inequality $\Omega_t^\sharp(s,x)\Omega_u^\sharp(r,x)\le \Omega_t^\sharp(r,s)$. 
Now taking the $\ell^p$-norm of the RHS of \eqref{eq:U-Xi-transfer1} with respect to the root vertex $x$ and using \eqref{eq:moments-iota} and H\"older's inequality gives, for $p\in\{1,2\}$,
\begin{align}\label{eq:U-Xi-transfer}
  \normB{\sum_{a,b} U(r,a)\Xi(s,\bullet)F_{\bullet}(a,b)}_{\ell^p}\le C(\log W)^2 (\ell_t^2/\ell_u^2)\ell_t^{2/p} \Omega_t^\sharp(r,s) .
\end{align}
The other mixed term is treated identically. Finally, for the term containing $\Xi(r,x)\Xi(s,x)$,
\eqref{eq:low-triangle-root-mass} and \eqref{eq:moments-iota} give
\[
 \sum_{a,b}(1+\iota_u(a,x))(1+\iota_u(b,x))
       \mathfrak P_u(x;a,b)\le C(\log W)^4\ell_u^4. 
\]
Then, by a similar argument, we get 
\begin{align}\label{eq:Xi-Xi-transfer}
  \normB{\sum_{a,b} \Xi(r,\bullet)\Xi(s,\bullet)F_{\bullet}(a,b)}_{\ell^p}\le C(\log W)^4\ell_u^4\normB{\Xi(r,\bullet)\Xi(s,\bullet)}_{\ell^p}\le  C(\log W)^4 \ell_t^{2/p} \Omega_t^\sharp(r,s) .
\end{align}
Combining \eqref{eq:U-U-transfer}, \eqref{eq:U-Xi-transfer}, and \eqref{eq:Xi-Xi-transfer}, and using $\ell_t/\ell_u<(\log W)^2$, proves \eqref{impl:root-transfer} in this regime. 

To prove \eqref{eq:marginal-scalar-response} and \eqref{eq:marginal-square-response}, we use \eqref{compact:kernel-bounds} and \eqref{compact:marginal-kernel} to obtain
\begin{align}
 \sum_a|f_x(a)|\,|\lambda_x(r,s;a)|
 &\le\sum_{a,b}|g_r^x(a)||g_s^x(b)|
 \pB{|f_x(a)|\mathsf K_u(a,b)
 +\mathsf K_u(a,b)|f_x(b)|
 +\sum_{a'}\mathsf K_u(a,a')|f_x(a')|\mathsf K_u(a',b)}\notag\\
 &\le CR_u^{-2}\sum_{|a-b|\le 2R_u}|g_r^x(a)||g_s^x(b)|(1+\iota_u(a,x)+\iota_u(b,x)) \mathfrak P_u(x;a,b).
 \label{dim:lift-product}
\end{align}
For the second inequality, the first two terms in brackets are supported on $|a-b|\le R_u$, where $\Omega_u^{\sharp}(a,x)\asymp \Omega_u^{\sharp}(b,x)$ and $\Omega_u^{\sharp}(a,b)\asymp 1$. Consequently, $\Omega_u^{\sharp}(a,x)\asymp \Omega_u^{\sharp}(b,x)\lesssim \mathfrak P_u(x;a,b)$. For the third term, we use the analogous profile comparisons on $|a-a'|,|b-a'|\le R_u$, together with 
\(
 R_u^{-2}\sum_{|a'-a|\le R_u}(1+\iota_u(a',x))
\le C,
\)
which follows from \eqref{eq:moments-iota}.
Finally, we may apply \eqref{impl:root-transfer} to the bound on the RHS of \eqref{dim:lift-product}. Taking the $\ell^p$-norm over $x$ for $p\in\{1,2\}$ and using $R_u=(\log W)\ell_u$ proves \eqref{eq:marginal-scalar-response} and \eqref{eq:marginal-square-response}.
\end{proof}

\subsection{The lift-defect drift}\label{compact:drift-section}
We first estimate the transport of the lift-defect drift $\sD_u\circ \mathsf f(u)$ in \eqref{eq:two-side-Z-Duhamel}. Since it has two zero marginals, we have $\Qsym_t\circ (\sD_u f)=\sD_u f$ for any function $f$, and \eqref{compact:projected-source} gives 
\begin{equation}
 [U(\sD_u f)U](r,s)
 =\sum_{a,b}g_r^x(a)(\sD_u f)(a,b)g_s^x(b) .
 \label{compact:projected-defect}
\end{equation}
Then, with \Cref{impl:transfer}, we obtain the following lemma. 

\begin{lemma}[Transport of the lift-defect drift]\label{compact:defect-transfer}
For any function $f:\Zn\to \C$ and all $0\le u\le t\le t_*$,
\begin{equation}
 \left|[U_{u,t}(\mathsf D_u\circ f)U_{u,t}](r,s)\right|
 \le C\|f\|_{\infty}(\log W)^8(1+\iota_t(r,s))\Omega_t^\sharp(r,s).
 \label{compact:defect-transfer-bound}
\end{equation}
Consequently, with $\mathsf f(u)$ as in \eqref{eq:mathsf-f}, the stopped lift-defect integral satisfies
\begin{equation}
 \frac{W^4}{\Pi_t(r,s)}
 \left|\int_0^t\mathbf1_{\{u\le \tau_*\}}
        [U_{u,t}(\mathsf D_u\circ \mathsf f(u))U_{u,t}](r,s)\dd u\right|
 \le C_D(\log W)^{14}\logpara^{5/2}.
 \label{compact:defect-budget}
\end{equation}
\end{lemma}

\begin{proof}
Define the matrix $F_x:=\mathsf D_u\circ\mathbf e_x$, where $\mathbf e_x$ is the standard unit vector along the $x$-th direction with $\mathbf e_x(a)=\delta_{ax}$,
and write $\mathfrak g_u=(1-u)^{-1}(\Id+\pi_u)$, where $\pi_u=(1-u)S^{(\sB)}\Theta_u$. Since $\mathscr A_u\circ A=(1-u)^{-1}(\pi_u A+A\pi_u)$,
expanding \eqref{compact:defect} using \eqref{impl:lift} gives
\begin{align}
 (1-u)F_x(a,b)
 = f_x(a,b)+\delta_{ax}h_u(x,b)+ \delta_{bx}h_u(a,x), 
\label{compact:defect-kernel-expansion}
\end{align}
where we denote $h_u(a,x):=(\mathsf K_u\pi_u)(a,x)-\mathsf K_u(a,x)-(1-u)\partial_u\mathsf K_u(a,x)$, and $f_x(a,b)$ contains all terms without $\delta$ factors:
\begin{align*}
 f_x(a,b):=&~ \pi_u(a,x)\mathsf K_u(x,b) +\mathsf K_u(a,x)\pi_u(x,b)-\pi_u(a,x)\mathsf K_u(a,b)
     -\mathsf K_u(a,b)\pi_u(b,x) + \mathsf K_u(a,x)\mathsf K_u(x,b)\\
&~-(\pi_u\mathsf K_u)(a,x)\mathsf K_u(x,b) -\mathsf K_u(a,x)(\mathsf K_u\pi_u)(x,b)+\sum_{a'}\mathsf K_u(a,a')\pi_u(a',x)\mathsf K_u(a',b)\\
&~ +(1-u)\qa{ \partial_u\mathsf K_u(a,x) \cdot \mathsf K_u(x,b)+\mathsf K_u(a,x)\cdot \partial_u\mathsf K_u(x,b)}.
\end{align*}

Using \eqref{eq:simple_Pt}, \eqref{compact:kernel-bounds}, and \eqref{eq:moments-iota}, we obtain that 
\begin{align}  
& |(\pi_u\mathsf K_u)(a,x)|\le C R_u^{-2}e^{-|a-x|/R_u}\cdot \eta_u \sum_{a'}[1+\iota_u(a,a')]e^{-\frac12c_{\ref{prop:tools-square-kernel}}\rho_u(a,a')}\le C R_u^{-2} e^{-|a-x|/R_u}.\label{eq:bound-piK}\end{align}
With a similar argument, we also have
\begin{align}
& \sum_{a'}\mathsf K_u(a,a')\pi_u(a',x)\mathsf K_u(a',b)
 \le CR_u^{-4}\mathbf1_{|a-b|\le2R_u}e^{-c|a-x|/R_u}.\label{eq:bound-KpiK}
\end{align}
Applying these bounds together with \eqref{eq:simple_Pt} and \eqref{compact:kernel-bounds} to the terms in $f_x$ and using $1-u\asymp \ell_u^{-2}$ gives
\begin{equation}
 \frac{|f_x(a,b)|}{1-u}
 \le CR_u^{-2}(1+\iota_u(a,x))(1+\iota_u(b,x))
                    \mathfrak P_u(x;a,b).
 \label{compact:defect-smooth-envelope}
\end{equation}
In deriving this bound, we used that the exponential factors in the applied estimates supply the triangle profile $\mathfrak P_u$. 
Applying \eqref{impl:root-transfer} then gives
\begin{equation}
 \sum_x\absbb{\sum_{a,b}g_r^x(a)\frac{f_x(a,b)}{1-u}g_s^x(b)}
 \le C(\log W)^8(1+\iota_t(r,s))\Omega_t^\sharp(r,s).
 \label{compact:defect-smooth-transfer}
\end{equation}
For the terms $\delta_{ax}h_u(x,b)+ \delta_{bx}h_u(a,x)$, the kernel bounds \eqref{eq:bound-piK} and \eqref{compact:kernel-bounds} give $|h_u(a,b)|\le CR_u^{-2}e^{-c|a-b|/R_u}$. With the identity $g_r^x(x)=\delta_{rx}$, we get
\begin{align*}  
 \frac1{1-u}\sum_x\absbb{
    g_r^x(x)\sum_bh_u(x,b)g_s^x(b)}
 &=\frac1{1-u}\absbb{\sum_bh_u(r,b)g_s^r(b)}\le \frac{C(\log W)^2}{(1-u)R_u^2} \xi_s(r)\Omega_t^\sharp(r,s)\\
&\le C\xi_s(r)\Omega_t^\sharp(r,s)\le C(\log W)(1+\iota_t(r,s))\Omega_t^\sharp(r,s).
\end{align*}
Here, the second step follows by applying \eqref{eq:near-endpoint-X} on dyadic shells and using the exponential decay of \(h_u\), as in \eqref{eq:sum-i>i*}.
We also used $(1-u)R_u^2\asymp(\log W)^2$ and $\xi_s(r)\le C(\log W)(1+\iota_t(r,s))$.
The other term $\delta_{bx}h_u(a,x)$ is treated similarly. Combining these estimates with \eqref{compact:defect-smooth-transfer} and $\mathsf D_u\circ f=\sum_xf(x)F_x$ proves
\eqref{compact:defect-transfer-bound}.

Finally, applying \eqref{compact:defect-transfer-bound} with $f=\mathsf f(u)$, using
$\|\mathsf f(u)\|_\infty\le C\mathfrak W(W^2\eta_u)^{-1}$ from
\eqref{def:Z-Y-f} and $J_1(u)\le1$, we bound the LHS of
\eqref{compact:defect-budget} by
\[
 W^4\,
 C(\log W)^8\frac{(\log W)^{6}\logpara^{3/2}}{W^4}
 \int_0^t\frac{\dd u}{\eta_u}
 \le C_D(\log W)^{14}\logpara^{5/2},
\]
where we used $\int_0^t\dd u/\eta_u\le C\logpara$.
This proves \eqref{compact:defect-budget}.
\end{proof}

\subsection{The quadratic cut and light-weight drifts}\label{sec:quadratic-cut-section}

We next estimate the transport of the drift term $\Qsym_u\circ\pa{\mathcal E(u)+\mathcal W(u)}$ in \eqref{eq:two-side-Z-Duhamel}.
We decompose $\mathcal E(u)+\mathcal W(u)$ as
\begin{equation}
\pa{\mathcal E(u)+\mathcal W(u)}(a,b)=\sum_xF_x(a,b),\quad F_x:=\frac12\sum_{y}(\wt F_{yx}+\wt F_{xy}),\label{eq:decompose-Ru}
\end{equation} 
where, suppressing the time $u$, we define
\begin{equation}
\begin{aligned}
 \wt F_{yx}(a,b):=W^2S^{(\sB)}_{yx}\bigl[
 \mathcal D_{ay}\mathcal D_{xb} +(\mathcal D_1)_y\mathcal L^{(+,+,-)}_{(a,x,b)}
 +\overline{(\mathcal D_1)_y}\mathcal L^{(+,-,-)}_{(a,b,x)}\bigr]
.
\end{aligned}
\label{impl:natural-allocation}
\end{equation}
The symmetrization ensures that $F_x$ has equal row and column marginals, denoted by $f_x$; see \Cref{lem:raw-source-marginals} below for the proof. 
Thus \eqref{compact:projected-source} gives
\begin{equation}
 \sum_x[U(\mathsf Q_u\circ F_x)U](r,s)
 =\sum_{x,a,b}g_r^x(a)F_x(a,b)g_s^x(b) -\sum_{x,a}f_x(a)\lambda_x(r,s;a) .
\label{eq:cut-and-light-response-split}
\end{equation}
We first estimate the contribution of the quadratic cut term. 

\begin{lemma}[Transfer of the quadratic cut]\label{compact:quadratic-cut-transfer}
On $\{u\le \tau_*\}$, we have
\begin{equation}
 \sum_x\absbb{\sum_{a,b}g_r^x(a)\mathcal E_{u,x}^{\rm cut}(a,b)g_s^x(b)}
 \le C(\log W)^{10}W^2\mathfrak w^2\ell_u^2 \Pi_t(r,s),
 \label{eq:quadratic-cut-transfer}
\end{equation}
where the quadratic cut contribution is defined by
\begin{equation}
\begin{aligned}
 \mathcal E_{u,x}^{\rm cut}(a,b)
 :=\frac{1}{2}W^2\sum_y S^{(\sB)}_{yx}\bigl[&
 \mathcal D_{ay}(u)\mathcal D_{xb}(u) +\mathcal D_{ax}(u)\mathcal D_{yb}(u)\bigr].
\end{aligned}
\label{eq:paired-quadratic-cut}
\end{equation}
Consequently, the stopped integral of the quadratic cut contribution satisfies
\begin{equation}
\begin{aligned}
 &\frac{W^4}{\Pi_t(r,s)}
 \int_0^t\mathbf1_{\{u\le \tau_*\}}
 \sum_x\absbb{\sum_{a,b}g_r^x(a)\mathcal E_{u,x}^{\rm cut}(a,b)g_s^x(b)}\dd u \le C(\log W)^{10}\frac{\logpara^{2\fb+1}}{W^2}.
\end{aligned}
\label{eq:log-cut-budget}
\end{equation}
\end{lemma}
\begin{proof}
Since $J_2(u)\le1$ on $\{u\le\tau_*\}$ and $S^{(\sB)}_{yx}=0$ unless $|y-x|\le1$, \eqref{eq:low-triangle-root-mass} gives
\[
 |\mathcal E_{u,x}^{\rm cut}(a,b)|
 \le CW^2\mathfrak w^2\mathfrak P_u(x;a,b)
       (1+\iota_u(a,x))(1+\iota_u(b,x)).
\]
Applying the $\ell^1$ bound in \eqref{impl:root-transfer} yields \eqref{eq:quadratic-cut-transfer}. The integral bound \eqref{eq:log-cut-budget} follows upon substituting
$\mathfrak w=\logpara^{\fb}/W^4$ from \eqref{eq:principal-envelopes} and using $\int_0^t\ell_u^2\dd u\le C\logpara$. \end{proof}

Before estimating the marginal correction in \eqref{eq:cut-and-light-response-split}, we compute the row and column marginals of $F_x$ defined in \eqref{eq:decompose-Ru} and verify that they coincide.
\begin{lemma}
  \label{lem:raw-source-marginals}
The matrices $F_x$ have equal row and column marginals,
$F_x\mathbf1=F_x^{\mathsf T}\mathbf1=f_x$, given by
\begin{equation}
 f_x(a)=\sum_y\frac{S^{(\sB)}_{yx}}{2\eta_u}\ha{
 \im\qa{(\mathcal D_1)_y\mathcal L^+_{ax}}  -\qa{\im (\mathcal D_1)_y} \mathcal K_{ax} + (x\leftrightarrow y)}.
\label{eq:raw-source-marginal}
\end{equation}
Here $(x\leftrightarrow y)$ denotes the two preceding terms inside the brackets with $x$ and $y$ interchanged.
Then, on $\{u\le\tau_*\}$, we have 
\begin{align}
 |f_x(a)|\le C\mathfrak W W^{-2}\eta_u^{-1}\Pi_u(a,x).
 \label{eq:allocated-marginal}
\end{align}
\end{lemma}
\begin{proof}
Recall the 3-loop Ward identities \eqref{eq:raw-third-loop-marginals}. Moreover, $\mathsf f(u)$ defined in \eqref{eq:mathsf-f} satisfies
\[
 \mathsf f_y(u)=\frac1{2\ii W^2\eta_u}\pb{(\mathcal D_1)_y-\overline{(\mathcal D_1)_y}} .
\]
These identities, together with the symmetry of the $(+,+)$ and $(-,-)$ 2-loops, yield
\begin{equation*}
\begin{aligned}
 (\wt F_{yx}\mathbf1)(a)
 ={}&\frac{S^{(\sB)}_{yx}}{2\ii\eta_u}
 \bigl[(\mathcal D_1)_y\mathcal L^+_{ax}
      -\overline{(\mathcal D_1)_y}\overline{\mathcal L^+_{ax}}\bigr] +W^2S^{(\sB)}_{yx}\bigl[
 \mathsf f_x(u)\mathcal D_{ay}
 -\mathsf f_y(u)\mathcal L_{ax}\bigr],\\
  (\wt F_{yx}^\top \mathbf1)(a)
 ={}&\frac{S^{(\sB)}_{yx}}{2\ii\eta_u}
 \bigl[(\mathcal D_1)_y\mathcal L^+_{ax}
      -\overline{(\mathcal D_1)_y}\overline{\mathcal L^+_{ax}}\bigr] +W^2S^{(\sB)}_{yx}\bigl[
 \mathsf f_y(u)\mathcal D_{xa}
 -\mathsf f_y(u)\mathcal L_{xa}\bigr].
\end{aligned}
\end{equation*}
Subtracting these expressions and using
$\mathcal L=\mathcal D+\mathcal K$
and the symmetry of $\mathcal K$, we obtain
\[
 (\wt F_{yx}\mathbf1)(a)-(\wt F_{yx}^\top \mathbf1)(a) = W^2S^{(\sB)}_{yx}
 \bigl[\mathsf f_x(u)\mathcal D_{ay}
      -\mathsf f_y(u)\mathcal D_{ax}\bigr].
\]
The RHS is antisymmetric under $x\leftrightarrow y$ and therefore cancels upon symmetrization. Thus $F_x$ has equal row and column marginals. Averaging the row-marginal formulas for \smash{$\wt F_{yx}$ and $\wt F_{xy}$} and summing over $y$ gives \eqref{eq:raw-source-marginal}. 
To show the bound \eqref{eq:allocated-marginal}, we decompose $f_x=f_x^{\rm det}+f_x^{\rm err}$ in \eqref{eq:raw-source-marginal} by substituting \smash{$\mathcal L^+=\mathcal K^++\mathcal D^+$} and the analogous identity for $--$, with $f_x^{\rm det}$ and $f_x^{\rm err}$ collecting the $\mathcal K^\star$ and $\mathcal D^\star$ terms, respectively. Then, the two parts are bounded readily using  
the two-loop bound $|\mathcal D^\star_{ab}(u)|\le\mathfrak w\Pi_u(a,b)$, the light-weight bound $|(\mathcal D_1)_a|\le\mathfrak W$, the explicit formula \eqref{eq:tools-theta}, and the propagator bound \eqref{eq:tools-P-point}.
\end{proof}

Combining \Cref{lem:raw-source-marginals} with \eqref{eq:marginal-scalar-response} immediately yields the following estimate for the second term in \eqref{eq:cut-and-light-response-split}. We omit the straightforward proof.

\begin{lemma}[Transfer of the marginal correction]\label{compact:joint-marginal-transfer}
For $f_x$ in \Cref{lem:raw-source-marginals}, on $\h{u\le \tau_*}$,
\begin{align*}
 \sum_x\absbb{\sum_af_x(a)\lambda_x(r,s;a)}
 &\le C(\log W)^8\mathfrak W W^{-2}\eta_u^{-1}\Pi_t(r,s).
\end{align*}
Consequently, we have 
\begin{equation}
\begin{aligned}
 &\frac{W^4}{\Pi_t(r,s)}
 \int_0^t\mathbf1_{\{u\le \tau_*\}}\sum_{x}\absbb{\sum_a f_x(a)\lambda_x(r,s;a)} \dd u \le
 C(\log W)^{14}\logpara^{5/2}.
\end{aligned}
 \label{eq:joint-marginal-budgets}
\end{equation}
\end{lemma}

We finally estimate the light-weight contribution $\sum_{x,a,b}g_r^x(a)(F_x-\mathcal E_{u,x}^{\rm cut})(a,b)g_s^x(b)$. By symmetry, it suffices to treat the second term on the RHS of \eqref{impl:natural-allocation}, together with its counterpart under $x\leftrightarrow y$. More precisely, the expression to estimate is
\begin{equation}
\frac12\sum_{x,y,a,b}W^2S^{(\sB)}_{yx} g_r^x(a)\qa{(\mathcal D_1)_y\mathcal L^{(+,+,-)}_{(a,x,b)}+ (\mathcal D_1)_x\mathcal L^{(+,+,-)}_{(a,y,b)}}g_s^x(b) .
\label{eq:LW-target}
\end{equation}
We write $\mathcal L^{(+,+,-)}_{(a,y,b)}=\tr(X_{ay}X_{yb}X_{ba}^{-})$. Since $(X_{ab}^{\sigma})^*=X_{ba}^{-\sigma}$,
\[
 |\mathcal L^{(+,+,-)}_{(a,y,b)}|
 =|\tr(X_{by}^{-}X_{ya}^{-}X_{ab})|.
\]
Thus the charge-uniform triangle estimates below also control this loop. We call $X^\sigma_{\alpha\beta}$ a short edge if $|\alpha-\beta|\le6C_\sharp R_u-1$, and a long edge otherwise. We first consider the simpler case in which one of the pairwise distances among $x,a,b$ is at least $12C_\sharp R_u$, the radius of the ball $B_x$ introduced in \Cref{end:master}.

\begin{lemma} \label{lem:static-far-tail}
Suppose $|x-a|\vee |x-b|\vee |a-b|\ge 12C_\sharp R_u$. Then on $\{u\le \tau_*\}$, 
\begin{equation}
 \max_{y:|y-x|\le 1}\max_{\bsig\in \{+,-\}^3}|\mathcal L^{(\bsig)}_{(y,a,b)}|
 \le C\mathfrak w^{3/2}\mathfrak P_u(x;a,b). 
 \label{eq:static-far-tail-bounds}
\end{equation}
\end{lemma}
\begin{proof}
By the triangle inequality, at least two $X$-edges are long. For each long edge, \eqref{eq:entropy-far-S2} gives $\|X_{\al\beta}\|_{S_2}\le C\sqrt{\mathfrak w}\psi_u(\al,\beta)$. If the remaining edge \smash{$X_{\al\beta}^\sigma$} is short, then $\psi_u(\alpha,\beta)\ge c>0$, while \eqref{eq:entropy-operator-output} gives $\|X_{\al\beta}\|_{\rm op}\le C(\log W)^{5}\logpara W^{-2}\le C\sqrt{\mathfrak w}\psi_u(\al,\beta)$.
If exactly two edges are long, we label them $X_1$ and $X_2$ and apply H\"older's inequality $|\tr(X_1X_2X_3)|\le\|X_1\|_{S_2}\|X_2\|_{S_2}\|X_3\|_{\rm op}$. If all three edges are long, we instead use $|\tr(X_1X_2X_3)|\le\|X_1\|_{S_2}\|X_2\|_{S_2}\|X_3\|_{S_2}$. Together with $\mathfrak P_u(y;a,b)\le C\mathfrak P_u(x;a,b)$, these imply \eqref{eq:static-far-tail-bounds}.
\end{proof}

In the complementary regime, we denote 
\[
I_x(r,s)=\max_{y:|y-x|\le 1}\max_{\bsig\in\{+,-\}^3}\sum_{a,b\in B_x: |a-b|\le 12C_\sharp R_u}|g_r^x(a)|\abs{\Tr(X^{\sigma_1}_{by}X^{\sigma_2}_{ya}X^{\sigma_3}_{ab})}|g_s^x(b)|.
\]

\begin{lemma}
  \label{lem:insertion-response}
For $0\le u\le t\le t_*$, on $\{u\le \tau_*\}$, we have 
\begin{align}
 \sum_x|I_x(r,s)|
 &\le C(\log W)^{10}\logpara^{1/2}{\ell_u^2}W^{-4}\Pi_t(r,s).
 \label{eq:near-envelopes}
\end{align}
\end{lemma}

\begin{proof}
Since the following proof is uniform in the charges of the \(X\)-matrices, we suppress the charge labels throughout and, for simplicity, present the argument with all charges equal to \(+\). For an \(X^-\)-factor, the identity \(\|X^-_{\alpha\beta}\|_{S_p}=\|X_{\beta\alpha}\|_{S_p}\) requires reversing the corresponding indices. To accommodate arbitrary charges explicitly, one may replace each factor \(\|X_{\alpha\beta}\|_{S_p}\) by \(\|X_{\alpha\beta}\|_{S_p}+\|X_{\beta\alpha}\|_{S_p}\). All estimates remain valid after enlarging the constants. For clarity of presentation, we leave this modification implicit. 

Before proceeding with the proof, we establish three technical estimates. First, combining \eqref{eq:entropy-operator-output} with the bound $\Omega_t^\sharp(r,x)\ge c>0$ for $r\in B_x$, we obtain 
\begin{equation}
\left\|{\mathbf1_{\{r\in B_x\}}\max_{|y-x|\le1}\|X_{yr}\|_{S_4}}\big /{\Omega_t^\sharp(r,x)} 
 \right\|_{\ell_x^4}
 \le C(\log W)^{5/2}\logpara^{1/4}\ell_u^{1/2}W^{-3/2},
 \label{eq:near-direct-endpoint-fourth}
\end{equation}
where $\ell_x^4$ denotes the $\ell^4$-norm with respect to $x$.
Second, decomposing $g_r^x=\delta_r+(g_r^x-\delta_r)$ and applying H\"older's inequality with exponents $4/3$ and $4$, we obtain
 \begin{align}
 {\displaystyle\max_{|y-x|\le1}
          \sum_{a\in B_x}|g_r^x(a)|\|X_{ya}\|_{S_4}}
        &\le
 {\mathbf1_{\{r\in B_x\}}\displaystyle\max_{|y-x|\le1}\|X_{yr}\|_{S_4}} +\|(g_r^x-\delta_r)\mathbf1_{B_x}\|_{\ell^{4/3}}
 \max_{|y-x|\le1}\pB{\sum_{a\in B_x}\|X_{ya}\|_{S_4}^4}^{1/4} \notag\\
 &\le {\mathbf1_{\{r\in B_x\}}\displaystyle\max_{|y-x|\le1}\|X_{yr}\|_{S_4}} +C(\log W)^{4}\logpara^{1/4}
 \xi_r(x)W^{-3/2}\Omega_t^\sharp(r,x), \label{eq:near-endpoint-Y}
 \end{align}
 where the second inequality follows from \eqref{eq:entropy-operator-output} and \eqref{eq:near-endpoint-X1} with $p=4/3$. 
Third, for each fixed $b\in B_x$,
\begin{align}
\sum_{a\in B_x}|g_r^x(a)-\delta_{ra}|\|X_{ya}\|_{S_4}\|X_{ab}\|_{S_2} &\le CW^{-1}\sum_{a\in B_x}|g_r^x(a)-\delta_{ra}|\|X_{ya}\|_{S_4}\sqrt{1+\iota_u(a,b)}\nonumber\\
&\le CW^{-1}\|(g_r^x-\delta_r)\mathbf1_{B_x}\|_2
 \pB{\sum_a\|X_{ya}\|_{S_4}^4}^{1/4}
 \pB{\sum_{a\in B_x}(1+\iota_u(a,b))^2}^{1/4}\notag\\
& \le C(\log W)^{4}\logpara^{1/4}W^{-5/2}\xi_r(x)\Omega_t^\sharp(r,x) ,\label{eq:near-endpoint-Y2}
\end{align}
where the first two inequalities follow from \eqref{eq:Gram-sharp-bounds} and H\"older's inequality with exponents $(2,4,4)$, and the third follows from \eqref{eq:near-endpoint-X1} with $p=2$, \eqref{eq:entropy-operator-output}, and \eqref{eq:moments-iota}.

Decompose $g_r^x(a)=\delta_{ra}+(g_r^x(a)-\delta_{ra})$ and $g_s^x(b)=\delta_{sb}+(g_s^x(b)-\delta_{sb})$, and split $I_x(r,s)$ into four terms. Then, applying the Schatten H\"older inequality \eqref{eq:block-Schatten-product} to $\abs{\Tr(X_{by}X_{ya}X_{ab})}$ with exponents $(4,4,2)$, we obtain 
{\medmuskip=1mu\thickmuskip=3mu\begin{align*}
  I_x(r,s)&\le \max_{y:|y-x|\le 1} \sum_{a,b\in B_x}\qa{\|X_{sy}\|_{S_4}\|X_{yr}\|_{S_4} \|X_{rs}\|_{S_2}\delta_{ra}\delta_{bs}+|g_r^x(a)-\delta_{ra}| \|X_{by}\|_{S_4}\|X_{ya}\|_{S_4} \|X_{ab}\|_{S_2} |g_s^x(b)-\delta_{sb}|}\\
 &+ \max_{y:|y-x|\le 1} \sum_{a,b\in B_x}\qa{\delta_{ra}\|X_{by}\|_{S_4}\|X_{yr}\|_{S_4} \|X_{rb}\|_{S_2} |g_s^x(b)-\delta_{sb}|+\delta_{bs}|g_r^x(a)-\delta_{ra}| \|X_{sy}\|_{S_4}\|X_{ya}\|_{S_4} \|X_{as}\|_{S_2}}.
\end{align*}}
Using \eqref{eq:Gram-sharp-bounds}, \eqref{eq:near-endpoint-Y}, \eqref{eq:near-endpoint-Y2}, and $\Omega_t^\sharp(r,x)\Omega_t^\sharp(s,x)\le\Omega_t^\sharp(r,s)$, we obtain
\begin{equation}
 |I_x(r,s)|
 \le C(\log W)^{2}W^{-1}\Omega_t^\sharp(r,s)
       \mathscr E_{\sqrt{1+\iota_u(r,s)}}(x).
 \label{eq:insertion-direct-endpoint-bound}
\end{equation}
Here, for any parameter $\mathfrak{r}\ge 1$, we define
\begin{align}
\mathscr E_{\mathfrak{r}}(x)&=
\mathfrak{r} \frac{\mathbf1_{\{r,s\in B_x\}}\bigl(\max_{|y-x|\le1}\|X_{yr}\|_{S_4}\bigr)\bigl(\max_{|y-x|\le1}\|X_{ys}\|_{S_4}\bigr)}{\Omega_t^\sharp(r,x)\Omega_t^\sharp(s,x)}  +  (\log W)^{6}\logpara^{1/2}W^{-3}\xi_r(x)\xi_s(x)
\label{merge:root-endpoint-majorant}\\
 &+ (\log W)^{3}\logpara^{1/4}W^{-3/2}
 \biggl[
 \frac{\mathbf1_{\{r\in B_x\}}\max_{|y-x|\le1}\|X_{yr}\|_{S_4}}
      {\Omega_t^\sharp(r,x)}\xi_s(x) +
 \frac{\mathbf1_{\{s\in B_x\}}\max_{|y-x|\le1}\|X_{sy}\|_{S_4}}
      {\Omega_t^\sharp(s,x)}\xi_r(x)\biggr].\notag
 \end{align}
More precisely, \eqref{eq:Gram-sharp-bounds} bounds the double $\delta$ term by $CW^{-1}\Omega_t^\sharp(r,s)$ times the first term in \eqref{merge:root-endpoint-majorant}. By \eqref{eq:near-endpoint-Y2}, the single $\delta$ terms are bounded by \smash{$C(\log W)W^{-1}\Omega_t^\sharp(r,s)$} times the two terms in the square brackets, including their prefactor, in \eqref{merge:root-endpoint-majorant}. For the remaining term, we apply \eqref{eq:near-endpoint-Y2} to the $a$-sum and \eqref{eq:near-endpoint-Y} to the $b$-sum, obtaining
\begin{align*}
&~\max_{|y-x|\le 1}\sum_{a,b\in B_x}|g_r^x(a)-\delta_{ra}| \|X_{by}\|_{S_4}\|X_{ya}\|_{S_4} \|X_{ab}\|_{S_2} |g_s^x(b)-\delta_{sb}|\\
\le&~ C(\log W)^{4}\logpara^{1/4}W^{-5/2}\xi_r(x)\Omega_t^\sharp(r,x)\max_{|y-x|\le 1}\sum_{b\in B_x} |g_s^x(b)-\delta_{sb}|\|X_{by}\|_{S_4} \\
\le&~ C(\log W)^{8}\logpara^{1/2}W^{-4}\xi_r(x)\xi_s(x)\Omega_t^\sharp(r,x)\Omega_t^\sharp(x,s). 
\end{align*}
This is bounded by $C(\log W)^2W^{-1}\Omega_t^\sharp(r,s)$ times the second term in \eqref{merge:root-endpoint-majorant}.

Next, we claim that, for $p\in\{1,2\}$,
\begin{equation}
 \|\mathscr E_{\mathfrak{r}}\|_p
 \le C(\log W)^{6+2/p}\logpara^{1/2}\ell_u^{2/p}W^{-3}\times 
 \begin{cases}
 \Psi^{(1)}_{u,t}(r,s)+{\mathfrak{r}},&p=1,\\
 (1+{\mathfrak{r}})\Psi^{(2)}_{u,t}(r,s),&p=2.
 \end{cases}
 \label{merge:root-endpoint-norms}
\end{equation}
Only the case $p=1$ is needed for the present proof of \Cref{lem:insertion-response}; the case $p=2$ will be used later in the martingale estimate. To prove \eqref{merge:root-endpoint-norms}, first observe that the second term in \eqref{merge:root-endpoint-majorant} is controlled by \eqref{eq:compact-L2-product}. For the mixed term containing $\xi_s(x)$, applying H\"older's inequality with exponents $4/p$ and $4/(4-p)$ and using \eqref{eq:near-direct-endpoint-fourth}, we obtain
\begin{align}\label{eq:bdd-mixed-lp}
\norma{\frac{\mathbf1_{\{r\in B_x\}}\max_{|y-x|\le1}\|X_{yr}\|_{S_4}} {\Omega_t^\sharp(r,x)}\xi_s(x)}_{\ell_x^p} \le C(\log W)^{5/2}\logpara^{1/4}\ell_u^{1/2}W^{-3/2} \|\mathbf1_{B(r,12C_\sharp R_u)}\xi_s\|_{\ell^{4p/(4-p)}}.
\end{align}
For every fixed $1\le q<\infty$, we have 
\begin{align}\label{eq:bdd-mixed-lq}
 \|\mathbf1_{B(r,12C_\sharp R_u)}\xi_s\|_q
 \le CR_u^{2/q}(1+|r-s|/R_u)^{-1}.
\end{align}
Indeed, if $|r-s|\le26C_\sharp R_u$, the ball $B(r,12C_\sharp R_u)$ is contained in $B(s,38C_\sharp R_u)$. The bound
\begin{equation}
 \sum_{s} \left[\log_+\frac{R}{1+|x-s|}\right]^q
 \le C_qR^2, \qquad \forall R\ge 1, \label{eq:endpoint-log-moments}
\end{equation}
proved as in \eqref{eq:moments-iota}, then gives
\eqref{eq:bdd-mixed-lq}. If $|r-s|>26C_\sharp R_u$, we instead use $\xi_s(x)\le CR_u/|r-s|$ on $B(r,12C_\sharp R_u)$. 
Substituting \eqref{eq:bdd-mixed-lq} with $q=4p/(4-p)$ into \eqref{eq:bdd-mixed-lp}, multiplying by the prefactor $(\log W)^{3}\logpara^{1/4}W^{-3/2}$, and using $R_u=(\log W)\ell_u$, we bound this mixed term by the RHS of \eqref{merge:root-endpoint-norms} for both $p=1$ and $p=2$. The other mixed term is treated symmetrically. 
Finally, applying \eqref{eq:near-direct-endpoint-fourth} to
both factors in the first term of
\eqref{merge:root-endpoint-majorant} and using the support
size $CR_u^2$, we obtain
\begin{align*}
 &\left\|\frac{\mathbf1_{\{r,s\in B_x\}}\bigl(\max_{|y-x|\le1}\|X_{yr}\|_{S_4}\bigr)\bigl(\max_{|y-x|\le1}\|X_{ys}\|_{S_4}\bigr)}{\Omega_t^\sharp(r,x)\Omega_t^\sharp(s,x)}\right\|_{\ell_x^p} \le C(\log W)^{5}\logpara^{1/2}R_u^{2/p-1}\ell_uW^{-3}\1_{|r-s|\le24C_\sharp R_u}.
\end{align*}
Multiplying by $\mathfrak r$ and using $R_u=(\log W)\ell_u$
gives the required bound: for $p=1$, the indicator is at most $1$, while for $p=2$ it is bounded by $C\Psi^{(2)}_{u,t}(r,s)$. This proves \eqref{merge:root-endpoint-norms}.

Combining \eqref{eq:insertion-direct-endpoint-bound} with
\eqref{merge:root-endpoint-norms} for $p=1$ and
$\mathfrak r=\sqrt{1+\iota_u(r,s)}$, we obtain
\begin{align*}
 \sum_x|I_x(r,s)|
 &\le C(\log W)^{10}\logpara^{1/2}\ell_u^2W^{-4}
       \Omega_t^\sharp(r,s) \bigl[\Psi^{(1)}_{u,t}(r,s)
                    +\sqrt{1+\iota_u(r,s)}\bigr]\\
 &\le C(\log W)^{10}\logpara^{1/2}\ell_u^2W^{-4}
       \Pi_t(r,s),
\end{align*}
where the second inequality follows from
$\Psi^{(1)}_{u,t}(r,s)\le C(1+\iota_t(r,s))$.
This proves \eqref{eq:near-envelopes}.
\end{proof}

Combining \Cref{lem:static-far-tail} with \eqref{impl:root-transfer} in the far regime and applying \Cref{lem:insertion-response} in the complementary regime, together with the light-weight bound $\max_a|(\cD_1)_a(u)|\le\mathfrak W$ on $\{u\le\tau_*\}$, yields the following estimates. We omit the straightforward proof.

\begin{lemma}[Transfer of the light-weight contribution]\label{lem:fine-source}
For $0\le u\le t\le t_*$, on $\{u\le\tau_*\}$, we have
\begin{equation}
 \sum_x\absbb{\sum_{a,b}g_r^x(a)(F_x-\mathcal E_{u,x}^{\rm cut})(a,b)g_s^x(b)} \le C\mathfrak W\ell_u^2\Pi_t(r,s)\bigl[
 (\log W)^{10}\logpara^{1/2}W^{-2}
 +(\log W)^{10}W^2\mathfrak w^{3/2}\bigr].
 \label{eq:LW-transfer}
\end{equation}
Consequently, under the parameter choices in \eqref{eq:bandwidth-margin} and \eqref{eq:principal-envelopes},
\begin{equation}
 \frac{W^4}{\Pi_t(r,s)}
 \int_0^t\mathbf1_{\{u\le\tau_*\}}
 \sum_x\absbb{\sum_{a,b}g_r^x(a)(F_x-\mathcal E_{u,x}^{\rm cut})(a,b)g_s^x(b)}\dd u \le C (\log W)^{16}\logpara^3.
\label{eq:LW-budget}
\end{equation}
\end{lemma}

\subsection{Martingale estimates}\label{new:physical-section}

Recall the independent real Brownian representation \eqref{eq:MG-Brownian-physical}--\eqref{eq:h_equal_marginal}.
In the following proof, we regard $\mathbf h_\mu(a,b)\equiv \mathbf h_{u,\mu}(a,b):=(h_{u,\phy}(a,b):\phy\in\cI_\mu)\in \C^{\cI_\mu}$ as an element of the Hilbert space $\ell^2(\cI_\mu)$, equipped with its standard $\ell^2$-norm, and apply the Banach-space version of \Cref{impl:transfer}. We retain the common marginal vectors $\mathbf f_{a,\mu}$ defined in \eqref{eq:mart-marginal}. For consistency with \Cref{impl:transfer}, we henceforth use $x,y$ to denote block indices in $\Zn$, rather than indices in $\ZL$. By \eqref{compact:projected-source}, we write
\begin{equation}
 \begin{gathered}
 [U_{u,t}(\mathsf Q_u\circ\mathbf h_x)U_{u,t}](r,s)
 =\sum_{a,b}g_r^x(a)\mathbf h_x(a,b)g_s^x(b)
 -\sum_a\mathbf f_{a,x}\lambda_x(r,s;a).
 \end{gathered}
 \label{ac:physical-centered-identity}
\end{equation}
We first estimate the marginal correction using
\eqref{eq:marginal-square-response}.

\begin{lemma}[Martingale marginal lift]\label{lem:physical-trace-lift}
For $0\le u\le t\le t_*$, on $\{u\le\tau_*\}$, we have
\begin{equation}
 \sum_x\normbb{\sum_a\mathbf f_{a,x}\lambda_x(r,s;a)}_{\ell^2}^2
 \le \frac{C(\log W)^{14}}{W^8\eta_u}
 \left[\logpara(1+\iota_t(r,s))^2
       +\frac{\logpara^{2\fb}}{W^2}\Psi_{u,t}(r,s)^2\right]
 \Omega_t^\sharp(r,s)^2.
 \label{eq:physical-marginal-lift}
\end{equation}
\end{lemma}
\begin{proof}
We first establish the bounds 
\begin{equation}
  \max_{a,x}\|\mathbf f_{a,x}\|_{\ell^2}^2
 \le\frac{C(\log W)^{10}\logpara^3}{W^8\eta_u^2},\qquad 
 \max_x\sum_a\pa{\|\mathbf f_{a,x}\|_{\ell^2}^2+\|\mathbf f_{x,a}\|_{\ell^2}^2}\le\frac{C(\log W)^{10}\logpara\ell_u^2}{W^8\eta_u^2} .
 \label{eq:trace-noise-both-energies}
\end{equation}
From \eqref{eq:ell2-fax} and Schatten H\"older's inequality, we get
\begin{align}
\|\mathbf f_{a,x}\|_{\ell^2(\cI_x)}^2
&\le CW^{-2}\eta_u^{-2}\sum_{y:|y-x|\le 1} (\|X_{ya}\|_{S_4}^{2}\|X_{ax}\|_{S_4}^{2}+\|X_{xa}\|_{S_4}^{2}\|X_{ay}\|_{S_4}^{2}).\label{eq:ell2-fax-2}
 \end{align}
The $O_4$ bound in \eqref{eq:entropy-operator-output} immediately yields the first estimate in \eqref{eq:trace-noise-both-energies}. For the second, apply Cauchy--Schwarz to the sum over index $x$ or $a$ and use the \smash{$\wt O_4$} bound in \eqref{eq:entropy-operator-output}, together with the fact that there are a bounded number of neighboring blocks satisfying $|y-x|\le1$.

We split the sum over $a$ into $a\in B_x$ and $a\notin B_x$. For the near part, choose $\mathscr D$ and $\mathscr E$ as the respective right-hand sides of \eqref{eq:trace-noise-both-energies}.
Then $\ell_u^2\mathscr D/\mathscr E\le C\logpara^2$, so $\log_+(\ell_u^2\mathscr D/\mathscr E)\le C\log W$. Applying \Cref{eq:techical-lambda-square} gives
\begin{equation}
 \sum_x\normbb{\sum_{a\in B_x}
 \mathbf f_{a,x}\lambda_x(r,s;a)}_{\ell^2}^2
 \le \frac{C(\log W)^{14}\logpara}{W^8\eta_u}
 \Pi_t(r,s)^2.
\label{eq:ell2-Pit-rs}
\end{equation}
For the far part, both $X$-matrices in \eqref{eq:ell2-fax} are far, since $|y-x|\le1$. Using $\|X_1X_2\|_{S_2}\le\|X_1\|_{S_2}\|X_2\|_{S_2}$ and applying \eqref{eq:entropy-far-S2} to both factors, we obtain
\[
 \|\mathbf f_{a,x}\|_{\ell^2} \mathbf1_{\{|a-x|>12C_\sharp R_u\}}
 \le C\mathfrak w W^{-1}\eta_u^{-1}\Omega_u^\sharp(a,x) .
\]
Applying \eqref{eq:marginal-square-response} therefore gives 
\begin{align*}
 &\sum_x\normbb{
 \sum_{|a-x|>12C_\sharp R_u}
 \mathbf f_{a,x}\lambda_x(r,s;a)}_{\ell^2}^2  \le
 \frac{C(\log W)^{14}\mathfrak w^2}{W^2\eta_u^2\ell_u^2}
 \Psi_{u,t}^2\Omega_t^\sharp(r,s)^2=\frac{C(\log W)^{14}\logpara^{2\fb}}{W^{10}\eta_u}
 \Psi_{u,t}^2\Omega_t^\sharp(r,s)^2,
\end{align*}
where we used $\mathfrak w=\logpara^{\fb}W^{-4}$ and
$\ell_u^{-2}=\eta_u$. Combining this estimate with \eqref{eq:ell2-Pit-rs} proves
\eqref{eq:physical-marginal-lift}.\end{proof}

\begin{lemma}\label{eq:techical-lambda-square}
Let $\mathbf g_{a,x}\in\ell^2(\cI_x)$ satisfy
\[
 \sup_x\sum_a\left(\|\mathbf g_{a,x}\|^2+\|\mathbf g_{x,a}\|^2\right)
 \le\mathscr E,\qquad
 \max_{a,x}\|\mathbf g_{a,x}\|^2\le\mathscr D,
\]
where $\mathscr E,\mathscr D>0$ and
$\log_+(\ell_u^2\mathscr D/\mathscr E)\le C\log W$.
Define $\mathcal T_x:=\sum_{a\in B_x}\mathbf g_{a,x}\lambda_x(r,s;a)$.
Then
\begin{equation}\label{compact:refined-physical-lift}
 \sum_x\|\mathcal T_x\|_{\ell^2}^2
 \le C(\log W)^4\ell_u^{-4}\mathscr E\,
                  \Omega_t^\sharp(r,s)^2.
\end{equation}
\end{lemma}

The proof of this lemma is postponed to \Cref{subsec:pf-techical-lambda-square}. We next control the quadratic variation of the first term on the RHS of \eqref{ac:physical-centered-identity}. 

\begin{lemma}[Transfer of the martingale]\label{lem:mg-source}
For $0\le u\le t\le t_*$, on $\{u\le\tau_*\}$, we have
\begin{equation}
\begin{aligned}
 &\sum_x\normbb{\sum_{a,b}g_r^x(a)\mathbf h_x(a,b)g_s^x(b)}_{\ell^2}^2 \le\frac{C(\log W)^{28}}{W^8\eta_u} \left[\logpara^2(1+\iota_t(r,s))
       +\frac{\logpara^{3\fb}}{W^2}\right]
 \Psi_{u,t}^2\Omega_t^\sharp(r,s)^2.
 \label{eq:mg-transfer}
\end{aligned}
\end{equation}
\end{lemma}
\begin{proof}
Using \eqref{eq:hue}, the definition of $V_{\phy}$, and $(X_{ab}^{\sigma})^*=X_{ba}^{-\sigma}$, we obtain
\begin{align}\label{eq:bh-ell2x}
  \|\mathbf h_{x}(a,b)\|_{\ell^2(\cI_x)}^2 &=\sum_{\phy\in\cI_x} |h_{u,\phy}(a,b)|^2 \le C W^2\max_{\bsig\in\{+,-\}^3}\sum_{y \in \Zn}S^{(\sB)}_{xy} \pa{\|X_{ya}^{\sigma_1}X_{ab}^{\sigma_2}X_{bx}^{\sigma_3}\|_{S_2}^2+\p{x\leftrightarrow y}}.
\end{align}  
We estimate the RHS and then apply the Banach-space version of the second bound in \eqref{impl:root-transfer}. 

As in the light-weight estimate, we decompose $\mathbf h_x$ into its far part, supported on $|x-a|\vee|x-b|\vee|a-b|\ge12C_\sharp R_u$, and its near part, supported on the complementary region. We denote these by $\mathbf h_x^{\rm far}$ and $\mathbf h_x^{\rm near}$, respectively. In the far region, at least two edges are long. We bound the Hilbert--Schmidt norm of each product using the $S_2$-norms of the long edges and the operator norm of the short edge, if present. The same block-norm estimates used in the proof of \Cref{lem:static-far-tail}, together with \eqref{eq:bh-ell2x}, yield
\begin{equation*}  
 \|\mathbf h_x^{\rm far}(a,b)\|_{\ell^2(\cI_x)}
 \le CW\mathfrak w^{3/2}\mathfrak P_u(x;a,b). 
\end{equation*}
The Banach-space estimate \eqref{impl:root-transfer} therefore yields
\begin{align}
\sum_x\normbb{\sum_{a,b}
 g_r^x(a)\mathbf h_x^{\rm far}(a,b)g_s^x(b)}_{\ell^2(\cI_x)}^2 
 \le C(\log W)^{18}W^2\mathfrak w^{3}\ell_u^2
 \Psi_{u,t}^2\Omega_t^\sharp(r,s)^2.\label{eq:hfar-transfer}  
\end{align}

For the near part, it suffices to estimate $\|X_{ya}^{\sigma_1}X_{ab}^{\sigma_2}X_{bx}^{\sigma_3}\|_{S_2}$; the $\p{x\leftrightarrow y}$ term is treated identically. Suppressing the charges, we apply \eqref{eq:spatial-operator-output} and Schatten H\"older's inequality with exponents $(4,\infty,4)$ to obtain
\[
 \|X_{ya}X_{ab}X_{bx}\|_{S_2}
 \le C(\log W)^{5}\logpara^{1/2}W^{-2} 
 \|X_{ya}\|_{S_4}\|X_{bx}\|_{S_4}
 \sqrt{1+\iota_u(a,b)}.
\]
Here we used $\vartheta_u\le C\logpara$ and $\iota_u(a,b)+\log W\le C(\log W)(1+\iota_u(a,b))$. We now decompose $g_r^x(a)=\delta_{ra}+(g_r^x(a)-\delta_{ra})$ and $g_s^x(b)=\delta_{sb}+(g_s^x(b)-\delta_{sb})$, as in the proof of \eqref{eq:insertion-direct-endpoint-bound}. The argument is the same, except that the middle-edge factor $W^{-1}\sqrt{1+\iota_u(a,b)}$ in \eqref{eq:near-endpoint-Y2} is replaced by $(\log W)^{5}\logpara^{1/2}W^{-2} \sqrt{1+\iota_u(a,b)}$. Including the factor $W$ from \eqref{eq:bh-ell2x}, we obtain
\begin{align*}
  \pbb{\sum_x\normB{\sum_{a,b}
 g_r^x(a)\mathbf h_x^{\rm near}(a,b)g_s^x(b)}_{\ell^2(\cI_x)}^2}^{1/2}\le C(\log W)^{7}\logpara^{1/2} W^{-1} \Omega_t^\sharp(r,s)
 \|\mathscr E_{\mathfrak r}\|_2,\quad \mathfrak{r}=\sqrt{1+\iota_u(r,s)}.
\end{align*}
Applying \eqref{merge:root-endpoint-norms} with $p=2$, and using $\Psi^{(2)}_{u,t}(r,s)\le C\Psi_{u,t}$ and $\iota_u(r,s)\le\iota_t(r,s)$, we obtain
\begin{align*}
 \sum_x \normbb{\sum_{a,b}
 g_r^x(a)\mathbf h_x^{\rm near}(a,b)g_s^x(b)}_{\ell^2(\cI_x)}^2 \le C(\log W)^{28} \logpara^2 \ell_u^2W^{-8}\Psi_{u,t}^2  
 (1+\iota_t(r,s))\Omega_t^\sharp(r,s)^2.
\end{align*}
Combining this estimate with \eqref{eq:hfar-transfer} and substituting $\mathfrak w=\logpara^{\fb}W^{-4}$ and $\ell_u^{-2}=\eta_u$ proves \eqref{eq:mg-transfer}.
\end{proof}

Combining \Cref{lem:physical-trace-lift,lem:mg-source} yields the following result. We write $\mathbf h_u:=\bigoplus_x\mathbf h_{u,x}$, viewed as an element of the Hilbert space $\mathcal H:=\bigoplus_x\ell^2(\cI_x)$.

\begin{lemma}[Martingale transport]\label{lem:mg-physical}
For $0\le u\le t\le t_*$, on $\{u\le \tau_*\}$, we have 
\begin{equation}
\begin{aligned}
 &\left\|\frac{[\mathcal U_{u,t}\circ(\mathsf Q_u\circ\mathbf h_u)](r,s)}{\Pi_t(r,s)}\right\|_{\mathcal H}^2 \le\frac{C(\log W)^{28}}{W^8\eta_u}\left(
 \logpara
 + \frac{\logpara^2\Psi_{u,t}^2}{1+\iota_t(r,s)}
 +\frac{\logpara^{3\fb}}{W^2}\frac{\Psi_{u,t}^2}{\p{1+\iota_t(r,s)}^2}\right).
\end{aligned}
\label{ac:physical-density-list}
\end{equation}
Consequently, for every fixed $D>0$, there exists a constant $C_D>0$ such that, with probability at least $1-N^{-D}$, the following estimate holds uniformly for $0\le t\le t_*$ and $r,s\in\Zn$:
\begin{equation}
\begin{aligned}
 &\frac{W^4}{\Pi_t(r,s)}
 \absa{\int_{0}^t\mathbf1_{\{u\le\tau_*\}}
 \left(\mathcal U_{u,t}\circ
       (\mathsf Q_{u}\circ\dd\mathcal M(u))\right)_{rs}} \le C_D(\log W)^{15}\pa{\logpara^{3/2}
       +\sqrt{\frac{\logpara^{3\fb+1}}{W^2}}}.
 \label{ac:physical-costs}
\end{aligned}
\end{equation}
\end{lemma}
\begin{proof}
The bound \eqref{ac:physical-density-list} follows immediately from \Cref{lem:physical-trace-lift,lem:mg-source}. We now focus on the proof of \eqref{ac:physical-costs} by combining \eqref{ac:physical-density-list} with \eqref{eq:stochastic-convolution-weighted-maximal} in \Cref{lem:stochastic-convolution-maximal}. To this end, we introduce deterministic diagonal weights and a normalized propagator:
\[
 w_t:=\frac{\eta_t}{\eta_0},\qquad
 \mathsf W_t(r,s):=
 \frac{W^4}{w_t\Pi_t(r,s)},\qquad
 \Phi_{u,t}:=\mathsf W_t\mathcal U_{u,t}\mathsf W_u^{-1}.
\]
Define the martingale $\widehat{\mathcal M}_t$ and the stochastic integral $N_t^{(w)}$, viewed as vectors indexed by $(r,s)\in(\Zn)^2$, by
\[
 \widehat{\mathcal M}_t(r,s)
 :=\int_0^t\mathbf1_{\{u\le\tau_*\}}
 \qa{\mathsf W_u\pa{\mathsf Q_u\circ\dd\mathcal M(u)}}_{rs},
 \quad
 N_t^{(w)}(r,s):=w_t\int_0^t\pa{\Phi_{u,t}\,\dd\widehat{\mathcal M}_u}_{rs}.
\]
Thus $|N_t^{(w)}(r,s)|$ is precisely the left-hand side of \eqref{ac:physical-costs}.

We now verify the two weighted bracket bounds in \eqref{eq:weight-bracket}. Taking $t=u$ in \eqref{ac:physical-density-list} gives the unpropagated-density majorant
\[
 q(u):=\frac{C(\log W)^{28}}{w_u^2\eta_u}\left(\logpara^2+\frac{\logpara^{3\fb}}{W^2}\right) \ \implies\ w_t^2\int_0^tq(u)\dd u\le C(\log W)^{28}\left(\logpara^2+\frac{\logpara^{3\fb}}{W^2}\right).
\]
For the propagated density, the definitions of $\Psi_{u,t}$ and $\iota_t$ imply, uniformly in $r,s$,
\begin{equation} 
 \frac{\Psi_{u,t}^2}{(1+\iota_t(r,s))^2}
 \le\frac{C(\log W)^2}{(1+\log(\ell_t^2/\ell_u^2))^2},\qquad
 \frac{\Psi_{u,t}^2}{1+\iota_t(r,s)}
 \le\frac{C\log W}{1+\log(\ell_t^2/\ell_u^2)}.
\label{eq:Psi-ut-2}
\end{equation}
Indeed, for $c\log(\ell_t^2/\ell_u^2)\le \log W$ these bounds follow from $1+\iota_t(r,s)\ge1$ and $\Psi_{u,t}\le1$. If $1+\iota_t(r,s)\ge c \log(\ell_t^2/\ell_u^2)$, these bounds are also immediate. Otherwise, suppose $c\log(\ell_t^2/\ell_u^2)> \log W$ and $1+\iota_t(r,s)< c \log(\ell_t^2/\ell_u^2)$ for a sufficiently small constant $c<1/4$. Then, we have $|r-s|\ge \ell_t^{1-2c}\ell_u^{2c}-1$, which gives $\Psi_{u,t}^2\le C(\log W)^{1/2}(\ell_t/\ell_u)^{-(1-2c)/2}$. This concludes \eqref{eq:Psi-ut-2}. 
By \eqref{eq:Psi-ut-2}, \eqref{ac:physical-density-list} gives
the propagated-density majorant
\[
 p_t(u):=\frac{C(\log W)^{28}}{w_t^2\eta_u}
 \left[\logpara  +
 \frac{\logpara^2\log W}{1+\log(\ell_t^2/\ell_u^2)}+\frac{\logpara^{3\fb}(\log W)^2}{W^2\bigl(1+\log(\ell_t^2/\ell_u^2)\bigr)^2}
 \right].
\]
Using $\dd u/\eta_u=-(\im m)^{-1}\dd\log(\ell_t^2/\ell_u^2)$,
$\log(\ell_t^2/\ell_0^2)\le C\logpara$, and
$\log(1+\logpara)\le C\log W$, we obtain
\[
 w_t^2\int_0^t p_t(u)\dd u
 \le v_*:=C(\log W)^{30}\left(\logpara^2+\frac{\logpara^{3\fb}}{W^2}\right).
\]

The generator of $\Phi_{u,t}$ is \(
 \cB_t=\mathsf W_t'\mathsf W_t^{-1}
 +\mathsf W_t\mathscr A_{t}\mathsf W_t^{-1},
\) where $\mathscr A_t$ is the abbreviation defined in \eqref{eq:L-minus-K-linear-operator} and generates \smash{$\mathcal U_{u,t}$}. We estimate the two terms separately. Note that both $w_t^{-1}$ and $\Pi_t(a,b)$ are absolutely continuous and nondecreasing. Since
$\mathsf W_t(a,b)=W^4/(w_t\Pi_t(a,b))$, we have almost everywhere
\[
 \left|\partial_t\log\mathsf W_t(a,b)\right|
 \le -{w_t'}/{w_t}+\partial_t\log\Pi_t(a,b).
\]
Consequently, we have 
\begin{align*}
 \int_0^{t_*}\|\mathsf W_t'\mathsf W_t^{-1}\|_{\infty\to\infty}\dd t \le \log\frac{w_0}{w_{t_*}}
 +\sum_{a,b}\log\frac{\Pi_{t_*}(a,b)}{\Pi_0(a,b)} \le C\logpara+CL^4(\logpara+\sqrt L) .
\end{align*}
Here we used $\log(w_0/w_{t_*})\le C\logpara$, and the definitions of $\Pi_t$, $\iota_t$, and $\Omega_t^\sharp$, which give $\log(\Pi_{t_*}(a,b)/\Pi_0(a,b))\le C(\logpara+\sqrt L)$.
For the conjugated evolution generator $\mathsf W_t\mathscr A_{t}\mathsf W_t^{-1}$, the scalar factor $w_t$ cancels in the weight ratios. By \eqref{eq:product_Omega} and $1\le1+\iota_t(a,b)\le C\logpara$, we have 
\[
 \frac{\mathsf W_t(a,b)}{\mathsf W_t(a',b)}
 =\frac{\Pi_t(a',b)}{\Pi_t(a,b)}
 \le C\logpara\,
 e^{\nu_\sharp\sqrt{\rho_t(a,a')}}.
\]
The same estimate holds for shifts of the second index.
With the definition \eqref{eq:L-minus-K-linear-operator}, it then follows that
\begin{align*}
 \|\mathsf W_t\mathscr A_t\mathsf W_t^{-1}\|_{\infty\to\infty}
 &\le C\logpara\sup_a\sum_{a'}
 (S^{(\sB)}\Theta_t)(a,a')
 e^{\nu_\sharp\sqrt{\rho_t(a,a')}} \le C\logpara\,\eta_t^{-1},
\end{align*}
where the second step follows from \eqref{eq:high-exponential-row-moment}.
Thus
\[
 \int_0^{t_*}
 \|\mathsf W_t\mathscr A_t\mathsf W_t^{-1}\|_{\infty\to\infty}\dd t
 \le C\logpara\int_0^{t_*}\frac{\dd t}{\eta_t}
 \le C\logpara^2.
\]
Combining the two estimates proves
\(
 A:=\int_0^{t_*}\|\cB_t\|_{\infty\to\infty}\dd t\le N^5.
\)
Applying \eqref{eq:stochastic-convolution-weighted-maximal} with
$R=C_D\sqrt{v_*\logpara}$ therefore yields
\[
 \mathbb P\left(\sup_{0\le t\le t_*}\|N_t^{(w)}\|_\infty
 >C_D(\log W)^{15}
 \left(\logpara^{3/2}
       +\sqrt{\frac{\logpara^{3\fb+1}}{W^2}}\right)\right)
 \le N^{-D}.
\]
Since the $\ell^\infty$-norm takes the maximum over all $r,s\in\Zn$, this proves \eqref{ac:physical-costs}.
\end{proof}

\subsection{Closure of the two-loop estimate}\label{new:release-section}

We now combine the preceding estimates with the stopped Duhamel formula \eqref{eq:two-side-Z-Duhamel} and \Cref{lem:closed-stopped-continuation} to establish the following two-loop bootstrap estimate.

\begin{proposition}[2-$\cD$-loop bootstrap estimate]
\label{lem:positive-output-2}
For every fixed $D>0$, there exists a constant $C_D>0$ such that, with probability at least $1-N^{-D}$, uniformly for
$0\le t\le\tau_*$ and $r,s\in\Zn$,
\begin{equation}
 \max_{\star\in\{\emptyset,+\}}|\difloop^\star_{rs}(t)|
 \le C_D(\log W)^{16}\logpara^3W^{-4}\Pi_t(r,s).
 \label{eq:L2-output}
\end{equation}
In particular, for sufficiently large $N$,
$\sup_{t\le\tau_*}J_2(t)\le1/2$ on the same event.
\end{proposition}
\begin{proof}
For alternating charges, we use \eqref{eq:D=Z+Y} and the stopped Duhamel formula \eqref{eq:two-side-Z-Duhamel}. The reconstruction bound \eqref{compact:reconstruction}, the drift bounds \eqref{compact:defect-budget}, \eqref{eq:log-cut-budget}, \eqref{eq:joint-marginal-budgets}, and \eqref{eq:LW-budget}, and the martingale bound \eqref{ac:physical-costs} give
\begin{align*}
 \frac{W^4|\difloop_{rs}(t)|}{\Pi_t(r,s)}
 \le C_D\biggl[&
 (\log W)^{4}\logpara^{3/2}
 +(\log W)^{14}\logpara^{5/2}
 +(\log W)^{10}\frac{\logpara^{2\fb+1}}{W^2}\\
 &+(\log W)^{16}\logpara^3
 +(\log W)^{15}
 \pbb{\logpara^{3/2}
       +\sqrt{\frac{\logpara^{3\fb+1}}{W^2}}}\biggr]
\end{align*}
uniformly on $[0,\tau_*]$, outside an event of probability at most $N^{-D}$. By \Cref{lem:closed-stopped-continuation}, the stopped representation agrees with the original process on this interval. 
Since $\fb=3+\epsilon/3$ and $W^2\ge\logpara^{4+\epsilon}$ by \eqref{ac:small-band-parameters} and \eqref{eq:bandwidth-margin}, the RHS is at most $C_D(\log W)^{16}\logpara^3$, proving \eqref{eq:L2-output}. 
The non-alternating case follows from the corresponding drift and martingale estimates with the stable kernels in \Cref{lem:tools-stable-transport}. 
\end{proof}

\subsection{Proof of \Cref{thm:local-law}}\label{sec:pf-main}
Choose the bandwidth margin $\epsilon$ and the bootstrap exponent $\fb$ as in \eqref{ac:small-band-parameters}.
For a fixed $z=E+\ii\eta\in\bS(\kappa)$, we again choose $t_0=t_0(z)$ and $\sE=\sE(z)$ as in \eqref{eq:t0E0}. We then apply \Cref{prop:small-flow-release} and \Cref{lem:positive-output-2} at an increased probability depth $D'>D$. They improve all three barriers in \eqref{eq:principal-stop}, so continuity gives $\tau_*=t_*$ outside an event of probability at most $N^{-D'}$. In particular, \eqref{eq:L2-output} holds at $t_0$.
Since $\mathcal D_{ab}(t_0)=\mathcal L_{ab}(t_0)-W^{-2}(\Theta_{t_0})_{ab}$, the identities in \eqref{eq:scaling} give
\[
 W^{-2}\tr|P_aG(z)P_b|^2
 -|m(z)|^2\pb{1-|m(z)|^2S^{(\sB)}}^{-1}_{ab}
 \stackrel{\rm d}=W^2t_0\mathcal D_{ab}(t_0).
\]
By \eqref{eq:tools-profile-nu}, $\eta_{t_0}\asymp_\kappa\eta$, and $\ell_{t_0}\asymp \eta^{-1/2}$, there is a structural constant $c_0>0$, independent of $D$ and $\e$, such that
$\Pi_{t_0}(a,b)\le C_\kappa\Omega_{c_0,\eta}(a,b)$.
Consequently, for every fixed $z$, with probability at least $1-N^{-D}$,
\begin{equation}\label{eq:quantum-diffusion-fixed}
 \max_{a,b}\frac{|
 W^{-2}\tr|P_aG(z)P_b|^2
 -|m(z)|^2\pb{1-|m(z)|^2S^{(\sB)}}^{-1}_{ab}
 |}{\Omega_{c_0,\eta}(a,b)}
 \le C_{\kappa,D}(\log W)^{16}\frac{\logpara^3}{W^2}.
\end{equation}

To extend this estimate uniformly to all \(z\in\bS(\kappa)\), choose a sufficiently large fixed \(K\) and an \(N^{-K}\)-net of \(\bS(\kappa)\) with cardinality at most \(CN^{2K}\). Applying \eqref{eq:quantum-diffusion-fixed} at probability depth \(D'=D+2K+10\) and taking a union bound gives the estimate simultaneously at all net points. To pass to the entire spectral domain, we must preserve the spatial decay in the resolvent perturbation estimates. 
Fix a sufficiently small structural constant \(0<c<c_0/4\). For a matrix \(A\), define
\[
 \|A\|_{c,\eta}:=\max_a\sum_b
 e^{c\sqrt{|a-b|\sqrt \eta /\log W}}\|P_aAP_b\|.
\]
At any net point \(z'\) where \eqref{eq:quantum-diffusion-fixed} holds, by the deterministic bound \eqref{eq:tools-P-point} and the inequality \(\|P_aG(z')P_b\|\le\|P_aG(z')P_b\|_{S_2}\), we obtain, for some structural constant \(C_0>0\),
\[
 \|G(z')\|_{c,\im z'}\le N^{C_0}.
\]

For any $z\in \bS(\kappa)$, choose $z'$ in the net so that $|z-z'|\le N^{-K}$. The weighted norm is submultiplicative by \eqref{eq:product_Omega}. Thus, choosing $K>C_0$, the Neumann expansion of
\smash{\(
 G(z)= G(z')\pb{\Id-(z-z')G(z')}^{-1} 
\)}
gives that
\[
 \|G(z)-G(z')\|_{c,\im z'}
 \le 2 |z-z'|N^{2C_0} \implies
 \|G(z)\|_{c,\im z'}\le 2N^{C_0}.
\]
Consequently, Cauchy–Schwarz and $\|P_aAP_b\|_{S_2}^2\le W^2\|P_aAP_b\|^2$ yield
\[
 W^{-2}\left|\tr|P_aG(z)P_b|^2-\tr|P_aG(z')P_b|^2\right|
 \le C|z-z'|N^{3C_0}
 e^{-2c\sqrt{|a-b|\sqrt{\im z'}/\log W}}.
\]
The deterministic term satisfies an analogous bound, possibly with a larger structural exponent, by the resolvent identity for \((1-|m(z)|^2S^{(\sB)})^{-1}\), \eqref{eq:tools-P-point}, and the uniform Lipschitz continuity of \(|m(z)|^2\) on the bulk domain. Since neighboring profiles are comparable, increasing \(K\) if necessary makes the interpolation error divided by \(\Omega_{c,\eta}(a,b)\) at most \(N^{-2}\). Thus \eqref{eq:quantum-diffusion-fixed} holds uniformly in \(z\), with rate \(c\) and an additional \(N^{-2}\) error, on an event of probability at least \(1-N^{-D}\). Taking \(c_{\ref{thm:local-law}}=c\) and absorbing the fixed powers of $\log W$ into $\logpara^\e$, we conclude \eqref{eq:quantum-diffusion}.

\subsection{Proof of \Cref{end:master}}  \label{sec:pf-endpoint-product}

We first prove the case $p=1$ in \eqref{eq:compact-L2-product}. If $|r-s|\le4R_u$, combining \eqref{eq:endpoint-log-moments} with Cauchy--Schwarz, and summing in dyadic shells of the form $x\in B(r,2^jR_u)\setminus B(r,2^{j-1}R_u)$ for $j\ge1$ and $x\in B(r,R_u)$ for $j=0$, we obtain 
\[
 \|\xi_r\xi_s\|_1 \le\|\xi_r\|_2\|\xi_s\|_2 \le CR_u^2\pB{1+\sum_{j\ge0}e^{-c\sqrt{2^jR_u/(R_u+\ell_t)}}} \le CR_u^2\left[1+\log_+\frac{R_u+\ell_t}{R_u}\right] \le C R_u^2\Psi^{(1)}_{u,t}(r,s).
\]
Suppose instead that $|r-s|>4R_u$. For $x\in B(r,|r-s|/2)$, the logarithmic factor of $\xi_s$ equals one and $\xi_s\le CR_u/|r-s|$, while
\smash{\(
 \sum_{x:|x-r|\le |r-s|/2}\xi_r(x)\le CR_u|r-s|.
\)}
Thus, $\sum_{x\in B(r,|r-s|/2)}|\xi_r(x)\xi_s(x)|$ contributes at most $CR_u^2$, as does the sum over the region $B(s,|r-s|/2)$. In the complement with $|x-r|\wedge |x-s|> |r-s|/2$, we sum over the subsets $2^{j-1}|r-s|<|x-r|\wedge |x-s|\le 2^j |r-s|$ for $j\ge 0$. The contribution of these sums is bounded by 
$$CR_u^2\sum_{j\ge0} e^{-c\sqrt{2^j|r-s|/(R_u+\ell_t)}}\le CR_u^2\left[1+\log_+\frac{R_u+\ell_t}{R_u+|r-s|}\right]
 \le CR_u^2\Psi^{(1)}_{u,t}(r,s). $$
This proves \eqref{eq:compact-L2-product} for $p=1$. 

For $p>1$, we may discard the exponential factors in $\xi_r$ and $\xi_s$. Applying equation~\eqref{eq:endpoint-log-moments} and summing outside $B(r,C_2R_u)$ shows that
\[
 \|\xi_r\|_{2p}\le CR_u^{1/p},\qquad p>1.
\]
Applying H\"older's inequality therefore gives $\|\xi_r\xi_s\|_p\le CR_u^{2/p}$, which already suffices when $|r-s|\le 4R_u$. For $|r-s|>4R_u$, we use the same decomposition as above. For $1< p \le 2$, we have 
\[
 \sum_{|x-r|\le |r-s|/2}\xi_r(x)^p\xi_s(x)^p
 \le C\frac{R_u^p}{|r-s|^p} \sum_{|x-r|\le |r-s|/2}\xi_r(x)^p \le CR_u^{2p}|r-s|^{2-2p} \left(1+\1_{p=2}\cdot \log\frac{|r-s|}{R_u}\right).
\]
The summation on the complementary region is bounded by 
\[\sum_{|x-r|\wedge|x-s|> |r-s|/2}\xi_r(x)^p\xi_s(x)^p \le CR_u^{2p}\sum_{j\ge 0}(2^j|r-s|)^{2-2p} \le CR_u^{2p}|r-s|^{2-2p},\] 
which is convergent. Combining these estimates and taking the $p$th root, we get
\[
 \|\xi_r\xi_s\|_p\le CR_u^{2/p}\left(\frac{R_u}{|r-s|}\right)^{2-2/p}\left(1+\mathbf1_{\{p=2\}}\log\frac{|r-s|}{R_u}\right)^{1/p}.
\]
Absorbing the logarithm as a small power of $|r-s|/R_u$ when $p=2$ yields \eqref{eq:compact-L2-product} for $1<p\le 2$.

For the proof of \eqref{eq:near-endpoint-X1} and \eqref{eq:near-endpoint-X}, we take $r=0$ by translation invariance. 
The cases $|a-x|=0$ and $t-u=0$ are trivial; henceforth, assume $|a-x|\ge1$ and $t-u>0$ in the following proof. Note that convolution with the finite-range kernel $S^{(\sB)}$ preserves \eqref{eq:tools-P-point} and \eqref{eq:gradient-envelope}, up to an adjustment of the constants: for $|e|=1$,
\begin{equation}
|\Xi(0,x)|\le C(t-u)\qa{1+\iota_t(0,x)}e^{-c|x|/\ell_t}, \quad \text{and} \quad |\Xi(0,x+e)-\Xi(0,x)|\le \frac{C(t-u)}{1+|x|}e^{-c|x|/\ell_t} .
\label{eq:gradient-B}
\end{equation}
Let $\mathcal N=\h{x:|x|\le 2|a-x|}\cup\h{x:|a|\le2|a-x|}$. The gradient bound in \eqref{eq:gradient-B} implies
\begin{align}  
 |\Xi(0,a)-\Xi(0,x)|
 &\le C(t-u)\qa{1 +  \log_+\frac{|a-x|}{1+|x|} + C \log_+\frac{|a-x|}{1+|a|} } ,  
 \quad x\in\mathcal N,
\label{eq:gradient-r-N}\\
  |\Xi(0,a)-\Xi(0,x)|&\le C(t-u)\frac{|a-x|}{1+|x|}e^{-c|x|/\ell_t}, \quad |x|>2|a-x|.\label{eq:gradient-r-Nc}
\end{align}
To prove \eqref{eq:gradient-r-N}, first suppose $|a-x|\le L/16$. Consider the dyadic annuli centered at the origin: $A_{k}=\{y: 2^{k-1}<|y|\le 2^{k}\}$ for $k\ge 0$, and let $k_0$ be the smallest integer such that $2^{k_0}\ge4|a-x|$. Connect $a$ and $x$ to points $y,y'\in A_{k_0}$, respectively, by outward coordinate paths, and then connect $y$ to $y'$ within $A_{k_0}$. These paths can be chosen so that their total length within each annulus $A_k$ is at most $C2^k$. By \eqref{eq:gradient-B}, each annulus therefore contributes $\OO(t-u)$. The bounded number of steps near the origin also contributes $\OO(t-u)$. The paths starting at $a$ and $x$ cross at most
\[
C+C \log_+\frac{|a-x|}{1+|a|}  
\quad\text{and}\quad
C+C \log_+\frac{|a-x|}{1+|x|} 
\]
annuli, respectively. The connecting path within $A_{k_0}$ contributes another $\OO(t-u)$, which proves \eqref{eq:gradient-r-N}. If $|a-x|>L/16$, we use an outer annulus of radius comparable to $L$ and the same path construction on the torus. Since $|a-x|\asymp L$, the resulting bounds have the same form. 
Whenever $|x|>2|a-x|$, a shortest path from $x$ to $a$ has length $|a-x|$, and each vertex $v$ on this path satisfies $|v|\ge |x|-|a-x|>|x|/2$. Telescoping along this path and applying \eqref{eq:gradient-B} proves \eqref{eq:gradient-r-Nc}.

We now prove the following strengthening of \eqref{eq:near-endpoint-X1}: for every $\mathfrak{r}\ge1$ and $1\le p\le2$,
\begin{align}
 \|[\Xi(0,\cdot)-\Xi(0,x)]\mathbf1_{B(x,\mathfrak{r}R_u)}(\cdot)\|_{\ell^p} \le &~C(t-u)(\mathfrak r R_u)^{2/p}\mathfrak r(1+\log\mathfrak r) \notag\\
 &~\times \exp\left(\nu_\sharp\sqrt{\frac{2\mathfrak r R_u}{\ell_t}} +c_\xi\sqrt{\frac{2\mathfrak r R_u}{R_u+\ell_t}}\right)\xi_0(x) \Omega_t^\sharp(0,x),
 \label{eq:near-gradient-row-bound}
\end{align}
where $C$ is independent of $\mathfrak r$.
Suppose first that $|x|\le2\mathfrak r R_u$. For
$a\in B(x,\mathfrak r R_u)$, we have $|a|\le3\mathfrak r R_u$.
On $\mathcal N$, \eqref{eq:gradient-r-N} gives
\begin{equation}
 |\Xi(0,a)-\Xi(0,x)|
 \le C(t-u)\left[
 1+\log_+\frac{\mathfrak r R_u}{1+|a|}
  +\log_+\frac{\mathfrak r R_u}{1+|x|}\right].
 \label{eq:diff-near}
\end{equation}
On $\mathcal N^c$, \eqref{eq:gradient-r-Nc} bounds the difference
by $C(t-u)$, so \eqref{eq:diff-near} holds throughout the ball.
The logarithmic moments given by \eqref{eq:endpoint-log-moments} and ball counting therefore yield
\[
 \|[\Xi(0,\cdot)-\Xi(0,x)]\mathbf1_{B(x,\mathfrak r R_u)}\|_{\ell^p}
 \le C(t-u)(\mathfrak r R_u)^{2/p}
 \left(1+\log_+\frac{\mathfrak r R_u}{1+|x|}\right) .
\]
The logarithmic factor is bounded by
\(
 (1+\log\mathfrak r)
 \q{1+\log_+\p{{R_u}/\p{1+|x|}}}.
\)
Moreover, on $|x|\le2\mathfrak r R_u$, the reciprocals of the
rational and exponential factors in $\xi_0(x)$ are bounded by
$C\mathfrak r$ and \smash{$\exp(c_\xi\sqrt{2\mathfrak r R_u/(R_u+\ell_t)})$}, respectively, while \smash{\(
 \Omega_t^\sharp(0,x)^{-1}
 \le\exp(\nu_\sharp\sqrt{{2\mathfrak r R_u}/{\ell_t}}).
\)} 
These comparisons prove \eqref{eq:near-gradient-row-bound} in the
region $|x|\le2\mathfrak r R_u$.
If $|x|>2\mathfrak r R_u$, every $a\in B(x,\mathfrak r R_u)$
satisfies $|a-x|<|x|/2$. Applying \eqref{eq:gradient-r-Nc} gives
\begin{equation}
 \|[\Xi(0,\cdot)-\Xi(0,x)]\mathbf1_{B(x,\mathfrak r R_u)}\|_{\ell^p}
 \le C(t-u)(\mathfrak r R_u)^{2/p}
 \frac{\mathfrak r R_u}{1+|x|}e^{-c|x|/\ell_t}.
\label{eq:far-mathfrakr}
\end{equation}
Here the logarithmic factor in $\xi_0(x)$ equals one,
$R_u/(1+|x|)\le CR_u/(R_u+|x|)$, and
\[
 \exp\left[-c\frac{|x|}{\ell_t}
 +\nu_\sharp\sqrt{\frac{|x|}{\ell_t}}
 +c_\xi\sqrt{\frac{|x|}{R_u+\ell_t}}\right]\le C.
\]
Thus the bound \eqref{eq:far-mathfrakr} when $|x|>2\mathfrak r R_u$ is at most $C(t-u)(\mathfrak r R_u)^{2/p}\mathfrak r
\xi_0(x)\Omega_t^\sharp(0,x)$, which is bounded by the RHS of \eqref{eq:near-gradient-row-bound}. This completes its proof, with constants independent of $\mathfrak r$.

To deduce \eqref{eq:near-endpoint-X1}, taking $\mathfrak r=12C_\sharp$, and using $g_0^x(a)-\delta_{0a}=\Xi(0,a)-\Xi(0,x)$, $R_u=(\log W)\ell_u$, $(t-u)\le C\ell_u^{-2}$, and $\nu_\sharp^2{R_u}/{\ell_t}\le C_\sharp^{-1}$, we obtain \eqref{eq:near-endpoint-X1} for $r=0$. The general case follows by translation invariance.

For the proof of \eqref{eq:near-endpoint-X}, we decompose $g_0^x=\delta_0+(g_0^x-\delta_0)$. For the $(g_0^x-\delta_0)$ part, we apply \eqref{eq:near-gradient-row-bound} with $p=1$ to the unweighted sum and with $p=2$, followed by Cauchy--Schwarz and \eqref{eq:moments-iota}, to the sum weighted by $\iota_u(a,x)$. This gives
\begin{align*}
 \sum_{a\in B(x,\mathfrak r R_u)}
 |g_0^x(a)-\delta_{0a}|(1+\iota_u(a,x)) \le&~ C(t-u)
 \bigl[(\mathfrak r R_u)^2+\ell_u\mathfrak r R_u\bigr]
 \mathfrak r(1+\log\mathfrak r)\\
 &~\times
 \exp\left(\nu_\sharp\sqrt{\frac{2\mathfrak r R_u}{\ell_t}}
 +c_\xi\sqrt{\frac{2\mathfrak r R_u}{R_u+\ell_t}}\right)
 \xi_0(x)\Omega_t^\sharp(0,x).
\end{align*}
Since $(t-u)\le C\ell_u^{-2}$, $R_u=(\log W)\ell_u$, and $\mathfrak r^3(1+\log\mathfrak r)\le C\mathfrak r^4$, this is bounded by the RHS of \eqref{eq:near-endpoint-X}.
For the $\delta_0$ part, we have
\[
(1+\iota_u(0,x))\1_{|x|<\mathfrak r R_u}\le \pa{1+\log_+\frac{R_u}{1+|x|}}\1_{|x|<\mathfrak r R_u}\le C\mathfrak r
 \exp\left(\nu_\sharp\sqrt{\frac{2\mathfrak r R_u}{\ell_t}}
 +c_\xi\sqrt{\frac{2\mathfrak r R_u}{R_u+\ell_t}}\right)
 \xi_0(x)\Omega_t^\sharp(0,x),
\]
which also satisfies \eqref{eq:near-endpoint-X}. This proves \eqref{eq:near-endpoint-X} for $r=0$, and the general case follows by translation invariance.

\subsection{Proof of \Cref{eq:techical-lambda-square}}\label{subsec:pf-techical-lambda-square}
Applying the estimate \eqref{eq:near-endpoint-X1} with $p=2$, the kernel bounds \eqref{compact:kernel-bounds}, and Minkowski's inequality, we get
\begin{align}\label{eq:ell2-grx}
 \|(g_r^x-\delta_r)\1_{B_x}\|_{\ell^2}+\|\pa{\mathsf K_ug_r^x}\1_{B_x}\|_{\ell^2}
 \le C(\log W)\ell_u^{-1}\xi_r(x)\Omega_t^\sharp(r,x).
\end{align}
On the other hand, applying \eqref{eq:near-endpoint-X1} with $p=1$ and \eqref{compact:kernel-bounds} gives
\begin{align}\label{eq:ellinfty-grx}
 \|\pa{\mathsf K_ug_r^x}\1_{B_x}\|_{\ell^\infty}  \le C\ell_u^{-2}\xi_r(x)\Omega_t^\sharp(r,x).
\end{align}
Indeed, for $a\in B_x$, we have
\begin{align*}
 |(\mathsf K_ug_r^x)(a)|
 &\le \mathsf K_u(a,r)
 +CR_u^{-2}\sum_{b:|b-a|\le R_u}|g_r^x(b)-\delta_{rb}| \le CR_u^{-2}\qa{\1_{|x-r|\le (12C_\sharp+1)R_u}+(\log W)^2\xi_r(x)\Omega_t^\sharp(r,x)},
\end{align*}
where the second inequality uses \eqref{eq:near-endpoint-X1} with $p=1$ on $B(x,(12C_\sharp+1)R_u)$; the estimate holds on this fixed enlargement of $B_x$ by \eqref{eq:near-gradient-row-bound}. Since $R_u=(\log W)\ell_u$ and $\xi_r(x)\Omega_t^\sharp(r,x)\ge c>0$ for $|x-r|\le (12C_\sharp+1)R_u$, both terms are bounded by $C\ell_u^{-2}\xi_r(x)\Omega_t^\sharp(r,x)$, proving \eqref{eq:ellinfty-grx}.

Expanding \eqref{compact:marginal-kernel}, we write
\(
 \lambda_x(r,s;a)=\delta_{ra}(\mathsf K_ug_s^x)(a)+\delta_{sa}(\mathsf K_ug_r^x)(a)+b_{x}(a),
\)
where 
$$ b_{x}(a):=(g_r^x(a)-\delta_{ra})(\mathsf K_ug_s^x)(a)+(\mathsf K_ug_r^x)(a)(g_s^x(a)-\delta_{sa})-(\mathsf K_ug_r^x)(a)(\mathsf K_ug_s^x)(a).$$ 
The estimates \eqref{eq:ell2-grx} and \eqref{eq:ellinfty-grx} give
\[
 \|b_{x}\1_{B_x}\|_{\ell^2}
 \le C(\log W)\ell_u^{-3}\xi_r(x)\xi_s(x)
 \Omega_t^\sharp(r,x)\Omega_t^\sharp(s,x).
\]
Then, using \eqref{eq:compact-L2-product} with $p=2$ and $\Omega_t^\sharp(r,x)\Omega_t^\sharp(s,x)\le\Omega_t^\sharp(r,s)$, we obtain
\[
 \sum_{x,a:|a-x|\le 12C_\sharp R_u}|b_{x}(a)|^2
 \le C(\log W)^4\ell_u^{-4}\Omega_t^\sharp(r,s)^2.
\]
Applying the Cauchy--Schwarz inequality in $a$ and using the assumption $\sup_x\sum_a\|\mathbf g_{a,x}\|^2\le\mathscr E$ then yields
\begin{equation}
\sum_{x}\normbb{\sum_{a\in B_x}\mathbf g_{a, x}b_x(a)}_{\ell^2}^2 \le 
 C(\log W)^4\ell_u^{-4} {\mathscr E}\,\Omega_t^\sharp(r,s)^2.
\label{eq:regular-contribution}
\end{equation}
For the $\delta$ terms, we consider $\delta_{ra}(\mathsf K_ug_s^x)(a)$; the other term is treated symmetrically. By \eqref{eq:ellinfty-grx}, for $r\in B_x$,
\begin{equation}
 |(\mathsf K_ug_s^x)(r)|
 \le C\ell_u^{-2}\Omega_t^\sharp(r,s)
 \left[1+\log_+\frac{R_u}{1+|x-s|}\right].
\label{eq:singular-contribution0}
\end{equation}
Here we used that the nonlogarithmic factors in $\xi_s$ are at most $1$ and that $\Omega_t^\sharp(s,x)\le C\Omega_t^\sharp(r,s)$ when $r\in B_x$. To bound this contribution,
we use the following elementary estimate: if $0\le a_x\le\mathscr D$ and $\sum_xa_x\le\mathscr E$, then
\begin{equation}
 \sum_xa_x\left[1+\log_+\frac{R_u}{1+|x-y|}\right]^2
 \le C\mathscr E
 \left(1+\log_+\frac{C\mathscr DR_u^2}{\mathscr E}\right)^2.
 \label{compact:log-energy-bound}
\end{equation}
Applying this bound with $a_x=\|\mathbf g_{r,x}\|^2\1_{\{r\in B_x\}}$ and $y=s$, and using $R_u=(\log W)\ell_u$ and $\log_+(\ell_u^2\mathscr D/\mathscr E)\le C\log W$, we obtain
\begin{align*} 
 C \ell_u^{-4}\Omega_t^\sharp(r,s)^2
  \sum_{x:|x-r|\le 12C_\sharp R_u}
 \|\mathbf g_{r,x}\|^2\left[1+\log_+\frac{R_u}{1+|x-s|}\right]^2 \le  C(\log W)^2\ell_u^{-4}\mathscr{E}\Omega_t^\sharp(r,s)^2.
\end{align*}
The same estimate holds for the symmetric contribution of $\delta_{sa}(\mathsf K_ug_r^x)(a)$. Combining these bounds with \eqref{eq:regular-contribution} proves \eqref{compact:refined-physical-lift}.

To prove \eqref{compact:log-energy-bound}, we may assume $\mathscr E,\mathscr D>0$, since otherwise all $a_x$ vanish. For $s\ge0$, the superlevel set of the logarithm is
\[
 \left\{x:\log_+\frac{R_u}{1+|x-y|}>s\right\}
 =\{x:1+|x-y|<R_ue^{-s}\}.
\]
By the assumptions $a_x\le\mathscr D$ and $\sum_xa_x\le\mathscr E$, we have
\[
 \sum_{x:\,\log_+[R_u/(1+|x-y|)]>s}a_x
 \le\min\{\mathscr E,C\mathscr DR_u^2e^{-2s}\}.
\]
The layer-cake formula therefore gives
\begin{align*}
 &\sum_xa_x\left[1+\log_+\frac{R_u}{1+|x-y|}\right]^2 \le C \mathscr E  +C\int_0^\infty s\min\{\mathscr E,
 C\mathscr DR_u^2e^{-2s}\}\dd s.
\end{align*}
Splitting the integral at $s_0:=\frac12\log_+(C\mathscr DR_u^2/\mathscr E)$ bounds it by $C\mathscr E(1+s_0)^2$, proving \eqref{compact:log-energy-bound}.

\end{document}